\documentclass[12pt,letterpaper,leqno]{amsart}
\usepackage{microtype}
\usepackage{amssymb}
\usepackage{amsfonts}
\usepackage{geometry}
\usepackage{hyperref}

\newtheorem{theorem}{Theorem}
\theoremstyle{plain}

\newtheorem{corollary}[theorem]{Corollary}

\newtheorem{definition}[theorem]{Definition}
\newtheorem{example}{Example}

\newtheorem{lemma}[theorem]{Lemma}

\newtheorem{proposition}[theorem]{Proposition}
\newtheorem{remark}[theorem]{Remark}

\numberwithin{equation}{section}
\numberwithin{theorem}{section}  
\input{tcilatex}
\begin{document}
\title[3D Quintic NLS as Mean-field Limit]{The Derivation of the $\mathbb{T}%
^{3}$ Energy-critical NLS from Quantum Many-body Dynamics}
\author{Xuwen Chen}
\address{Department of Mathematics, University of Rochester, Rochester, NY
14627}
\email{chenxuwen@math.umd.edu}
\urladdr{http://www.math.rochester.edu/people/faculty/xchen84/}
\author{Justin Holmer}
\address{Department of Mathematics, Brown University, 151 Thayer Street,
Providence, RI 02912}
\email{holmer@math.brown.edu}
\urladdr{http://www.math.brown.edu/\symbol{126}holmer/}
\date{V3 for Invent. Math., 01/14/2019}

\begin{abstract}
We derive the 3D energy critical quintic NLS from quantum many-body dynamics
with $3$-body interaction in the $\mathbb{T}^{3}$ (periodic) setting. Due to
the known complexity of the energy critical setting, previous progress was
limited in comparison to the $2$-body interaction case yielding energy
subcritical cubic NLS. Previously, the only result for the 3D energy
critical case was \cite{HTX}, which proved the uniqueness part of the
argument in the case of small solutions. In the main part of this paper, we
develop methods to prove the convergence of the BBGKY hierarchy to the
infinite Gross-Pitaevskii (GP) hierarchy, and separately, the uniqueness of
large GP solutions. Since the trace estimate used in the previous proofs of
convergence is the false endpoint trace estimate in our setting, we instead
introduce a new frequency interaction analysis and apply the finite
dimensional quantum de Finetti theorem. For the large solution uniqueness
argument, we discover the new HUFL (hierarchical uniform frequency
localization) property for the GP hierarchy and use it to prove a new type
of uniqueness theorem. The HUFL property reduces to a new statement even for
NLS. With the help of \cite{CKSTT,IP} which proved the global well-posedness
for the quintic NLS, this new uniqueness theorem establishes global
uniqueness.
\end{abstract}

\maketitle
\tableofcontents

\section{Introduction\label{sec:Introduction}}

The energy-critical NLS in three-dimension has been studied in \cite%
{CKSTT,IP}, but it has been an open problem for a while to prove that the 3D
energy-critical nonlinear Schr\"{o}dinger equation (NLS)%
\begin{equation}
i\partial _{t}\phi =-\Delta \phi +\left\vert \phi \right\vert ^{4}\phi \text{
in }\mathbb{R}\times \Lambda \text{,}  \label{eqn:3d Quintic with 1 Coupling}
\end{equation}%
where $\Lambda =\mathbb{R}^{3}$ or $\mathbb{T}^{3}$, is the mean-field limit
of the $N$-body Bosonic Schr\"{o}dinger equation with three-body interaction%
\begin{equation}
i\partial _{t}\psi _{N}=H_{N}\psi _{N}\text{ in }\mathbb{R}\times \Lambda
^{N}  \label{eqn:N-body Schrodinger}
\end{equation}%
where $\psi _{N}\in L_{s}^{2}(\Lambda ^{N})$ and the $N$-body Hamiltonian is
given by%
\begin{equation}
H_{N}=\sum_{j=1}^{N}\left( -\Delta _{x_{j}}\right) +\frac{1}{N^{2}}%
\sum_{1\leqslant i<j<k\leqslant N}V_{N}(x_{i}-x_{j},x_{i}-x_{k})
\label{Hamiltonian:quintic N-body}
\end{equation}%
where $V_{N}(x,y)=N^{6\beta }V(N^{\beta }(x_{i}-x_{j}),N^{\beta
}(x_{i}-x_{k}))$ and $V\geqslant 0$.\footnote{%
When $\Lambda =\mathbb{R}^{3},$ $V_{N}$ is just a rescaling of the
three-body interaction $V$. When $\Lambda =\mathbb{T}^{3}$, $V_{N}$ should
be understood as the periodic extension of the rescaling of $V$ with $V$
being compactly supported on $\mathbb{R}^{3}\times \mathbb{R}^{3}$.} That
is, if we adopt the notation $\mathbf{x}_{k}=\left(
x_{1},x_{2},...,x_{k}\right) \in \Lambda ^{k}$ and define the $k$-particle
marginal density associated with $\psi _{N}$ by%
\begin{eqnarray}
\gamma _{N}^{(k)} &=&\int_{\Lambda ^{N-k}}\bar{\psi}_{N}(\mathbf{x}_{k},%
\mathbf{x}_{N-k})\psi _{N}(\mathbf{x}_{k}^{\prime },\mathbf{x}_{N-k})d%
\mathbf{x}_{N-k}  \label{def:marginals} \\
&=&\limfunc{Tr}\nolimits_{x_{k+1},...,x_{N}}\left\vert \psi
_{N}\right\rangle \left\langle \psi _{N}\right\vert  \notag
\end{eqnarray}%
then the problem asks to prove $\lim_{N\rightarrow \infty }\gamma
_{N}^{(k)}(t)=\left\vert \phi (t)\right\rangle \left\langle \phi
(t)\right\vert ^{\otimes k}$ if $\lim_{N\rightarrow \infty }\gamma
_{N}^{(1)}(0)=\left\vert \phi _{0}\right\rangle \left\langle \phi
_{0}\right\vert $ and $\phi $ solves (\ref{eqn:3d Quintic with 1 Coupling})
subject to initial datum $\phi _{0}$. In this paper, we answer the open
question for the $\Lambda \mathbb{=T}^{3}$ case, while the $\mathbb{R}^{3}$
case is still open. We note that for the $\mathbb{T}^{3}$

\begin{theorem}
\label{Thm:Main1Introuction}Let $\Lambda \mathbb{=T}^{3}$. Assume the
three-body interaction $V\ $is a nonnegative compactly supported smooth
function on $\mathbb{R}^{3}\times \mathbb{R}^{3}$ with the property that $%
V\left( x,y\right) =V\left( y,x\right) $. Let $\psi _{N}\left( t,\mathbf{x}%
_{N}\right) $ be the $N-body$ Hamiltonian evolution $e^{itH_{N}}\psi _{N}(0)$%
, with the Hamiltonian given by (\ref{Hamiltonian:quintic N-body}) for some $%
\beta \in \left( 0,\frac{1}{9}\right) $. Let $\left\{ \gamma
_{N}^{(k)}\right\} $ be the family of marginal densities associated with $%
\psi _{N}$. Suppose that the initial datum $\psi _{N}(0)$ satisfies the
following conditions:

(a) the initial datum is normalized, that is 
\begin{equation*}
\left\Vert \psi _{N}(0)\right\Vert _{L^{2}}=1,
\end{equation*}

(b) the initial datum is asymptotically factorized, in the sense that,%
\begin{equation}
\lim_{N\rightarrow \infty }\limfunc{Tr}\left\vert \gamma
_{N}^{(1)}(0)-\left\vert \phi _{0}\right\rangle \left\langle \phi
_{0}\right\vert \right\vert =0,  \label{eqn:asym factorized}
\end{equation}%
for some one particle wave function $\phi _{0}\in H^{1}\left( \mathbb{T}%
^{3}\right) $.

(c) initially, we have bounded energy per particle, 
\begin{equation}
\sup_{N}\frac{1}{N}\left\langle \psi _{N}(0),H_{N}\psi _{N}(0)\right\rangle
<\infty .  \label{Condition:FiniteKineticOnManyBodyInitialData}
\end{equation}%
Then $\forall t\geqslant 0$, $\forall k\geqslant 1$, we have the convergence
in the trace norm or the propagation of chaos%
\begin{equation*}
\lim_{N\rightarrow \infty }\limfunc{Tr}\left\vert \gamma
_{N}^{(k)}(t)-\left\vert \phi (t)\right\rangle \left\langle \phi
(t)\right\vert ^{\otimes k}\right\vert =0,
\end{equation*}%
where $\phi (t,x)$ solves the 3D defocusing quintic / energy-critical NLS%
\begin{eqnarray}
i\partial _{t}\phi &=&-\Delta \phi +b_{0}\left\vert \phi \right\vert
^{4}\phi \text{ in }\mathbb{R}\times \mathbb{T}^{3},
\label{equation:TargetQuinticNLS} \\
\phi (0,x) &=&\phi _{0}(x),  \notag
\end{eqnarray}%
with the coupling constant $b_{0}=\int_{\mathbb{T}^{3}\times \mathbb{T}%
^{3}}V(x,y)dxdy.$
\end{theorem}

The above theorem is well-known to be equivalent to the following theorem
using a smooth approximation argument.\footnote{%
See, for example, \cite{E-S-Y2}.}

\begin{theorem}
\label{THM:Main Theorem}Let $\Lambda \mathbb{=T}^{3}$. Assume the three-body
interaction $V\ $is a nonnegative compactly supported smooth function on $%
\mathbb{R}^{3}\times \mathbb{R}^{3}$ with the property that $V\left(
x,y\right) =V\left( y,x\right) $. Let $\psi _{N}\left( t,\mathbf{x}%
_{N}\right) $ be the $N-body$ Hamiltonian evolution $e^{itH_{N}}\psi _{N}(0)$%
, with the Hamiltonian given by (\ref{Hamiltonian:quintic N-body}) for some $%
\beta \in \left( 0,\frac{1}{9}\right) $. Let $\left\{ \gamma
_{N}^{(k)}\right\} $ be the family of marginal densities associated with $%
\psi _{N}$. Suppose that the initial datum $\psi _{N}(0)$ is normalized and
asymptotically factorized in the sense of (a) and (b) in Theorem \ref%
{Thm:Main1Introuction} and verifies the following energy condition:

(c') there is a $C>0$ independent of $N$ or $k$ such that 
\begin{equation}
\left\langle \psi _{N}(0),H_{N}^{k}\psi _{N}(0)\right\rangle \leqslant
C^{k}N^{k},\text{ }\forall k\geqslant 1,
\label{Condition:EnergyBoundOnManyBodyInitialData}
\end{equation}

Then $\forall t\geqslant 0$, $\forall k\geqslant 1$, we have the convergence
in the trace norm or the propagation of chaos%
\begin{equation*}
\lim_{N\rightarrow \infty }\limfunc{Tr}\left\vert \gamma
_{N}^{(k)}(t)-\left\vert \phi (t)\right\rangle \left\langle \phi
(t)\right\vert ^{\otimes k}\right\vert =0,
\end{equation*}%
where $\phi (t,x)$ is the solution to the 3D defocusing quintic /
energy-critical NLS (\ref{equation:TargetQuinticNLS}).
\end{theorem}

This problem arises from the study of Bose-Einstein condensate (BEC). Since
the Nobel prize winning experiments \cite{Anderson, Davis} of BEC in
interacting many-particle systems using laser cooling techniques, this new
state of matter has attracted a lot of attention in physics and mathematics.
BEC is a peculiar gaseous state in which particles of integer spin (bosons)
occupy a macroscopic quantum state. In short, BEC means that every particle
in the $N$-body system takes the same quantum state ("does the same thing").
This new state of matter can be used to explore fundamental questions in
quantum mechanics, such as the emergence of interference, decoherence,
superfluidity and quantized vortices. Investigating various condensates has
become one of the most active areas of contemporary research.

A single particle in quantum mechanics is governed by a linear one-body Schr%
\"{o}dinger equation. If $N$ bodies interact quantum mechanically, they are
governed by the $N$-body \textit{linear} Schr\"{o}dinger equation due to the
superposition principle. Thus we arrive at the $N$-particle dynamic (\ref%
{eqn:N-body Schrodinger}) for the analysis of BEC. It is self-evident that (%
\ref{eqn:N-body Schrodinger}) is impossible to solve or simulate when $N$ is
large. In fact, the largest system one could simulate at the moment only
allows $N\sim 10^{2}$ while $N\sim 10^{3}$ for very dilute Bose-Einstein
samples and $N$ is up to values of the order $10^{30}$ in Boson stars. Thus,
it is necessary to find reductions or approximations. It is long believed
that NLS like (\ref{equation:TargetQuinticNLS}) is the mean-field limit for
these $N$-body systems. Such a belief can be formally seen, through the
following heuristics.

The marginal densities $\left\{ \gamma _{N}^{(k)}\right\} $ defined via (\ref%
{def:marginals}) satisfy the 3D quintic
Bogoliubov--Born--Green--Kirkwood--Yvon (BBGKY) hierarchy:

\begin{eqnarray}
&&i\partial _{t}\gamma _{N}^{(k)}  \label{hierarchy:quintic BBGKY} \\
&=&\sum_{j=1}^{k}\left[ -\bigtriangleup _{x_{j}},\gamma _{N}^{(k)}\right] +%
\frac{1}{N^{2}}\sum_{1\leqslant i<j<l\leqslant k}\left[
V_{N}(x_{i}-x_{j},x_{i}-x_{l}),\gamma _{N}^{(k)}\right]  \notag \\
&&+\frac{(N-k)}{N^{2}}\sum_{1\leqslant i<j\leqslant k}\limfunc{Tr}%
\nolimits_{k+1}\left[ V_{N}(x_{i}-x_{j},x_{i}-x_{k+1}),\gamma _{N}^{(k+1)}%
\right]  \notag \\
&&+\frac{\left( N-k\right) (N-k-1)}{N^{2}}\sum_{j=1}^{k}\limfunc{Tr}%
\nolimits_{k+1,k+2}\left[ V_{N}(x_{j}-x_{k+1},x_{j}-x_{k+2}),\gamma
_{N}^{(k+2)}\right] .  \notag
\end{eqnarray}%
Here, $V_{N}(x,y)=N^{6\beta }V(N^{\beta }x,N^{\beta }y)$ and we do not
distinguish the kernels and the operators they define. If there is some
regular $N\rightarrow \infty $ limit $\gamma ^{(k)}$ of $\gamma _{N}^{(k)}$,
if the sizes of the seemingly error terms, $\left[
V_{N}(x_{i}-x_{j},x_{i}-x_{l}),\gamma _{N}^{(k)}\right] $ and $\left[
V_{N}(x_{i}-x_{j},x_{i}-x_{k+1}),\gamma _{N}^{(k+1)}\right] $, are
manageable so that they tend to zero when combined with the $\frac{1}{N^{2}}$
and $\frac{(N-k)}{N^{2}}$ in front of them, and if the coupling term $\left[
V_{N}(x_{j}-x_{k+1},x_{j}-x_{k+2}),\gamma _{N}^{(k+2)}\right] $ behaves
nicely, then, as $N\rightarrow \infty $, (\ref{hierarchy:quintic BBGKY})
becomes the 3D energy-critical defocusing Gross--Pitaevskii (GP) hierarchy 
\begin{equation}
i\partial _{t}\gamma ^{(k)}(t)=\sum_{j=1}^{k}\left[ -\bigtriangleup
_{x_{j}},\gamma ^{(k)}\right] +b_{0}\sum_{j=1}^{k}\limfunc{Tr}%
\nolimits_{k+1,k+2}\left[ \delta (x_{j}-x_{k+1})\delta
(x_{j}-x_{k+2}),\gamma ^{(k+2)}\right] .
\label{hierarchy:quintic GP in differential form}
\end{equation}%
When $\gamma ^{(k)}(0)=\left\vert \phi _{0}\right\rangle \left\langle \phi
_{0}\right\vert ^{\otimes k}$, one solution to (\ref{hierarchy:quintic GP in
differential form}) is $\left\vert \phi (t)\right\rangle \left\langle \phi
(t)\right\vert ^{\otimes k}$ given by (\ref{equation:TargetQuinticNLS}).
That is, if the solution to (\ref{hierarchy:quintic GP in differential form}%
) is unique, then $\gamma ^{(k)}(t)=\left\vert \phi (t)\right\rangle
\left\langle \phi (t)\right\vert ^{\otimes k}$ which is the conclusion of
Theorems \ref{Thm:Main1Introuction} and \ref{THM:Main Theorem}.

However, such a heuristics has to be studied via rigorous mathematical
proofs, not only because computers are of no use in such a large system,
also because there are experiments \cite{Cornish,JILA2} proving that such
reasonings are not true, or in other words, at least one of the four "if"s
in the above heuristics does not hold, when one considers the focusing
quantum many-body dynamics which is a current research hot spot.

It was Erd\"{o}s, Schlein, and Yau who first proved this type of mean-field
limit for the case of defocusing pair interaction and $\Lambda =\mathbb{R}%
^{3}$, in which, the mean-field equation would be a defocusing cubic NLS, in 
\cite{E-S-Y2,E-S-Y3,E-S-Y5} around 2005.\footnote{%
See also \cite{AGT} for the 1D defocusing pair interaction case around the
same time.} They first proved an a-priori $L_{T}^{\infty }H_{x}^{1}$-type
bound using condition (\ref{Condition:EnergyBoundOnManyBodyInitialData}) but
in the 2-body interaction setting. They then proved using the a-priori $%
L_{T}^{\infty }H_{x}^{1}$-type bound, which enables the application of the
Sobolev inequality%
\begin{equation}
\left\Vert \left\langle \nabla _{x_{1}}\right\rangle ^{-1}\left\langle
\nabla _{x_{2}}\right\rangle ^{-1}V_{2}(x_{1}-x_{2})\left\langle \nabla
_{x_{1}}\right\rangle ^{-1}\left\langle \nabla _{x_{2}}\right\rangle
^{-1}\right\Vert _{op}\leqslant C\left\Vert V_{2}\right\Vert _{L^{1}}
\label{estimate:OldESYSobolevInequality}
\end{equation}%
in operator form or 
\begin{equation*}
\left\Vert f(x_{1},x_{1})\right\Vert _{L_{x_{1}}^{2}(\mathbb{R}%
^{3})}\leqslant C\left\Vert \left\langle \nabla _{x_{1}}\right\rangle
\left\langle \nabla _{x_{2}}\right\rangle f\right\Vert _{L^{2}},
\end{equation*}%
in usual form, that the sequence of all $k$-particle marginal densities in
the particle number $N$ is a compact sequence with respect to a topology on
the trace class operators, and proved that every limit point must satisfy an
infinite limiting hierarchy (the defocusing cubic GP hierarchy in $\mathbb{R}%
^{3}$) to which the cubic defocusing NLS generates a solution. Finally, they
proved that there is a unique solution to the cubic GP hierarchy satisfying
the a-priori $L_{T}^{\infty }H_{x}^{1}$-type bound. Their uniqueness proof
uses sophisticated Feynman graph analysis which is closely related to
renormalization methods in quantum field theory, and was regarded as the
most involved part of their analysis.\footnote{%
When revisiting \cite{E-S-Y2,E-S-Y3,E-S-Y5} in \cite{SchleinNew}, the
authors actually wrote that proving the $L_{T}^{\infty }H_{x}^{1}$ bound was
tough.}

In 2007, inspired by their work on the wave equation \cite%
{KlainermanMachedonNullForm} and the combinatorial argument in \cite{E-S-Y2}%
, Klainerman and Machedon gave a different uniqueness theorem, with a short
analysis of PDE style proof, for the cubic GP hierarchy in $\mathbb{R}^{3}$,
which, instead of a $L_{T}^{\infty }H_{x}^{1}$-type bound, requires a
Strichartz type bound. Though the desired limit generated by the cubic NLS
easily satisfies both of the $L_{T}^{\infty }H_{x}^{1}$-type bound and the
Strichartz type bound, at that time, it was unknown how to prove that the
limits coming from the $N$-body dynamics actually satisfy the Strichartz
type bound.

In 2008, Kirkpatrick, Schlein, and Staffilani discovered that the
Klainerman-Machedon Strichartz type bound can be obtained via a simple trace
theorem in $\mathbb{R}^{2}$ and $\mathbb{T}^{2}$\ and hence greatly
simplified the argument and derived the 2D cubic NLS in \cite{Kirpatrick}.%
\footnote{%
See \cite{HerrSohinger} later on for the general tori case.} In fact,
Kirkpatrick, Schlein, and Staffilani also introduced some fine tunings to
the compactness and the convergence arguments in \cite{Kirpatrick}, like the
Fourier analysis proof of the Poincar\'{e} / approximation of identity type
lemmas. The tuned scheme became standard in the later work, for example, 
\cite{ChenAnisotropic,C-H3Dto2D,Sohinger}. However, how to check the
Klainerman-Machedon Strichartz type bound in the 3D cubic case remained
fully open at that time.

Later in 2008, T. Chen and Pavlovic initiated the study of the repelling
three-body interaction case / the quintic case in \cite{TChenAndNP} and
proved that the 1D and 2D defocusing quintic NLS arise as the mean-field
limits. They also showed in \cite{TChenAndNP} that the 2D quintic case,
which is usually considered the same as the 3d cubic case because the
corresponding NLS are both $H^{\frac{1}{2}}$-critical, does satisfy the
Klainerman-Machedon Strichartz type bound though proving it for the 3D cubic
case was still open. They also considered more general data to attack the
problem in \cite{TChenAndNpGP1,TChenAndNP1.5,TChenAndNP2}. Then in 2011, T.
Chen and Pavlovic proved that the 3D cubic Klainerman-Machedon Strichartz
type bound does hold for the defocusing $\beta <1/4$ case in \cite%
{Chen3DDerivation}. The result was quickly improved to $\beta \leqslant 2/7$
by X.C. in \cite{Chen3DDerivation}. In \cite{C-H2/3,C-H<1}, X.C. and J.H.
proved the bound up to the almost optimal case, $\beta <1$, by lifting the $%
X_{1,b}$ space techniques from NLS theory into the field.

In 2013, T. Chen, Hainzl, Pavlovic, and Seiringer, introduced the quantum de
Finetti theorem, from \cite{Lewin}, to the derivation of the time-dependent
power-type NLS and provided, in \cite{TCNPdeFinitte}, a simplified proof of
the $L_{T}^{\infty }H_{x}^{1}$-type 3D cubic uniqueness theorem in \cite%
{E-S-Y2}. The proof in \cite{TCNPdeFinitte} inspired work for more refined
uniqueness theorems in the cubic case like \cite%
{Sohinger3,HoTaXi14,C-PUniqueness} and the uniqueness problem for the cubic
case is settled away from the critical ones.

Convergence rate for the cubic defocusing cases have also been studied using
Fock space methods and even the metaplectic representation. See \cite%
{SchleinNew,GM1,GM2,Kuz,Kuz2}, and also \cite%
{Frolich,E-Y1,KnowlesAndPickl,RodnianskiAndSchlein,MichelangeliSchlein,GMM1,GMM2,Chen2ndOrder,AKS,Ammari2,Ammari1,Lewin}
for the $\beta =0$ case (Hartree dynamics\footnote{%
See also the Hartree-Fock case, for example, in \cite{BPS}.}). The general
well-posedness theory of the limiting GP hierarchy as a nonlinear PDE is
also of interest. See \cite{TCNPNT,Sohinger2,SoSt13,MNPS}.

We can see that the 3D quintic / energy-critical problem we are dealing with
in this paper is a main open problem in the defocusing setting. From the
analysis of PDE point of view, it is an interesting but at the same time
difficult problem since the energy-critical case has always been a very
delicate case in the theory of defocusing dispersive equations. See, for
example, Colliander-Keel-Staffilani-Takaoka-Tao \cite{CKSTT} and
Ionescu-Pausader \cite{IP} for the NLS case, and Grillakis \cite{ManosNLW}
for the NLW case.

As already pointed out in \cite{TChenAndNP}, in many situations,
interactions more general than pair interactions are of importance. For
instance, if the Bose gas interacts with a background field of matter (such
as phonons or photons), averaging over the latter will typically lead to a
linear combination of effective (renormalized, in the sense of quantum field
theory) $n$-particle interactions, $n=2,3,....$ For systems exhibiting
effective interactions of this general structure, it remains a key problem
to determine the mean-field dynamics. Thus there is a strong physical reason
to study the 3D quintic / energy-critical problem in the defocusing setting.
What's more interesting is that, in fact, there is also a very strong reason
in theoretic physics to analyze the 3D defocusing quintic / energy-critical
problem coming from the theory of focusing quantum many-body systems.

While the focusing case is a natural continuation of the defocusing
problems, the first full focusing result \cite{C-HFocusing}, which derives
the 1D focusing cubic NLS, did not come around until 2013. The mathematical
reason is that when the interaction is nonpositive, the a-priori $%
L_{T}^{\infty }H_{x}^{1}$-type bound, required in the compactness,
convergence, and uniqueness arguments, cannot be deduced by the standard
process due to an accumulation of constants and the technical difficulty
that one estimates a trace instead of a pure power in NLS theory. By
introducing the 2-body operator and a new argument, X.C. and J.H.
circumvented this problem in the mass-subcritical case and worked out \cite%
{C-HFocusing} and later a 3D to 1D reduction \cite{C-HFocusingII}. But the $%
L^{2}$ critical case (2D cubic), did not see any progress until \cite%
{LewinFocusing}, in which, Lewin, Nam, and Rougerie showed in the static
case that the ground state energy of the 2D $N$-body focusing Hamiltonian
right below the Gagliardo-Nirenberg threshold is still described by a NLS
ground state energy using a version the finite dimensional quantum de
Finetti theorem from \cite{OriginalDeFinette}. In \cite{C-HFocusingIII},
using the finite dimensional quantum de Finetti theorem in \cite%
{LewinFocusing} to exchange a trace with a power with an manageable error in
the analysis of the 2-body Hamiltonian, X.C. and J.H. obtained the a-priori $%
L_{T}^{\infty }H_{x}^{1}$-type bound and hence pushed the time-dependent
case to the aforementioned threshold as well. In \cite{LNR2018}, Lewin, Nam,
and Rougerie proved in the static case that, if the potential approaches the
Gagliardo-Nirenberg threshold sufficiently slowly, the ground state of the
2D $N$-body focusing Hamiltonian tends to infinity at the origin as $N$
tends to infinity. That is, at the moment, the mathematical study of the
focusing quantum many-body problem stops right before any possible blow ups
in time. In fact, the theoretical study in physics stops right at this place
as well, while, in the focusing experiments \cite{Cornish,JILA2},\ we have
already observed blow ups in time.

On \cite[p.4-5]{JILA2}, Cornell and Wiemann,\footnote{%
2001 Nobel Laureates} wrote, "At the moment all the theoretical calculation
in physics use a mean-field approach, and describe the condensate dynamics
using the NLS equation. None of the predictions in these papers match our
measurements except for the general feature that atoms are lost from the
condensate. Also, we see several phenomena that are not discussed in these
papers." That is, the 3D cubic focusing NLS is not the right model nor the
correct mean-field limit. As Einstein said, "A theory can be proved by
experiment; but no path leads from experiment to the birth of a theory." We
need a new theory. Interestingly, Cornell and Wiemann also suggested that
"the missing mechanism could be a three-body recombination." Hence, it is
natural to assume that, to have more insight of the focusing problem, one
has to first understand the case with the simplest 3D three-body interaction
which is exactly the defocusing energy-critical case we are dealing with
here.\footnote{%
Experiments are also trying to go to this direction. See, for example, \cite%
{3-body Experiment}.} After this paper, it would be interesting to consider
the two-body and three-body interactions combined which would result in a
cubic + quintic equation as the next step. For the 1D and 2D case, this has
been done by Xie \cite{Xie}.

There are certainly difficulties so that this problem remained open for so
long. For quite a while, it was widely perceived that the difficulty lies
solely on the uniqueness part because the Klainerman-Machedon argument
requires, by scaling, $\frac{3}{2}$-derivatives to work. After the
introduction of the quantum de Finetti theorem, which can deal with traces
like a superposition of powers, to the uniqueness theory, in \cite%
{TCNPdeFinitte}, this problem was then conceived as solvable via the
standardized procedure from \cite{E-S-Y2,Kirpatrick,TCNPdeFinitte}. However,
the most recent work \cite{HTX} by Hong, Taliferro, and Xie showed that such
a belief is actually not true. They could only prove a uniqueness theorem
regarding small solutions to (\ref{hierarchy:quintic GP in differential form}%
) and could not reproduce any of the compactness nor the convergence
results. This is exactly caused by the fact that, the problem we are dealing
with here is energy-critical.

We now give more details about the difficulties.\footnote{%
It is definitely interesting to analyze this energy-critical problem with
the Fock space formalism which is a newer method and should even yield a
convergence rate result without the difficulties mentioned here. However, to
apply such a method, away from the many more error terms one has to chewed
through, one would have to assume much more than $H^{1}$ regularity for the
initial datum of the $H^{1}$-critical NLS (\ref{equation:TargetQuinticNLS})
and kind of defeat the purpose of the problem. (In fact, already at least $%
H^{4}$ in the cubic setting \cite{SchleinNew,GM1,GM2,Kuz}.) Moreover, one
would need to prove "uniform in $N$" global well-posedness results for the
auxiliary middle equation $i\partial _{t}\phi =-\Delta \phi +(V_{N}\ast
\left\vert \phi \right\vert ^{4})\phi $ and compare it with (\ref%
{equation:TargetQuinticNLS}) in $H^{1}$. Considering this equation's $%
N\rightarrow \infty $ limit is (\ref{equation:TargetQuinticNLS}) and the
tremendous effort in \cite{CKSTT, IP} to prove the global well-posedness of (%
\ref{equation:TargetQuinticNLS}), this does not sound like any easier than
what we are doing here.} For the compactness and the convergence part, say
one would like to estimate a term like%
\begin{equation}
\left\vert \limfunc{Tr}J^{(k)}V_{N}(x_{i}-x_{j},x_{i}-x_{k+1})\gamma
_{N}^{(k+1)}\right\vert  \label{eqn:sample term of compactness in intro}
\end{equation}%
where $J^{(k)}$ is a test function. In the standard argument, one would
first regroup as the following%
\begin{eqnarray*}
(\ref{eqn:sample term of compactness in intro}) &\leqslant &\left\Vert
S_{i}^{-1}S_{j}^{-1}J^{(k)}S_{i}S_{j}\right\Vert _{op}\left\Vert
S_{i}^{-1}S_{j}^{-1}S_{k+1}^{-1}V_{N}(x_{i}-x_{j},x_{i}-x_{k+1})S_{i}^{-1}S_{j}^{-1}S_{k+1}^{-1}\right\Vert _{op}
\\
&&\limfunc{Tr}S_{i}S_{j}S_{k+1}\gamma _{N}^{(k+1)}S_{i}S_{j}S_{k+1},
\end{eqnarray*}%
where $S_{i}=\left\langle \nabla _{x_{i}}\right\rangle $, then use estimates
similar to (\ref{estimate:OldESYSobolevInequality}) which in this case would
be 
\begin{equation*}
\left\Vert
S_{i}^{-1}S_{j}^{-1}S_{k+1}^{-1}V_{N}(x_{i}-x_{j},x_{i}-x_{k+1})S_{i}^{-1}S_{j}^{-1}S_{k+1}^{-1}\right\Vert _{op}\leqslant C\left\Vert V_{N}\right\Vert _{L^{1}}
\end{equation*}%
or%
\begin{equation}
\left\Vert f(x_{1},x_{1},x_{1})\right\Vert _{L^{2}(\Lambda )}\leqslant
C\left\Vert \left\langle \nabla _{x_{1}}\right\rangle \left\langle \nabla
_{x_{2}}\right\rangle \left\langle \nabla _{x_{3}}\right\rangle f\right\Vert
_{L^{2}(\Lambda ^{3})}.  \label{estimate:failed sobolev 1}
\end{equation}%
Alert readers can immediately tell that (\ref{estimate:failed sobolev 1}) is
exactly the sharp endpoint Sobolev trace theorem and is certainly not true.
Due to the same reason, the usual approximation of identity / the Poincar%
\'{e} type inequality used to prove the convergence of the key coupling
term, which reads 
\begin{align}
\hspace{0.3in}& \hspace{-0.3in}\left\vert \limfunc{Tr}J^{(k)}\left( \rho
_{\alpha }\left( x_{j}-x_{k+1},x_{j}-x_{k+2}\right) -\delta \left(
x_{j}-x_{k+1}\right) \delta \left( x_{j}-x_{k+2}\right) \right) \gamma
^{(k+2)}\right\vert  \label{estimate:failed sobolev 2} \\
& \leqslant C_{\rho }\alpha ^{\kappa }C_{J}\limfunc{Tr}%
S_{j}^{2}S_{k+1}^{2}S_{k+2}^{2}\gamma ^{(k+2)}  \notag
\end{align}%
in the needed version here, also fails. Hence, neither compactness nor
convergence follow from the standard argument.

Notice that, estimating a trace, which is the case for the BBGKY / GP
hierarchies, needs an $\varepsilon $ more derivative than estimating a
power, which is the case of the NLS.

Despite how closely related the GP hierarchy and its corresponding NLS
equation are, proving the $L_{T}^{\infty }H_{x}^{1}$ uniqueness for the GP
hierarchy is more like proving the local existence which iteratively uses
the Strichartz estimates, instead of proving the unconditional uniqueness
which usually involves absorbing the nonlinear term to the left, for NLS.
Recall from \cite{Tao} the local existence theory for the NLS%
\begin{equation*}
i\partial _{t}u=-\Delta u\pm \left\vert u\right\vert ^{p-1}u\text{ in }%
\mathbb{R}\times \Lambda ,
\end{equation*}%
we would prove%
\begin{eqnarray}
&&\left\Vert \left\vert u\right\vert ^{p-1}u-\left\vert v\right\vert
^{p-1}v\right\Vert _{L_{T}^{2}W_{x}^{1,\frac{6}{5}}}
\label{estimate:NLSwell-posedness} \\
&\lesssim &T^{\frac{5-p}{10}}\left( \left\Vert u\right\Vert
_{L_{T}^{10}W_{x}^{1,\frac{30}{13}}}^{p-1}+\left\Vert v\right\Vert
_{L_{T}^{10}W_{x}^{1,\frac{30}{13}}}^{p-1}\right) \left\Vert u-v\right\Vert
_{L_{T}^{10}W_{x}^{1,\frac{30}{13}}}.  \notag
\end{eqnarray}%
That is, when we are dealing with the energy subcritical $p<5$ case there is
a small factor $T^{\frac{5-p}{10}}$ in front so that we can get a
contraction, and when we are dealing with the energy-critical $p=5$ case, we
would have to use the size of the solution to build a contraction.

Very vaguely and heuristically speaking, proving $L_{T}^{\infty }H_{x}^{1}$
uniqueness for the GP hierarchy is like, taking $\left( \left\vert
u\right\vert ^{p-1}u-\left\vert v\right\vert ^{p-1}v\right) $ to be $\gamma
^{(k)}$ (or $B_{j,k+1}\gamma ^{(k+1)}$) and $\left( u-v\right) $ as $\gamma
^{(k+1)\text{ }}$, and using the hierarchical structure to iterate (\ref%
{estimate:NLSwell-posedness}). That is, one would end up with a zero and
thus uniqueness if one has a small factor in front. When $p=5$, the factor $%
T^{\frac{5-p}{10}}$ is gone,\footnote{%
To be precise, the power of $T$ for the GP hierarchy is not $(5-p)/10$. We
put that here just for illustration purpose.} therefore, to conclude
uniqueness, one would then have to use the size of the solution. This is why
the standard argument could only reach a uniqueness regarding small
solutions.

\subsection{Outline of the Paper}

We first state and prove, Theorem \ref{Thm:energy estimate}, which derives
the a-priori estimate (\ref{estimate:key a-priori estimate}) from (\ref%
{Condition:EnergyBoundOnManyBodyInitialData}) in \S \ref{sec:energy estimate}%
. We obtain an improvement in $\beta $ over the standard method for the
defocusing case by using the focusing scheme. It turns out that the focusing
scheme is fairly straightforward here since we are dealing with a defocusing
case. Readers familiar with this matter could skip \S \ref{sec:energy
estimate}, assume the $L_{T}^{\infty }H_{x}^{1}$ a-priori bound, and start
from \S \ref{sec:Compactness}. Because \S \ref{sec:energy estimate} is the
first and necessary step, in the sense that, compactness, convergence, and
uniqueness all require (\ref{estimate:key a-priori estimate}) to hold, even
in this defocusing setting, we did not put it in the appendix.

In \S \ref{sec:Compactness}, we investigate terms like (\ref{eqn:sample term
of compactness in intro}) in detail. By looking into their defining
structure, we are able to cleverly use the conserved energy and a
subcritical Sobolev inequality to prove their boundedness and therefore
establish the compactness of the sequence $\left\{ \Gamma _{N}\right\}
_{N}=\left\{ \left\{ \gamma _{N}^{(k)}\right\} _{k=1}^{N}\right\} _{N}$ with
respect to the product topology $\tau _{prod}$ defined in \S \ref%
{sec:Compactness}. This technique in \S \ref{sec:Compactness} is new and a
minor novelty of this paper.

In \S \ref{sec:convergence}, while we bound the irrelevant error terms
generated by $\left[ V_{N}(x_{i}-x_{j},x_{i}-x_{l}),\gamma _{N}^{(k)}\right] 
$ and $\left[ V_{N}(x_{i}-x_{j},x_{i}-x_{k+1}),\gamma _{N}^{(k+1)}\right] $
with the energy method we used in \S \ref{sec:Compactness}, the main concern
is how to prove the convergence%
\begin{equation}
\limfunc{Tr}\nolimits_{k+1,k+2}\left[ V_{N}(x_{j}-x_{k+1},x_{j}-x_{k+2}),%
\gamma _{N}^{(k+2)}\right] \rightarrow \limfunc{Tr}\nolimits_{k+1,k+2}\left[
\delta (x_{j}-x_{k+1})\delta (x_{j}-x_{k+2}),\gamma ^{(k+2)}\right] ,
\label{convergence:key coupling}
\end{equation}%
without the endpoint Sobolev. We found that, one could circumvent this $%
\varepsilon $ loss with a delicate frequency interaction analysis up to $%
\beta <\frac{1}{9}$. Since the $\beta <\frac{1}{9}$ argument is very
different from the standard argument, we first present, in \S \ref%
{Sec:Convergence 1/12}, an easier argument for $\beta <\frac{1}{12}$, which
is a bit closer to the usual argument and conveys the basic idea of why the
frequency decomposition helps. There are three main terms: high frequency
approximation error, low frequency approximation error, and low frequency
main term. We notice that, when one variable is at frequency higher than $%
N^{2\beta }$, without considering interaction, we can replace the powers of $%
N$ in $V_{N}$ by derivatives on $\gamma _{N}^{(k+2)}$ and decrease the
strength of the singularity to obtain a decay. When all variables rest at a
relatively low frequency, at the price of a low frequency approximation
error, a version of the finite dimensional quantum de Finetti theorem
effectively turns the leftover problem from estimating a trace to dealing
with a power. We then prove Lemma \ref{lem:approximation of identity for
products}, the approximation of identity for the low frequency main term,
without an $\varepsilon $ loss and use some part of the usual argument to
conclude the convergence. This usage of the finite dimensional quantum de
Finetti theorem is where we require $\Lambda =\mathbb{T}^{3}$ in this paper.
After the easier-to-understand $\beta <\frac{1}{12}$ argument, we then
present the $\beta <\frac{1}{9}$ proof which uses only the frequency
analysis language in \S \ref{Sec:Convergence 1/9}. In \S \ref%
{Sec:Convergence 1/9}, we use two frequency parameters, one of them depends
on $N$ and the other one does not. We investigate the low-high interaction
inside (\ref{convergence:key coupling}), we set the $N$-dependent frequency
splitting at $N^{\beta }$ this time and further decompose the high region
into four different interaction regions. Using the fact that for frequencies
above $N^{\beta }$, there is a coupling between $x_{k+1}$ and $%
x_{k+1}^{\prime }$, and a coupling between $x_{k+2}$ and $x_{k+2}^{\prime }$
induced by $V_{N}$, we can then squeeze a gain from the trace lemma applied
to the primed-factor. For the other terms, we use the advantage of the $N$%
-independent frequency and obtain a decay. This frequency analysis in \S \ref%
{sec:convergence} is fully new and is a novelty of this paper.

In \S \ref{sec:uniqueness}, which is the main part of this paper, we
establish Theorem \ref{Thm:TotalUniqueness}, uniqueness of solutions to
hierarchy (\ref{hierarchy:quintic GP in differential form}) regardless of
the size of the solution. We are able to prove Theorem \ref%
{Thm:TotalUniqueness} because we have discovered the fully new \emph{%
hierarchically uniform frequency localization} (HUFL) property, Definition %
\ref{def:HUFL}, for solutions to the GP hierarchy and hence for solutions to
NLS as well. We break the proof of Theorem \ref{Thm:TotalUniqueness} into
two main parts, Theorems \ref{Thm:UTFLinTime} and \ref%
{THM:UniquessAssumingUFL}.

We first prove that all admissible $L_{T}^{\infty }H_{x}^{1}$ solutions to
hierarchy (\ref{hierarchy:quintic GP in differential form}) satisfy HUFL
uniformly for a small time independent of $k$ as Theorem \ref{Thm:UTFLinTime}
in \S \ref{subsec:GPUTFL}. At a glance, the hierarchical structure does
nothing but produces numerous "Deja vu" terms from the NLS energy. However,
if one treats them like in the NLS case in which there is no hierarchical
structure, one ends up with a $k$-dependent time period which has to be
avoided. Thus the proof of Theorem \ref{Thm:UTFLinTime} is structured in a
special way to avoid any $k$-dependent time emerging from the final
estimate. Via a delicate computation in which we need to estimate 42 terms,
we prove a growth estimate (\ref{E:main-1}) for the high frequency part of
the energy of the hiearchy (\ref{hierarchy:quintic GP in differential form})
in \S \ref{Sec:Proof of Main High Energy Estimate}. With (\ref{E:main-1})
and some refined Sobolev inequalities in Lemma \ref{lemma:refined sobolev
for continuity} which allow a gain through the intermediate kinetic energy,
we argue by contradiction that the high kinetic energy must stay small in \S %
\ref{Sec:Proof of High Kinetic Eats High Energy} if we assume the continuity
of the intermediate kinetic energy, which implies a strong continuity result
of the de Finetti measure. Then in \S \ref{Sec:Proof of High Kinetic
Continuity}, we prove the continuity of the intermediate kinetic energy. \S %
\ref{Sec:Proof of Main High Energy Estimate}-\S \ref{Sec:Proof of High
Kinetic Continuity} together prove Theorem \ref{Thm:UTFLinTime}.

After we have proved Theorem \ref{Thm:UTFLinTime}, we examine the couplings
in the Duhamel-Born expansion of $\gamma ^{(1)}$ carefully in \S \ref%
{subsec:GPUTFLUniqueness}. We find out that we can classify the couplings as
unclogged couplings and congested couplings as defined in Definition \ref%
{def:type of coupling}. In the unclogged couplings, we explore uniform in
time HUFL and manage to gain a small factor via uniform in time HUFL even in
this energy-critical case, by applying the refined multilinear estimates in
Lemma \ref{Lem:MultilinearWithFreqLocal}. In the congested couplings, due to
its structure which consists of a nonlinear term solely made of nonlinear
terms, one cannot gain anything via HUFL. We then show that, at least $4/5$
of the total couplings has to be unclogged couplings. This allows us to
prove uniqueness for a small time assuming HUFL initial datum, which is
Theorem \ref{THM:UniquessAssumingUFL}.

We then prove Lemma \ref{Lem:MultilinearWithFreqLocal} in \S \ref%
{sec:UniquenessMultilinearEstimate} with a meticulous frequency interaction
argument. The proof of Lemma \ref{Lem:MultilinearWithFreqLocal} happens to
be highly technical, and in fact, relies heavily on the scale invariant $%
\mathbb{T}^{3}$ Strichartz estimates and bilinear Strichartz estimates
recently proved in Bougain-Demeter \cite{BD}, Killip-Visan \cite{KV}, and
Zhang \cite{Z}. This is due to the fact that the $\mathbb{T}^{3}$ Strichartz
estimates have loss and we are dealing with an "endpoint" problem in which
we could not afford any loss.

With Theorems \ref{Thm:UTFLinTime} and \ref{THM:UniquessAssumingUFL}, we
prove by contradiction that uniqueness must persist for all time, which is
Theorem \ref{Thm:TotalUniqueness} and establishes uniqueness for the
energy-critical GP hierarchy with critical regularity regardless of the size
of the solution.\footnote{%
We do not know if Theorem \ref{Thm:TotalUniqueness} should be classified as
a conditional uniqueness theorem or an unconditional uniqueness theorem. See
Remark \ref{Remark:QtoGigliola} for detail.} \S \ref{sec:uniqueness} records
this new discovery and is the main novelty of this paper.

Through \S \ref{sec:energy estimate} to \S \ref{sec:uniqueness}, we have
concluded that, under the setting of Theorem \ref{THM:Main Theorem}, $\gamma
_{N}^{(k)}(t)\rightarrow \left\vert \phi (t)\right\rangle \left\langle \phi
(t)\right\vert ^{\otimes k}$ weak* as trace class operators. The standard
argument then proves $\gamma _{N}^{(k)}(t)\rightarrow \left\vert \phi
(t)\right\rangle \left\langle \phi (t)\right\vert ^{\otimes k}$ strongly as
trace class operators, hence deduces Theorem \ref{THM:Main Theorem}. Again,
by the standard argument, one can prove Theorem \ref{Thm:Main1Introuction}
via Theorem \ref{THM:Main Theorem}. We neglect the details here because it
is not new.\footnote{%
As usual, proof of Theorems \ref{Thm:Main1Introuction} and \ref{THM:Main
Theorem} work for initial datum more general than (\ref{eqn:asym factorized}%
). We are not stating the main theorems that way because such a fact is
well-known by now.}

For completeness, we include an appendix on how HUFL works for NLS (\ref%
{equation:TargetQuinticNLS}).

\bigskip

\noindent\textbf{Acknowledgements.}  X.C. was supported in part by NSF grant
DMS-1464869.  J.H. was partially supported by the NSF grant DMS-1500106 (until August 2017). Since September 2017, J.H. has been serving as a Program Director in the Division of Mathematical Sciences at the National Science Foundation (NSF), USA, and as a component of this position, J.H. received support from NSF for research, which included work on this paper. Any opinions, findings, and conclusions or recommendations expressed in this material are those of the authors and do not necessarily reflect the views of the National Science Foundation. 

X.C. would like to thank Thomas Chen and Natasa Pavlovic for their heartwarming hospitality and the delightful discussions during his visit to Austin. Moreover, X.C. and J.H. would like to thank the referees for their careful and detailed reading and checking of the paper and their many insightful comments and helpful suggestions which have made the paper better.

\section{Energy Estimates\label{sec:energy estimate}}

\begin{theorem}
\label{Thm:energy estimate}Assume $\beta <\frac{1}{8}.$ $\forall c_{1}\in %
\left[ 0,1\right) $, $\forall k\geqslant 1$, there exists an $%
N_{0}(c_{1},k)>0$ such that 
\begin{equation}
\left\langle \psi _{N},\left( N^{-1}H_{N}+1\right) ^{k}\psi
_{N}\right\rangle \geqslant c_{1}^{k}\left( \left\Vert S^{(1,k)}\psi
_{N}\right\Vert _{L_{x}^{2}}^{2}+\frac{1}{N}\left\Vert S_{1}S^{(1,k-1)}\psi
_{N}\right\Vert _{L_{x}^{2}}^{2}\right) \text{, }
\label{energy etimate: stability of matter}
\end{equation}%
for all $N>N_{0}$ and for all $\psi _{N}\in L_{s}^{2}(\Lambda ^{N})$. Here, $%
S_{j}=\left\langle \nabla _{x_{j}}\right\rangle $, 
\begin{equation*}
S^{(\alpha ,k)}=\dprod\limits_{j=1}^{k}\left\langle \nabla
_{x_{j}}\right\rangle ^{\alpha }
\end{equation*}%
and $N_{0}$ grows to infinity as $c_{1}$ approaches $1$.

In particular, under the assumptions of Theorem \ref{THM:Main Theorem}, (\ref%
{Condition:EnergyBoundOnManyBodyInitialData}) yields%
\begin{equation}
\sup_{t\in \left[ 0,T\right] }S^{(1,k)}\gamma _{N}^{(k)}S^{(1,k)}\leqslant
C^{k}\text{ for all large enough }N  \label{estimate:key a-priori estimate}
\end{equation}
\end{theorem}

We prove Theorem \ref{Thm:energy estimate} with the scheme in the focusing
literature \cite{C-HFocusing,C-HFocusingII,C-HFocusingIII} though it is a
defocusing problem. While this proof is longer and more complicated than 
\cite[Proposition 2.1]{TChenAndNP}, it yields $\beta <\frac{1}{8}$. If one
follows \cite[Proposition 2.1]{TChenAndNP} directly, one obtains $\beta <%
\frac{1}{12}$ instead. As we go through the proof, we see that the reason of
the improvement is that, the focusing scheme, though much more complicated,
groups terms finer.

Recall the $N$-body Hamiltonian (\ref{Hamiltonian:quintic N-body}) 
\begin{equation}
H_{N}+N=\sum_{j=1}^{N}S_{j}^{2}+\frac{1}{N^{2}}\sum_{1\leqslant
i<j<k\leqslant N}V_{N}(x_{i}-x_{j},x_{i}-x_{k}).  \label{hamiltonian:H_N+N}
\end{equation}%
Since $V\geqslant 0$, it is trivial to conclude the $k=1$ case%
\begin{equation}
\left\langle \psi _{N},\left( \frac{H_{N}}{N}+1\right) \psi
_{N}\right\rangle \geqslant c_{1}\left\Vert S_{1}\psi _{N}\right\Vert
_{L_{x}^{2}}^{2},  \label{eqn:nergy when n=1}
\end{equation}%
regardless of the parameter $\beta $, away from the even more trivial $k=0$
case.

To prove Theorem \ref{Thm:energy estimate} for general $k$, we prove the $%
k=n+2$ case 
\begin{equation*}
\left\langle \psi _{N},\left( \frac{H_{N}}{N}+1\right) ^{n+2}\psi
_{N}\right\rangle \geqslant c_{1}^{n+2}\left( \left\Vert S^{(1,n+2)}\psi
_{N}\right\Vert _{L_{x}^{2}}^{2}+\frac{1}{N}\left\Vert S_{1}S^{(1,n+1)}\psi
_{N}\right\Vert _{L_{x}^{2}}^{2}\right) ,
\end{equation*}%
provided that the weaker$\ $version of the $k=n$ case%
\begin{equation}
\left\langle \psi _{N},\left( \frac{H_{N}}{N}+1\right) ^{n}\psi
_{N}\right\rangle \geqslant c_{1}^{n}\left\Vert S^{(1,n)}\psi
_{N}\right\Vert _{L_{x}^{2}}^{2}  \label{eqn:induction hypothesis in energy}
\end{equation}%
holds. We will use the following lemma.

\begin{lemma}[{\protect\cite[Lemma A.2]{C-HFocusing}}]
\quad \label{LemmaInEnergyEstimate:commuteop}If $A_{1}\geq A_{2}\geq 0$, $%
B_{1}\geq B_{2}\geq 0$ and $A_{i}B_{j}=B_{j}A_{i}$ for all $1\leq i,j\leq 2$%
, then $A_{1}B_{1}\geq A_{2}B_{2}$.
\end{lemma}

By (\ref{eqn:induction hypothesis in energy}), we have%
\begin{equation*}
\left\langle \psi _{N},\left( \frac{H_{N}}{N}+1\right) ^{n+2}\psi
_{N}\right\rangle \geqslant c_{1}^{n}\left\langle S^{(1,n)}\left( \frac{H_{N}%
}{N}+1\right) \psi _{N},S^{(1,n)}\left( \frac{H_{N}}{N}+1\right) \psi
_{N}\right\rangle .
\end{equation*}%
To continue estimating, we decompose the $N$-body Hamiltonian into its $3$%
-body summands. We define the $3$-body Hamiltonian as 
\begin{eqnarray*}
H_{i,j,k} &=&S_{i}^{2}+S_{j}^{2}+S_{k}^{2} \\
&&+\frac{(N-1)(N-2)}{N^{2}}V_{N}(x_{i}-x_{j},x_{i}-x_{k}) \\
&=&S_{i}^{2}+S_{j}^{2}+S_{k}^{2}+V_{N,i,j,k},
\end{eqnarray*}%
then we can decompose (\ref{hamiltonian:H_N+N}) with $H_{i,j,k}$ into%
\begin{eqnarray*}
\frac{H_{N}}{N}+1 &=&\frac{1}{N(N-1)(N-2)}\sum_{1\leqslant i<j<k\leqslant
N}H_{i,j,k} \\
&=&\frac{1}{N(N-1)(N-2)}\left( \sum_{\substack{ 1\leqslant i<j<k\leqslant N 
\\ i\leqslant n}}H_{i,j,k}+\sum_{\substack{ 1\leqslant i<j<k\leqslant N  \\ %
i>n}}H_{i,j,k}\right) .
\end{eqnarray*}%
We then expand 
\begin{equation*}
\left\langle S^{(1,n)}\psi _{N},\left( \frac{H_{N}}{N}+1\right)
^{2}S^{(1,n)}\psi _{N}\right\rangle \equiv K+E+P
\end{equation*}%
with%
\begin{equation*}
K\equiv \frac{\left\langle S^{(1,n)}\left( \sum_{\substack{ 1\leqslant
i_{1}<j_{1}<k_{1}\leqslant N  \\ i_{1}>n}}H_{i_{1},j_{1},k_{1}}\right) \psi
_{N},S^{(1,n)}\left( \sum_{\substack{ 1\leqslant i_{2}<j_{2}<k_{2}\leqslant
N  \\ i_{2}>n}}H_{i_{2},j_{2},k_{2}}\right) \psi _{N}\right\rangle }{%
N^{2}(N-1)^{2}(N-2)^{2}}\geqslant 0,
\end{equation*}

\begin{equation*}
E\equiv \frac{2\func{Re}\left\langle S^{(1,n)}\left( \sum_{\substack{ %
1\leqslant i_{1}<j_{1}<k_{1}\leqslant N  \\ i_{1}\leqslant n}}%
H_{i_{1},j_{1},k_{1}}\right) \psi _{N},S^{(1,n)}\left( \sum_{\substack{ %
1\leqslant i_{2}<j_{2}<k_{2}\leqslant N  \\ i_{2}>n}}H_{i_{2},j_{2},k_{2}}%
\right) \psi _{N}\right\rangle }{N^{2}(N-1)^{2}(N-2)^{2}},
\end{equation*}%
\begin{equation*}
P\equiv \frac{\left\langle S^{(1,n)}\left( \sum_{\substack{ 1\leqslant
i_{1}<j_{1}<k_{1}\leqslant N  \\ i_{1}\leqslant n}}H_{i_{1},j_{1},k_{1}}%
\right) \psi _{N},S^{(1,n)}\left( \sum_{\substack{ 1\leqslant
i_{2}<j_{2}<k_{2}\leqslant N  \\ i_{2}\leqslant n}}H_{i_{2},j_{2},k_{2}}%
\right) \psi _{N}\right\rangle }{N^{2}(N-1)^{2}(N-2)^{2}}\geqslant 0.
\end{equation*}%
On the one hand, $P\geqslant 0$. On the other hand, there are only $\sim
n^{2}N^{2}$ terms inside $P$. That is, $P$ will not be our main term. Hence,
we discard $P$.

\subsubsection{Estimate for $K$}

We decompose $K$ into four terms 
\begin{equation*}
K\equiv K_{1}+K_{2}+K_{3}+K_{4}\text{,}
\end{equation*}%
where $K_{1}$ consists of the terms in which $\left\vert \left\{
i_{1},j_{1},k_{1}\right\} \cap \left\{ i_{2},j_{2},k_{2}\right\} \right\vert
=0$, $K_{2}$ consists of the terms in which $\left\vert \left\{
i_{1},j_{1},k_{1}\right\} \cap \left\{ i_{2},j_{2},k_{2}\right\} \right\vert
=1$, $K_{3}$ consists of the terms in which $\left\vert \left\{
i_{1},j_{1},k_{1}\right\} \cap \left\{ i_{2},j_{2},k_{2}\right\} \right\vert
=2$, $K_{4}$ consists of the terms $\left\vert \left\{
i_{1},j_{1},k_{1}\right\} \cap \left\{ i_{2},j_{2},k_{2}\right\} \right\vert
=3$. Since $i_{1},i_{2}>n$, we can commute $H_{i,j,k}$ with $S^{(1,n)}$.
Hence, by symmetry, we have%
\begin{equation*}
K_{1}=a_{1,N}\left\langle S^{(1,n)}\psi
_{N},H_{n+1,n+2,n+3}H_{n+4,n+5,n+6}S^{(1,n)}\psi _{N}\right\rangle ,
\end{equation*}%
\begin{equation*}
K_{2}=b_{1,N}\left\langle S^{(1,n)}\psi
_{N},H_{n+1,n+2,n+3}H_{n+3,n+4,n+5}S^{(1,n)}\psi _{N}\right\rangle ,
\end{equation*}%
\begin{equation*}
K_{3}=c_{1,N}\left\langle S^{(1,n)}\psi
_{N},H_{n+1,n+2,n+3}H_{n+2,n+3,n+4}S^{(1,n)}\psi _{N}\right\rangle ,
\end{equation*}%
\begin{equation*}
K_{4}=d_{1,N}\left\langle S^{(1,n)}\psi
_{N},H_{n+1,n+2,n+3}H_{n+1,n+2,n+3}S^{(1,n)}\psi _{N}\right\rangle ,
\end{equation*}%
with 
\begin{equation*}
a_{1,N}=\frac{\left( N-n\right) \left( N-n-1\right) \left( N-n-2\right)
\left( N-n-3\right) \left( N-n-4\right) \left( N-n-5\right) }{%
N^{2}(N-1)^{2}(N-2)^{2}}\sim 1,
\end{equation*}%
\begin{equation*}
b_{1,N}=\frac{\left( N-n\right) \left( N-n-1\right) \left( N-n-2\right)
\left( N-n-3\right) \left( N-n-4\right) }{N^{2}(N-1)^{2}(N-2)^{2}}\sim \frac{%
1}{N},
\end{equation*}%
\begin{equation*}
c_{1,N}=\frac{\left( N-n\right) \left( N-n-1\right) \left( N-n-2\right)
\left( N-n-3\right) }{N^{2}(N-1)^{2}(N-2)^{2}}\sim \frac{1}{N^{2}},
\end{equation*}%
\begin{equation*}
d_{1,N}=\frac{\left( N-n\right) \left( N-n-1\right) \left( N-n-2\right) }{%
N^{2}(N-1)^{2}(N-2)^{2}}\sim \frac{1}{N^{3}}.
\end{equation*}%
where the constants $C$ inside $b_{1,N}$, $c_{1,N}$, $d_{1,N}$ are some
combinatorics number between $1$ and $6!$ and does not matter at all.

Since $\left[ H_{i_{1},j_{1},k_{1}},H_{i_{2},j_{2},k_{2}}\right] =0$ and $%
\left[ S_{i_{1}}^{2}+S_{j_{1}}^{2}+S_{k_{1}}^{2},H_{i_{2},j_{2},k_{2}}\right]
=0$ whenever $\left\vert \left\{ i_{1},j_{1},k_{1}\right\} \cap \left\{
i_{2},j_{2},k_{2}\right\} \right\vert =0$ and $H_{i,j,k}\geqslant
S_{i}^{2}+S_{j}^{2}+S_{k}^{2}$ trivially, by Lemma \ref%
{LemmaInEnergyEstimate:commuteop}, we know 
\begin{eqnarray}
K_{1} &\geqslant &\left\langle S^{(1,n)}\psi
_{N},(S_{n+1}^{2}+S_{n+2}^{2}+S_{n+3}^{2})(S_{n+4}^{2}+S_{n+5}^{2}+S_{n+6}^{2})S^{(1,n)}\psi _{N}\right\rangle
\label{energy estimate:K1} \\
&\geqslant &9\left\langle S^{(1,n+2)}\psi _{N},S^{(1,n+2)}\psi
_{N}\right\rangle  \notag \\
&\geqslant &\left\Vert S^{(1,n+2)}\psi _{N}\right\Vert _{L_{\mathbf{x}%
}^{2}}^{2}.  \notag
\end{eqnarray}

For $K_{2},$ we have%
\begin{eqnarray*}
K_{2} &=&\frac{1}{N}\left\langle S^{(1,n)}\psi
_{N},H_{n+1,n+2,n+3}H_{n+3,n+4,n+5}S^{(1,n)}\psi _{N}\right\rangle \\
&=&\frac{1}{N}\left\langle S^{(1,n)}\psi _{N},\left(
S_{n+1}^{2}+S_{n+2}^{2}+S_{n+3}^{2}\right) \left(
S_{n+3}^{2}+S_{n+4}^{2}+S_{n+5}^{2}\right) S^{(1,n)}\psi _{N}\right\rangle \\
&&+\frac{1}{N}\left\langle S^{(1,n)}\psi _{N},\left(
S_{n+1}^{2}+S_{n+2}^{2}+S_{n+3}^{2}\right) V_{N,n+3,n+4,n+5}S^{(1,n)}\psi
_{N}\right\rangle \\
&&+\frac{1}{N}\left\langle S^{(1,n)}\psi _{N},V_{N,n+1,n+2,n+3}\left(
S_{n+3}^{2}+S_{n+4}^{2}+S_{n+5}^{2}\right) S^{(1,n)}\psi _{N}\right\rangle \\
&&+\frac{1}{N}\left\langle S^{(1,n)}\psi
_{N},V_{N,n+1,n+2,n+3}V_{N,n+3,n+4,n+5}S^{(1,n)}\psi _{N}\right\rangle
\end{eqnarray*}%
Dropping the nonnegative 4th term and combining the 2nd and the 3rd gives%
\begin{eqnarray*}
K_{2} &\geqslant &\frac{1}{N}\left\langle S_{1}S^{(1,n+1)}\psi
_{N},S_{1}S^{(1,n+1)}\psi _{N}\right\rangle \\
&&-\frac{2}{N}\left\vert \func{Re}\left\langle S_{n+3}S^{(1,n)}\psi _{N},%
\left[ S_{n+3},V_{N,n+3,n+4,n+5}\right] S^{(1,n)}\psi _{N}\right\rangle
\right\vert
\end{eqnarray*}%
For the 2nd term above, 
\begin{eqnarray}
&&\left\vert \left\langle S_{n+3}S^{(1,n)}\psi _{N},\left[
S_{n+3},V_{N,n+3,n+4,n+5}\right] S^{(1,n)}\psi _{N}\right\rangle \right\vert
\label{energy estimate:an error term in K} \\
&=&N^{\beta }\left\vert \left\langle S_{n+3}S^{(1,n)}\psi
_{N},V_{N,n+3,n+4,n+5}^{\prime }S^{(1,n)}\psi _{N}\right\rangle \right\vert ,
\notag
\end{eqnarray}%
we first apply Cauchy-Schwarz%
\begin{eqnarray*}
&\leqslant &\varepsilon N^{\beta }\left\vert \left\langle
S_{n+3}S^{(1,n)}\psi _{N},\left\vert V^{\prime }\right\vert
_{N,n+3,n+4,n+5}S_{n+3}S^{(1,n)}\psi _{N}\right\rangle \right\vert \\
&&+\frac{N^{\beta }}{\varepsilon }\left\vert \left\langle S^{(1,n)}\psi
_{N},\left\vert V^{\prime }\right\vert _{N,n+3,n+4,n+5}S^{(1,n)}\psi
_{N}\right\rangle \right\vert ,
\end{eqnarray*}%
then use H\"{o}lder's inequality and Sobolev's inequality%
\begin{eqnarray*}
&\leqslant &\varepsilon N^{\beta }\left\Vert V_{N,n+3,n+4,n+5}^{\prime
}\right\Vert _{L_{x_{n+4}}^{\infty }L_{x_{n+5}}^{\frac{3}{2}}}\left\Vert
S_{n+3}S^{(1,n)}\psi _{N}\right\Vert _{L_{\mathbf{x}%
}^{2}L_{x_{_{n+5}}}^{6}}^{2} \\
&&\frac{N^{\beta }}{\varepsilon }\left\Vert V_{N,n+3,n+4,n+5}^{\prime
}\right\Vert _{L_{x_{n+4}}^{\frac{3}{2}}L_{x_{n+5}}^{\frac{3}{2}}}\left\Vert
S^{(1,n)}\psi _{N}\right\Vert _{L_{\mathbf{x}%
}^{2}L_{x_{_{n+4}}}^{6}L_{x_{_{n+5}}}^{6}}^{2} \\
&\leqslant &C\left( \varepsilon N^{5\beta }+\frac{N^{3\beta }}{\varepsilon }%
\right) \left\Vert S^{(1,n+2)}\psi _{N}\right\Vert _{L_{\mathbf{x}}^{2}}^{2}.
\end{eqnarray*}%
Selecting $\varepsilon =\frac{1}{N^{\beta }}$ yields%
\begin{eqnarray*}
&&\left\vert \left\langle S^{(1,n)}\psi _{N},\left[ S_{n+3},V_{N,n+3,n+4,n+5}%
\right] S^{(1,n)}\psi _{N}\right\rangle \right\vert \\
&\leqslant &CN^{4\beta }\left\Vert S^{(1,n+2)}\psi _{N}\right\Vert _{L_{%
\mathbf{x}}^{2}}^{2}.
\end{eqnarray*}%
That is,%
\begin{equation}
K_{2}\geqslant \frac{1}{N}\left\Vert S_{1}S^{(1,n+1)}\psi _{N}\right\Vert
_{L_{\mathbf{x}}^{2}}^{2}-\frac{C}{N^{1-4\beta }}\left\Vert S^{(1,n+2)}\psi
_{N}\right\Vert _{L_{\mathbf{x}}^{2}}^{2}.  \label{energy estimate:K2}
\end{equation}

For $K_{3}$, we write out%
\begin{eqnarray*}
K_{3} &=&\frac{1}{N^{2}}\left\langle S^{(1,n)}\psi
_{N},H_{n+1,n+2,n+3}H_{n+2,n+3,n+4}S^{(1,n)}\psi _{N}\right\rangle \\
&=&\frac{1}{N^{2}}\left\langle S^{(1,n)}\psi _{N},\left(
S_{n+1}^{2}+S_{n+2}^{2}+S_{n+3}^{2}\right) \left(
S_{n+2}^{2}+S_{n+3}^{2}+S_{n+4}^{2}\right) S^{(1,n)}\psi _{N}\right\rangle \\
&&+\frac{1}{N^{2}}\left\langle S^{(1,n)}\psi _{N},\left(
S_{n+1}^{2}+S_{n+2}^{2}+S_{n+3}^{2}\right) V_{N,n+2,n+3,n+4}S^{(1,n)}\psi
_{N}\right\rangle \\
&&+\frac{1}{N^{2}}\left\langle S^{(1,n)}\psi _{N},V_{N,n+1,n+2,n+3}\left(
S_{n+2}^{2}+S_{n+3}^{2}+S_{n+4}^{2}\right) S^{(1,n)}\psi _{N}\right\rangle \\
&&+\frac{1}{N^{2}}\left\langle S^{(1,n)}\psi
_{N},V_{N,n+1,n+2,n+3}V_{N,n+2,n+3,n+4}S^{(1,n)}\psi _{N}\right\rangle .
\end{eqnarray*}%
Discarding the nonnegative terms, we have%
\begin{equation}
K_{3}\geqslant -\frac{4}{N^{2}}\left\vert \func{Re}\left\langle
S_{n+2}S^{(1,n)}\psi _{N},\left[ S_{n+2},V_{N,n+2,n+3,n+4}\right]
S^{(1,n)}\psi _{N}\right\rangle \right\vert  \label{energy estimate:PreK3}
\end{equation}%
which can be treated like $\left( \ref{energy estimate:an error term in K}%
\right) $. Hence, 
\begin{equation}
K_{3}\geqslant -\frac{C}{N^{2-4\beta }}\left\Vert S^{(1,n+2)}\psi
_{N}\right\Vert _{L_{\mathbf{x}}^{2}}^{2}.  \label{energy estimate:K3}
\end{equation}

Noticing that $K_{4}\geqslant 0$, putting together (\ref{energy estimate:K1}%
), (\ref{energy estimate:K2}), and (\ref{energy estimate:K3}) yields%
\begin{equation}
K\geqslant (1-\frac{C}{N^{1-4\beta }})\left( \left\Vert S^{(1,n+2)}\psi
_{N}\right\Vert _{L_{\mathbf{x}}^{2}}^{2}+\frac{1}{N}\left\Vert
S_{1}S^{(1,n+1)}\psi _{N}\right\Vert _{L_{\mathbf{x}}^{2}}^{2}\right) .
\label{energy estimate:K}
\end{equation}

\subsubsection{Estimate of $E$}

We group the summands inside $E$ into three groups%
\begin{equation}
E\equiv E_{1}+E_{2}+E_{3}.  \label{energu estimate:decomposition of error}
\end{equation}%
by the size of $\left\vert \{j_{1},k_{1}\}\cap \left\{
i_{2},j_{2},k_{2}\right\} \right\vert $.

$E_{1}$ consists of the terms in which $\left\vert \{j_{1},k_{1}\}\cap
\left\{ i_{2},j_{2},k_{2}\right\} \right\vert =0,$ and there are $\sim N^{5}$
such terms if $j_{1}>n$, $\sim N^{4}$ such terms if $\left( j_{1}\leqslant
n\right) \&\&(k_{1}>n)$, and $\sim N^{3}$ such terms if $\left(
j_{1}\leqslant n\right) \&\&(k_{1}\leqslant n)$. Thus, by symmetry,%
\begin{eqnarray*}
E_{1} &=&E_{11}+E_{12}+E_{13} \\
&=&\frac{1}{N^{3}}\left\langle S^{(1,n)}H_{1,2,3}\psi
_{N},H_{n+1,n+2,n+3}S^{(1,n)}\psi _{N}\right\rangle \\
&&+\frac{1}{N^{2}}\left\langle S^{(1,n)}H_{1,2,n+1}\psi
_{N},H_{n+2,n+3,n+4}S^{(1,n)}\psi _{N}\right\rangle \\
&&+\frac{1}{N}\left\langle S^{(1,n)}H_{1,n+1,n+2}\psi
_{N},H_{n+3,n+4,n+5}S^{(1,n)}\psi _{N}\right\rangle .
\end{eqnarray*}

$E_{2}$ consists of the terms in which $\left\vert \{j_{1},k_{1}\}\cap
\left\{ i_{2},j_{2},k_{2}\right\} \right\vert =1,$ and there are $\sim N^{4}$
such terms if $j_{1}>n$, and $\sim N^{3}$ such terms if $j_{1}\leqslant n$.
Thus, by symmetry,%
\begin{eqnarray*}
E_{2} &=&E_{21}+E_{22} \\
&=&\frac{1}{N^{2}}\left\langle S^{(1,n)}H_{1,n+1,n+2}\psi
_{N},H_{n+2,n+3,n+4}S^{(1,n)}\psi _{N}\right\rangle \\
&&+\frac{1}{N^{3}}\left\langle S^{(1,n)}H_{1,2,n+1}\psi
_{N},H_{n+1,n+2,n+3}S^{(1,n)}\psi _{N}\right\rangle .
\end{eqnarray*}

$E_{3}$ consists of the terms in which $\left\vert \{j_{1},k_{1}\}\cap
\left\{ i_{2},j_{2},k_{2}\right\} \right\vert =2,$ and there are $\sim N^{3}$
such terms. Thus, by symmetry,%
\begin{equation*}
E_{3}=\frac{1}{N^{3}}\left\langle S^{(1,n)}H_{1,n+1,n+2}\psi
_{N},H_{n+1,n+2,n+3}S^{(1,n)}\psi _{N}\right\rangle .
\end{equation*}

Due to the $\frac{1}{N}$ factor, $E_{13}$ is the worst term. The other
terms, $E_{11}$, $E_{12}$, $E_{2}$, and $E_{3}$ can be estimated crudely by
simply taking the $L^{\infty }$ at $V_{N}$ and the relevant derivatives.
Hence, we only analyze $E_{11}$ and $E_{13}$ in detail. Writing out $E_{11}$%
, we have%
\begin{eqnarray*}
E_{11} &=&\frac{1}{N^{3}}\left\langle S^{(1,n)}\left(
S_{1}^{2}+S_{2}^{2}+S_{3}^{2}\right) \psi _{N},\left(
S_{n+1}^{2}+S_{n+2}^{2}+S_{n+3}^{2}\right) S^{(1,n)}\psi _{N}\right\rangle \\
&&+\frac{1}{N^{3}}\left\langle S^{(1,n)}\left(
S_{1}^{2}+S_{2}^{2}+S_{3}^{2}\right) \psi
_{N},V_{N,n+1,n+2,n+3}S^{(1,n)}\psi _{N}\right\rangle \\
&&+\frac{1}{N^{3}}\left\langle S^{(1,n)}V_{N,1,2,3}\psi _{N},\left(
S_{n+1}^{2}+S_{n+2}^{2}+S_{n+3}^{2}\right) S^{(1,n)}\psi _{N}\right\rangle \\
&&+\frac{1}{N^{3}}\left\langle S^{(1,n)}V_{N,1,2,3}\psi
_{N},V_{N,n+1,n+2,n+3}S^{(1,n)}\psi _{N}\right\rangle
\end{eqnarray*}%
Discarding the nonnegative terms, it becomes%
\begin{eqnarray*}
E_{11} &\geqslant &-\frac{3}{N^{3}}\left\vert \left\langle \frac{S^{(1,n+1)}%
}{S_{1}S_{2}S_{3}}\psi _{N},\left[ S_{1}S_{2}S_{3},V_{N,1,2,3}\right]
S^{(1,n+1)}\psi _{N}\right\rangle \right\vert \\
&&-\frac{1}{N^{3}}\left\vert \left\langle \frac{S^{(1,n+1)}}{S_{1}S_{2}S_{3}}%
\psi _{N},\left[ S_{1}S_{2}S_{3},V_{N,1,2,3}\right]
V_{N,n+1,n+2,n+3}S^{(1,n)}\psi _{N}\right\rangle \right\vert
\end{eqnarray*}%
Putting $\left[ S_{1}S_{2}S_{3},V_{N,1,2,3}\right] $ and $\left[
S_{1}S_{2}S_{3},V_{N,1,2,3}\right] V_{N,n+1,n+2,n+3}$ in $L^{\infty }$, then
Cauchy-Schwarz, we have%
\begin{eqnarray}
E_{11} &\geqslant &-\frac{C}{N^{3-9\beta }}\left\Vert \frac{S^{(1,n+1)}}{%
S_{1}S_{2}S_{3}}\psi _{N}\right\Vert _{L_{x}^{2}}\left\Vert S^{(1,n+1)}\psi
_{N}\right\Vert _{L_{x}^{2}}  \label{estimate:E11} \\
&&-\frac{C}{N^{3-15\beta }}\left\Vert \frac{S^{(1,n+1)}}{S_{1}S_{2}S_{3}}%
\psi _{N}\right\Vert _{L_{x}^{2}}\left\Vert S^{(1,n)}\psi _{N}\right\Vert
_{L_{x}^{2}}  \notag \\
&\geqslant &-\frac{C}{N^{3-15\beta }}\left\Vert S^{(1,n+2)}\psi
_{N}\right\Vert _{L_{x}^{2}}^{2}  \notag
\end{eqnarray}

We can apply similar argument to $E_{12}$, $E_{2}$, and $E_{3}$ and conclude%
\begin{equation}
E_{12}\geqslant -\frac{C}{N^{2-14\beta }}\left\Vert S^{(1,n+2)}\psi
_{N}\right\Vert _{L_{x}^{2}}^{2},  \label{energy estimate:E12}
\end{equation}%
\begin{equation}
E_{21}\geqslant -\frac{C}{N^{2-13\beta }}\left\Vert S^{(1,n+2)}\psi
_{N}\right\Vert _{L_{\mathbf{x}}^{2}}^{2},  \label{energy estimate:E21}
\end{equation}%
\begin{equation}
E_{22}\geqslant -\frac{C}{N^{3-14\beta }}\left\Vert S^{(1,n+2)}\psi
_{N}\right\Vert _{L_{\mathbf{x}}^{2}}^{2},  \label{energy estimate:E22}
\end{equation}%
\begin{equation}
E_{3}\geqslant -\frac{C}{N^{3-13\beta }}\left\Vert S^{(1,n+2)}\psi
_{N}\right\Vert _{L_{\mathbf{x}}^{2}}^{2}.  \label{energy estimate:E3}
\end{equation}

We are left to deal with $E_{13}$ which cannot be estimated this crudely.
For $E_{13}$, we have%
\begin{eqnarray*}
E_{13} &=&\frac{1}{N}\left\langle S^{(1,n)}\left(
S_{1}^{2}+S_{n+1}^{2}+S_{n+2}^{2}\right) \psi _{N},\left(
S_{n+3}^{2}+S_{n+4}^{2}+S_{n+5}^{2}\right) S^{(1,n)}\psi _{N}\right\rangle \\
&&+\frac{1}{N}\left\langle S^{(1,n)}\left(
S_{1}^{2}+S_{n+1}^{2}+S_{n+2}^{2}\right) \psi
_{N},V_{N,n+3,n+4,n+5}S^{(1,n)}\psi _{N}\right\rangle \\
&&+\frac{1}{N}\left\langle S^{(1,n)}V_{N,1,n+1,n+2}\psi _{N},\left(
S_{n+3}^{2}+S_{n+4}^{2}+S_{n+5}^{2}\right) S^{(1,n)}\psi _{N}\right\rangle \\
&&+\frac{1}{N}\left\langle S^{(1,n)}V_{N,1,n+1,n+2}\psi
_{N},V_{N,n+3,n+4,n+5}S^{(1,n)}\psi _{N}\right\rangle
\end{eqnarray*}%
Putting away the nonnegative terms,%
\begin{eqnarray*}
E_{13} &\geqslant &-\frac{1}{N}\left\vert \left\langle \frac{S^{(1,n)}S_{n+3}%
}{S_{1}}\psi _{N},\left[ S_{1},V_{N,1,n+1,n+2}\right] S^{(1,n)}S_{n+3}\psi
_{N}\right\rangle \right\vert \\
&&-\frac{1}{N}\left\vert \left\langle \frac{S^{(1,n)}}{S_{1}}\psi _{N},\left[
S_{1},V_{N,1,n+1,n+2}\right] V_{N,n+3,n+4,n+5}S^{(1,n)}\psi
_{N}\right\rangle \right\vert \\
&=&E_{131}+E_{132}
\end{eqnarray*}%
We apply Cauchy-Schwarz with a weight and attain%
\begin{eqnarray*}
E_{131} &\geqslant &-\frac{\varepsilon _{1}C}{N^{1-\beta }}\left\vert
\left\langle \frac{S^{(1,n)}S_{n+3}}{S_{1}}\psi _{N},\left\vert V^{\prime
}\right\vert _{N,1,n+1,n+2}\frac{S^{(1,n)}S_{n+3}}{S_{1}}\right\rangle
\right\vert \\
&&-\frac{\varepsilon _{1}^{-1}C}{N^{1-\beta }}\left\vert \left\langle
S^{(1,n)}S_{n+3}\psi _{N},\left\vert V^{\prime }\right\vert
_{N,1,n+1,n+2}S^{(1,n)}S_{n+3}\psi _{N}\right\rangle \right\vert ,
\end{eqnarray*}%
and%
\begin{eqnarray*}
E_{132} &\geqslant &-\frac{\varepsilon _{2}C}{N^{1-\beta }}\left\vert
\left\langle \frac{S^{(1,n)}}{S_{1}}\psi _{N},\left\vert V^{\prime
}\right\vert _{N,1,n+1,n+2}V_{N,n+3,n+4,n+5}\frac{S^{(1,n)}}{S_{1}}\psi
_{N}\right\rangle \right\vert \\
&&-\frac{\varepsilon _{2}^{-1}C}{N^{1-\beta }}\left\vert \left\langle
S^{(1,n)}\psi _{N},\left\vert V^{\prime }\right\vert
_{N,1,n+1,n+2}V_{N,n+3,n+4,n+5}S^{(1,n)}\psi _{N}\right\rangle \right\vert
\end{eqnarray*}%
with $\varepsilon _{1}$ and $\varepsilon _{2}$ to be determined. Utilize H%
\"{o}lder and Sobolev,%
\begin{eqnarray*}
E_{131} &\geqslant &-\frac{\varepsilon _{1}C}{N^{1-\beta }}\left\Vert
\left\vert V^{\prime }\right\vert _{N,1,n+1,n+2}\right\Vert
_{L_{x_{n+1}}^{1+}L_{x_{n+2}}^{1+}}\left\Vert \frac{S^{(1,n)}S_{n+3}}{S_{1}}%
\psi _{N}\right\Vert _{L_{x}^{2}L_{x_{n+1}}^{\infty -}L_{x_{n+2}}^{\infty
-}}^{2} \\
&&-\frac{\varepsilon _{1}^{-1}C}{N^{1-\beta }}\left\Vert \left\vert
V^{\prime }\right\vert _{N,1,n+1,n+2}\right\Vert _{L_{x_{n+1}}^{\infty
}L_{x_{n+2}}^{\frac{3}{2}}}\left\Vert S^{(1,n)}S_{n+3}\psi _{N}\right\Vert
_{L_{x}^{2}L_{x_{n+2}}^{6}}^{2} \\
&\geqslant &-\frac{C}{N^{1-\beta }}\left( \varepsilon _{1}N^{+}+\varepsilon
_{1}^{-1}N^{4\beta }\right) \left\Vert S^{(1,n+2)}\psi _{N}\right\Vert
_{L_{x}^{2}}^{2} \\
&\geqslant &-\frac{C}{N^{1-3\beta -}}\left\Vert S^{(1,n+2)}\psi
_{N}\right\Vert _{L_{x}^{2}}^{2}
\end{eqnarray*}%
and%
\begin{eqnarray*}
E_{132} &\geqslant &-\frac{\varepsilon _{2}C}{N^{1-\beta }}\left\Vert
\left\vert V^{\prime }\right\vert _{N,1,n+1,n+2}\right\Vert
_{L_{x_{n+1}}^{\infty }L_{x_{n+2}}^{\frac{3}{2}}}\left\Vert
V_{N,n+3,n+4,n+5}\right\Vert _{L_{x_{n+4}}^{\frac{3}{2}}L_{x_{n+5}}^{\frac{3%
}{2}}} \\
&&\times \left\Vert \frac{S^{(1,n)}}{S_{1}}\psi _{N}\right\Vert
_{L_{x}^{2}L_{x_{n+2}}^{6}L_{x_{n+4}}^{6}L_{x_{n+5}}^{6}}^{2} \\
&&-\frac{\varepsilon _{2}^{-1}C}{N^{1-\beta }}\left\Vert \left\vert
V^{\prime }\right\vert _{N,1,n+1,n+2}\right\Vert _{L_{x_{n+1}}^{\infty
}L_{x_{n+2}}^{\infty }}\left\Vert V_{N,n+3,n+4,n+5}\right\Vert
_{L_{x_{n+4}}^{\frac{3}{2}}L_{x_{n+5}}^{\frac{3}{2}}} \\
&&\times \left\Vert S^{(1,n)}\psi _{N}\right\Vert
_{L_{x}^{2}L_{x_{n+4}}^{6}L_{x_{n+5}}^{6}}^{2} \\
&\geqslant &-\frac{C}{N^{1-\beta }}\left( \varepsilon _{2}N^{6\beta }+\frac{%
N^{8\beta }}{\varepsilon _{2}}\right) \left\Vert S^{(1,n+2)}\psi
_{N}\right\Vert _{L_{x}^{2}}^{2} \\
&=&-\frac{C}{N^{1-8\beta }}\left\Vert S^{(1,n+2)}\psi _{N}\right\Vert
_{L_{x}^{2}}^{2}
\end{eqnarray*}%
Therefore, we have the worst term%
\begin{equation}
E_{13}\geqslant -\frac{C}{N^{1-8\beta }}\left\Vert S^{(1,n+2)}\psi
_{N}\right\Vert _{L_{x}^{2}}^{2}.  \label{estimate:E13}
\end{equation}

Putting (\ref{estimate:E11}), (\ref{energy estimate:E12}), (\ref%
{estimate:E13}), (\ref{energy estimate:E21}), (\ref{energy estimate:E22}), (%
\ref{energy estimate:E3}) together, we have%
\begin{equation}
E\geqslant -\frac{C}{N^{1-8\beta }}\left\Vert S^{(1,n+2)}\psi
_{N}\right\Vert _{L_{x}^{2}}^{2}  \label{energy estimate:E}
\end{equation}

With (\ref{energy estimate:K}), we have%
\begin{eqnarray*}
&&\left\langle \psi _{N},\left( \frac{H_{N}}{N}+1\right) ^{n+2}\psi
_{N}\right\rangle \\
&\geqslant &c_{1}^{n}\left\langle S^{(1,n)}\left( \frac{H_{N}}{N}+1\right)
\psi _{N},S^{(1,n)}\left( \frac{H_{N}}{N}+1\right) \psi _{N}\right\rangle \\
&\geqslant &c_{1}^{n}(1-\frac{C}{N^{1-8\beta }})\left( \left\Vert
S^{(1,n+2)}\psi _{N}\right\Vert _{L_{\mathbf{x}}^{2}}^{2}+\frac{1}{N}%
\left\Vert S_{1}S^{(1,n+1)}\psi _{N}\right\Vert _{L_{\mathbf{x}%
}^{2}}^{2}\right)
\end{eqnarray*}%
as claimed. So we have proved (\ref{energy etimate: stability of matter}).
One can then use the argument on \cite[p.465]{C-HFocusing} to prove (\ref%
{estimate:key a-priori estimate}) as usual. We omit the details here.

\section{Compactness via Energy\label{sec:Compactness}}

We start by introducing an appropriate topology on the density matrices as
was previously done in \cite{E-Y1, E-S-Y2,E-S-Y5,
E-S-Y3,Kirpatrick,TChenAndNP,ChenAnisotropic,Chen3DDerivation,C-H3Dto2D,C-H2/3}%
. Denote the spaces of compact operators and trace class operators on $%
L^{2}\left( \mathbb{T}^{3k}\right) $ as $\mathcal{K}_{k}$ and $\mathcal{L}%
_{k}^{1}$, respectively. Then $\left( \mathcal{K}_{k}\right) ^{\prime }=%
\mathcal{L}_{k}^{1}$. By the fact that $\mathcal{K}_{k}$ is separable, we
select a dense countable subset $\{J_{i}^{(k)}\}_{i\geqslant 1}\subset 
\mathcal{K}_{k}$ in the unit ball of $\mathcal{K}_{k}$ (so $\Vert
J_{i}^{(k)}\Vert _{\func{op}}\leqslant 1$ where $\left\Vert \cdot
\right\Vert _{\func{op}}$ is the operator norm). For $\gamma ^{(k)},\tilde{%
\gamma}^{(k)}\in \mathcal{L}_{k}^{1}$, we then define a metric $d_{k}$ on $%
\mathcal{L}_{k}^{1}$ by 
\begin{equation*}
d_{k}(\gamma ^{(k)},\tilde{\gamma}^{(k)})=\sum_{i=1}^{\infty
}2^{-i}\left\vert \limfunc{Tr}J_{i}^{(k)}\left( \gamma ^{(k)}-\tilde{\gamma}%
^{(k)}\right) \right\vert .
\end{equation*}%
A uniformly bounded sequence $\gamma _{N}^{(k)}\in \mathcal{L}_{k}^{1}$
converges to $\gamma ^{(k)}\in \mathcal{L}_{k}^{1}$ with respect to the
weak* topology if and only if 
\begin{equation*}
\lim_{N\rightarrow \infty }d_{k}(\gamma _{N}^{(k)},\gamma ^{(k)})=0.
\end{equation*}%
For fixed $T>0$, let $C\left( \left[ 0,T\right] ,\mathcal{L}_{k}^{1}\right) $
be the space of functions of $t\in \left[ 0,T\right] $ with values in $%
\mathcal{L}_{k}^{1}$ which are continuous with respect to the metric $d_{k}.$
On $C\left( \left[ 0,T\right] ,\mathcal{L}_{k}^{1}\right) ,$ we define the
metric 
\begin{equation*}
\hat{d}_{k}(\gamma ^{(k)}\left( \cdot \right) ,\tilde{\gamma}^{(k)}\left(
\cdot \right) )=\sup_{t\in \left[ 0,T\right] }d_{k}(\gamma ^{(k)}\left(
t\right) ,\tilde{\gamma}^{(k)}\left( t\right) ),
\end{equation*}%
and denote by $\tau _{prod}$ the topology on the space $\oplus _{k\geqslant
1}C\left( \left[ 0,T\right] ,\mathcal{L}_{k}^{1}\right) $ given by the
product of topologies generated by the metrics $\hat{d}_{k}$ on $C\left( %
\left[ 0,T\right] ,\mathcal{L}_{k}^{1}\right) $.

\begin{theorem}
\label{Theorem:CompactnessOfBBGKY}For every $T\in \left( 0,\infty \right) $,
the sequence $\left\{ \Gamma _{N}(t)=\left\{ \gamma _{N}^{(k)}\right\}
_{k=1}^{N}\right\} \subset \bigoplus_{k\geqslant 1}C\left( \left[ 0,T\right]
,\mathcal{L}_{k}^{1}\right) ,$ which satisfies equation (\ref%
{hierarchy:quintic BBGKY}) and the a-priori bound (\ref{estimate:key
a-priori estimate}) and has finite energy per particle initially%
\begin{equation*}
\frac{1}{N}\limfunc{Tr}H_{N}\gamma _{N}^{(k)}(0)=\frac{1}{N}\left\langle
\psi _{N}(0),H_{N}\psi _{N}(0)\right\rangle \leqslant C,\footnote{%
(\ref{Condition:EnergyBoundOnManyBodyInitialData}) certainly implies both of
(\ref{estimate:key a-priori estimate}) and finite energy per particle.}
\end{equation*}%
is compact with respect to the product topology $\tau _{prod}$. For any
limit point $\Gamma (t)=\left\{ \gamma ^{(k)}\right\} _{k=1}^{N},$ $\gamma
^{(k)}$ is a symmetric nonnegative trace class operator with trace bounded
by $1,$ and it verifies the kinetic energy bound%
\begin{equation}
\sup_{t\in \left[ 0,T\right] }\limfunc{Tr}S^{(1,k)}\gamma
^{(k)}S^{(1,k)}\leqslant C^{k}.  \label{EnergyBound:GP}
\end{equation}
\end{theorem}

As in the usual argument, it suffices to show that for every "nice" test
kernel / operator $J^{(k)}(\mathbf{x}_{k},\mathbf{x}_{k}^{\prime })$, we have%
\begin{equation*}
\left\vert \partial _{t}\limfunc{Tr}\left( J^{(k)}\gamma _{N}^{(k)}\right)
\right\vert \leqslant C_{k,J}.
\end{equation*}%
However, as pointed out in \S \ref{sec:Introduction}, due to the fact that
the case we are dealing with is energy-critical and we are estimating a
trace instead of a product, one will miss an $\varepsilon $ if one estimates
the terms involving $V_{N}$ with the usual method. This can be circumvented
using the conservation of energy:%
\begin{equation*}
\frac{1}{N}\limfunc{Tr}H_{N}\gamma _{N}^{(k)}(t)=\frac{1}{N}\limfunc{Tr}%
H_{N}\gamma _{N}^{(k)}(0)\leqslant C.
\end{equation*}

Rewrite (\ref{hierarchy:quintic BBGKY}) as, 
\begin{equation*}
\left\vert \partial _{t}\limfunc{Tr}\left( J^{(k)}\gamma _{N}^{(k)}\right)
\right\vert \leqslant I+II+III+IV
\end{equation*}%
where%
\begin{equation*}
I\equiv \sum_{j=1}^{k}\left\vert \limfunc{Tr}\left( J^{(k)}\left[
-\bigtriangleup _{x_{j}},\gamma _{N}^{(k)}\right] \right) \right\vert ,
\end{equation*}%
\begin{equation*}
II\equiv \frac{1}{N^{2}}\sum_{1\leqslant i<j<l\leqslant k}\left\vert 
\limfunc{Tr}\left( J^{(k)}\left[ V_{N}(x_{i}-x_{j},x_{i}-x_{l}),\gamma
_{N}^{(k)}\right] \right) \right\vert ,
\end{equation*}%
\begin{equation*}
III\equiv \frac{(N-k)}{N^{2}}\sum_{1\leqslant i<j\leqslant k}\left\vert 
\limfunc{Tr}\left( J^{(k)}\left[ V_{N}(x_{i}-x_{j},x_{i}-x_{k+1}),\gamma
_{N}^{(k+1)}\right] \right) \right\vert ,
\end{equation*}%
and%
\begin{equation*}
IV=\frac{\left( N-k\right) (N-k-1)}{N^{2}}\sum_{j=1}^{k}\left\vert \limfunc{%
Tr}\left( J^{(k)}\left[ V_{N}(x_{j}-x_{k+1},x_{j}-x_{k+2}),\gamma
_{N}^{(k+2)}\right] \right) \right\vert .
\end{equation*}

We can estimate $I$ as usual since it does not contain $V_{N}$.%
\begin{eqnarray*}
&&\left\vert \limfunc{Tr}\left( J^{(k)}\left[ -\bigtriangleup
_{x_{j}},\gamma _{N}^{(k)}\right] \right) \right\vert \\
&=&\left\vert \limfunc{Tr}\left( J^{(k)}\left[ S_{j}^{2},\gamma _{N}^{(k)}%
\right] \right) \right\vert \\
&\leqslant &\left\vert \limfunc{Tr}\left( S_{j}^{-1}J^{(k)}S_{j}S_{j}\gamma
_{N}^{(k)}S_{j}\right) \right\vert +\left\vert \limfunc{Tr}\left(
S_{j}J^{(k)}S_{j}^{-1}S_{j}\gamma _{N}^{(k)}S_{j}\right) \right\vert \\
&\leqslant &\left( \left\Vert S_{j}^{-1}J^{(k)}S_{j}\right\Vert
_{op}+\left\Vert S_{j}J^{(k)}S_{j}^{-1}\right\Vert _{op}\right) \limfunc{Tr}%
\left( S_{j}\gamma _{N}^{(k)}S_{j}\right) \\
&\leqslant &C_{J}C.
\end{eqnarray*}

For $II$, $III$, and $IV$, we estimate a typical term 
\begin{equation}
\limfunc{Tr}J^{(k)}V_{N}(x_{j}-x_{k+1},x_{j}-x_{k+2})\gamma _{N}^{(k+2)}.
\label{estimate:typical term estimate in compactness}
\end{equation}%
The composition is given\ by%
\begin{eqnarray*}
&&J^{(k)}V_{N}(x_{j}-x_{k+1},x_{j}-x_{k+2})\gamma _{N}^{(k+2)} \\
&=&\int d\mathbf{y}_{k+2}\left[ J^{(k)}(\mathbf{x}_{k},\mathbf{y}_{k})\delta
(x_{k+1}-y_{k+1})\delta (x_{k+2}-y_{k+2})\right] \\
&&\times \left[ V_{N}(y_{j}-y_{k+1},y_{j}-y_{k+2})\gamma _{N}^{(k+2)}(%
\mathbf{y}_{k+2},\mathbf{x}_{k+2}^{\prime })\right] \\
&=&\int d\mathbf{y}_{k}J^{(k)}(\mathbf{x}_{k},\mathbf{y}%
_{k})V_{N}(y_{j}-x_{k+1},y_{j}-x_{k+2})\gamma _{N}^{(k+2)}(\mathbf{y}%
_{k},x_{k+1},x_{k+2},\mathbf{x}_{k+2}^{\prime }).
\end{eqnarray*}%
Hence,%
\begin{eqnarray*}
&&\limfunc{Tr}J^{(k)}V_{N}(x_{j}-x_{k+1},x_{j}-x_{k+2})\gamma _{N}^{(k+2)} \\
&=&\int d\mathbf{x}_{k+2}\int d\mathbf{y}_{k}J^{(k)}(\mathbf{x}_{k},\mathbf{y%
}_{k})V_{N}(y_{j}-x_{k+1},y_{j}-x_{k+2})\gamma _{N}^{(k+2)}(\mathbf{y}%
_{k},x_{k+1},x_{k+2},\mathbf{x}_{k+2})
\end{eqnarray*}%
Putting in the definition of $\gamma _{N}^{(k+2)}$ yields%
\begin{eqnarray*}
&&\left\vert \limfunc{Tr}J^{(k)}V_{N}(x_{j}-x_{k+1},x_{j}-x_{k+2})\gamma
_{N}^{(k+2)}\right\vert \\
&\leqslant &\int d\mathbf{x}_{k+2}\int d\mathbf{y}_{k}\left\vert J^{(k)}(%
\mathbf{x}_{k},\mathbf{y}_{k})\right\vert V_{N}(y_{j}-x_{k+1},y_{j}-x_{k+2})
\\
&&\times \left\vert \int \bar{\psi}_{N}(t,\mathbf{y}_{k},x_{k+1},x_{k+2},%
\mathbf{x}_{N-k-2})\psi _{N}(t,\mathbf{x}_{k+2},\mathbf{x}_{N-k-2})d\mathbf{x%
}_{N-k-2}\right\vert .
\end{eqnarray*}%
Cauchy-Schwarz at $d\mathbf{x}_{N-k}$,%
\begin{eqnarray*}
&\leqslant &\int d\mathbf{x}_{k+2}\int d\mathbf{y}_{k}\left\vert J^{(k)}(%
\mathbf{x}_{k},\mathbf{y}_{k})\right\vert V_{N}(y_{j}-x_{k+1},y_{j}-x_{k+2})
\\
&&\times \left( \int \left\vert \psi _{N}\right\vert ^{2}(t,\mathbf{y}%
_{k},x_{k+1},x_{k+2},\mathbf{x}_{N-k-2})d\mathbf{x}_{N-k-2}\right) ^{\frac{1%
}{2}}\left( \int \left\vert \psi _{N}\right\vert ^{2}(t,\mathbf{x}_{k+2},%
\mathbf{x}_{N-k-2})d\mathbf{x}_{N-k-2}\right) ^{\frac{1}{2}}.
\end{eqnarray*}%
Cauchy-Schwarz at $d\mathbf{x}_{k+2}d\mathbf{y}_{k}$,%
\begin{eqnarray*}
&\leqslant &\left( \int \left\vert J^{(k)}(\mathbf{x}_{k},\mathbf{y}%
_{k})\right\vert V_{N}(y_{j}-x_{k+1},y_{j}-x_{k+2})\left\vert \psi
_{N}\right\vert ^{2}(t,\mathbf{y}_{k},x_{k+1},x_{k+2},\mathbf{x}_{N-k-2})d%
\mathbf{x}_{N}d\mathbf{y}_{k}\right) ^{\frac{1}{2}} \\
&&\times \left( \int \left\vert J^{(k)}(\mathbf{x}_{k},\mathbf{y}%
_{k})\right\vert V_{N}(y_{j}-x_{k+1},y_{j}-x_{k+2})\left\vert \psi
_{N}\right\vert ^{2}(t,\mathbf{x}_{k+2},\mathbf{x}_{N-k-2})d\mathbf{x}_{N}d%
\mathbf{y}_{k}\right) ^{\frac{1}{2}} \\
&\equiv &A^{\frac{1}{2}}B^{\frac{1}{2}}.
\end{eqnarray*}%
We estimate $A$ by%
\begin{eqnarray*}
A &\leqslant &\left\Vert J^{(k)}\right\Vert _{L_{\mathbf{y}_{k}}^{\infty }L_{%
\mathbf{x}_{k}}^{1}}\int V_{N}(y_{j}-x_{k+1},y_{j}-x_{k+2}) \\
&&\times \left\vert \psi _{N}\right\vert ^{2}(t,\mathbf{y}%
_{k},x_{k+1},x_{k+2},\mathbf{x}_{N-k-2})d\mathbf{x}_{N-k-2}dx_{k+1}dx_{k+2}d%
\mathbf{y}_{k}
\end{eqnarray*}%
Because the quantity%
\begin{equation*}
\int V_{N}(y_{j}-x_{k+1},y_{j}-x_{k+2})\left\vert \psi _{N}\right\vert
^{2}(t,\mathbf{y}_{k},x_{k+1},x_{k+2},\mathbf{x}_{N-k-2})d\mathbf{x}%
_{N-k-2}dx_{k+1}dx_{k+2}d\mathbf{y}_{k}
\end{equation*}%
is part of the energy and we are dealing with a defocusing case, we have 
\begin{equation*}
A\leqslant C_{J}C.
\end{equation*}

Though $B$ is not part of the energy, $B$ is in fact not as singular as $A$
because $V_{N}(y_{j}-x_{k+1},y_{j}-x_{k+2})\left\vert \psi _{N}\right\vert
^{2}(t,\mathbf{x}_{k+2},\mathbf{x}_{N-k})$ effectively only set two
variables inside $\left\vert \psi _{N}\right\vert ^{2}$ to be equal. But
there is a technical point inside. For simpler notation, let us assume $j=k$%
. Denote a Littlewood-Paley projector by $P$. Decompose $\psi _{N}$ in
frequency that%
\begin{equation}
\psi _{N}=\psi _{N,1}+\psi _{N,2}
\label{eqn:freq decomposition in compactness}
\end{equation}%
where 
\begin{equation*}
\psi _{N,1}=P_{M_{k+1}\geqslant M_{k+2}}\psi _{N},
\end{equation*}%
and 
\begin{equation*}
\psi _{N,2}=P_{M_{k+2}\geqslant M_{k+1}}\psi _{N}.
\end{equation*}%
That is, inside $\psi _{N,1}$, the $M_{k+1}$ frequency dominates the $%
M_{k+2} $ frequency, while the $M_{k+2}$ frequency dominates the $M_{k+1}$
frequency inside $\psi _{N,2}$.

Plugging (\ref{eqn:freq decomposition in compactness}) into $B$ yields,%
\begin{eqnarray*}
B &\leqslant &2\left\Vert J^{(k)}\right\Vert _{L_{\mathbf{x}_{k}}^{\infty
}L_{y_{k}}^{\infty }L_{\mathbf{y}_{k-1}}^{1}} \\
&&\times \int V_{N}(y_{k}-x_{k+1},y_{k}-x_{k+2})\left\vert \psi
_{N,1}\right\vert ^{2}(t,\mathbf{x}_{k+2},\mathbf{x}_{N-k-2})d\mathbf{x}%
_{k+2}dy_{k}d\mathbf{x}_{N-k-2} \\
&&+2\left\Vert J^{(k)}\right\Vert _{L_{\mathbf{x}_{k}}^{\infty
}L_{y_{k}}^{\infty }L_{\mathbf{y}_{k-1}}^{1}} \\
&&\times \int V_{N}(y_{k}-x_{k+1},y_{k}-x_{k+2})\left\vert \psi
_{N,2}\right\vert ^{2}(t,\mathbf{x}_{k+2},\mathbf{x}_{N-k-2})d\mathbf{x}%
_{k+2}y_{k}d\mathbf{x}_{N-k-2}
\end{eqnarray*}%
The second term actually equals to the first term by symmetry. We then
estimate%
\begin{equation*}
B\leqslant 4C_{J}\left( \int
V_{N}(y_{k}-x_{k+1},y_{k}-x_{k+2})dy_{k}dx_{k+2}\right) \left\Vert \psi
_{N,1}\right\Vert _{L_{x_{k+2}}^{\infty }L_{\mathbf{x}_{k+1}}^{2}L_{\mathbf{x%
}_{N-k-2}}^{2}}^{2}
\end{equation*}%
By Sobolev,%
\begin{equation*}
B\leqslant C_{J}b_{0}\left\Vert S_{k+2}^{2}\psi _{N,1}\right\Vert
_{L_{x_{k+2}}^{2}L_{\mathbf{x}_{k+1}}^{2}L_{\mathbf{x}_{N-k-2}}^{2}}^{2}
\end{equation*}%
But the $M_{k+1}$ frequency dominates the $M_{k+2}$ frequency in $\psi
_{N,1} $, so 
\begin{eqnarray*}
B &\leqslant &C_{J}\left\Vert S_{k+1}S_{k+2}\psi _{N,1}\right\Vert _{L_{%
\mathbf{x}_{N}}^{2}}^{2} \\
&\leqslant &C_{J}C^{2}.
\end{eqnarray*}%
That is%
\begin{equation*}
\left\vert \limfunc{Tr}J^{(k)}V_{N}(x_{j}-x_{k+1},x_{j}-x_{k+2})\gamma
_{N}^{(k+2)}\right\vert \leqslant C_{J}C^{\frac{3}{2}}\text{.}
\end{equation*}%
All terms inside $II$, $III$, and $IV$ can be estimated similarly. We do not
repeat the process here. Putting the estimates for $I$, $II$, $III$, and $IV$
together, we have proved 
\begin{equation*}
\left\vert \partial _{t}\limfunc{Tr}\left( J^{(k)}\gamma _{N}^{(k)}\right)
\right\vert \leqslant C_{k,J}
\end{equation*}%
and hence Theorem \ref{Theorem:CompactnessOfBBGKY}.

\section{Convergence without the Sharp Trace Theorem\label{sec:convergence}}

\begin{theorem}
\label{Theorem:Convergence}Assume $\beta <\frac{1}{9}$, let $\Gamma
(t)=\left\{ \gamma ^{(k)}\right\} _{k=1}^{\infty }$ $\in \oplus _{k\geqslant
1}C\left( \left[ 0,T\right] ,\mathcal{L}_{k}^{1}\right) $ be a limit point
of the sequence $\left\{ \Gamma _{N}(t)=\left\{ \gamma _{N}^{(k)}\right\}
_{k=1}^{N}\right\} $ in Theorem \ref{Theorem:CompactnessOfBBGKY}, with
respect to the product topology $\tau _{prod}$, then $\Gamma (t)$ is a
solution to the energy-critical GP hierarchy (\ref{hierarchy:quintic GP in
differential form}) subject to initial data $\gamma ^{(k)}\left( 0\right)
=\left\vert \phi _{0}\right\rangle \left\langle \phi _{0}\right\vert
^{\otimes k}$ with coupling constant $b_{0}=$ $\int_{\mathbb{T}^{3}\times 
\mathbb{T}^{3}}V\left( x,y\right) dxdy$, which, written in integral form, is 
\begin{eqnarray}
&&\gamma ^{(k)}\left( t\right)  \label{hierarchy:quintic GP in integral form}
\\
&=&U^{(k)}(t)\gamma ^{(k)}\left( 0\right)
-ib_{0}\sum_{j=1}^{k}\int_{0}^{t}U^{(k)}(t-s)\limfunc{Tr}\nolimits_{k+1,k+2}%
\left[ \delta (x_{j}-x_{k+1})\delta (x_{j}-x_{k+2}),\gamma ^{(k+2)}(s)\right]
ds.  \notag
\end{eqnarray}%
where $U^{(k)}(t)=\dprod\limits_{j=1}^{k}e^{it\bigtriangleup
_{x_{j}}}e^{-it\bigtriangleup _{x_{j}^{\prime }}}$.
\end{theorem}

Use the short hand notation $J_{t}^{(k)}=J^{(k)}U^{(k)}(t)$, we rewrite (\ref%
{hierarchy:quintic BBGKY}) as%
\begin{eqnarray}
&&\limfunc{Tr}J^{(k)}\gamma _{N}^{(k)}\left( t\right)
\label{hierarchy:Integral BBGKY with test function} \\
&\equiv &I-\frac{i}{N^{2}}\sum_{1\leqslant i<j<l\leqslant k}II-\frac{i(N-k)}{%
N^{2}}\sum_{1\leqslant i<j\leqslant k}III-\frac{i\left( N-k\right) (N-k-1)}{%
N^{2}}\sum_{j=1}^{k}IV  \notag
\end{eqnarray}%
where%
\begin{equation*}
I=\limfunc{Tr}J_{t}^{(k)}\gamma _{N}^{(k)}\left( 0\right) ,
\end{equation*}%
\begin{equation*}
II=\int_{0}^{t}\limfunc{Tr}J_{t-s}^{(k)}\left[
V_{N}(x_{i}-x_{j},x_{i}-x_{l}),\gamma _{N}^{(k)}\right] ds,
\end{equation*}%
\begin{equation*}
III=\int_{0}^{t}\limfunc{Tr}J_{t-s}^{(k)}\left[
V_{N}(x_{i}-x_{j},x_{i}-x_{k+1}),\gamma _{N}^{(k+1)}(s)\right] ds,
\end{equation*}%
\begin{equation*}
IV=\int_{0}^{t}\limfunc{Tr}J_{t-s}^{(k)}\left[ V_{N}\left(
x_{j}-x_{k+1},x_{j}-x_{k+2}\right) ,\gamma _{N}^{(k+2)}(s)\right] ds.
\end{equation*}

By Theorem \ref{Theorem:CompactnessOfBBGKY}, we already know $\Gamma
(t)=\left\{ \gamma ^{(k)}\right\} _{k=1}^{\infty }$ exists and is in $%
\bigoplus_{k\geqslant 1}C\left( \left[ 0,T\right] ,\mathcal{L}%
_{k}^{1}\right) $. We will prove that, for every "nice" test kernel /
operator $J^{(k)}$, we have (\ref{hierarchy:Integral BBGKY with test
function}) converges to%
\begin{eqnarray}
&&\limfunc{Tr}J^{(k)}\gamma ^{(k)}\left( t\right)
\label{hierarchy:Integral GP with test function} \\
&=&\limfunc{Tr}J_{t}^{(k)}\gamma ^{(k)}\left( 0\right)  \notag \\
&&-ib_{0}\sum_{j=1}^{k}\int_{0}^{t}\limfunc{Tr}J_{t-s}^{(k)}\left[ \delta
(x_{j}-x_{k+1})\delta (x_{j}-x_{k+2}),\gamma ^{(k+2)}(s)\right] ds.  \notag
\end{eqnarray}

Passing to further subsequences if necessary, let us assume 
\begin{equation*}
\lim_{N\rightarrow \infty }\sup_{t\in \lbrack 0,T]}\limfunc{Tr}J^{(k)}\left(
\gamma _{N}^{(k)}-\gamma ^{(k)}\right) =0,\forall J^{(k)}\in \mathcal{K}_{k}.
\end{equation*}%
By the standard argument,%
\begin{equation*}
\lim_{N\rightarrow \infty }\limfunc{Tr}J^{(k)}\gamma _{N}^{(k)}\left(
t\right) =\limfunc{Tr}J^{(k)}\gamma ^{(k)}\left( t\right)
\end{equation*}%
and%
\begin{equation*}
\lim_{N\rightarrow \infty }I=\limfunc{Tr}J^{(k)}U^{(k)}(t)\gamma
^{(k)}\left( 0\right) \text{.}
\end{equation*}%
So the rest of the proof will only deal with $II,III,IV$.

We first notice that, terms like 
\begin{equation*}
\limfunc{Tr}J_{t-s}^{(k)}V_{N}\left( x_{j}-x_{k+1},x_{j}-x_{k+2}\right)
\gamma _{N}^{(k+2)},
\end{equation*}%
inside (\ref{hierarchy:Integral BBGKY with test function}), are at least
bounded using the method to handle (\ref{estimate:typical term estimate in
compactness}) in the proof of Theorem \ref{Theorem:CompactnessOfBBGKY} since 
$J_{t-s}^{(k)}$ is but another test functions like $J^{(k)}$. Hence,%
\begin{eqnarray*}
&&\frac{1}{N^{2}}\sum_{1\leqslant i<j<l\leqslant k}\left\vert II\right\vert +%
\frac{(N-k)}{N^{2}}\sum_{1\leqslant i<j\leqslant k}\left\vert III\right\vert
+\frac{\left( N-k\right) (N-k-1)-N^{2}}{N^{2}}\sum_{j=1}^{k}\left\vert
IV\right\vert \\
&\leqslant &C\left( \frac{k^{3}}{N^{2}}+\frac{\left( N-k\right) k^{2}}{N^{2}}%
+\frac{k\left( N-k\right) (N-k-1)-N^{2}}{N^{2}}\right) C_{J}C^{\frac{3}{2}}
\\
&\rightarrow &0\text{ as }N\rightarrow \infty \text{.}
\end{eqnarray*}

We are left to prove the key coupling term's convergence, that is, 
\begin{equation}
\limfunc{Tr}J_{t-s}^{(k)}V_{N}\left( x_{j}-x_{k+1},x_{j}-x_{k+2}\right)
\gamma _{N}^{(k+2)}(s)  \label{eqn:key convergence term}
\end{equation}%
converges to%
\begin{equation}
\limfunc{Tr}J_{t-s}^{(k)}b_{0}\delta (x_{j}-x_{k+1})\delta
(x_{j}-x_{k+2})\gamma ^{(k+2)}(s).  \label{eqn:key convergence target}
\end{equation}%
as $N\rightarrow \infty $. We prove (\ref{eqn:key convergence term})$%
\rightarrow $(\ref{eqn:key convergence target}) as $N\rightarrow \infty $
for $\beta <\frac{1}{12}$ in \S \ref{Sec:Convergence 1/12} first with a
partially frequency analysis and partially standard argument for an
incremental presentation. We then give the $\beta <\frac{1}{9}$ proof in \S %
\ref{Sec:Convergence 1/9}. At the moment, we first explain why the usual
method does not work in this energy-critical setting. If we proceed as usual,%
\begin{eqnarray*}
&&V_{N}\left( x_{j}-x_{k+1},x_{j}-x_{k+2}\right) \gamma
_{N}^{(k+2)}-b_{0}\delta (x_{j}-x_{k+1})\delta (x_{j}-x_{k+2})\gamma ^{(k+2)}
\\
&=&\left( V_{N}\left( x_{j}-x_{k+1},x_{j}-x_{k+2}\right) \gamma
_{N}^{(k+2)}-b_{0}\delta (x_{j}-x_{k+1})\delta (x_{j}-x_{k+2})\gamma
_{N}^{(k+2)}\right) \\
&&+b_{0}\left( \delta (x_{j}-x_{k+1})\delta (x_{j}-x_{k+2})\gamma
_{N}^{(k+2)}-\rho _{\alpha }\left( x_{j}-x_{k+1},x_{j}-x_{k+2}\right) \gamma
_{N}^{(k+2)}\right) \\
&&+b_{0}\left( \rho _{\alpha }\left( x_{j}-x_{k+1},x_{j}-x_{k+2}\right)
\left( \gamma _{N}^{(k+2)}-\gamma ^{(k+2)}\right) \right) \\
&&+b_{0}(\rho _{\alpha }\left( x_{j}-x_{k+1},x_{j}-x_{k+2}\right) \gamma
^{(k+2)}-\delta (x_{j}-x_{k+1})\delta (x_{j}-x_{k+2})\gamma ^{(k+2)}) \\
&\equiv &A+B+D+F.
\end{eqnarray*}%
where $\rho _{\alpha }=\frac{1}{\alpha ^{6}}\rho (\frac{x}{\alpha },\frac{y}{%
\alpha })$, which is defined like $V_{N}$, namely, first rescaled and then
extended periodically with $\rho $ being a compactly supported smooth
probability measure on $\mathbb{R}^{3}\times \mathbb{R}^{3}$, then we will
run into either the sharp Sobolev trace theorem (\ref{estimate:failed
sobolev 2}) which is not true or the quantity 
\begin{equation*}
\limfunc{Tr}S^{(1+\varepsilon ,3)}\gamma _{N}^{(k+2)}S^{(1+\varepsilon ,3)}
\end{equation*}%
which is not uniformly bounded in $N$, at the Sobolev endpoint when we
estimate $A$ or $B$ using the usual Poincar\'{e} type inequality in the
non-critical setting. One seemingly fix is to use the boundedness of $\frac{1%
}{N}\left\Vert S_{1}S^{(1,k-1)}\psi _{N}\right\Vert _{L_{x}^{2}}^{2}$ inside
(\ref{energy etimate: stability of matter}). But then one is forced to
select the $\alpha $ in $B$ to be dependent of $N$ to counter balance a $%
N^{\varepsilon }$, which comes from interpolating between $\left\Vert
S^{(k)}\psi _{N}\right\Vert _{L_{x}^{2}}^{2}$ and $\frac{1}{N}\left\Vert
S_{1}S^{(1,k-1)}\psi _{N}\right\Vert _{L_{x}^{2}}^{2}$. This will cause
trouble when dealing with $D$ since $\rho _{\alpha }\left(
x_{j}-x_{k+1},x_{j}-x_{k+2}\right) $ is no longer independent of $N$.

\subsection{Convergence of the Key Coupling Term for $\protect\beta <1/12$
(An Easier Case)\label{Sec:Convergence 1/12}}

Let $M$ be a dyadic frequency, let $P_{\leqslant M}^{j}$ ($P_{>M}^{j}$)
denote the Littlewood-Paley projector which projects the $x_{j}$ variable at
frequency $\leqslant M$ ($>M$), and write 
\begin{equation*}
P_{\leqslant M}^{(k)}=\dprod\limits_{j=1}^{k}P_{\leqslant M}^{j}\text{ (}%
P_{>M}^{(k)}=\dprod\limits_{j=1}^{k}P_{>M}^{j}\text{).}
\end{equation*}%
We will use the usual Bernstein inequalities from \cite{Tao}, which are
easily adapted to $\mathbb{T}^{d}$ via estimates on the periodic
Littlewood-Paley convolution kernels and widely used.\footnote{%
See, for example, \cite{KV}.} We include an example in Appendix \ref%
{Sec:AppendixBernstein}. We also need the following Theorem.

\begin{theorem}[\protect\cite{OriginalDeFinette,LewinFocusing}]
\footnote{%
To be precise, this version we are using is from \cite{LewinFocusing}.}\label%
{Lem:finite QdF}Assume $\left\{ \gamma _{N}^{(k)}\right\} _{k=1}^{N}$ is the
marginal density generated by a $N$-body wave function $\psi _{N}\in
L_{s}^{2}(\mathbb{T}^{3N})$. Then there is a family of positive Borel
measure $\left\{ d\mu _{N,M,k}\right\} $ supported on the unit sphere of $%
P_{\leqslant M}\left( L^{2}(\mathbb{T}^{3})\right) $ such that%
\begin{equation*}
\limfunc{Tr}\left\vert P_{\leqslant M}^{(k)}\gamma _{N}^{(k)}P_{\leqslant
M}^{(k)}-\tilde{\gamma}_{N,M}^{(k)}\right\vert \leqslant \frac{4kD_{M}}{N},
\end{equation*}%
if%
\begin{equation*}
\tilde{\gamma}_{N,M}^{(k)}=\int_{S(P_{\leqslant M}\left( L^{2}(\mathbb{T}%
^{3})\right) )}\left\vert \phi ^{\otimes k}\right\rangle \left\langle \phi
^{\otimes k}\right\vert d\mu _{N,k}(\phi )
\end{equation*}%
where $D_{M}$ is the dimension of $P_{\leqslant M}\left( L^{2}(\mathbb{T}%
^{3})\right) $ which is controlled by $CM^{3}$ and $d\mu _{N,M,k}$ has the
property that 
\begin{equation*}
\int_{S(P_{\leqslant M}\left( L^{2}(\mathbb{T}^{3})\right) )}d\mu
_{N,k}(\phi )=\limfunc{Tr}P_{\leqslant M}^{(k)}\gamma
_{N}^{(k)}(t)P_{\leqslant M}^{(k)}\leqslant 1.
\end{equation*}
\end{theorem}

\begin{corollary}
\label{Corollary:Regularity for Tensor product finite qde}As long as $%
M\leqslant N^{\frac{1}{6}}$, we have 
\begin{equation*}
\sup_{t\in \left[ 0,T\right] }\limfunc{Tr}S^{(1,3)}\tilde{\gamma}%
_{N,M}^{(k)}(t)S^{(1,3)}\leqslant C^{3}
\end{equation*}%
Moreover, when $M=N^{2\beta +\varepsilon }$ and $\beta <\frac{1}{12}$ or $%
M=N^{\beta }$ and $\beta <\frac{1}{9}$, $\tilde{\gamma}_{N,M}^{(k)}%
\rightarrow \gamma ^{(k)}$ weak* as trace class operators as well.
\end{corollary}

\begin{proof}
Directly from Theorem \ref{Lem:finite QdF}, we have 
\begin{eqnarray*}
&&\limfunc{Tr}\left\vert S^{(1,3)}P_{\leqslant M}^{(k)}\gamma
_{N}^{(k)}(t)P_{\leqslant M}^{(k)}S^{(1,3)}-S^{(1,3)}\tilde{\gamma}%
_{N,M}^{(k)}(t)S^{(1,3)}\right\vert \\
&\leqslant &M^{6}\limfunc{Tr}\left\vert P_{\leqslant M}^{(k)}\gamma
_{N}^{(k)}(t)P_{\leqslant M}^{(k)}-\tilde{\gamma}_{N,M}^{(k)}(t)\right\vert
\\
&\leqslant &\frac{4kCM^{6}}{N}
\end{eqnarray*}%
while we have, from (\ref{estimate:key a-priori estimate}), that%
\begin{equation*}
\sup_{t\in \left[ 0,T\right] }\limfunc{Tr}S^{(1,3)}P_{\leqslant
M}^{(k)}\gamma _{N}^{(k)}(t)P_{\leqslant M}^{(k)}S^{(1,3)}\leqslant
\sup_{t\in \left[ 0,T\right] }\limfunc{Tr}S^{(1,3)}\gamma
_{N}^{(k)}(t)S^{(1,3)}\leqslant C^{3}.
\end{equation*}%
The simple triangles inequality%
\begin{eqnarray*}
&&\left\vert \limfunc{Tr}J^{(k)}\left( \gamma ^{(k)}-\tilde{\gamma}%
_{N,M}^{(k)}\right) \right\vert \\
&\leqslant &\left\vert \limfunc{Tr}J^{(k)}\left( \gamma ^{(k)}-P_{\leqslant
M}^{(k)}\gamma _{N}^{(k)}P_{\leqslant M}^{(k)}\right) \right\vert
+\left\vert \limfunc{Tr}J^{(k)}\left( P_{\leqslant M}^{(k)}\gamma
_{N}^{(k)}P_{\leqslant M}^{(k)}-\tilde{\gamma}_{N,M}^{(k)}\right) \right\vert
\\
&\leqslant &\left\vert \limfunc{Tr}J^{(k)}\left( \gamma ^{(k)}-P_{\leqslant
M}^{(k)}\gamma _{N}^{(k)}P_{\leqslant M}^{(k)}\right) \right\vert
+\left\Vert J^{(k)}\right\Vert _{op}\limfunc{Tr}\left\vert P_{\leqslant
M}^{(k)}\gamma _{N}^{(k)}P_{\leqslant M}^{(k)}-\tilde{\gamma}%
_{N,M}^{(k)}\right\vert
\end{eqnarray*}%
shows that $\tilde{\gamma}_{N,M}^{(k)}\rightarrow \gamma ^{(k)}$ weak* as
trace class operators as well.
\end{proof}

We decompose%
\begin{eqnarray*}
&&V_{N}\left( x_{j}-x_{k+1},x_{j}-x_{k+2}\right) \gamma
_{N}^{(k+2)}-b_{0}\delta (x_{j}-x_{k+1})\delta (x_{j}-x_{k+2})\gamma ^{(k+2)}
\\
&=&V_{N}\left( x_{j}-x_{k+1},x_{j}-x_{k+2}\right) \left( \gamma
_{N}^{(k+2)}-P_{\leqslant M}^{(k+2)}\gamma _{N}^{(k+2)}P_{\leqslant
M}^{(k+2)}\right) \\
&&+V_{N}\left( x_{j}-x_{k+1},x_{j}-x_{k+2}\right) \left( P_{\leqslant
M}^{(k+2)}\gamma _{N}^{(k+2)}P_{\leqslant M}^{(k+2)}-\tilde{\gamma}%
_{N,M}^{(k+2)}\right) \\
&&+\left[ V_{N}\left( x_{j}-x_{k+1},x_{j}-x_{k+2}\right) -b_{0}\delta
(x_{j}-x_{k+1})\delta (x_{j}-x_{k+2})\right] \tilde{\gamma}_{N,M}^{(k+2)} \\
&&+b_{0}\left[ \delta (x_{j}-x_{k+1})\delta (x_{j}-x_{k+2})-\rho _{\alpha
}\left( x_{j}-x_{k+1},x_{j}-x_{k+2}\right) \right] \tilde{\gamma}%
_{N,M}^{(k+2)} \\
&&+b_{0}\left( \rho _{\alpha }\left( x_{j}-x_{k+1},x_{j}-x_{k+2}\right)
\left( \tilde{\gamma}_{N,M}^{(k+2)}-\gamma ^{(k+2)}\right) \right) \\
&&+b_{0}(\rho _{\alpha }\left( x_{j}-x_{k+1},x_{j}-x_{k+2}\right) \gamma
^{(k+2)}-\delta (x_{j}-x_{k+1})\delta (x_{j}-x_{k+2})\gamma ^{(k+2)}) \\
&\equiv &I+II+III+IV+VI+VII
\end{eqnarray*}%
and we choose 
\begin{equation}
M=N^{2\beta +\varepsilon }.  \label{eqn:choice of M in convergence}
\end{equation}%
We estimate $I,II,III,IV,$and $VII$ in detail. With our method, the estimate
for $VI$ follows from the standard argument \cite{Kirpatrick} since $\tilde{%
\gamma}_{N,M}^{(k)}\rightarrow \gamma ^{(k)}$ weak* as trace class
operators, and hence we skip it.

\subsubsection{Estimate for $I$}

When $k$ is large, the difference $\gamma _{N}^{(k+2)}(s)-P_{\leqslant
M}^{(k+2)}\gamma _{N}^{(k+2)}(s)P_{\leqslant M}^{(k+2)}$ is a big sum. We
group them into 
\begin{equation*}
\gamma _{N}^{(k+2)}(s)-P_{\leqslant M}^{(k+2)}\gamma
_{N}^{(k+2)}(s)P_{\leqslant M}^{(k+2)}=I_{0}(s)+I_{1}(s)+I_{2}(s)+I_{3}(s)
\end{equation*}%
depending on how many of $x_{j}$, $x_{k+1}$ and $x_{k+2}$ are at high. That
is, for $h=j,k+1$ and $k+2$ 
\begin{equation*}
P_{>M}^{h}I_{0}(s)=0;
\end{equation*}%
there is one of $h=j,h=k+1$, and $h=k+2$ such that%
\begin{equation*}
P_{>M}^{h}I_{1}(s)=I_{1}(s);
\end{equation*}%
there are two of $h=j,h=k+1$, and $h=k+2$ such that%
\begin{equation*}
P_{>M}^{h}I_{2}(s)=I_{2(s)};
\end{equation*}%
and%
\begin{equation*}
P_{>M}^{j}P_{>M}^{k+1}P_{>M}^{k+2}I_{3}(s)=I_{3}(s).
\end{equation*}

With this decomposition,%
\begin{eqnarray*}
\left\vert \limfunc{Tr}J_{t-s}^{(k)}I\right\vert &\leqslant &\left\vert 
\limfunc{Tr}J_{t-s}^{(k)}V_{N}\left( x_{j}-x_{k+1},x_{j}-x_{k+2}\right)
I_{0}(s)\right\vert \\
&&+\left\vert \limfunc{Tr}J_{t-s}^{(k)}V_{N}\left(
x_{j}-x_{k+1},x_{j}-x_{k+2}\right) I_{1}(s)\right\vert \\
&&+\left\vert \limfunc{Tr}J_{t-s}^{(k)}V_{N}\left(
x_{j}-x_{k+1},x_{j}-x_{k+2}\right) I_{2}(s)\right\vert \\
&&+\left\vert \limfunc{Tr}J_{t-s}^{(k)}V_{N}\left(
x_{j}-x_{k+1},x_{j}-x_{k+2}\right) I_{3}(s)\right\vert .
\end{eqnarray*}%
We estimate the best term $I_{3}$ and the worst term $I_{0}$. The main idea
is that when at least one variable is at frequency higher than $N^{2\beta }$%
, we can exchange the powers of $N^{\beta }$ in $V_{N}$ with derivatives on $%
\gamma _{N}^{(k+2)}$ to decrease the strength of singularity and hence
obtain a decay.

Without lose of generality, say $x_{j-1}$ is at high in a typical term
inside the worst group $I_{0}$, 
\begin{eqnarray*}
&&\left\vert \limfunc{Tr}J_{t-s}^{(k)}V_{N}\left(
x_{j}-x_{k+1},x_{j}-x_{k+2}\right) P_{>M}^{j-1}\gamma
_{N}^{(k+2)}(s)\right\vert \\
&=&\frac{1}{N^{\varepsilon }}|\limfunc{Tr}\left(
S_{k+1}^{-1}S_{k+2}^{-1}J_{t-s}^{(k)}S_{k+1}S_{k+2}\right) \\
&&\left( S_{k+1}^{-1}S_{k+2}^{-1}\frac{V_{N}\left(
x_{j}-x_{k+1},x_{j}-x_{k+2}\right) }{N^{2\beta }}S_{k+1}^{-1}S_{k+2}^{-1}%
\right) \\
&&\left( S_{k+1}S_{k+2}P_{>M}^{j-1}N^{2\beta +\varepsilon }\gamma
_{N}^{(k+2)}(s)S_{k+1}S_{k+2}\right) |
\end{eqnarray*}%
By the Sobolev inequality $\left\Vert S^{-1}f\right\Vert _{L^{6}}\leqslant
C\left\Vert f\right\Vert _{L^{2}}$, 
\begin{eqnarray*}
&&\left\Vert S_{k+1}^{-1}S_{k+2}^{-1}\frac{V_{N}\left(
x_{j}-x_{k+1},x_{j}-x_{k+2}\right) }{N^{2\beta }}S_{k+1}^{-1}S_{k+2}^{-1}%
\right\Vert _{op} \\
&\leqslant &\frac{C}{N^{2\beta }}\left\Vert V_{N}(x,y)\right\Vert _{L_{x}^{%
\frac{3}{2}}L_{y}^{\frac{3}{2}}}=C\left\Vert V\right\Vert _{L^{\frac{3}{2}}}%
\text{.}
\end{eqnarray*}%
Hence,%
\begin{eqnarray}
&&\left\vert \limfunc{Tr}J_{t-s}^{(k)}V_{N}\left(
x_{j}-x_{k+1},x_{j}-x_{k+2}\right) P_{>M}^{j-1}\gamma
_{N}^{(k+2)}(s)\right\vert
\label{estimate:a typical term in VI_0, convergence part} \\
&\leqslant &\frac{1}{N^{\varepsilon }}\left\Vert J_{t-s}^{(k)}\right\Vert
_{op}C\left\Vert V\right\Vert _{L^{\frac{3}{2}}}\limfunc{Tr}\left\vert
S_{k+1}S_{k+2}P_{>M}^{j-1}N^{2\beta +\varepsilon }\gamma
_{N}^{(k+2)}(s)S_{k+1}S_{k+2}\right\vert  \notag
\end{eqnarray}%
Cauchy-Schwarz inside $\gamma _{N}^{(k+2)}$, we have%
\begin{eqnarray*}
&&\limfunc{Tr}\left\vert S_{k+1}S_{k+2}P_{>M}^{j-1}N^{2\beta +\varepsilon
}\gamma _{N}^{(k+2)}(s)S_{k+1}S_{k+2}\right\vert \\
&\leqslant &\left( \limfunc{Tr}S_{k+1}S_{k+2}P_{>M}^{j-1}N^{2\beta
+\varepsilon }\gamma _{N}^{(k+2)}(s)S_{k+1}S_{k+2}P_{>M}^{j-1}N^{2\beta
+\varepsilon }\right) ^{\frac{1}{2}} \\
&&\times \left( \limfunc{Tr}S_{k+1}S_{k+2}\gamma
_{N}^{(k+2)}(s)S_{k+1}S_{k+2}\right) ^{\frac{1}{2}}.
\end{eqnarray*}%
Using the fact that $P_{>M}^{j-1}N^{2\beta }\leqslant S_{j-1}P_{>M}^{j-1}$
whenever $M=N^{2\beta +\varepsilon }$, we reach%
\begin{eqnarray*}
&&\limfunc{Tr}\left\vert S_{k+1}S_{k+2}P_{>M}^{j-1}N^{2\beta +\varepsilon
}\gamma _{N}^{(k+2)}(s)S_{k+1}S_{k+2}\right\vert \\
&\leqslant &\left( \limfunc{Tr}S_{k+1}S_{k+2}S_{j-1}P_{>M}^{j-1}\gamma
_{N}^{(k+2)}(s)S_{k+1}S_{k+2}S_{j-1}P_{>M}^{j-1}\right) ^{\frac{1}{2}} \\
&&\times \left( \limfunc{Tr}S_{k+1}S_{k+2}\gamma
_{N}^{(k+2)}(s)S_{k+1}S_{k+2}\right) ^{\frac{1}{2}} \\
&\leqslant &\left( \limfunc{Tr}S_{k+1}S_{k+2}S_{j-1}\gamma
_{N}^{(k+2)}(s)S_{k+1}S_{k+2}S_{j-1}\right) ^{\frac{1}{2}} \\
&&\times \left( \limfunc{Tr}S_{k+1}S_{k+2}\gamma
_{N}^{(k+2)}(s)S_{k+1}S_{k+2}\right) ^{\frac{1}{2}} \\
&\leqslant &C^{\frac{3}{2}}C^{\frac{2}{2}}\text{. }
\end{eqnarray*}%
Putting the above back into (\ref{estimate:a typical term in VI_0,
convergence part}), we have 
\begin{equation*}
\left\vert \limfunc{Tr}J_{t-s}^{(k)}V_{N}\left(
x_{j}-x_{k+1},x_{j}-x_{k+2}\right) P_{>M}^{j-1}\gamma
_{N}^{(k+2)}(s)\right\vert \leqslant \frac{C_{V,J,k}}{N^{\varepsilon }}.
\end{equation*}%
That is, 
\begin{equation*}
\left\vert \limfunc{Tr}J_{t-s}^{(k)}V_{N}\left(
x_{j}-x_{k+1},x_{j}-x_{k+2}\right) I_{0}(s)\right\vert \leqslant \frac{%
4^{k}C_{V,J,k}}{N^{\varepsilon }}=\frac{C_{V,J,k}}{N^{\varepsilon }}
\end{equation*}%
after counting terms.

For the best term $I_{3}$, apply similar method, we have 
\begin{eqnarray*}
&&\left\vert \limfunc{Tr}J_{t-s}^{(k)}V_{N}\left(
x_{j}-x_{k+1},x_{j}-x_{k+2}\right) I_{3}\right\vert \\
&=&\frac{1}{N^{3\varepsilon }}\left\vert \limfunc{Tr}J_{t-s}^{(k)}\frac{%
V_{N}\left( x_{j}-x_{k+1},x_{j}-x_{k+2}\right) }{N^{6\beta }}%
P_{>M}^{j}N^{2\beta +\varepsilon }P_{>M}^{k+1}N^{2\beta +\varepsilon
}P_{>M}^{k+2}N^{2\beta +\varepsilon }I_{3}(s)\right\vert \\
&\leqslant &\frac{1}{N^{3\varepsilon }}\left\Vert J_{t-s}^{(k)}\right\Vert
_{op}\left\Vert V\right\Vert _{L^{\infty }}\limfunc{Tr}\left\vert
S_{j}P_{>M}^{j}S_{k+1}P_{>M}^{k+1}S_{k+2}P_{>M}^{k+2}I_{3}(s)\right\vert \\
&\leqslant &\frac{C_{V,J,k}}{N^{3\varepsilon }}C^{\frac{3}{2}}1=\frac{%
C_{V,J,k}}{N^{3\varepsilon }}\text{.}
\end{eqnarray*}%
Therefore,%
\begin{equation*}
\left\vert \limfunc{Tr}J_{t-s}^{(k)}I\right\vert \leqslant \frac{C_{V,J,k}}{%
N^{\varepsilon }}\rightarrow 0\text{ as }N\rightarrow \infty \text{.}
\end{equation*}%
as long as $M=N^{2\beta +\varepsilon }$.

\subsubsection{Estimates for $II$, $III$, $IV$, and $VII$}

The estimate for $II$ directly follows from Theorem \ref{Lem:finite QdF}:%
\begin{eqnarray}
&&\left\vert \limfunc{Tr}J_{t-s}^{(k)}II\right\vert
\label{estimate:II in convergence} \\
&\leqslant &\left\Vert J_{t-s}^{(k)}\right\Vert _{op}\left\Vert V_{N}\left(
x_{j}-x_{k+1},x_{j}-x_{k+2}\right) \right\Vert _{op}\limfunc{Tr}\left\vert
P_{\leqslant M}^{(k)}\gamma _{N}^{(k+2)}(s)P_{\leqslant M}^{(k)}-\tilde{%
\gamma}_{N,M}^{(k)}(s)\right\vert  \notag \\
&\leqslant &\left\Vert J^{(k)}\right\Vert _{op}N^{6\beta }\left\Vert
V\right\Vert _{L^{\infty }}\frac{4kM^{3}}{N}  \notag
\end{eqnarray}%
Since $M=N^{2\beta +\varepsilon }$, 
\begin{equation*}
\left\vert \limfunc{Tr}J_{t-s}^{(k)}II\right\vert \leqslant \frac{%
C_{J,V,k}N^{12\beta +3\varepsilon }}{N}\rightarrow 0\text{ as }N\rightarrow
\infty \text{, }
\end{equation*}%
as long as $\beta <\frac{1}{12}$.

After the estimates of $I$ and $II$, we are now dealing with a a
superposition of tensor product $\left( \left\vert \phi \right\rangle
\left\langle \phi \right\vert \right) ^{\otimes k+2}$ in $III$, $IV$, and $%
VII$. Hence we can follow the standard argument because we can prove Lemma %
\ref{lem:approximation of identity for products} for the approximation of
identity without an $\varepsilon $ loss. Since Lemma \ref{lem:approximation
of identity for products} only applies to tensor products, its proof is
elementary. However, we do need this version of the $H^{1}\rightarrow L^{6}$
Sobolev estimate to hold in sharp form, thus we prove it in detail. The
proof we provide is more like \cite[Lemma 8.2]{E-S-Y2} with a real space
method, instead of the Fourier method after \cite[Lemma A.2]{Kirpatrick},
since we are proving a $L^{6}$ estimate.\footnote{%
Putting in the Littlewood-Paley square functions might help but we are not
sure.}

\begin{lemma}
\label{lem:approximation of identity for products}Let $\rho \in L^{1}\left( 
\mathbb{T}^{6}\right) $ be a probability measure and let $\rho _{\alpha
}\left( x,y\right) =\alpha ^{-6}\rho \left( \frac{x}{\alpha },\frac{y}{%
\alpha }\right) $. Then, there exists $C>0$ s.t. 
\begin{align*}
\hspace{0.3in}& \hspace{-0.3in}\left\vert \limfunc{Tr}J^{(k)}\left( \rho
_{\alpha }\left( x_{j}-x_{k+1},x_{j}-x_{k+2}\right) -\delta \left(
x_{j}-x_{k+1}\right) \delta \left( x_{j}-x_{k+2}\right) \right) \left(
\left\vert \phi \right\rangle \left\langle \phi \right\vert \right)
^{\otimes k+2}\right\vert \\
& \leqslant C_{\rho }\alpha ^{\frac{1}{2}}C_{J}\limfunc{Tr}%
S_{j}^{2}S_{k+1}^{2}S_{k+2}^{2}\left( \left\vert \phi \right\rangle
\left\langle \phi \right\vert \right) ^{\otimes k+2}.
\end{align*}%
This estimate is not true if one replaces the tensor product $\left(
\left\vert \phi \right\rangle \left\langle \phi \right\vert \right)
^{\otimes k+2}$ by a general symmetric nonnegative $\gamma ^{(k+2)}\in 
\mathcal{L}^{1}\left( L^{2}\left( \mathbb{T}^{3k+6}\right) \right) $.
\end{lemma}

\begin{proof}
It suffices to prove the estimate for $k=1$. The trace, written out
explicitly, is%
\begin{eqnarray*}
&&\left\vert \limfunc{Tr}J^{(1)}\left( \rho _{\alpha }\left(
x_{1}-x_{2},x_{1}-x_{3}\right) -\delta \left( x_{1}-x_{2}\right) \delta
\left( x_{1}-x_{3}\right) \right) \left( \left\vert \phi \right\rangle
\left\langle \phi \right\vert \right) ^{\otimes 3}\right\vert \\
&=&\left\vert \left\langle \psi \otimes \phi ^{\otimes 2},\left( \rho
_{\alpha }\left( x_{1}-x_{2},x_{1}-x_{3}\right) -\delta \left(
x_{1}-x_{2}\right) \delta \left( x_{1}-x_{3}\right) \right) \phi ^{\otimes
3}\right\rangle \right\vert \\
&\leqslant &\int dx_{1}\left\vert \bar{\psi}(x_{1})\right\vert \left\vert
\phi (x_{1})\right\vert |\int \left[ \rho _{\alpha }\left(
x_{1}-x_{2},x_{1}-x_{3}\right) -\delta \left( x_{1}-x_{2}\right) \delta
\left( x_{1}-x_{3}\right) \right] \\
&&\times \left\vert \phi \right\vert ^{2}(x_{2})\left\vert \phi \right\vert
^{2}(x_{3})dx_{2}dx_{3}|
\end{eqnarray*}%
where $\psi =J^{(1)}\left( \phi \right) $. Since $\int \rho _{\alpha }=1$,
we rewrite%
\begin{eqnarray*}
&&|\int \left[ \rho _{\alpha }\left( x_{1}-x_{2},x_{1}-x_{3}\right) -\delta
\left( x_{1}-x_{2}\right) \delta \left( x_{1}-x_{3}\right) \right]
\left\vert \phi \right\vert ^{2}(x_{2})\left\vert \phi \right\vert
^{2}(x_{3})dx_{2}dx_{3}| \\
&=&|\int \rho _{\alpha }\left( x_{1}-x_{2},x_{1}-x_{3}\right) \left(
\left\vert \phi \right\vert ^{2}(x_{2})\left\vert \phi \right\vert
^{2}(x_{3})-\left\vert \phi \right\vert ^{2}(x_{1})\left\vert \phi
\right\vert ^{2}(x_{1})\right) dx_{2}dx_{3}| \\
&\leqslant &I(x_{1})+II(x_{1})
\end{eqnarray*}%
where%
\begin{eqnarray*}
I(x_{1}) &=&\int \rho _{\alpha }\left( x_{1}-x_{2},x_{1}-x_{3}\right)
\left\vert \left\vert \phi \right\vert ^{2}(x_{2})\left\vert \phi
\right\vert ^{2}(x_{3})-\left\vert \phi \right\vert ^{2}(x_{1})\left\vert
\phi \right\vert ^{2}(x_{3})\right\vert dx_{2}dx_{3}, \\
II(x_{1}) &=&\int \rho _{\alpha }\left( x_{1}-x_{2},x_{1}-x_{3}\right)
\left\vert \left\vert \phi \right\vert ^{2}(x_{1})\left\vert \phi
\right\vert ^{2}(x_{3})-\left\vert \phi \right\vert ^{2}(x_{1})\left\vert
\phi \right\vert ^{2}(x_{1})\right\vert dx_{2}dx_{3}.
\end{eqnarray*}%
Using the estimate that $\rho _{\alpha }(x,y)\leqslant \frac{C_{\rho }}{%
\left\vert B_{\alpha }\right\vert ^{2}}\chi _{B_{\alpha }}(x)\chi
_{B_{\alpha }}(y)$ where $B_{\alpha }=\left\{ x:\left\vert x\right\vert
\leqslant \alpha \right\} $, 
\begin{equation*}
I(x_{1})\leqslant \int \frac{C_{\rho }}{\left\vert B_{\alpha }\right\vert
^{2}}\chi _{B_{\alpha }}(x_{1}-x_{2})\chi _{B_{\alpha
}}(x_{1}-x_{3})\left\vert \phi \right\vert ^{2}(x_{3})\left\vert \left\vert
\phi \right\vert ^{2}(x_{2})-\left\vert \phi \right\vert
^{2}(x_{2})\right\vert dx_{2}dx_{3}
\end{equation*}%
Apply Poincar\'{e}'s inequality or just \cite[$\left( A.14\right) $]{E-S-Y2},%
\begin{eqnarray*}
I(x_{1}) &\leqslant &\int \frac{C_{\rho }}{\left\vert B_{\alpha }\right\vert 
}\chi _{B_{\alpha }}(x_{1}-x_{3})\left\vert \phi \right\vert ^{2}(x_{3})%
\frac{\chi _{B_{\alpha }}(x_{1}-x_{2})}{\left\vert x_{1}-x_{2}\right\vert
^{2}}\left\vert \nabla \left( \left\vert \phi \right\vert ^{2}\right)
(x_{2})\right\vert dx_{2}dx_{3} \\
&\leqslant &2\int \frac{C_{\rho }}{\left\vert B_{\alpha }\right\vert }\chi
_{B_{\alpha }}(x_{1}-x_{3})\left\vert \phi \right\vert ^{2}(x_{3})\frac{\chi
_{B_{\alpha }}(x_{1}-x_{2})}{\left\vert x_{1}-x_{2}\right\vert ^{2}}%
\left\vert \phi (x_{2})\nabla \phi (x_{2})\right\vert dx_{2}dx_{3}.
\end{eqnarray*}%
We can then estimate 
\begin{eqnarray*}
\left\Vert I\right\Vert _{L_{x_{1}}^{\frac{6}{5}}} &\leqslant &C_{\rho
}\left\Vert \int \frac{1}{\left\vert B_{\alpha }\right\vert }\chi
_{B_{\alpha }}(x_{1}-x_{3})\left\vert \phi \right\vert
^{2}(x_{3})dx_{3}\right\Vert _{L_{x_{1}}^{3}} \\
&&\left\Vert \int \frac{\chi _{B_{\alpha }}(x_{1}-x_{2})}{\left\vert
x_{1}-x_{2}\right\vert ^{2}}\left\vert \phi (x_{2})\nabla \phi
(x_{2})\right\vert dx_{2}\right\Vert _{L_{x_{1}}^{2}}.
\end{eqnarray*}%
Apply Young's inequality%
\begin{equation*}
\left\Vert I\right\Vert _{L_{x_{1}}^{\frac{6}{5}}}\leqslant C_{\rho
}\left\Vert \frac{1}{\left\vert B_{\alpha }\right\vert }\chi _{B_{\alpha
}}\right\Vert _{L^{1}}\left\Vert \phi \right\Vert
_{L_{x_{1}}^{6}}^{2}\left\Vert \frac{\chi _{B_{\alpha }}(x_{1})}{\left\vert
x_{1}\right\vert ^{2}}\right\Vert _{L_{x_{1}}^{\frac{6}{5}}}\left\Vert \phi
(x_{2})\nabla \phi (x_{2})\right\Vert _{L_{x_{2}}^{\frac{3}{2}}}.
\end{equation*}%
Use H\"{o}lder and Sobolev%
\begin{equation}
\left\Vert I\right\Vert _{L_{x_{1}}^{\frac{6}{5}}}\leqslant C_{\rho
}\left\Vert \phi \right\Vert _{L_{x_{1}}^{6}}^{2}\alpha ^{\frac{1}{2}%
}\left\Vert \phi \right\Vert _{L_{x_{1}}^{6}}\left\Vert \nabla \phi
\right\Vert _{L_{x_{2}}^{2}}\leqslant C_{\rho }\alpha ^{\frac{1}{2}%
}\left\Vert \nabla \phi \right\Vert _{L_{x_{1}}^{2}}^{4}.
\label{Estimate:I in Trace Lemma}
\end{equation}%
Similarly, for $II(x_{1})$, we have%
\begin{eqnarray*}
II(x_{1}) &\leqslant &\int \frac{C_{\rho }}{\left\vert B_{\alpha
}\right\vert ^{2}}\chi _{B_{\alpha }}(x_{1}-x_{2})\chi _{B_{\alpha
}}(x_{1}-x_{3})\left\vert \phi \right\vert ^{2}(x_{1})\left\vert \left\vert
\phi \right\vert ^{2}(x_{3})-\left\vert \phi \right\vert
^{2}(x_{1})\right\vert dx_{2}dx_{3} \\
&=&C_{\rho }\int \frac{1}{\left\vert B_{\alpha }\right\vert }\chi
_{B_{\alpha }}(x_{1}-x_{3})\left\vert \phi \right\vert ^{2}(x_{1})\left\vert
\left\vert \phi \right\vert ^{2}(x_{3})-\left\vert \phi \right\vert
^{2}(x_{1})\right\vert dx_{3} \\
&\leqslant &C_{\rho }\left\vert \phi \right\vert ^{2}(x_{1})\int \frac{\chi
_{B_{\alpha }}(x_{1}-x_{3})}{\left\vert x_{1}-x_{3}\right\vert ^{2}}%
\left\vert \phi (x_{3})\nabla \phi (x_{3})\right\vert dx_{3}
\end{eqnarray*}%
and hence%
\begin{eqnarray}
\left\Vert II\right\Vert _{L_{x_{1}}^{\frac{6}{5}}} &\leqslant &C_{\rho
}\left\Vert \left\vert \phi \right\vert ^{2}\right\Vert
_{L_{x_{1}}^{3}}\left\Vert \int \frac{\chi _{B_{\alpha }}(x_{1}-x_{3})}{%
\left\vert x_{1}-x_{3}\right\vert ^{2}}\left\vert \phi (x_{3})\nabla \phi
(x_{3})\right\vert dx_{3}\right\Vert _{L_{x_{1}}^{2}}
\label{Estimate:II in Trace Lemma} \\
&\leqslant &C_{\rho }\left\Vert \left\vert \phi \right\vert ^{2}\right\Vert
_{L_{x_{1}}^{3}}\left\Vert \frac{\chi _{B_{\alpha }}(x_{1})}{\left\vert
x_{1}\right\vert ^{2}}\right\Vert _{L_{x_{1}}^{\frac{6}{5}}}\left\Vert \phi
(x_{3})\nabla \phi (x_{3})\right\Vert _{L_{x_{3}}^{\frac{3}{2}}}  \notag \\
&\leqslant &C_{\rho }\alpha ^{\frac{1}{2}}\left\Vert \nabla \phi \right\Vert
_{L_{x_{1}}^{2}}^{4}.  \notag
\end{eqnarray}%
Finally, since%
\begin{eqnarray}
\left\Vert \phi J^{(1)}\left( \phi \right) \right\Vert _{L^{6}} &\leqslant
&\left\Vert \int J^{(1)}(x_{1},x_{1}^{\prime })\phi (x_{1}^{\prime
})dx_{1}^{\prime }\right\Vert _{L_{x_{1}}^{\infty }}\left\Vert \phi
\right\Vert _{L^{6}}  \label{Estimate:Leftover in Trace Lemma} \\
&\leqslant &\left\Vert J^{(1)}(x_{1},x_{1}^{\prime })\right\Vert
_{L_{x_{1}}^{\infty }L_{x_{1}^{\prime }}^{\frac{6}{5}}}\left\Vert \phi
\right\Vert _{L^{6}}^{2}  \notag \\
&\leqslant &C\left\Vert J^{(1)}(x_{1},x_{1}^{\prime })\right\Vert
_{L_{x_{1}}^{\infty }L_{x_{1}^{\prime }}^{\frac{6}{5}}}\left\Vert \nabla
\phi \right\Vert _{L^{2}}^{2}.  \notag
\end{eqnarray}%
Putting (\ref{Estimate:I in Trace Lemma}), (\ref{Estimate:II in Trace Lemma}%
), and (\ref{Estimate:Leftover in Trace Lemma}) together, we have%
\begin{eqnarray*}
&&\left\vert \limfunc{Tr}J^{(1)}\left( \rho _{\alpha }\left(
x_{1}-x_{2},x_{1}-x_{3}\right) -\delta \left( x_{1}-x_{2}\right) \delta
\left( x_{1}-x_{3}\right) \right) \left( \left\vert \phi \right\rangle
\left\langle \phi \right\vert \right) ^{\otimes 3}\right\vert \\
&\leqslant &CC_{\rho }\left\Vert J^{(1)}(x_{1},x_{1}^{\prime })\right\Vert
_{L_{x_{1}}^{\infty }L_{x_{1}^{\prime }}^{\frac{6}{5}}}\alpha ^{\frac{1}{2}%
}\left\Vert \nabla \phi \right\Vert _{L^{2}}^{6} \\
&=&CC_{\rho }\left\Vert J^{(1)}(x_{1},x_{1}^{\prime })\right\Vert
_{L_{x_{1}}^{\infty }L_{x_{1}^{\prime }}^{\frac{6}{5}}}\alpha ^{\frac{1}{2}}%
\limfunc{Tr}S_{1}^{2}S_{2}^{2}S_{3}^{2}\left( \left\vert \phi \right\rangle
\left\langle \phi \right\vert \right) ^{\otimes 3}
\end{eqnarray*}%
as claimed. So we have proved Lemma \ref{lem:approximation of identity for
products}.

We remark that, one cannot use Young's inequality to get (\ref{Estimate:I in
Trace Lemma}) or (\ref{Estimate:II in Trace Lemma}) if one replaces the
tensor product $\left( \left\vert \phi \right\rangle \left\langle \phi
\right\vert \right) ^{\otimes k+2}$ by a general symmetric nonnegative $%
\gamma ^{(k+2)}\in \mathcal{L}^{1}\left( L^{2}\left( \mathbb{T}%
^{3k+6}\right) \right) $. Hence, we are not violating the fact that the
sharp trace theorem is false.
\end{proof}

By Corollary \ref{Corollary:Regularity for Tensor product finite qde}, we
can use Lemma \ref{lem:approximation of identity for products} and conclude%
\begin{eqnarray*}
\left\vert \limfunc{Tr}J_{t-s}^{(k)}III\right\vert &\leqslant &b_{0}C_{\rho
}N^{-\frac{\beta }{2}}C_{J}\int_{S(P_{\leqslant M}\left( L_{s}^{2}(\mathbb{T}%
^{3})\right) )}\limfunc{Tr}S_{j}^{2}S_{k+1}^{2}S_{k+2}^{2}\left\vert \phi
^{\otimes k+2}\right\rangle \left\langle \phi ^{\otimes k+2}\right\vert d\mu
_{N,k,s}(\phi ) \\
&=&b_{0}C_{V}N^{-\frac{\beta }{2}}C_{J}\limfunc{Tr}%
S_{j}^{2}S_{k+1}^{2}S_{k+2}^{2}\tilde{\gamma}_{N,M}^{(k+2)} \\
&\leqslant &C_{J,V}N^{-\frac{\beta }{2}} \\
&\rightarrow &0\text{ as }N\rightarrow \infty \text{.}
\end{eqnarray*}%
and%
\begin{eqnarray*}
\left\vert \limfunc{Tr}J_{t-s}^{(k)}IV\right\vert &\leqslant &b_{0}C_{\rho
}\alpha ^{\frac{1}{2}}C_{J}\int_{S(P_{\leqslant M}\left( L_{s}^{2}(\mathbb{T}%
^{3})\right) )}\limfunc{Tr}S_{j}^{2}S_{k+1}^{2}S_{k+2}^{2}\left\vert \phi
^{\otimes k+2}\right\rangle \left\langle \phi ^{\otimes k+2}\right\vert d\mu
_{N,k,s}(\phi ) \\
&=&b_{0}C_{\rho }\alpha ^{\frac{1}{2}}C_{J}\limfunc{Tr}%
S_{j}^{2}S_{k+1}^{2}S_{k+2}^{2}\tilde{\gamma}_{N,M}^{(k+2)} \\
&\leqslant &C_{J,V}\alpha ^{\frac{1}{2}} \\
&\rightarrow &0\text{ as }\alpha \rightarrow 0\text{ uniform in }N\text{.}
\end{eqnarray*}%
For $VII$, we can use the ordinary quantum de Finetti theorem and represent $%
\gamma ^{(k+2)}=\int_{\mathbb{B(}L^{2}\mathbb{)}}\left( \left\vert \phi
\right\rangle \left\langle \phi \right\vert \right) ^{\otimes k+2}d\mu (\phi
)$.

\begin{definition}
\label{def:Admissible}Given a sequence of nonnegative symmetric trace class
operators $\Gamma =\left\{ \gamma ^{(k)}\right\} _{k=1}^{\infty }$. We say $%
\Gamma $ is weakly admissible if it is a weak* limit of a bosonic quantum $N$%
-body dynamic and strongly admissible if $\Gamma $ satisfies $\limfunc{Tr}%
\gamma ^{(k)}=1$ and $\limfunc{Tr}_{k+1}\gamma ^{(k+1)}=\gamma ^{(k)}$ for
all $k$.
\end{definition}

\begin{theorem}[\protect\cite{Lewin}]
\label{thm:qde}Given a weakly (respectively strongly)\ admissible sequence
of nonnegative symmetric trace class operators $\Gamma =\left\{ \gamma
^{(k)}\right\} _{k=1}^{\infty }$, there exists a probability measure $d\mu $
supported on $\mathbb{B(}L^{2}\mathbb{)}$ (respectively $\mathbb{S}(L^{2})$)
such that%
\begin{equation*}
\gamma ^{(k)}=\int_{\mathbb{B(}L^{2}\mathbb{)}}\left( \left\vert \phi
\right\rangle \left\langle \phi \right\vert \right) ^{\otimes k}d\mu (\phi )%
\text{ (respectively }\gamma ^{(k)}=\int_{\mathbb{S(}L^{2}\mathbb{)}}\left(
\left\vert \phi \right\rangle \left\langle \phi \right\vert \right)
^{\otimes k}d\mu (\phi )\text{)}
\end{equation*}
\end{theorem}

Then we can still use Lemma \ref{lem:approximation of identity for products}%
, and conclude that%
\begin{eqnarray*}
\left\vert \limfunc{Tr}J_{t-s}^{(k)}VII\right\vert &\leqslant &b_{0}C_{\rho
}\alpha ^{\frac{1}{2}}C_{J}\int_{\mathbb{B(}L^{2}\mathbb{)}}\limfunc{Tr}%
S_{j}^{2}S_{k+1}^{2}S_{k+2}^{2}\left\vert \phi ^{\otimes k+2}\right\rangle
\left\langle \phi ^{\otimes k+2}\right\vert d\mu _{s}(\phi ) \\
&=&b_{0}C_{\rho }\alpha ^{\frac{1}{2}}C_{J}\limfunc{Tr}%
S_{j}^{2}S_{k+1}^{2}S_{k+2}^{2}\gamma ^{(k+2)} \\
&\leqslant &C_{J,V}\alpha ^{\frac{1}{2}} \\
&\rightarrow &0\text{ as }\alpha \rightarrow 0\text{ uniform in }N\text{.}
\end{eqnarray*}%
We remind the readers that, with our method, the estimate for $VI$ follows
from the standard argument \cite{Kirpatrick}, and hence we skip it.
Therefore, we have proved (\ref{eqn:key convergence term})$\rightarrow $(\ref%
{eqn:key convergence target}) as $N\rightarrow \infty $ and thence Theorem %
\ref{Theorem:Convergence} for $\beta <\frac{1}{12}$.

\subsection{Convergence of the Key Coupling Term for $\protect\beta <1/9 
\label{Sec:Convergence 1/9}$}

Let us recall our target, we need to prove that, for a fixed smooth kernel $%
J^{(k)}\left( \mathbf{x}_{k},\mathbf{x}_{k}^{\prime }\right) $ and
corresponding operator $J^{(k)}$ 
\begin{equation}
\sup_{0\leq t\leq T}\left\vert \func{Tr}J^{(k)}\int_{0}^{t}[U^{(k)}(t-s)F(s,%
\bullet ;\bullet )](\mathbf{x}_{k};\mathbf{x}_{k}^{\prime })ds\right\vert
\rightarrow 0  \label{eqn:recall target in convergence}
\end{equation}%
and $N\rightarrow \infty $ where 
\begin{eqnarray*}
&&F(s,\mathbf{x}_{k};\mathbf{x}_{k}^{\prime }) \\
&=&\func{Tr}_{k+1,k+2}\left[ V_{N}(x_{j}-x_{k+1},x_{j}-x_{k+2}),\gamma
_{N}^{(k+2)}(s)\right] \\
&&-\func{Tr}_{k+1,k+2}\left[ b_{0}\delta (x_{j}-x_{k+1})\delta
(x_{j}-x_{k+2}),\gamma ^{(k+2)}(s)\right]
\end{eqnarray*}%
By writing out the trace explicitly, we can rewrite (\ref{eqn:recall target
in convergence}) as 
\begin{equation*}
\sup_{0\leq t\leq T}\left\vert \int_{0}^{t}\int_{\mathbf{x}_{k}}\int_{%
\mathbf{x}_{k}^{\prime }}J^{(k)}\left( \mathbf{x}_{k}^{\prime },\mathbf{x}%
_{k}\right) \left[ U^{(k)}(t-s)F(s,\bullet ;\bullet )\right] (\mathbf{x}_{k};%
\mathbf{x}_{k}^{\prime })\,d\mathbf{x}_{k}\,d\mathbf{x}_{k}^{\prime
}\,ds\right\vert \rightarrow 0
\end{equation*}%
Or%
\begin{equation*}
\sup_{0\leq t\leq T}\left\vert \int_{0}^{t}\int_{\mathbf{x}_{k}}\int_{%
\mathbf{x}_{k}^{\prime }}[U^{(k)}(s-t)J^{(k)}\left( \bullet ^{\prime
},\bullet \right) ](\mathbf{x}_{k};\mathbf{x}_{k}^{\prime })F(s,\mathbf{x}%
_{k};\mathbf{x}_{k}^{\prime })]\,d\mathbf{x}_{k}\,d\mathbf{x}_{k}^{\prime
}\,ds\right\vert \rightarrow 0
\end{equation*}%
if we put $U^{(k)}$ on $J^{(k)}$. Since $J^{(k)}$ is smooth, we have $\Vert
U^{(k)}J^{(k)}\Vert _{L^{\infty }}\leqslant C_{J}$ and $\Vert
U^{(k)}J^{(k)}\Vert _{L^{2}}\leqslant C_{J}$. Thus it suffices to show 
\begin{equation*}
\Vert F(t,\mathbf{x}_{k};\mathbf{x}_{k}^{\prime })\Vert _{L_{0\leq t\leq
T}^{1}L_{\hat{\mathbf{x}}_{k}\mathbf{x}_{k}^{\prime
}}^{2}L_{x_{j}}^{1}}\rightarrow 0
\end{equation*}%
where $\hat{\mathbf{x}}_{k}$ means $x_{j}$ omitted. We will in fact prove 
\begin{equation}
\Vert F(t,\mathbf{x}_{k};\mathbf{x}_{k}^{\prime })\Vert _{L_{0\leq t\leq
T}^{\infty }L_{\hat{\mathbf{x}}_{k}\mathbf{x}_{k}^{\prime
}}^{2}L_{x_{j}}^{1}}\rightarrow 0  \label{eqn:real target in convergence 1/9}
\end{equation}%
and thus we drop the $t$ dependence from the notation. Moreover, by
Parseval, we see that $J^{(k)}$ in fact carries a frequency localization to
effectively force $\left\vert \xi _{i}\right\vert ,\left\vert \xi
_{i}^{\prime }\right\vert \lesssim Q$, $\forall i\leqslant k$, for some
frequency $Q$ independent of $N$. Thus we will prove (\ref{eqn:real target
in convergence 1/9}) with $\left\vert \xi _{i}\right\vert ,\left\vert \xi
_{i}^{\prime }\right\vert \lesssim Q$, $\forall i\leqslant k$,

By triangle inequality, it suffices to treat one summand in the commutator,
and, without lose of generality, we may assume $j=k$ so that we can write
out the variables explicitly and avoid heavy notation. That is, 
\begin{eqnarray*}
F(\mathbf{x}_{k},\mathbf{x}_{k}^{\prime }) &=&[\int
dx_{k+1}dx_{k+2}V_{N}(x_{k}-x_{k+1},x_{k}-x_{k+2}) \\
&&\times \gamma _{N}^{(k+2)}(\mathbf{x}_{k},x_{k+1},x_{k+2};\mathbf{x}%
_{k},x_{k+1},x_{k+2})] \\
&&-\gamma ^{(k+2)}(\mathbf{x}_{k},x_{k},x_{k};\mathbf{x}_{k},x_{k},x_{k})
\end{eqnarray*}

We take two frequency thresholds $M$ and $R\footnote{%
Together with $Q$, we actually have three frequency thresholds here. But we
only mark down $M$ and $R$ because $Q$ is fixed.}$ this time where $M\gg Q$
and is independent of $N$ and $N^{\beta }\ll R\ll N^{(1-6\beta )/3}$. In
this setting, for sufficiently large $N$, we have 
\begin{equation*}
Q\ll M\ll N^{\beta }\ll R\ll N^{(1-6\beta )/3}.
\end{equation*}%
Since we are only working with kernels here, let us adopt the notation $%
P_{\leqslant M}^{j^{\prime }}$ ($P_{>M}^{j^{\prime }}$) which denotes the
Littlewood-Paley projector which projects the $x_{j}^{\prime }$ variable at
frequency $\leqslant M$ ($>M$), and write 
\begin{equation*}
P_{\leqslant M}^{(k^{\prime })}=\dprod\limits_{j=1}^{k}P_{\leqslant
M}^{j^{\prime }}\text{ (}P_{>M}^{(k^{\prime
})}=\dprod\limits_{j=1}^{k^{\prime }}P_{>M}^{j}\text{).}
\end{equation*}%
We decompose $F=F_{1}+F_{2}+F_{3}-F_{4}$, where%
\begin{eqnarray*}
&&F_{1}(\mathbf{x}_{k},\mathbf{x}_{k}^{\prime }) \\
&=&\int dx_{k+1}dx_{k+2}V_{N}(x_{k}-x_{k+1},x_{k}-x_{k+2}) \\
&&\times \left[ \left( 1-P_{\leqslant R}^{(k+2)}P_{\leqslant
R}^{(k+2^{\prime })}\right) \gamma _{N}^{(k+2)}\right] (\mathbf{x}%
_{k},x_{k+1},x_{k+2};\mathbf{x}_{k}^{\prime },x_{k+1},x_{k+2})
\end{eqnarray*}%
\begin{eqnarray*}
&&F_{2}(\mathbf{x}_{k},\mathbf{x}_{k}^{\prime }) \\
&=&\int dx_{k+1}dx_{k+2}V_{N}(x_{k}-x_{k+1},x_{k}-x_{k+2}) \\
&&\times \left[ \left( P_{\leqslant R}^{(k+2)}P_{\leqslant R}^{(k+2^{\prime
})}-P_{\leqslant M}^{(k+2)}P_{\leqslant M}^{(k+2^{\prime })}\right) \gamma
_{N}^{(k+2)}\right] (\mathbf{x}_{k},x_{k+1},x_{k+2};\mathbf{x}_{k}^{\prime
},x_{k+1},x_{k+2})
\end{eqnarray*}%
\begin{eqnarray*}
&&F_{3}(\mathbf{x}_{k},\mathbf{x}_{k}^{\prime }) \\
&=&\int dx_{k+1}dx_{k+2}V_{N}(x_{k}-x_{k+1},x_{k}-x_{k+2}) \\
&&\times \left( P_{\leqslant M}^{(k+2)}P_{\leqslant M}^{(k+2^{\prime
})}\gamma _{N}^{(k+2)}\right) (\mathbf{x}_{k},x_{k+1},x_{k+2};\mathbf{x}%
_{k}^{\prime },x_{k+1},x_{k+2}) \\
&&-b_{0}\left( P_{\leqslant M}^{(k+2)}P_{\leqslant M}^{(k+2^{\prime
})}\gamma ^{(k+2)}\right) (\mathbf{x}_{k},x_{k},x_{k};\mathbf{x}_{k}^{\prime
},x_{k},x_{k})
\end{eqnarray*}%
and 
\begin{equation*}
F_{4}=b_{0}\left[ \left( 1-P_{\leqslant M}^{(k+2)}P_{\leqslant
M}^{(k+2^{\prime })}\right) \gamma ^{(k+2)}\right] (\mathbf{x}%
_{k},x_{k},x_{k};\mathbf{x}_{k}^{\prime },x_{k},x_{k})
\end{equation*}

\subsubsection{Estimate for $F_{1}$}

$F_{1}$, like the $I$ term in \S \ref{Sec:Convergence 1/12}, is again the
most interesting term. We will prove that $\Vert F_{1}(t,\mathbf{x}_{k};%
\mathbf{x}_{k}^{\prime })\Vert _{L_{0\leq t\leq T}^{\infty }L_{\mathbf{x}%
_{k-1}\mathbf{x}_{k}^{\prime }}^{2}L_{x_{k}}^{1}}\rightarrow 0$ as $%
N\rightarrow \infty $, independent of $M$. To see the interaction between
the frequencies, let us at the moment\ switch to the partial Fourier side 
\begin{align*}
& \hat{F}_{1}(\mathbf{x}_{k-1},\xi _{k};\mathbf{x}_{k}^{\prime }) \\
=& \int d\xi _{k+1}\,d\xi _{k+2}\,d\xi _{k+1}^{\prime }\,d\xi _{k+2}^{\prime
}\hat{V}_{N}(\xi _{k+1}+\xi _{k+1}^{\prime };\xi _{k+2}+\xi _{k+2}^{\prime })%
\left[ \left( 1-P_{\leqslant R}^{(k+2)}P_{\leqslant R}^{(k+2^{\prime
})}\right) \gamma _{N}^{(k+2)}\right] ^{\symbol{94}}(\mathbf{x}_{k-1}, \\
& \quad \xi _{k}-\xi _{k+1}-\xi _{k+1}^{\prime }-\xi _{k+2}-\xi
_{k+2}^{\prime },\xi _{k+1},\xi _{k+2};\mathbf{x}_{k}^{\prime },\xi
_{k+1}^{\prime },\xi _{k+2}^{\prime })\,.
\end{align*}%
Due to the factor $\left( 1-P_{\leqslant R}^{(k+2)}P_{\leqslant
R}^{(k+2^{\prime })}\right) $, at least one of the following components is
at frequency $\geqslant R$

\begin{itemize}
\item $\left\vert \xi _{k+1}\right\vert >R$; in this case, $\left\vert \xi
_{k+1}+\xi _{k+1}^{\prime }\right\vert \lesssim N^{\beta }\ll R$
(effectively) due to the $\hat{V}_{N}$ term, and thus $\left\vert \xi
_{k+1}^{\prime }\right\vert \gtrsim R$.

\item $\left\vert \xi _{k+2}\right\vert >R$; in this case, $\left\vert \xi
_{k+2}+\xi _{k+2}^{\prime }\right\vert \lesssim N^{\beta }\ll R$
(effectively) due to the $\hat{V}_{N}$ term, and thus $\left\vert \xi
_{k+2}^{\prime }\right\vert \gtrsim R$.

\item $\left\vert \xi _{k+1}^{\prime }\right\vert >R$; in this case, $%
\left\vert \xi _{k+1}+\xi _{k+1}^{\prime }\right\vert \lesssim N^{\beta }\ll
R$ (effectively) due to the $\hat{V}_{N}$ term, and thus $\left\vert \xi
_{k+1}\right\vert \gtrsim R$.

\item $\left\vert \xi _{k+2}^{\prime }\right\vert >R$; in this case, $%
\left\vert \xi _{k+2}+\xi _{k+2}^{\prime }\right\vert \lesssim N^{\beta }\ll
R$ (effectively) due to the $\hat{V}_{N}$ term, and thus $\left\vert \xi
_{k+2}\right\vert \gtrsim R$.

\item the $k$ component; in this case $\left\vert \xi _{k}-\xi _{k+1}-\xi
_{k+1}^{\prime }-\xi _{k+2}-\xi _{k+2}^{\prime }\right\vert \geqslant R$,
and due to the presence of the $V_{N}$ term, we still have (effectively) $%
\left\vert \xi _{k+1}+\xi _{k+1}^{\prime }\right\vert \lesssim N^{\beta }\ll
R$ and $\left\vert \xi _{k+2}+\xi _{k+2}^{\prime }\right\vert \lesssim
N^{\beta }\ll R$, and thus $\left\vert \xi _{k}\right\vert \gtrsim R$. But
this case cannot arise since we only consider the case $\left\vert \xi
_{i}\right\vert ,\left\vert \xi _{i}^{\prime }\right\vert \lesssim Q$, $%
\forall i\leqslant k$ for some frequency $Q$ independent of $N$ because of
the effect of $J^{(k)}$. Same argument shows that $\xi _{i}$, $\xi
_{i}^{\prime }$ cannot be high, $\forall i\leqslant k$.
\end{itemize}

We notice that in all four possible cases, either the $(k+1)^{\prime }$ or $%
(k+2)^{\prime }$ component will be at high $\gtrsim R$ frequency. Say it is
the $(k+2)^{\prime }$ component being high, we can then write it out as%
\begin{eqnarray*}
&&\Vert F_{1}(t,\mathbf{x}_{k};\mathbf{x}_{k}^{\prime })\Vert
_{L_{x_{k}}^{1}} \\
&=&\int dx_{k}dx_{k+1}dx_{k+2}V_{N}(x_{k}-x_{k+1},x_{k}-x_{k+2}) \\
&&\left\vert \int \psi _{N}(\mathbf{x}_{k},x_{k+1},x_{k+2},\mathbf{x}%
_{N-k-2})\left( P_{>R}^{k+2^{\prime }}\bar{\psi}_{N}\right) (\mathbf{x}%
_{k}^{\prime },x_{k+1},x_{k+2},\mathbf{x}_{N-k-2})d\mathbf{x}%
_{N-k-2}\right\vert
\end{eqnarray*}%
then Cauchy-Schwarz to obtain%
\begin{eqnarray*}
&\leqslant &\left( \int V_{N}(x_{k}-x_{k+1},x_{k}-x_{k+2})\left\vert \psi
_{N}\right\vert ^{2}(\mathbf{x}_{k},x_{k+1},x_{k+2},\mathbf{x}_{N-k-2})d%
\mathbf{x}_{N-k+1}\right) ^{\frac{1}{2}} \\
&&[\int dx_{k}dx_{k+1}dx_{k+2}d\mathbf{x}%
_{N-k-2}V_{N}(x_{k}-x_{k+1},x_{k}-x_{k+2}) \\
&&\times \left\vert P_{>R}^{k+2^{\prime }}\psi _{N}\right\vert ^{2}(\mathbf{x%
}_{k}^{\prime },x_{k+1},x_{k+2},\mathbf{x}_{N-k-2})]^{\frac{1}{2}}
\end{eqnarray*}%
That is,%
\begin{eqnarray*}
&&\Vert F_{1}(t,\mathbf{x}_{k};\mathbf{x}_{k}^{\prime })\Vert _{L_{0\leq
t\leq T}^{\infty }L_{\mathbf{x}_{k-1}\mathbf{x}_{k}^{\prime
}}^{2}L_{x_{k}}^{1}} \\
&\leqslant &\left( \int V_{N}(x_{k}-x_{k+1},x_{k}-x_{k+2})\left\vert \psi
_{N}\right\vert ^{2}(\mathbf{x}_{k},x_{k+1},x_{k+2},\mathbf{x}_{N-k-2})d%
\mathbf{x}_{N}\right) ^{\frac{1}{2}} \\
&&\left( \int V_{N}(x_{k}-x_{k+1},x_{k}-x_{k+2})\left\vert
P_{>R}^{k+2^{\prime }}\psi _{N}\right\vert ^{2}(\mathbf{x}_{k}^{\prime
},x_{k+1},x_{k+2},\mathbf{x}_{N-k-2})dx_{k}d\mathbf{x}_{N-k}d\mathbf{x}%
_{k}^{\prime }\right) ^{\frac{1}{2}} \\
&=&I^{\frac{1}{2}}\times II
\end{eqnarray*}%
Recall that $I$ is part of the energy $\langle N^{-1}H_{N}\psi _{N},\psi
_{N}\rangle $ and we are dealing with a defocusing case, so 
\begin{equation*}
I\lesssim C^{\frac{1}{2}}
\end{equation*}%
For the second term, noticing that extra $dx_{k}$, we can use the trace
theorem at the $x_{k+1}$ or the $x_{k+2}$ variable to obtain 
\begin{equation*}
II\lesssim \Vert \langle \nabla _{x_{k+1}}\rangle ^{\frac{3}{4}+}\langle
\nabla _{x_{k+2}}\rangle ^{\frac{3}{4}+}P_{>R}^{k+2^{\prime }}\psi _{N}(%
\mathbf{x}_{k}^{\prime },x_{k+1},x_{k+2},\mathbf{x}_{N-k-2})\Vert _{L_{%
\mathbf{x}_{N-k},\mathbf{x}_{k}^{\prime }}^{2}}
\end{equation*}%
Since the $(k+2)^{\prime }$ component is at frequency $\gtrsim R$, we can
obtain the gain in passing to full derivatives 
\begin{equation*}
II\lesssim R^{-\frac{1}{4}+}\Vert \langle \nabla _{k+1}\rangle \langle
\nabla _{k+2}\rangle \psi _{N}\Vert _{L_{\mathbf{x}_{N}}^{2}}\leqslant CR^{-%
\frac{1}{4}+}
\end{equation*}%
Recall $N^{\beta }\ll R\ll N^{(1-6\beta )/3}$, we have obtained that%
\begin{equation*}
\Vert F_{1}(t,\mathbf{x}_{k};\mathbf{x}_{k}^{\prime })\Vert _{L_{0\leq t\leq
T}^{\infty }L_{\mathbf{x}_{k-1}\mathbf{x}_{k}^{\prime
}}^{2}L_{x_{k}}^{1}}\leqslant CR^{-\frac{1}{4}+}\rightarrow 0\text{ as }%
N\rightarrow \infty
\end{equation*}%
independent of $M$.

\subsubsection{Estimates for $F_{2}$, $F_{3}$, and $F_{4}$}

$F_{2}$ and $F_{4}$ are the same type of terms. We first estimate $F_{4}$.
Due to $\left( 1-P_{\leqslant M}^{(k+2)}P_{\leqslant M}^{(k+2^{\prime
})}\right) $ inside the definition of $F_{4}$, some frequency is $>M$. Since 
$\left\vert \xi _{i}\right\vert ,\left\vert \xi _{i}^{\prime }\right\vert
\lesssim Q$, $\forall i\leqslant k$ and $Q\ll M$, we can assume at least one
of $x_{k+1},x_{k+2},x_{k+1}^{\prime },x_{k+2}^{\prime }$ is at high $>M$.
Say $x_{k+1}$ is at high in a typical term in $F_{4}$, then, by Theorem \ref%
{thm:qde}, this typical term reads%
\begin{eqnarray*}
&&\left\Vert P_{>M}^{k+1}\gamma ^{(k+2)}(\mathbf{x}_{k-1},x_{k},x_{k},x_{k};%
\mathbf{x}_{k}^{\prime },x_{k},x_{k})\right\Vert _{L_{0\leq t\leq T}^{\infty
}L_{\mathbf{x}_{k-1}\mathbf{x}_{k}^{\prime }}^{2}L_{x_{k}}^{1}} \\
&=&\sup_{0\leq t\leq T}\int_{\mathbb{B(}L^{2}\mathbb{)}}\left\Vert \phi
\right\Vert _{L^{2}}^{k-1}\left\Vert \phi \right\Vert _{L^{2}}^{k}\left\Vert
\left( P_{>M}\phi \right) \bar{\phi}\phi \bar{\phi}\phi \right\Vert
_{L^{1}}d\mu _{t}(\phi ) \\
&\leqslant &\sup_{0\leq t\leq T}\int_{\mathbb{B(}L^{2}\mathbb{)}}\left\Vert
\phi \right\Vert _{L^{2}}^{2k-1}\left\Vert P_{>M}\phi \right\Vert
_{L^{5}}\left\Vert \phi \right\Vert _{L^{5}}^{4}d\mu _{t}(\phi )
\end{eqnarray*}%
By Sobolev, 
\begin{eqnarray*}
&\leqslant &C\sup_{0\leq t\leq T}\int_{\mathbb{B(}L^{2}\mathbb{)}}\left\Vert
\phi \right\Vert _{L^{2}}^{2k-1}\left\Vert \left\vert \nabla \right\vert ^{%
\frac{9}{10}}P_{>M}\phi \right\Vert _{L^{2}}\left\Vert \left\vert \nabla
\right\vert ^{\frac{9}{10}}\phi \right\Vert _{L^{2}}^{4}d\mu _{t}(\phi ) \\
&\leqslant &CM^{-\frac{1}{10}}\sup_{0\leq t\leq T}\int_{\mathbb{B(}L^{2}%
\mathbb{)}}\left\Vert \phi \right\Vert _{L^{2}}^{2k-1}\left\Vert \left\vert
\nabla \right\vert P_{>M}\phi \right\Vert _{L^{2}}\left\Vert \left\langle
\nabla \right\rangle \phi \right\Vert _{L^{2}}^{4}d\mu _{t}(\phi ) \\
&\leqslant &CM^{-\frac{1}{10}}\sup_{0\leq t\leq T}\int_{\mathbb{B(}L^{2}%
\mathbb{)}}\left\Vert \phi \right\Vert _{L^{2}}^{2k-2}\left\Vert
\left\langle \nabla \right\rangle \phi \right\Vert _{L^{2}}^{6}d\mu
_{t}(\phi ) \\
&=&CM^{-\frac{1}{10}}\sup_{0\leq t\leq T}\limfunc{Tr}%
S_{k}^{2}S_{k+1}^{2}S_{k+2}^{2}\gamma ^{(k+2)}\leqslant CM^{-\frac{1}{10}%
}C^{3}.
\end{eqnarray*}%
That is%
\begin{equation*}
\left\Vert F_{4}\right\Vert _{L_{0\leq t\leq T}^{\infty }L_{\mathbf{x}_{k-1}%
\mathbf{x}_{k}^{\prime }}^{2}L_{x_{k}}^{1}}\leqslant C_{k}M^{-\frac{1}{10}}.
\end{equation*}%
In other words, $F_{4}\rightarrow 0$ as $M\rightarrow \infty $ independent
of $N$.

Like the $II$ term in \S \ref{Sec:Convergence 1/12}, with Theorem \ref%
{Lem:finite QdF}, one can estimate $F_{2}$ in the same way as $F_{4}$ at the
price of a $CR^{3}N^{6\beta }N^{-1}$ error, because%
\begin{equation*}
\left( P_{\leqslant R}^{(k+2)}P_{\leqslant R}^{(k+2^{\prime })}-P_{\leqslant
M}^{(k+2)}P_{\leqslant M}^{(k+2^{\prime })}\right) =\left( 1-P_{\leqslant
M}^{(k+2)}P_{\leqslant M}^{(k+2^{\prime })}\right) P_{\leqslant
R}^{(k+2)}P_{\leqslant R}^{(k+2^{\prime })}.
\end{equation*}%
The final estimate of $F_{2}$ reads%
\begin{equation*}
\left\Vert F_{2}\right\Vert _{L_{0\leq t\leq T}^{\infty }L_{\mathbf{x}_{k-1}%
\mathbf{x}_{k}^{\prime }}^{2}L_{x_{k}}^{1}}\leqslant CR^{3}N^{6\beta
}N^{-1}+C_{k}M^{-\frac{1}{10}}\text{.}
\end{equation*}%
Recall $R\ll N^{(1-6\beta )/3}$ and $\beta <\frac{1}{9}$, therefore $%
F_{2}\rightarrow 0$ if we first let $N\rightarrow \infty $ then let $%
M\rightarrow \infty $.

$F_{3}$ is in fact the simplest term\footnote{%
One can even use the nonsharp Sobolev trace theorems to estimate $F_{3}$ at
the price of a $M^{+}$ since we only need $\lim_{M\rightarrow \infty
}\lim_{N\rightarrow \infty }F_{3}=0$.} to estimate due to the localization $%
P_{\leqslant M}^{(k+2)}P_{\leqslant M}^{(k+2^{\prime })}$ independent of $N$%
. In fact, if we denote the kernel of $P_{\leqslant M}$ by $\chi _{M}$, the
smoothing and compact operator kernel%
\begin{eqnarray*}
K_{N} &=&\int V_{N}(x_{j}-x_{k+1},x_{j}-x_{k+2})\chi
_{M}(x_{k+1}-y_{k+1})\chi _{M}(x_{k+2}-y_{k+2}) \\
&&\times \chi _{M}(x_{k+1}-y_{k+1}^{\prime })\chi
_{M}(x_{k+2}-y_{k+2}^{\prime })dx_{k+1}dx_{k+2}
\end{eqnarray*}%
tends to the smoothing and compact operator kernel 
\begin{equation*}
K=b_{0}\chi _{M}(x_{j}-y_{k+1})\chi _{M}(x_{j}-y_{k+2})\chi
_{M}(x_{j}-y_{k+1}^{\prime })\chi _{M}(x_{j}-y_{k+2}^{\prime })
\end{equation*}%
as $N\rightarrow \infty $ for every finite $M$. Thus a Cantor
diagonalization argument together with the weak* convergence show that $%
\lim_{N\rightarrow \infty }\Vert F_{3}(t,\mathbf{x}_{k};\mathbf{x}%
_{k}^{\prime })\Vert _{L_{0\leq t\leq T}^{\infty }L_{\mathbf{x}_{k-1}\mathbf{%
x}_{k}^{\prime }}^{2}L_{x_{k}}^{1}}=0$ for every $M$. We skip the details
here.

Up to this point, we have proved the convergence of the key coupling term
for $\beta <1/9$, and thence Theorem \ref{Theorem:Convergence} for $\beta <%
\frac{1}{9}$.

\section{Uniqueness for Large Solutions to the Energy-critical GP Hierarchy 
\label{sec:uniqueness}}

\begin{definition}[HUFL]
\label{def:HUFL}We say that $\Gamma (t)=\left\{ \gamma ^{(k)}(t)\right\}
_{k=1}^{\infty }\in \oplus _{k\geqslant 1}\mathcal{L}_{k}^{1}$ satisfies 
\emph{hierarchically uniform frequency localization} at time $t$ (HUFL) if $%
\forall \varepsilon >0$, $\exists M(t,\varepsilon )$ such that 
\begin{equation}
\limfunc{Tr}S^{(1,k)}P_{>M}^{(k)}\gamma
^{(k)}(t)P_{>M}^{(k)}S^{(1,k)}\leqslant \varepsilon ^{2k}\ \text{for all }k%
\text{.}  \label{condition:HUFL}
\end{equation}
\end{definition}

\begin{theorem}
\label{Thm:TotalUniqueness}Assume the initial datum $\left\{ \gamma
^{(k)}(0)\right\} _{k=1}^{\infty }$ to hierarchy (\ref{hierarchy:quintic GP
in differential form}), with $b_{0}\geqslant 0$, satisfies HUFL. Then there
is exactly one admissible\footnote{%
Here, "admissible" is as defined in Definition \ref{def:Admissible}.}
solution to hierarchy (\ref{hierarchy:quintic GP in differential form}) in $%
[0,+\infty )$ subject to this initial datum satisfying the condition that
there is a constant $C_{0}$ such that 
\begin{equation}
\sup_{t\in \left[ 0,T\right] }\limfunc{Tr}S^{(1,k)}\gamma
^{(k)}(t)S^{(1,k)}\leqslant C_{0}^{2k}\text{ for all }k.
\label{condition:kinetic energy}
\end{equation}
\end{theorem}

\begin{remark}
\label{Remark:QtoGigliola}We are not sure if Theorem \ref%
{Thm:TotalUniqueness} should be called a GP conditional uniqueness theorem
or an GP unconditional uniqueness theorem. Theorem \ref{Thm:TotalUniqueness}
requires more than (\ref{condition:kinetic energy}), the usual assumption of
the GP unconditional uniqueness theorems (see \cite%
{E-S-Y2,TCNPdeFinitte,HoTaXi14,HTX,Sohinger3,C-PUniqueness,C-HFocusingIII},
for examples). However, it does not require a Strichartz type bound like the
GP conditional uniqueness theorems (see \cite%
{KlainermanAndMachedon,Kirpatrick,TChenAndNP,ChenAnisotropic,Sohinger,C-HFocusing, HerrSohinger}%
, for examples). If one applies Theorem \ref{Thm:TotalUniqueness} to
factorized datum $\gamma ^{(k)}(0)=\left\vert \phi _{0}\right\rangle
\left\langle \phi _{0}\right\vert ^{\otimes k}$, then the HUFL condition is
trivially satisfied and Theorem \ref{Thm:TotalUniqueness} implies the
unconditional uniqueness of (\ref{equation:TargetQuinticNLS}) at critical
regularity.
\end{remark}

\begin{corollary}
There is at most one $H^{1}$ solution to (\ref{equation:TargetQuinticNLS}).
\end{corollary}

The proof of Theorem \ref{Thm:TotalUniqueness} is splitted into two main
parts, Theorems \ref{Thm:UTFLinTime} and \ref{THM:UniquessAssumingUFL}. We
prove Theorem \ref{Thm:UTFLinTime} in \S \ref{subsec:GPUTFL} and Theorem \ref%
{THM:UniquessAssumingUFL} in \S \ref{subsec:GPUTFLUniqueness}.

\subsection{Local-in-time HUFL Estimates\label{subsec:GPUTFL}}

\begin{theorem}
\label{Thm:UTFLinTime}Assume the initial datum $\left\{ \gamma
^{(k)}(0)\right\} _{k=1}^{\infty }$ to hierarchy (\ref{hierarchy:quintic GP
in differential form}) satisfies HUFL. Let $\varepsilon >0$ and $M>0$ and
suppose that 
\begin{equation}
\limfunc{Tr}S^{(1,k)}P_{>M}^{(k)}\gamma
^{(k)}(0)P_{>M}^{(k)}S^{(1,k)}\leqslant \left( \tfrac{\varepsilon }{4}%
\right) ^{2k}\ \text{for all }k\text{.}  \label{eqn:HUFL}
\end{equation}%
Then there exists $T>0$, depending on $\varepsilon >0$ and $M>0$, such that
any admissible solution $\Gamma =\left\{ \gamma ^{(k)}\right\}
_{k=1}^{\infty }$ $\in \oplus _{k\geqslant 1}C\left( \left[ 0,T\right] ,%
\mathcal{L}_{k}^{1}\right) $ to (\ref{hierarchy:quintic GP in differential
form}) in $\left[ 0,T\right] $, subject to this initial datum, satisfying (%
\ref{condition:kinetic energy}), must satisfy 
\begin{equation}
\sup_{t\in \left[ 0,T\right] }\limfunc{Tr}S^{(1,k)}P_{>M}^{(k)}\gamma
^{(k)}(t)P_{>M}^{(k)}S^{(1,k)}\leqslant \varepsilon ^{2k}\ \text{for all }k%
\text{.}  \label{HUFLinTime}
\end{equation}
\end{theorem}

\begin{remark}
We note that the length of the time interval $T$ in Theorem \ref%
{Thm:UTFLinTime} depends on $\varepsilon >0$ and $M>0$ (which in turn
depends on the specifics of the initial condition). We cannot, in fact,
claim the above result for arbitrarily long times, or even for short times
independent of $\varepsilon $ or the initial condition. Nevertheless,
Theorem \ref{Thm:UTFLinTime} is a type of uniform-time,
uniform-in-hierarchy, frequency localization estimate, and is sufficient to
prove Theorem \ref{THM:UniquessAssumingUFL} and hence Theorem \ref%
{Thm:TotalUniqueness}.
\end{remark}

By scaling, it is enough to show%
\begin{equation}
\sup_{t\in \left[ 0,T\right] }\limfunc{Tr}R^{(1,k)}P_{>M}^{(k)}\gamma
^{(k)}(t)P_{>M}^{(k)}R^{(1,k)}\leqslant \varepsilon ^{2k}
\label{estimate:HUFLinTime}
\end{equation}%
where $R^{(\alpha ,k)}=\dprod\limits_{j=}^{k}\left\vert \nabla
_{x_{j}}\right\vert ^{\alpha }$. Due to the hierachical structure of (\ref%
{hierarchy:quintic GP in differential form}), we cannot directly estimate%
\begin{equation*}
\left\vert \partial _{t}\limfunc{Tr}R^{(1,k)}P_{>M}^{(k)}\gamma
^{(k)}(t)P_{>M}^{(k)}R^{(1,k)}\right\vert .\footnote{%
One could try computing this quantity like in the proof of Lemma \ref%
{lem:intermediate kinetic energy is continuous} and see the difficulties.}
\end{equation*}%
We have to use the energy as our passage.

For shorter formulas, let us adopt the abbreviations $-\Delta
_{x_{j}}=-\Delta _{j}$, and 
\begin{eqnarray*}
B^{(k+2)}\gamma ^{(k+2)} &=&\sum_{j=1}^{k}B_{j;k+1,k+2}\gamma ^{(k+2)} \\
&=&\sum_{j=1}^{k}\left( B_{j;k+1,k+2}^{+}-B_{j;k+1,k+2}^{-}\right) \gamma
^{(k+2)} \\
&=&\sum_{j=1}^{k}\limfunc{Tr}\nolimits_{k+1,k+2}\left[ \delta
(x_{j}-x_{k+1})\delta (x_{j}-x_{k+2}),\gamma ^{(k+2)}\right] .
\end{eqnarray*}%
and let $b_{0}=1$. Then written in kernel form, (\ref{hierarchy:quintic GP
in differential form}) is 
\begin{equation}
i\partial _{t}\gamma ^{(k)}=-\Delta _{\mathbf{k}}\gamma ^{(k)}+\Delta _{%
\mathbf{k}^{\prime }}\gamma
^{(k)}+\sum_{j=1}^{k}(B_{j;k+1,k+2}^{+}-B_{j;k+1,k+2}^{-})\gamma ^{(k+2)}.
\label{E:GP}
\end{equation}

Let 
\begin{equation*}
H^{j}=\func{Tr}_{j\rightarrow j+2}\left( -\Delta _{j}\right) +\func{Tr}_{j}%
\frac{1}{3}B_{j;j+1,j+2}^{+}
\end{equation*}%
and define the $k$-th order GP energy by 
\begin{equation*}
E^{(k)}=H^{1}H^{4}\cdots H^{3k-2}\gamma ^{(3k)},
\end{equation*}%
then it is proved in \cite{TChenAndNP2} that $E^{(k)}$ is a conserved
quantity. Actually, if $\gamma ^{(k)}(t)$ is indeed our target $\left\vert
\phi (t)\right\rangle \left\langle \phi (t)\right\vert ^{\otimes k}$ with $%
\phi $ solving (\ref{equation:TargetQuinticNLS}), then $E^{(k)}$ reduces to 
\begin{equation*}
E^{(k)}=(E_{\text{NLS}})^{k}
\end{equation*}%
with $E_{\text{NLS}}(\phi )=\int |\nabla \phi |^{2}+\frac{1}{3}\int |\phi
|^{6}$.

We will define the high and low pieces, $H_{L}^{j}$ and $H_{H}^{j}$
respectively, of $H^{j}$, so that we can define the high piece of $E^{(k)}$
by 
\begin{equation*}
E_{H}^{(k)}=H_{H}^{1}H_{H}^{4}\cdots H_{H}^{3k-2}\gamma ^{(3k)}.
\end{equation*}%
Because $P_{\geqslant M}^{(k)}+$ $P_{<M}^{(k)}\neq 1$, it is not true that $%
E^{(k)}-E_{H}^{(k)}=H_{L}^{1}H_{L}^{4}\cdots H_{L}^{3k-2}\gamma ^{(3k)}$. To
help understanding the definition of $H_{L}^{j}$ and $H_{H}^{j}$, we first
write out what they are for the NLS (\ref{eqn:3d Quintic with 1 Coupling}).

For NLS, we can decompose $E_{\text{NLS}}$ into a low frequency component $%
E_{\text{NLS}}^{L}$ and a high frequency component $E_{\text{NLS}}^{H}$, as
follows: In the expansion of $|\phi _{L}+\phi _{H}|^{6}$, all of the terms
that have $0$, $1$, or $2$ of $\phi _{H}$ or $\bar{\phi}_{H}$ are placed
into $E_{\text{NLS}}^{L}$, and all of the terms that have $3$, $4$, $5$, or $%
6$ of $\phi _{H}$ or $\bar{\phi}_{H}$ go into $E_{\text{NLS}}^{H}$. The
expansion is 
\begin{equation*}
|\phi _{H}+\phi _{L}|^{6}=(\phi _{H}+\phi _{L})^{3}(\bar{\phi}_{H}+\bar{\phi}%
_{L})^{3}
\end{equation*}%
and thus the coefficients can be determined by counting the number of bar
and non bar H terms.

\begin{itemize}
\item zero H terms: $|\phi_L|^6$

\item one H terms: This breaks into two cases

\begin{itemize}
\item If the H term is non bar, then there are three possibilities, so $%
3\bar\phi_L^3 \phi_L^2\phi_H = 3|\phi_L|^2 \bar \phi_L \phi_H$.

\item If the H term is bar, then there are three possibilities, so $%
3\phi_L^3 \bar \phi_L ^2\bar \phi_H = 3 |\phi_L|^2 \phi_L \bar \phi_H$
\end{itemize}

\item two H terms: This breaks into three cases

\begin{itemize}
\item 2 H are non bar, zero H are bar. There are 3 ways to choose the non
bar L term, so $3\phi_H^2\phi_L \bar \phi_L^3 = 3\phi_H^2 \bar \phi_L^2
|\phi_L|^2 $

\item 2 H are bar, zero H are non bar. There are 3 ways to choose the bar L
term, so $3\bar \phi_H^2 \bar \phi_L \phi_L^3 = 3\bar \phi_H^2 \phi_L^2
|\phi_L|^2$

\item 1 H is bar, 1 H is non bar. There are 3 ways to choose the non bar H
term, and 3 ways to choose the bar H term, so 9 total possibilities, so $%
9\phi _{H}\bar{\phi}_{H}\phi _{L}^{2}\bar{\phi}_{L}^{2}=9|\phi
_{H}|^{2}|\phi _{L}|^{4}$.
\end{itemize}
\end{itemize}

In summary, we define 
\begin{eqnarray}
E_{\text{NLS}}^{L} &=&\int |\nabla \phi _{H}|^{2}+\frac{1}{3}\int (|\phi
_{L}|^{6}+3|\phi _{L}|^{2}\bar{\phi}_{L}\phi _{H}+3|\phi _{L}|^{2}\phi _{L}%
\bar{\phi}_{H}  \label{eqn:E_NLS^L} \\
&&+3\phi _{H}^{2}\bar{\phi}_{L}^{2}|\phi _{L}|^{2}+3\bar{\phi}_{H}^{2}\phi
_{L}^{2}|\phi _{L}|^{2}+9|\phi _{H}|^{2}|\phi _{L}|^{4})  \notag
\end{eqnarray}%
and 
\begin{equation*}
E_{\text{NLS}}^{H}=E_{\text{NLS}}-E_{\text{NLS}}^{L}
\end{equation*}

We can now define $H_{H}^{j}$ and $H_{L}^{j}$. From the definition of $E_{%
\text{NLS}}^{L}$ (\ref{eqn:E_NLS^L}), we see that we will have to write out
a large sum of high/low projections of many variables explicitly in the
interaction part when defining $H_{H}^{j}$ and $H_{L}^{j}$ and hence proving
Theorem \ref{Thm:UTFLinTime}. To shorten the formulas, for a fixed $M$, we
shorten the notation by 
\begin{equation*}
P_{H}^{j}=P_{>M}^{j},P_{L}^{j}=P_{\leqslant M}^{j},P_{H}^{j^{\prime
}}=P_{>M}^{j^{\prime }},\;P_{L}^{j^{\prime }}=P_{\leqslant M}^{j^{\prime }}
\end{equation*}

The $H_{L}^{j}$ consists of terms in which $0$, $1$, or $2$ variables in the
interaction lie at high frequency. The possiblities are

\begin{itemize}
\item 0 variable is at high: $P_{L}^{1}P_{L}^{2}P_{L}^{3}P_{L}^{1^{\prime
}}P_{L}^{2^{\prime }}P_{L}^{3^{\prime }}$

\item 1 variable is at high: This breaks into two cases

\begin{itemize}
\item 1 nonprime variable is at high, then there are three possibilities, so 
$3P_{H}^{1}P_{L}^{2}P_{L}^{3}P_{L}^{1^{\prime }}P_{L}^{2^{\prime
}}P_{L}^{3^{\prime }}.$

\item 1 prime variable is at high, then there are three possibilities, so $%
3P_{L}^{1}P_{L}^{2}P_{L}^{3}P_{H}^{1^{\prime }}P_{L}^{2^{\prime
}}P_{L}^{3^{\prime }}.$
\end{itemize}

\item Two variables are at high: This breaks into three cases

\begin{itemize}
\item 2 nonprime variables are at high and zero prime variable is at high.
There are 3 ways to choose the low nonprime variable, so $%
3P_{L}^{1}P_{H}^{2}P_{H}^{3}P_{L}^{1^{\prime }}P_{L}^{2^{\prime
}}P_{L}^{3^{\prime }}.$

\item 2 prime variables are at high and zero nonprime variable is at high.
There are 3 ways to choose the low prime variable, so $%
3P_{L}^{1}P_{L}^{2}P_{L}^{3}P_{L}^{1^{\prime }}P_{H}^{2^{\prime
}}P_{H}^{3^{\prime }}.$

\item 1 prime variables are at high and 1 nonprime variable is at high.
There are 3 ways to choose the high nonprime variable, and 3 ways to choose
the high prime variable, that is 9 total possibilities, so $%
9P_{H}^{1}P_{L}^{2}P_{L}^{3}P_{H}^{1^{\prime }}P_{L}^{2^{\prime
}}P_{L}^{3^{\prime }}.$
\end{itemize}
\end{itemize}

By the above analysis, we define 
\begin{align}
H_{L}^{j}& =\func{Tr}_{j\rightarrow j+2}\left( -\Delta _{j}P_{L}^{j}\right) +%
\frac{1}{3}\func{Tr}%
_{j}B_{j;j+1,j+2}^{+}[P_{L}^{j}P_{L}^{j+1}P_{L}^{j+2}P_{L}^{j^{\prime
}}P_{L}^{(j+1)^{\prime }}P_{L}^{(j+2)^{\prime }}  \label{eqn:H_L^j} \\
& \qquad +3P_{H}^{j}P_{L}^{j+1}P_{L}^{j+2}P_{L}^{j^{\prime
}}P_{L}^{(j+1)^{\prime }}P_{L}^{(j+2)^{\prime
}}+3P_{L}^{j}P_{L}^{j+1}P_{L}^{j+2}P_{H}^{j^{\prime }}P_{L}^{(j+1)^{\prime
}}P_{L}^{(j+2)^{\prime }}  \notag \\
& \qquad +3P_{L}^{j}P_{H}^{j+1}P_{H}^{j+2}P_{L}^{j^{\prime
}}P_{L}^{(j+1)^{\prime }}P_{L}^{(j+2)^{\prime
}}+3P_{L}^{j}P_{L}^{j+1}P_{L}^{j+2}P_{L}^{j^{\prime }}P_{H}^{(j+1)^{\prime
}}P_{H}^{(j+2)^{\prime }}  \notag \\
& \qquad +9P_{H}^{j}P_{L}^{j+1}P_{L}^{j+2}P_{H}^{j^{\prime
}}P_{L}^{(j+1)^{\prime }}P_{L}^{(j+2)^{\prime }}].  \notag
\end{align}%
We can then define $H_{H}^{j}$ so that 
\begin{equation*}
\func{Tr}_{1\rightarrow 3}(H^{1}-H_{L}^{1}-H_{H}^{1})\rho ^{(3)}=0
\end{equation*}%
for any symmetric $\rho ^{(3)}$, but importantly we remark that $H^{1}\neq
H_{L}^{1}+H_{H}^{1}$

The $H_{H}^{j}$ consists of terms in which $3$, $4$, $5$, or $6$ variables
in the interaction lie at high frequency. The possiblities are

\begin{itemize}
\item 3 variables are at high. This breaks into subclasses

\begin{itemize}
\item 3 nonprime variables at high, 0 prime variables at high. 1 possibility%
\newline
$P^j_H P^{j+1}_H P_H^{j+2} P_L^{j^{\prime }} P_L^{(j+1)^{\prime }}
P_L^{(j+2)^{\prime }}$

\item 2 nonprime variables at high (equivalently 1 nonprime at low), 1 prime
variable at high. There are $3\times 3=9$ ways to choose. $9 P^j_L P^{j+1}_H
P^{j+2}_H P_H^{j^{\prime }} P_L^{(j+1)^{\prime }} P_L^{(j+2)^{\prime }}$

\item 1 nonprime variables at high, 2 prime variables at high (equivalently
1 prime at low). There are $3\times 3=9$ waves to choose. $9 P^j_H P^{j+1}_L
P^{j+2}_L P_L^{j^{\prime }} P^{(j+1)^{\prime }}_H P^{(j+2)^{\prime }}_H$.

\item 0 nonprime variables at high, 3 prime variables at high. 1 possibility%
\newline
$P^j_L P^{j+1}_L P^{j+2}_L P^{j^{\prime }}_H P^{(j+1)^{\prime }}_H
P^{(j+2)^{\prime }}_H$.
\end{itemize}

\item 4 variable are at high, or equivalently 2 variables are at low. This
breaks into subclasses

\begin{itemize}
\item 0 nonprime variables at low, 2 prime variables at low (equivalently 1
prime at high). There are $3$ ways to choose. $3P^j_H P^{j+1}_H P^{j+2}_H
P^{j^{\prime }}_H P^{(j+1)^{\prime }}_L P^{(j+2)^{\prime }}_L$.

\item 1 nonprime variables at low, 1 prime variables at low. There are $%
3\times 3=9$ ways to choose. $9P^j_L P^{j+1}_H P^{j+2}_H P^{j^{\prime }}_L
P^{(j+1)^{\prime }}_H P^{(j+2)^{\prime }}_H$.

\item 2 nonprime variables at low (equivalently 1 prime at high), 0 prime
variables at low. There are $3$ ways to choose. $3P^j_H P^{j+1}_L P^{j+2}_L
P^{j^{\prime }}_H P^{(j+1)^{\prime }}_H P^{(j+2)^{\prime }}_H$.
\end{itemize}

\item 5 variable are at high, or equivalently 1 variables are at low. This
breaks into subclasses

\begin{itemize}
\item 0 nonprime variables at low, 1 prime variables at low. There are $3$
ways to choose. $3P^j_H P^{j+1}_H P^{j+2}_H P^{j^{\prime }}_L
P^{(j+1)^{\prime }}_H P^{(j+2)^{\prime }}_H$.

\item 1 nonprime variables at low, 0 prime variables at low. There are $3$
ways to choose. $3P^j_L P^{j+1}_H P^{j+2}_H P^{j^{\prime }}_H
P^{(j+1)^{\prime }}_H P^{(j+2)^{\prime }}_H$.
\end{itemize}

\item 6 variable are at high, or equivalently 0 variables are at low. There
is only one possibility. $P^j_H P^{j+1}_H P^{j+2}_H P^{j^{\prime }}_H
P^{(j+1)^{\prime }}_H P^{(j+2)^{\prime }}_H$.
\end{itemize}

By the above analysis, we define%
\begin{eqnarray}  \label{E:highdef}
H_{H}^{j} &=&\func{Tr}_{j\rightarrow j+2}\left( -\Delta _{j}P_{H}^{j}\right)
+\frac{1}{3}\func{Tr}%
_{j}B_{j;j+1,j+2}^{+}[P_{H}^{j}P_{H}^{j+1}P_{H}^{j+2}P_{L}^{j^{\prime
}}P_{L}^{(j+1)^{\prime }}P_{L}^{(j+2)^{\prime }}  \label{eqn:H_H^j} \\
&&+9P_{L}^{j}P_{H}^{j+1}P_{H}^{j+2}P_{H}^{j^{\prime }}P_{L}^{(j+1)^{\prime
}}P_{L}^{(j+2)^{\prime }}+9P_{H}^{j}P_{L}^{j+1}P_{L}^{j+2}P_{L}^{j^{\prime
}}P_{H}^{(j+1)^{\prime }}P_{H}^{(j+2)^{\prime }}  \notag \\
&&+P_{L}^{j}P_{L}^{j+1}P_{L}^{j+2}P_{H}^{j^{\prime }}P_{H}^{(j+1)^{\prime
}}P_{H}^{(j+2)^{\prime }}+3P_{H}^{j}P_{H}^{j+1}P_{H}^{j+2}P_{H}^{j^{\prime
}}P_{L}^{(j+1)^{\prime }}P_{L}^{(j+2)^{\prime }}  \notag \\
&&+9P_{L}^{j}P_{H}^{j+1}P_{H}^{j+2}P_{L}^{j^{\prime }}P_{H}^{(j+1)^{\prime
}}P_{H}^{(j+2)^{\prime }}+3P_{H}^{j}P_{L}^{j+1}P_{L}^{j+2}P_{H}^{j^{\prime
}}P_{H}^{(j+1)^{\prime }}P_{H}^{(j+2)^{\prime }}  \notag \\
&&+3P_{H}^{j}P_{H}^{j+1}P_{H}^{j+2}P_{L}^{j^{\prime }}P_{H}^{(j+1)^{\prime
}}P_{H}^{(j+2)^{\prime }}+3P_{L}^{j}P_{H}^{j+1}P_{H}^{j+2}P_{H}^{j^{\prime
}}P_{H}^{(j+1)^{\prime }}P_{H}^{(j+2)^{\prime }}  \notag \\
&&+P_{H}^{j}P_{H}^{j+1}P_{H}^{j+2}P_{H}^{j^{\prime }}P_{H}^{(j+1)^{\prime
}}P_{H}^{(j+2)^{\prime }}].  \notag
\end{eqnarray}

Now, let us recall $E_{H}^{(k)}=H_{H}^{1}H_{H}^{4}\cdots H_{H}^{3k-2}\gamma
^{(3k)}$. There are three main parts of the proof of Theorem \ref%
{Thm:UTFLinTime}. In the first part, \S \ref{Sec:Proof of Main High Energy
Estimate}, our goal is to establish 
\begin{equation}
\left\vert \partial _{t}E_{H}^{(k)}\right\vert \lesssim k\left\langle
M\right\rangle ^{2}\left( E^{(k)}\right) ^{1/k}\left( E_{H}^{(k)}\right)
^{(k-1)/k}  \label{E:main-1}
\end{equation}%
If we can achieve \eqref{E:main-1}\footnote{%
For our purpose here $E_{H}^{(k)}$ is always a nonnegative quantity, see \S %
\ref{Sec:Proof of High Kinetic Eats High Energy}.}, and we assume that there
is a $C>0$ such that, $E^{(k)}\leqslant C^{k}$ for all $k$, then (\ref%
{E:main-1}) implies%
\begin{equation}
\left\vert \partial _{t}\left[ \left( E_{H}^{(k)}\right) ^{1/k}\right]
\right\vert \lesssim C\langle M\rangle ^{2},
\label{stimate:E_high^1/k's final derivative estimate}
\end{equation}%
since 
\begin{equation*}
\partial _{t}\left[ \left( E_{H}^{(k)}\right) ^{1/k}\right] =\frac{1}{k}%
\left( E_{H}^{(k)}\right) ^{-(k-1)/k}\left( \partial _{t}E_{H}^{(k)}\right) .
\end{equation*}%
Integrate (\ref{stimate:E_high^1/k's final derivative estimate}) on $[0,T]$,
and using that \eqref{eqn:HUFL} implies $\left\vert
E_{H}^{(k)}(0)\right\vert \leqslant (\frac{1}{4}\varepsilon )^{2k}$, we
obtain 
\begin{equation}
\left( E_{H}^{(k)}\left( t\right) \right) ^{1/k}\leqslant (\tfrac{1}{4}%
\varepsilon )^{2}+C\langle M\rangle ^{2}T\text{, }\forall t\in \lbrack
0,T].\,  \label{E:main-2}
\end{equation}%
That is, if $T=\frac{1}{8}\left( C\langle M\rangle ^{2}\right)
^{-1}\varepsilon ^{2}$, we have%
\begin{equation}
\left( E_{H}^{(k)}\left( t\right) \right) ^{1/k}\leqslant \frac{\varepsilon
^{2}}{4}\text{, }\forall t\in \lbrack 0,T].
\label{eq conclusion of high energy growth}
\end{equation}

In the second part, \S \ref{Sec:Proof of High Kinetic Eats High Energy},
using (\ref{eq conclusion of high energy growth}), we argue by contradiction
that as long as $T\leqslant \frac{1}{8}\left( C\langle M\rangle ^{2}\right)
^{-1}\varepsilon ^{2}$, (\ref{estimate:HUFLinTime}) holds with the aid of
the refined Sobolev inequalities in Lemma \ref{lemma:refined sobolev for
continuity}, assuming that the quantity 
\begin{equation}
K_{M\leq \bullet \leq R}(t)=\lim_{k\rightarrow \infty }\left( \limfunc{Tr}%
R^{(1,k)}P_{<R}^{(k)}P_{>M}^{(k)}\gamma
^{(k)}(t_{b})P_{<R}^{(k)}P_{>M}^{(k)}R^{(1,k)}\right) ^{\frac{1}{k}},
\label{eqn:def of intermediate kinetic}
\end{equation}%
which can be understood as the intermediate kinetic energy, is a continuous
function in time $t$. Notice that, if we denote $\mu _{t}(\phi )$ the
quantum de Finetti measure, associated with $\gamma ^{(k)}(t)$ at any time $%
t $, $K_{M\leq \bullet \leq R}(t)$ actually becomes 
\begin{eqnarray*}
K_{M\leq \bullet \leq R}(t) &=&\lim_{k\rightarrow \infty }\left( \limfunc{Tr}%
R^{(1,k)}P_{<R}^{(k)}P_{>M}^{(k)}\gamma
^{(k)}(t_{b})P_{<R}^{(k)}P_{>M}^{(k)}R^{(1,k)}\right) ^{\frac{1}{k}} \\
&=&\lim_{k\rightarrow \infty }\left( \int_{\phi }\Vert \nabla \phi _{M\leq
\bullet \leq R}\Vert _{L_{x}^{2}}^{2k}d\mu _{t}(\phi )\right) ^{1/k} \\
&=&\lim_{k\rightarrow \infty }\Vert \Vert \nabla \phi _{M\leq \bullet \leq
R}\Vert _{L_{x}^{2}}^{2}\Vert _{L_{d\mu _{t}(\phi )}^{2k}} \\
&=&\Vert \Vert \nabla \phi _{M\leq \bullet \leq R}\Vert _{L^{2}}^{2}\Vert
_{L_{d\mu _{t}(\phi )}^{\infty }}.
\end{eqnarray*}%
Because of that $L_{d\mu _{t}(\phi )}^{\infty }$, the continuity of $%
K_{M\leq \bullet \leq R}(t)$ could be interpreted a strong continuity result
of the de Finetti measure.

Finally, in the third part, \S \ref{Sec:Proof of High Kinetic Continuity},
we prove, as Lemma \ref{lem:intermediate kinetic energy is continuous}, that 
$K_{M\leq \bullet \leq R}(t)$ the intermediate kinetic energy, is indeed a
continuous function in time $t$. \S \ref{Sec:Proof of Main High Energy
Estimate}-\ref{Sec:Proof of High Kinetic Continuity} proves Theorem \ref%
{Thm:UTFLinTime}.

We see that the proof of Theorem \ref{Thm:UTFLinTime} is very delicate. For
the NLS case, we compute in Appendix \ref{Sec:AppendixUTFL} that 
\begin{equation}
\left\vert \partial _{t}E_{\text{NLS}}^{H}\right\vert =\left\vert \partial
_{t}(E_{\text{NLS}}-E_{\text{NLS}}^{L})\right\vert =\left\vert \partial
_{t}E_{\text{NLS}}^{L}\right\vert \leqslant C\langle M\rangle ^{2}E_{\text{%
NLS}}  \label{eqn:key estimate for NLS UTFL in GP UTFL}
\end{equation}%
by exchanging derivatives off of the high frequency terms onto the low
frequency terms, and allowing for the $M^{2}$ penalty using Bernstein.
Though we have defined $H_{L}^{j}$, $H_{H}^{j}$ for (\ref{E:GP}) similarly,
the proof of the NLS case in Appendix \ref{Sec:AppendixUTFL} does not go
through for (\ref{E:GP}). If one uses the NLS argument for GP, even if one
assumes Theorem \ref{thm:qde} which at a glance, makes the NLS problem and
the GP problem look similar, one will end up with a $k-$depending $t$ in the
end and hence fails completely. This is certainly because special solutions (%
\ref{eqn:3d Quintic with 1 Coupling}) cannot imply the general case (\ref%
{E:GP}). In fact, notice that, when $k>1$, (\ref{E:main-1}) is more accurate
than (\ref{eqn:key estimate for NLS UTFL in GP UTFL}) even if one assumes a
pure tensor product solution to (\ref{E:GP}). On the one hand, derivatives
do not work very well with even a pure tensor product when it has an
unbounded amount of factors and hence making any uniform in $k$ estimates
difficult. On the other hand, as we have already pointed out earlier, 
\begin{equation*}
E^{(k)}-E_{H}^{(k)}\neq \func{Tr}H_{L}^{1}H_{L}^{4}\cdots H_{L}^{3k-2}\gamma
^{(3k)}
\end{equation*}%
because $P_{\geqslant M}^{(k)}+$ $P_{<M}^{(k)}\neq 1$, that $\partial _{t}E_{%
\text{NLS}}^{H}=\partial _{t}(E_{\text{NLS}}-E_{\text{NLS}}^{L})$ middle
step in (\ref{eqn:key estimate for NLS UTFL in GP UTFL}) which greatly helps
the argument, does not exist. Therefore, the proof of Theorem \ref%
{Thm:UTFLinTime} has to be much more accurate and subtle.

\subsubsection{Proof of (\protect\ref{E:main-1})\label{Sec:Proof of Main
High Energy Estimate}}

We start with 
\begin{equation*}
\partial _{t}E_{H}^{(k)}=H_{H}^{1}H_{H}^{4}\cdots H_{H}^{3k-2}\partial
_{t}\gamma ^{(3k)}
\end{equation*}%
Now plug in \eqref{E:GP} with $k$ replaced by $3k$ 
\begin{align*}
\partial _{t}E_{H}^{(k)}& =i\sum_{j=1}^{3k}(H_{H}^{1}H_{H}^{4}\cdots
H_{H}^{3k-2})(\Delta _{j}-\Delta _{j^{\prime }})\gamma ^{(3k)} \\
& \qquad -i\sum_{j=1}^{3k}(H_{H}^{1}H_{H}^{4}\cdots
H_{H}^{3k-2})(B_{j;3k+1,3k+2}^{+}-B_{j;3k+1,3k+2}^{-})\gamma ^{(3k+2)}
\end{align*}%
By symmetry 
\begin{equation*}
i\sum_{j=1}^{3k}(H_{H}^{1}H_{H}^{4}\cdots H_{H}^{3k-2})(\Delta _{j}-\Delta
_{j^{\prime }})\gamma ^{(3k)}=ik\sum_{j=1}^{3}(H_{H}^{1}H_{H}^{4}\cdots
H_{H}^{3k-2})(\Delta _{j}-\Delta _{j^{\prime }})\gamma ^{(3k)}
\end{equation*}%
Let 
\begin{equation}
\sigma ^{(3)}(x_{1},x_{2},x_{3};x_{1}^{\prime },x_{2}^{\prime
},x_{3}^{\prime })=H_{H}^{4}\cdots H_{H}^{3k-2}\gamma ^{(3k)}
\label{eqn:sigma_3}
\end{equation}%
Then since $\Delta _{j}-\Delta _{j^{\prime }}$ and $H_{H}^{4}\cdots
H_{H}^{3k-2}$ commute 
\begin{equation*}
i\sum_{j=1}^{3k}(H_{H}^{1}H_{H}^{4}\cdots H_{H}^{3k-2})(\Delta _{j}-\Delta
_{j^{\prime }})\gamma ^{(3k)}=ik\sum_{j=1}^{3}H_{H}^{1}(\Delta _{j}-\Delta
_{j^{\prime }})\sigma ^{(3)}
\end{equation*}%
By symmetry 
\begin{eqnarray}
&&i\sum_{j=1}^{3k}(H_{H}^{1}H_{H}^{4}\cdots
H_{H}^{3k-2})(B_{j;3k+1,3k+2}^{+}-B_{j;3k+1,3k+2}^{-})\gamma ^{(3k+2)}
\label{eqn:inhomogeneous term in energy hierarchy} \\
&=&ik\sum_{j=1}^{3}(H_{H}^{1}H_{H}^{4}\cdots
H_{H}^{3k-2})(B_{j;3k+1,3k+2}^{+}-B_{j;3k+1,3k+2}^{-})\gamma ^{(3k+2)} 
\notag \\
&=&ik%
\sum_{j=1}^{3}H_{H}^{1}(B_{j;3k+1,3k+2}^{+}-B_{j;3k+1,3k+2}^{-})H_{H}^{4}%
\cdots H_{H}^{3k-2}\gamma ^{(3k+2)}  \notag
\end{eqnarray}%
Let 
\begin{equation}
\sigma ^{(5)}(x_{1},x_{2},x_{3},x_{3k+1},x_{3k+2};x_{1}^{\prime
},x_{2}^{\prime },x_{3}^{\prime },x_{3k+1}^{\prime },x_{3k+2}^{\prime
})=H_{H}^{4}\cdots H_{H}^{3k-2}\gamma ^{(3k+2)}  \label{eqn:sigma_5}
\end{equation}%
Then 
\begin{equation*}
\text{(\ref{eqn:inhomogeneous term in energy hierarchy})}=ik%
\sum_{j=1}^{3}H_{H}^{1}(B_{j;4,5}^{+}-B_{j;4,5}^{-})\sigma ^{(5)}
\end{equation*}%
In summary, we have 
\begin{equation*}
\partial _{t}E_{H}^{(k)}=ik\sum_{j=1}^{3}H_{H}^{1}(\Delta _{j}-\Delta
_{j^{\prime }})\sigma
^{(3)}-ik\sum_{k=1}^{3}H_{H}^{1}(B_{j;4,5}^{+}-B_{j;4,5}^{-})\sigma ^{(5)}
\end{equation*}%
Since $\rho ^{(3)}=\sum_{j=1}^{3}(\Delta _{j}-\Delta _{j^{\prime }})\sigma
^{(3)}$ is symmetric and $\rho
^{(3)}=\sum_{j=1}^{3}(B_{j;4,5}^{+}-B_{j;4,5}^{-})\sigma ^{(5)}$ is
symmetric (carry the sums to the inside), we can replace $H_{H}^{1}$ by $%
H^{1}-H_{L}^{1}$. 
\begin{eqnarray*}
\partial _{t}E_{H}^{(k)} &=&ik\sum_{j=1}^{3}H^{1}(\Delta _{j}-\Delta
_{j^{\prime }})\sigma
^{(3)}-ik\sum_{k=1}^{3}H^{1}(B_{j;4,5}^{+}-B_{j;4,5}^{-})\sigma ^{(5)} \\
&&-ik\sum_{j=1}^{3}H_{L}^{1}(\Delta _{j}-\Delta _{j^{\prime }})\sigma
^{(3)}+ik\sum_{k=1}^{3}H_{L}^{1}(B_{j;4,5}^{+}-B_{j;4,5}^{-})\sigma ^{(5)}
\end{eqnarray*}%
The first two terms of $\partial _{t}E_{H}^{(k)}$ combines to zero by
exactly the same calculation as the energy conservation \cite[\S 4]%
{TChenAndNP2}. Thus%
\begin{eqnarray}
\frac{i}{k}\partial _{t}E_{H}^{(k)} &=&-\sum_{j=1}^{3}H_{L}^{1}(\Delta
_{j}-\Delta _{j^{\prime }})\sigma
^{(3)}+\sum_{k=1}^{3}H_{L}^{1}(B_{j;4,5}^{+}-B_{j;4,5}^{-})\sigma ^{(5)}
\label{E:main-3} \\
&=&1+...+6  \notag
\end{eqnarray}%
so that the first sum is terms $1$, $2$ and $3$, and the second sum is terms 
$4$, $5$, and $6$. Now let us enumerate the terms of $H_{L}^{1}$ defined in (%
\ref{eqn:H_L^j}) as follows 
\begin{align*}
H_{L}^{1}=& \func{Tr}_{1\rightarrow 3}\left( -\Delta _{1}P_{L}^{1}\right) & 
& \text{term A} \\
& +\frac{1}{3}\func{Tr}%
_{1}B_{1;2,3}^{+}P_{L}^{1}P_{L}^{2}P_{L}^{3}P_{L}^{1^{\prime
}}P_{L}^{2^{\prime }}P_{L}^{3^{\prime }} & & \text{term B} \\
& +\func{Tr}_{1}B_{1;2,3}^{+}P_{H}^{1}P_{L}^{2}P_{L}^{3}P_{L}^{1^{\prime
}}P_{L}^{2^{\prime }}P_{L}^{3^{\prime }} & & \text{term C} \\
& +\func{Tr}_{1}B_{1;2,3}^{+}P_{L}^{1}P_{L}^{2}P_{L}^{3}P_{H}^{1^{\prime
}}P_{L}^{2^{\prime }}P_{L}^{3^{\prime }} & & \text{term D} \\
& +\func{Tr}_{1}B_{1;2,3}^{+}P_{L}^{1}P_{H}^{2}P_{H}^{3}P_{L}^{1^{\prime
}}P_{L}^{2^{\prime }}P_{L}^{3^{\prime }} & & \text{term E} \\
& +\func{Tr}_{1}B_{1;2,3}^{+}P_{L}^{1}P_{L}^{2}P_{L}^{3}P_{L}^{1^{\prime
}}P_{H}^{2^{\prime }}P_{H}^{3^{\prime }} & & \text{term F} \\
& +\func{Tr}_{1}3B_{1;2,3}^{+}P_{H}^{1}P_{L}^{2}P_{L}^{3}P_{H}^{1^{\prime
}}P_{L}^{2^{\prime }}P_{L}^{3^{\prime }} & & \text{term G}
\end{align*}%
When these terms are substituted into \eqref{E:main-3}, we obtain terms $A1$
through $G6$, a total of 42 terms. We will investigate each of these below.
Some of them can be grouped together and share the same estimating methods.

In (\ref{eqn:sigma_3}) and (\ref{eqn:sigma_5}) we have introduced $\sigma
^{(3)}$ and $\sigma ^{(5)}$. By Theorem \ref{thm:qde}, $\sigma ^{(3)}$ and $%
\sigma ^{(5)}$ can be represented as 
\begin{equation*}
\sigma ^{(3)}(x_{1},x_{2},x_{3};x_{1}^{\prime },x_{2}^{\prime
},x_{3}^{\prime })=\int \phi (x_{1})\phi (x_{2})\phi (x_{3})\bar{\phi}%
(x_{1}^{\prime })\bar{\phi}(x_{2}^{\prime })\bar{\phi}(x_{3}^{\prime })E_{%
\text{NLS}}^{H}(\phi )^{k-1}\,d\mu (\phi )
\end{equation*}%
and 
\begin{eqnarray*}
&&\sigma ^{(5)}(x_{1},x_{2},x_{3},x_{3k+1},x_{3k+2};x_{1}^{\prime
},x_{2}^{\prime },x_{3}^{\prime },x_{3k+1}^{\prime },x_{3k+2}^{\prime }) \\
&=&\int \phi (x_{1})\phi (x_{2})\phi (x_{3})\phi (x_{3k+1})\phi (x_{3k+2})%
\bar{\phi}(x_{1}^{\prime })\bar{\phi}(x_{2}^{\prime })\bar{\phi}%
(x_{3}^{\prime })\bar{\phi}(x_{3k+1}^{\prime })\bar{\phi}(x_{3k+2}^{\prime
})E_{\text{NLS}}^{H}(\phi )^{k-1}\,d\mu (\phi )
\end{eqnarray*}%
We will also need the property of the de Finetti measure 
\begin{equation}
\mu \left( \{\phi \in L^{2}:\left\Vert S\phi \right\Vert
_{L_{x}^{2}}>C_{0}\}\right) =0  \label{eqn:property of measure 1}
\end{equation}%
proved in \cite[Lemma 4.5]{TCNPdeFinitte}, whenever (\ref{condition:kinetic
energy}) holds, and $E_{i}$ be the operator that acts on a function $%
f(x_{i},x_{i}^{\prime })$ and gives 
\begin{equation*}
(E_{i}f)(x_{i})=f(x_{i},x_{i})
\end{equation*}

\noindent \textbf{Term A1, A2, A3}. We use 
\begin{equation*}
E_{j}\Delta _{j}=-E_{j}\nabla _{j}\cdot \nabla _{j}^{\prime }+\nabla
_{j}\cdot E_{j}\nabla _{j}
\end{equation*}%
and 
\begin{equation*}
-E_{j}\Delta _{j}^{\prime }=E_{j}\nabla _{j}\cdot \nabla _{j}^{\prime
}-\nabla _{j}^{\prime }\cdot E_{j}\nabla _{j}^{\prime }
\end{equation*}%
applied to $\Delta _{1}P_{L}^{1}\sigma ^{(3)}$. This gives 
\begin{equation*}
\iiint_{x_{1},x_{2},x_{3}}(\Delta _{j}-\Delta _{j}^{\prime })\Delta
_{1}P_{L}^{1}\sigma ^{3}(x_{1},x_{2},x_{3};x_{1},x_{2},x_{3})\,d\mathbf{x}=0
\end{equation*}

\noindent \textbf{Term A4}. (the $B^{+}$ case only, the $B^{-}$ case is
similar) By the quantum deFinetti expansion and Fubini 
\begin{align*}
\text{A4}& =\func{Tr}_{1\rightarrow 3}(-\Delta
_{1}P_{L}^{1})(B_{1;4,5}^{+}-B_{1;4,5}^{-})\sigma ^{(5)} \\
& =-\int \func{Tr}_{1\rightarrow 3}\Delta _{1}P_{L}^{1}[\phi (x_{1})|\phi
(x_{1})|^{4}]\phi (x_{2})\phi (x_{3})\bar{\phi}(x_{1}^{\prime })\bar{\phi}%
(x_{2}^{\prime })\bar{\phi}(x_{3}^{\prime })E_{\text{NLS}}^{H}(\phi
)^{k-1}\,d\mu (\phi ) \\
& =-\int \int_{x_{1}}(\Delta P_{L}[\phi (x_{1})|\phi (x_{1})|^{4}])\,\bar{%
\phi}(x_{1})\,dx_{1}\left\Vert \phi \right\Vert _{L^{2}}^{4}E_{\text{NLS}%
}^{H}(\phi )^{k-1}\,d\mu (\phi )
\end{align*}%
Since $\Vert \phi \Vert _{L^{2}}\leqslant 1$, by H\"{o}lder in $x_{1}$, 
\begin{equation*}
\leqslant \int \Vert \Delta P_{L}[\phi (x_{1})|\phi (x_{1})|^{4}]\Vert
_{L_{x_{1}}^{6/5}}\Vert \phi \Vert _{L^{6}}E_{\text{NLS}}^{H}(\phi
)^{k-1}\,d\mu (\phi )
\end{equation*}%
By Bernstein, 
\begin{align*}
& \lesssim M^{2}\int \Vert \phi \Vert _{L^{6}}^{6}\,E_{\text{NLS}}^{H}(\phi
)^{k-1}\,d\mu (\phi ) \\
& \lesssim M^{2}\int E_{\text{NLS}}(\phi )\,E_{\text{NLS}}^{H}(\phi
)^{k-1}\,d\mu (\phi )
\end{align*}%
By H\"{o}lder in $\phi $ with exponents $k$ and $\frac{k}{k-1}$, we obtain 
\begin{align*}
& \lesssim M^{2}\left( \int E_{\text{NLS}}(\phi )^{k}\,d\mu (\phi )\right)
^{1/k}\left( \int E_{\text{NLS}}^{H}(\phi )^{k}\,d\mu (\phi )\right)
^{k/(k-1)} \\
& \lesssim M^{2}(E^{(k)})^{1/k}(E_{H}^{(k)})^{(k-1)/k}
\end{align*}

\noindent \textbf{Term A5, A6}. We write only the A5 term (so $j=2$), since
the A6 term is identical after permutation of index 2 and 3. We denote by
A5p the $B^{+}$ contribution and A5m the $B^{-}$ contribution, so that A$%
\text{5}=\text{A5p}-\text{A5m}$. By the quantum deFinetti expansion and
Fubini 
\begin{align*}
\text{A5p}& =\func{Tr}_{1\rightarrow 3}(-\Delta
_{1}P_{L}^{1})(B_{2;4,5}^{+})\sigma ^{(5)} \\
& =-\int \func{Tr}_{1\rightarrow 3}\Delta _{1}P_{L}^{1}\phi (x_{1})\phi
(x_{2})|\phi (x_{2})|^{4}\phi (x_{3})\bar{\phi}(x_{1}^{\prime })\bar{\phi}%
(x_{2}^{\prime })\bar{\phi}(x_{3}^{\prime })E_{\text{NLS}}^{H}(\phi
)^{k-1}\,d\mu (\phi ) \\
& =-\int \int_{x_{1}}(\Delta P_{L}\phi (x_{1}))\bar{\phi}(x_{1})\,dx_{1}%
\Vert \phi \Vert _{L^{6}}^{6}\left\Vert \phi \right\Vert _{L^{2}}^{2}E_{%
\text{NLS}}^{H}(\phi )^{k-1}\,d\mu (\phi )
\end{align*}%
Likewise 
\begin{equation*}
\text{A5m}=-\int \int_{x_{1}}\phi (x_{1})\,\overline{\Delta P_{L}\phi (x_{1})%
}\,dx_{1}\Vert \phi \Vert _{L^{6}}^{6}\left\Vert \phi \right\Vert
_{L^{2}}^{2}E_{\text{NLS}}^{H}(\phi )^{k-1}\,d\mu (\phi )
\end{equation*}%
Hence $\text{A5}=\text{A5p}-\text{A5m}=0$.

Now we proceed to the contracted terms B, C, D, E, F, and G. Recall that B
has no components projected to high frequency, C and D have exactly one term
projected to high frequency, and E, F, and G have exactly two terms
projected to high frequency.

\noindent \textbf{Term B1, B2, B3}. We have 
\begin{eqnarray*}
&&\text{B1, B2, B3} \\
&=&\frac{1}{3}\func{Tr}%
_{1}B_{1;2,3}^{+}P_{L}^{1}P_{L}^{2}P_{L}^{3}P_{L}^{1^{\prime
}}P_{L}^{2^{\prime }}P_{L}^{3^{\prime }}\left( \Delta _{j}-\Delta
_{j}^{\prime }\right) \sigma ^{(3)} \\
&=&\frac{1}{3}\int \int_{x}(P_{L}\Delta \phi (x)P_{L}\bar{\phi}(x)-P_{L}\phi
(x)\,P_{L}\Delta \bar{\phi}(x))\,|P_{L}\phi (x)|^{4}\,dx\,E_{\text{NLS}%
}^{H}(\phi )^{k-1}d\mu (\phi ) \\
&=&\frac{2}{3}i\func{Im}\int \int_{x}\left( P_{L}\Delta \phi (x)P_{L}\bar{%
\phi}(x)\,|P_{L}\phi (x)|^{4}\right) \,dx\,E_{\text{NLS}}^{H}(\phi
)^{k-1}d\mu (\phi )
\end{eqnarray*}%
Hence, by H\"{o}lder%
\begin{equation*}
\left\vert \text{B1}\right\vert \text{, }\left\vert \text{B2}\right\vert 
\text{, }\left\vert \text{B3}\right\vert \lesssim \int \left\Vert
P_{L}\Delta \phi \right\Vert _{L^{6}}\left\Vert P_{L}\bar{\phi}\,|P_{L}\phi
|^{4}\right\Vert _{L^{\frac{6}{5}}}\,E_{\text{NLS}}^{H}(\phi )^{k-1}d\mu
(\phi )
\end{equation*}%
By Bernstein%
\begin{equation*}
\lesssim M^{2}\int \left\Vert P_{L}\phi \right\Vert _{L^{6}}^{6}\,E_{\text{%
NLS}}^{H}(\phi )^{k-1}d\mu (\phi )\lesssim M^{2}\int E_{\text{NLS}}(\phi
)\,E_{\text{NLS}}^{H}(\phi )^{k-1}d\mu (\phi )
\end{equation*}%
Finally, By H\"{o}lder in $\phi $ with exponents $k$ and $\frac{k}{k-1}$,%
\begin{eqnarray*}
&&\left\vert \text{B1}\right\vert \text{, }\left\vert \text{B2}\right\vert 
\text{, }\left\vert \text{B3}\right\vert \\
&\lesssim &M^{2}\left( \int E_{\text{NLS}}(\phi )^{k}d\mu (\phi )\right)
^{1/k}\left( \int E_{\text{NLS}}^{H}(\phi )^{k}d\mu (\phi )\right) ^{(k-1)/k}
\\
&=&CM^{2}(E^{(k)})^{1/k}(E_{H}^{(k)})^{(k-1)/k}.
\end{eqnarray*}

\noindent \textbf{Term B4, B5, B6}. We will only write B4 since B5 and B6
are identical. We have 
\begin{align*}
\text{B4p}& =\frac{1}{3}\func{Tr}%
_{1}B_{1;2,3}^{+}P_{L}^{1}P_{L}^{2}P_{L}^{3}P_{L}^{1^{\prime
}}P_{L}^{2^{\prime }}P_{L}^{3^{\prime }}B_{1;4,5}^{+}\sigma ^{(5)} \\
& =\frac{1}{3}\int \int_{x}P_{L}(\phi (x)|\phi (x)|^{4})\,\overline{%
P_{L}\phi (x)}|P_{L}\phi (x)|^{4}\,dx\,E_{\text{NLS}}^{H}(\phi )^{k-1}d\mu
(\phi )
\end{align*}%
and 
\begin{align*}
\text{B4m}& =\frac{1}{3}\func{Tr}%
_{1}B_{1;2,3}^{+}P_{L}^{1}P_{L}^{2}P_{L}^{3}P_{L}^{1^{\prime
}}P_{L}^{2^{\prime }}P_{L}^{3^{\prime }}B_{1;4,5}^{-}\sigma ^{(5)} \\
& =\frac{1}{3}\int \int_{x}\overline{P_{L}(\phi (x)|\phi (x)|^{4})}%
\,P_{L}\phi (x)|P_{L}\phi (x)|^{4}\,dx\,E_{\text{NLS}}^{H}(\phi )^{k-1}d\mu
(\phi )
\end{align*}%
and hence 
\begin{align*}
\text{B4}& =\text{B4p}-\text{B4m} \\
& =\frac{2}{3}i\func{Im}\int \int_{x}P_{L}(\phi (x)|\phi (x)|^{4})\,%
\overline{P_{L}\phi (x)}|P_{L}\phi (x)|^{4}\,dx\,E_{\text{NLS}}^{H}(\phi
)^{k-1}d\mu (\phi )
\end{align*}%
By H\"{o}lder, 
\begin{equation*}
\left\vert \text{B4}\right\vert \lesssim \int \Vert P_{L}(\phi |\phi
|^{4})\Vert _{L^{6}}\Vert P_{L}\phi \Vert _{L^{6}}^{5}E_{\text{NLS}%
}^{H}(\phi )^{k-1}d\mu (\phi )
\end{equation*}%
We apply Bernstein in the form $\Vert P_{L}f\Vert _{L_{x}^{6}}\lesssim
M^{2}\Vert f\Vert _{L_{x}^{6/5}}$ with $f=\phi |\phi |^{4}$ to obtain 
\begin{equation*}
\left\vert \text{B4}\right\vert \lesssim M^{2}\int \Vert \phi \Vert
_{L^{6}}^{10}\,E_{\text{NLS}}^{H}(\phi )^{k-1}d\mu (\phi )\lesssim M^{2}\int
\Vert \phi \Vert _{L^{6}}^{6}\left\Vert \nabla \phi \right\Vert
_{L^{2}}^{4}\,E_{\text{NLS}}^{H}(\phi )^{k-1}d\mu (\phi )
\end{equation*}%
By (\ref{eqn:property of measure 1}),%
\begin{equation*}
\lesssim C_{0}^{4}M^{2}\int \Vert \phi \Vert _{L^{6}}^{6}\,E_{\text{NLS}%
}^{H}(\phi )^{k-1}d\mu (\phi )\lesssim M^{2}\int E_{\text{NLS}}(\phi )\,E_{%
\text{NLS}}^{H}(\phi )^{k-1}d\mu (\phi )
\end{equation*}%
By H\"{o}lder in $\phi $ with exponents $k$ and $\frac{k}{k-1}$, 
\begin{eqnarray*}
&\lesssim &M^{2}\left( \int E_{\text{NLS}}(\phi )^{k}d\mu (\phi )\right)
^{1/k}\left( \int E_{\text{NLS}}^{H}(\phi )^{k}d\mu (\phi )\right) ^{(k-1)/k}
\\
&=&CM^{2}(E^{(k)})^{1/k}(E_{H}^{(k)})^{(k-1)/k}.
\end{eqnarray*}%
\textbf{Terms C1, C2, C3, D1, D2, D3. }Both of C1 and D1 have $\Delta $ land
on the high term and is similar. We only deal with C1. 
\begin{eqnarray*}
\text{C1} &=&\func{Tr}_{1}B_{1;2,3}^{+}P_{H}^{1}P_{L}^{2}P_{L}^{3}P_{L}^{1^{%
\prime }}P_{L}^{2^{\prime }}P_{L}^{3^{\prime }}(\Delta _{1}-\Delta
_{1^{\prime }})\sigma ^{(3)} \\
&=&\int \left[ \left( \Delta P_{H}\phi \right) (x)\left( P_{L}\bar{\phi}%
\right) (x)-\left( P_{H}\phi \right) (x)\left( \Delta P_{L}\bar{\phi}\right)
(x)\right] \left\vert P_{L}\phi \right\vert ^{4}(x)dxE_{\text{NLS}}^{H}(\phi
)^{k-1}\,d\mu (\phi )
\end{eqnarray*}%
Intergating by parts in the 1st summand and then applying Holder in the 2nd
summand, we have 
\begin{eqnarray*}
\left\vert \text{C1}\right\vert &\leqslant &\left\vert \int \left( \nabla
P_{H}\phi \right) (x)\cdot \nabla \left( P_{L}\bar{\phi}\left\vert P_{L}\phi
\right\vert ^{4}\right) (x)dxE_{\text{NLS}}^{H}(\phi )^{k-1}\,d\mu (\phi
)\right\vert \\
&&+\int \left\Vert \Delta P_{L}\phi \right\Vert _{L^{6}}\left\Vert P_{H}\phi
\right\Vert _{L^{6}}\left\Vert P_{L}\phi \right\Vert _{L^{6}}^{4}dxE_{\text{%
NLS}}^{H}(\phi )^{k-1}\,d\mu (\phi )
\end{eqnarray*}%
Use Holder again in the 1st summand and Bernstein in the 2nd summand, 
\begin{eqnarray*}
&\lesssim &\int \left\Vert \nabla P_{H}\phi \right\Vert _{L^{2}}\left\Vert
\nabla P_{L}\phi \right\Vert _{L^{6}}\left\Vert P_{L}\phi \right\Vert
_{L^{12}}^{4}E_{\text{NLS}}^{H}(\phi )^{k-1}\,d\mu (\phi ) \\
&&+M^{2}\int \left\Vert P_{L}\phi \right\Vert _{L^{6}}\left\Vert P_{H}\phi
\right\Vert _{L^{6}}\left\Vert P_{L}\phi \right\Vert _{L^{6}}^{4}E_{\text{NLS%
}}^{H}(\phi )^{k-1}\,d\mu (\phi )
\end{eqnarray*}%
Applying Bernstein in the 1st summand, then we can apply Holder to get to%
\begin{eqnarray*}
\left\vert \text{C1}\right\vert &\lesssim &\int \left\Vert \nabla \phi
\right\Vert _{L^{2}}^{\frac{2}{3}}\left\Vert \nabla P_{H}\phi \right\Vert
_{L^{2}}^{1-\frac{2}{3}}M\left\Vert P_{L}\phi \right\Vert _{L^{6}}\left( M^{%
\frac{1}{4}}\left\Vert P_{L}\phi \right\Vert _{L^{6}}\right) ^{4}E_{\text{NLS%
}}^{H}(\phi )^{k-1}\,d\mu (\phi ) \\
&&+M^{2}\int \left\Vert \phi \right\Vert _{L^{6}}^{6}E_{\text{NLS}}^{H}(\phi
)^{k-1}\,d\mu (\phi ) \\
&\lesssim &C_{0}^{\frac{2}{3}}M^{2}\int E_{\text{NLS}}(\phi )\,E_{\text{NLS}%
}^{H}(\phi )^{k-1}d\mu (\phi ) \\
&\lesssim &M^{2}(E^{(k)})^{1/k}(E_{H}^{(k)})^{(k-1)/k}.
\end{eqnarray*}%
C2, C3, D2, D3 are similar and they do not have $\Delta $ land on the high
term. We only deal with C2.%
\begin{eqnarray*}
\text{C2} &=&\func{Tr}_{1}B_{1;2,3}^{+}P_{H}^{1}P_{L}^{2}P_{L}^{3}P_{L}^{1^{%
\prime }}P_{L}^{2^{\prime }}P_{L}^{3^{\prime }}(\Delta _{2}-\Delta
_{2^{\prime }})\sigma ^{(3)} \\
&=&\int \left[ \left( \Delta P_{L}\phi \right) (x)\left( P_{L}\bar{\phi}%
\right) (x)-\left( P_{L}\phi \right) (x)\left( \Delta P_{L}\bar{\phi}\right)
(x)\right] \\
&&\times \left( P_{H}\phi \right) (x)\left\vert P_{L}\phi \right\vert
^{2}(x)\left( P_{L}\bar{\phi}\right) (x)dxE_{\text{NLS}}^{H}(\phi
)^{k-1}\,d\mu (\phi )
\end{eqnarray*}%
Put everything in $L^{6}$%
\begin{eqnarray*}
\left\vert \text{C2}\right\vert &\lesssim &\int \left\Vert \Delta P_{L}\phi
\right\Vert _{L^{6}}\left\Vert P_{H}\phi \right\Vert _{L^{6}}\left\Vert
P_{L}\phi \right\Vert _{L^{6}}^{4}E_{\text{NLS}}^{H}(\phi )^{k-1}\,d\mu
(\phi ) \\
&\lesssim &M^{2}\int E_{\text{NLS}}(\phi )\,E_{\text{NLS}}^{H}(\phi
)^{k-1}d\mu (\phi ) \\
&\lesssim &M^{2}(E^{(k)})^{1/k}(E_{H}^{(k)})^{(k-1)/k}.
\end{eqnarray*}%
\textbf{Terms C4, C5, C6, D4, D5, D6. }The terms C4, C5, C6, D4, D5, D6 are
similar. We will only deal with C4 in detail.%
\begin{eqnarray*}
\text{C4p} &=&\func{Tr}%
_{1}B_{1;2,3}^{+}P_{H}^{1}P_{L}^{2}P_{L}^{3}P_{L}^{1^{\prime
}}P_{L}^{2^{\prime }}P_{L}^{3^{\prime }}B_{1;4,5}^{+}\sigma ^{(5)} \\
&=&\int \left[ P_{H}\left( \left\vert \phi \right\vert ^{4}\phi \right) %
\right] (x)\left\vert \left( P_{L}\phi \right) (x)\right\vert ^{4}\left(
P_{L}\bar{\phi}\right) (x)dxE_{\text{NLS}}^{H}(\phi )^{k-1}\,d\mu (\phi )
\end{eqnarray*}%
\begin{eqnarray*}
\text{C4m} &=&\func{Tr}%
_{1}B_{1;2,3}^{+}P_{H}^{1}P_{L}^{2}P_{L}^{3}P_{L}^{1^{\prime
}}P_{L}^{2^{\prime }}P_{L}^{3^{\prime }}B_{1;4,5}^{-}\sigma ^{(5)} \\
&=&\int \left( P_{H}\phi \left\vert P_{L}\phi \right\vert ^{4}\right)
(x)P_{L}\left( \left\vert \phi \right\vert ^{4}\bar{\phi}\right) (x)dxE_{%
\text{NLS}}^{H}(\phi )^{k-1}\,d\mu (\phi )
\end{eqnarray*}%
Thus%
\begin{eqnarray*}
\left\vert \text{C4p}\right\vert &\lesssim &\int \left\Vert P_{H}\left(
\left\vert \phi \right\vert ^{4}\phi \right) \right\Vert _{L^{\frac{6}{5}%
}}\left\Vert \left\vert \left( P_{L}\phi \right) \right\vert ^{4}P_{L}\bar{%
\phi}\right\Vert _{L^{6}}E_{\text{NLS}}^{H}(\phi )^{k-1}\,d\mu (\phi ) \\
&\lesssim &\int \left\Vert \phi \right\Vert _{L^{6}}^{5}\left\Vert P_{L}\phi
\right\Vert _{L^{30}}^{5}E_{\text{NLS}}^{H}(\phi )^{k-1}\,d\mu (\phi ) \\
&\lesssim &\int \left\Vert \nabla \phi \right\Vert _{L^{2}}^{4}\left\Vert
\phi \right\Vert _{L^{6}}^{1}\left\Vert \left\vert \nabla \right\vert ^{%
\frac{2}{5}}P_{L}\phi \right\Vert _{L^{6}}^{5}E_{\text{NLS}}^{H}(\phi
)^{k-1}\,d\mu (\phi )
\end{eqnarray*}%
Use Bernstein,%
\begin{eqnarray*}
\left\vert \text{C4p}\right\vert &\lesssim &C_{0}^{4}M^{2}\int \left\Vert
\phi \right\Vert _{L^{6}}^{6}E_{\text{NLS}}^{H}(\phi )^{k-1}\,d\mu (\phi ) \\
&\lesssim &M^{2}\int E_{\text{NLS}}(\phi )\,E_{\text{NLS}}^{H}(\phi
)^{k-1}d\mu (\phi ) \\
&\lesssim &M^{2}(E^{(k)})^{1/k}(E_{H}^{(k)})^{(k-1)/k}.
\end{eqnarray*}%
On the other hand, for C4m, we again apply Bernstein in the form $\Vert
P_{L}f\Vert _{L_{x}^{6}}\lesssim M^{2}\Vert f\Vert _{L_{x}^{6/5}}$ with $%
f=\phi |\phi |^{4}$%
\begin{eqnarray*}
\left\vert \text{C4m}\right\vert &\lesssim &\int \left\Vert P_{H}\phi
\left\vert P_{L}\phi \right\vert ^{4}\right\Vert _{L^{\frac{6}{5}%
}}\left\Vert P_{L}\left( \left\vert \phi \right\vert ^{4}\bar{\phi}\right)
\right\Vert _{L^{6}}E_{\text{NLS}}^{H}(\phi )^{k-1}\,d\mu (\phi ) \\
&\lesssim &M^{2}\int \left\Vert \phi \right\Vert _{L^{6}}^{5}\left\Vert
\left\vert \phi \right\vert ^{4}\bar{\phi}\right\Vert _{L^{\frac{6}{5}}}E_{%
\text{NLS}}^{H}(\phi )^{k-1}\,d\mu (\phi ) \\
&\lesssim &C_{0}^{4}M^{2}\int E_{\text{NLS}}(\phi )\,E_{\text{NLS}}^{H}(\phi
)^{k-1}d\mu (\phi ) \\
&\lesssim &M^{2}(E^{(k)})^{1/k}(E_{H}^{(k)})^{(k-1)/k}.
\end{eqnarray*}%
That is%
\begin{equation*}
\left\vert \text{C4}\right\vert \lesssim \left\vert \text{C4p}\right\vert
+\left\vert \text{C4m}\right\vert \lesssim
M^{2}(E^{(k)})^{1/k}(E_{H}^{(k)})^{(k-1)/k}.
\end{equation*}%
\textbf{Term E1, E2, E3, F1, F2, F3. }E1 and F1 are similar, while E2, E3,
F2, and F3 are similar. We only deal with E1 and E2 in detail. On the one
hand,%
\begin{eqnarray*}
\text{E1} &=&\func{Tr}_{1}B_{1;2,3}^{+}P_{L}^{1}P_{H}^{2}P_{H}^{3}P_{L}^{1^{%
\prime }}P_{L}^{2^{\prime }}P_{L}^{3^{\prime }}(\Delta _{1}-\Delta
_{1^{\prime }})\sigma ^{(3)} \\
&=&\int \left[ \left( P_{L}\Delta \phi \right) (x)\left( P_{L}\bar{\phi}%
\right) (x)-\left( P_{L}\phi \right) (x)\left( P_{L}\Delta \bar{\phi}\right)
(x)\right] \\
&&\times \left( P_{H}\phi \right) (x)\left( P_{H}\phi \right) (x)\left( P_{L}%
\bar{\phi}\right) (x)\left( P_{L}\bar{\phi}\right) (x)dxE_{\text{NLS}%
}^{H}(\phi )^{k-1}\,d\mu (\phi ).
\end{eqnarray*}%
Put everything in $L^{6}$,%
\begin{eqnarray*}
\left\vert \text{E1}\right\vert &\leqslant &\int \left\Vert P_{L}\Delta \phi
\right\Vert _{L^{6}}\left\Vert P_{L}\Delta \phi \right\Vert
_{L^{6}}^{3}\left\Vert P_{H}\Delta \phi \right\Vert _{L^{6}}^{2}E_{\text{NLS}%
}^{H}(\phi )^{k-1}\,d\mu (\phi ) \\
&\lesssim &M^{2}\int \left\Vert \phi \right\Vert _{L^{6}}^{6}E_{\text{NLS}%
}^{H}(\phi )^{k-1}\,d\mu (\phi ) \\
&\lesssim &M^{2}\int E_{\text{NLS}}(\phi )\,E_{\text{NLS}}^{H}(\phi
)^{k-1}\,d\mu (\phi ) \\
&\lesssim &M^{2}(E^{(k)})^{1/k}(E_{H}^{(k)})^{(k-1)/k}.
\end{eqnarray*}%
On the other hand,%
\begin{eqnarray*}
\text{E2} &=&\func{Tr}_{1}B_{1;2,3}^{+}P_{L}^{1}P_{H}^{2}P_{H}^{3}P_{L}^{1^{%
\prime }}P_{L}^{2^{\prime }}P_{L}^{3^{\prime }}(\Delta _{2}-\Delta
_{2^{\prime }})\sigma ^{(3)} \\
&=&\int \left[ \left( \Delta P_{H}\phi \right) (x)\left( P_{L}\bar{\phi}%
\right) (x)-\left( P_{H}\phi \right) (x)\left( \Delta P_{L}\bar{\phi}\right)
(x)\right] \\
&&\times \left( P_{H}\phi \right) (x)\left\vert P_{L}\phi \right\vert
^{2}(x)\left( P_{L}\bar{\phi}\right) (x)dxE_{\text{NLS}}^{H}(\phi
)^{k-1}\,d\mu (\phi )
\end{eqnarray*}%
Integrating by parts in the 1st summand and putting everything in $L^{6}$ in
the 2nd summand, we get to%
\begin{eqnarray*}
\left\vert \text{E2}\right\vert &\lesssim &\left\vert \int \left( \nabla
P_{H}\phi \right) (x)\left( P_{L}\bar{\phi}\right) (x)\left( \nabla
P_{H}\phi \right) (x)\left\vert P_{L}\phi \right\vert ^{2}(x)\left( P_{L}%
\bar{\phi}\right) (x)dxE_{\text{NLS}}^{H}(\phi )^{k-1}\,d\mu (\phi
)\right\vert \\
&&+\left\vert \int \left( \nabla P_{H}\phi \right) (x)\left( \nabla P_{L}%
\bar{\phi}\right) (x)\left( P_{H}\phi \right) (x)\left\vert P_{L}\phi
\right\vert ^{2}(x)\left( P_{L}\bar{\phi}\right) (x)dxE_{\text{NLS}%
}^{H}(\phi )^{k-1}\,d\mu (\phi )\right\vert \\
&&+\int \left\Vert P_{H}\phi \right\Vert _{L^{6}}^{2}\left\Vert \Delta
P_{L}\phi \right\Vert _{L^{6}}\left\Vert P_{L}\phi \right\Vert
_{L^{6}}^{3}E_{\text{NLS}}^{H}(\phi )^{k-1}\,d\mu (\phi )
\end{eqnarray*}%
Apply Holder in the first two term as well,%
\begin{eqnarray*}
\left\vert \text{E2}\right\vert &\lesssim &\int \left\Vert \nabla P_{H}\phi
\right\Vert _{L^{2}}^{2}\left\Vert P_{L}\phi \right\Vert _{L^{\infty
}}^{4}E_{\text{NLS}}^{H}(\phi )^{k-1}\,d\mu (\phi ) \\
&&+\int \left\Vert \nabla P_{H}\phi \right\Vert _{L^{2}}\left\Vert P_{H}\phi
\right\Vert _{L^{6}}\left\Vert \nabla P_{L}\bar{\phi}\right\Vert
_{L^{3}}\left\Vert P_{L}\phi \right\Vert _{L^{\infty }}^{3}E_{\text{NLS}%
}^{H}(\phi )^{k-1}\,d\mu (\phi ) \\
&&+M^{2}\int \left\Vert \phi \right\Vert _{L^{6}}^{6}E_{\text{NLS}}^{H}(\phi
)^{k-1}\,d\mu (\phi )
\end{eqnarray*}%
Use the Bernstein inequalities $\left\Vert P_{L}\phi \right\Vert _{L^{\infty
}}\lesssim M^{\frac{1}{2}}\left\Vert \nabla P_{L}\phi \right\Vert _{L^{2}}$
and $\left\Vert \nabla P_{L}\phi \right\Vert _{L^{3}}\lesssim M^{\frac{1}{2}%
}\left\Vert \nabla \phi \right\Vert _{L^{2}}$ for the first two terms, we
then have 
\begin{eqnarray*}
\left\vert \text{E2}\right\vert &\lesssim &M^{2}\int \left\Vert \nabla \phi
\right\Vert _{L^{2}}^{4}\left\Vert \nabla \phi \right\Vert _{L^{2}}^{2}E_{%
\text{NLS}}^{H}(\phi )^{k-1}\,d\mu (\phi ) \\
&&+M^{2}\int \left\Vert \phi \right\Vert _{L^{6}}^{6}E_{\text{NLS}}^{H}(\phi
)^{k-1}\,d\mu (\phi )
\end{eqnarray*}%
By (\ref{eqn:property of measure 1}), 
\begin{eqnarray*}
\left\vert \text{E2}\right\vert &\lesssim &C_{0}^{4}M^{2}\int \left\Vert
\nabla \phi \right\Vert _{L^{2}}^{2}E_{\text{NLS}}^{H}(\phi )^{k-1}\,d\mu
(\phi ) \\
&&+M^{2}\int E_{\text{NLS}}(\phi )E_{\text{NLS}}^{H}(\phi )^{k-1}\,d\mu
(\phi ) \\
&\lesssim &M^{2}\int E_{\text{NLS}}(\phi )E_{\text{NLS}}^{H}(\phi
)^{k-1}\,d\mu (\phi ) \\
&\lesssim &M^{2}(E^{(k)})^{1/k}(E_{H}^{(k)})^{(k-1)/k}.
\end{eqnarray*}%
\textbf{Term G1, G2, G3. }We start with G1.%
\begin{eqnarray*}
\text{G1} &=&\func{Tr}%
_{1}3B_{1;2,3}^{+}P_{H}^{1}P_{L}^{2}P_{L}^{3}P_{H}^{1^{\prime
}}P_{L}^{2^{\prime }}P_{L}^{3^{\prime }}(\Delta _{1}-\Delta _{1^{\prime
}})\sigma ^{(3)} \\
&=&3\int \left[ \left( \Delta P_{H}\phi \right) (x)\left( P_{H}\bar{\phi}%
\right) (x)-\left( \Delta P_{H}\phi \right) (x)\left( \Delta P_{H}\bar{\phi}%
\right) (x)\right] \\
&&\times \left\vert P_{L}\phi \right\vert ^{4}(x)E_{\text{NLS}}^{H}(\phi
)^{k-1}\,d\mu (\phi ) \\
&=&6i\func{Im}\int \left( \Delta P_{H}\phi \right) (x)\left( P_{H}\bar{\phi}%
\right) (x)\left\vert P_{L}\phi \right\vert ^{4}(x)dxE_{\text{NLS}}^{H}(\phi
)^{k-1}\,d\mu (\phi )
\end{eqnarray*}%
Integrating by parts gives%
\begin{eqnarray*}
\left\vert \text{G1}\right\vert &\lesssim &\left\vert \int \left( \nabla
P_{H}\phi \right) (x)\left( \nabla P_{H}\bar{\phi}\right) (x)\left\vert
P_{L}\phi \right\vert ^{4}(x)dxE_{\text{NLS}}^{H}(\phi )^{k-1}\,d\mu (\phi
)\right\vert \\
&&+\left\vert \int \left( \nabla P_{H}\phi \right) (x)\left( P_{H}\bar{\phi}%
\right) (x)\left\vert P_{L}\phi \right\vert ^{2}(x)\left( \nabla P_{L}\phi
\right) (x)\left( P_{L}\bar{\phi}\right) (x)dxE_{\text{NLS}}^{H}(\phi
)^{k-1}\,d\mu (\phi )\right\vert
\end{eqnarray*}%
which can be estimated like E2, so we omit the details here and conclude%
\begin{equation*}
\left\vert \text{G1}\right\vert \lesssim
M^{2}(E^{(k)})^{1/k}(E_{H}^{(k)})^{(k-1)/k}.
\end{equation*}%
G2 and G3 are similar, so we only deal with G2.%
\begin{eqnarray*}
\text{G2} &=&\func{Tr}%
_{1}3B_{1;2,3}^{+}P_{H}^{1}P_{L}^{2}P_{L}^{3}P_{H}^{1^{\prime
}}P_{L}^{2^{\prime }}P_{L}^{3^{\prime }}(\Delta _{2}-\Delta _{2^{\prime
}})\sigma ^{(3)} \\
&=&3\int \left[ \left( \Delta P_{L}\phi \right) (x)\left( P_{L}\bar{\phi}%
\right) (x)-\left( P_{L}\phi \right) (x)\left( \Delta P_{L}\bar{\phi}\right)
(x)\right] \\
&&\times \left\vert P_{H}\phi \right\vert ^{2}(x)\left\vert P_{L}\phi
\right\vert ^{2}(x)dxE_{\text{NLS}}^{H}(\phi )^{k-1}\,d\mu (\phi ) \\
&=&6\func{Im}\int \left( \Delta P_{L}\phi \right) (x)\left( P_{L}\bar{\phi}%
\right) (x)\left\vert P_{H}\phi \right\vert ^{2}(x)\left\vert P_{L}\phi
\right\vert ^{2}(x)dxE_{\text{NLS}}^{H}(\phi )^{k-1}\,d\mu (\phi )
\end{eqnarray*}%
Putting everything in $L^{6}$, we have%
\begin{eqnarray*}
\left\vert \text{G2}\right\vert &\lesssim &\int \left\Vert \Delta P_{L}\phi
\right\Vert _{L^{6}}\left\Vert P_{H}\phi \right\Vert _{L^{6}}^{2}\left\Vert
P_{L}\phi \right\Vert _{L^{6}}^{3}E_{\text{NLS}}^{H}(\phi )^{k-1}\,d\mu
(\phi ) \\
&\lesssim &M^{2}\int \left\Vert \phi \right\Vert _{L^{6}}^{6}E_{\text{NLS}%
}^{H}(\phi )^{k-1}\,d\mu (\phi ) \\
&\lesssim &M^{2}\int E_{\text{NLS}}(\phi )E_{\text{NLS}}^{H}(\phi
)^{k-1}\,d\mu (\phi ) \\
&\lesssim &M^{2}(E^{(k)})^{1/k}(E_{H}^{(k)})^{(k-1)/k}.
\end{eqnarray*}%
\textbf{Term E4, E5, E6, F4, F5, F6}. E4 and F4 are similar while E5, E6, F5
and F6 are similar. So we only deal with E4 and E5. On the one hand,%
\begin{eqnarray*}
\text{E4} &=&\func{Tr}_{1}B_{1;2,3}^{+}P_{L}^{1}P_{H}^{2}P_{H}^{3}P_{L}^{1^{%
\prime }}P_{L}^{2^{\prime }}P_{L}^{3^{\prime
}}(B_{1;4,5}^{+}-B_{1;4,5}^{-})\sigma ^{(5)} \\
&=&\int \left[ P_{L}\left( \left\vert \phi \right\vert ^{4}\phi \right)
(x)\left( P_{L}\bar{\phi}\right) (x)-\left( P_{L}\phi \right) (x)P_{L}\left(
\left\vert \phi \right\vert ^{4}\bar{\phi}\right) (x)\right] \\
&&\times \left( P_{H}\phi \right) (x)\left( P_{H}\phi \right) (x)\left( P_{L}%
\bar{\phi}\right) (x)\left( P_{L}\bar{\phi}\right) (x)dxE_{\text{NLS}%
}^{H}(\phi )^{k-1}\,d\mu (\phi ) \\
&=&2i\func{Im}\int P_{L}\left( \left\vert \phi \right\vert ^{4}\phi \right)
(x)\left( P_{L}\bar{\phi}\right) (x) \\
&&\times \left( P_{H}\phi \right) (x)\left( P_{H}\phi \right) (x)\left( P_{L}%
\bar{\phi}\right) (x)\left( P_{L}\bar{\phi}\right) (x)dxE_{\text{NLS}%
}^{H}(\phi )^{k-1}\,d\mu (\phi )
\end{eqnarray*}%
Again, putting everything in $L^{6}$ and the Bernstein $\Vert P_{L}f\Vert
_{L_{x}^{6}}\lesssim M^{2}\Vert f\Vert _{L_{x}^{6/5}}$ give%
\begin{eqnarray*}
\left\vert \text{E4}\right\vert &\lesssim &\int \left\Vert P_{L}\left(
\left\vert \phi \right\vert ^{4}\phi \right) \right\Vert _{L^{6}}\left\Vert
P_{L}\phi \right\Vert _{L^{6}}^{3}\left\Vert P_{H}\phi \right\Vert
_{L^{6}}^{2}E_{\text{NLS}}^{H}(\phi )^{k-1}\,d\mu (\phi ) \\
&\lesssim &M^{2}\int \left\Vert \left\vert \phi \right\vert ^{4}\phi
\right\Vert _{L^{\frac{6}{5}}}\left\Vert \phi \right\Vert _{L^{6}}^{5}E_{%
\text{NLS}}^{H}(\phi )^{k-1}\,d\mu (\phi ) \\
&\lesssim &M^{2}\int \left\Vert \phi \right\Vert _{L^{6}}^{6}\left\Vert
\nabla \phi \right\Vert _{L^{2}}^{4}E_{\text{NLS}}^{H}(\phi )^{k-1}\,d\mu
(\phi ) \\
&\lesssim &C_{0}^{4}M^{2}\int E_{\text{NLS}}(\phi )E_{\text{NLS}}^{H}(\phi
)^{k-1}\,d\mu (\phi ) \\
&\lesssim &M^{2}(E^{(k)})^{1/k}(E_{H}^{(k)})^{(k-1)/k}.
\end{eqnarray*}%
On the other hand,%
\begin{eqnarray*}
\text{E5} &=&\func{Tr}_{1}B_{1;2,3}^{+}P_{L}^{1}P_{H}^{2}P_{H}^{3}P_{L}^{1^{%
\prime }}P_{L}^{2^{\prime }}P_{L}^{3^{\prime
}}(B_{2;4,5}^{+}-B_{2;4,5}^{-})\sigma ^{(5)} \\
&=&\int \left[ P_{H}\left( \left\vert \phi \right\vert ^{4}\phi \right)
(x)\left( P_{L}\bar{\phi}\right) (x)-P_{H}\left( \left\vert \phi \right\vert
^{4}\bar{\phi}\right) (x)\left( P_{L}\phi \right) (x)\right] \\
&&\times \left( P_{H}\phi \right) (x)\left\vert P_{L}\phi \right\vert
^{2}(x)\left( P_{L}\bar{\phi}\right) (x)dxE_{\text{NLS}}^{H}(\phi
)^{k-1}\,d\mu (\phi ) \\
&=&2i\func{Im}\int P_{H}\left( \left\vert \phi \right\vert ^{4}\phi \right)
(x)\left( P_{L}\bar{\phi}\right) (x) \\
&&\times \left( P_{H}\phi \right) (x)\left\vert P_{L}\phi \right\vert
^{2}(x)\left( P_{L}\bar{\phi}\right) (x)dxE_{\text{NLS}}^{H}(\phi
)^{k-1}\,d\mu (\phi )
\end{eqnarray*}%
That is%
\begin{eqnarray*}
\left\vert \text{E5}\right\vert &\lesssim &\int \left\Vert P_{H}\left(
\left\vert \phi \right\vert ^{4}\phi \right) \right\Vert _{L^{\frac{6}{5}%
}}\left\Vert P_{L}\phi \right\Vert _{L^{\infty }}^{4}\left\Vert P_{H}\phi
\right\Vert _{L^{6}}E_{\text{NLS}}^{H}(\phi )^{k-1}\,d\mu (\phi ) \\
&\lesssim &M^{2}\int \left\Vert \phi \right\Vert _{L^{6}}^{6}\left\Vert
\nabla P_{L}\phi \right\Vert _{L^{2}}^{4}E_{\text{NLS}}^{H}(\phi
)^{k-1}\,d\mu (\phi ) \\
&\lesssim &C_{0}^{4}M^{2}\int \left\Vert \phi \right\Vert _{L^{6}}^{6}E_{%
\text{NLS}}^{H}(\phi )^{k-1}\,d\mu (\phi ) \\
&\lesssim &M^{2}\int E_{\text{NLS}}(\phi )E_{\text{NLS}}^{H}(\phi
)^{k-1}\,d\mu (\phi ) \\
&\lesssim &M^{2}(E^{(k)})^{1/k}(E_{H}^{(k)})^{(k-1)/k}.
\end{eqnarray*}%
\textbf{Term G4, G5, G6.}%
\begin{eqnarray*}
\text{G4} &=&\func{Tr}%
_{1}3B_{1;2,3}^{+}P_{H}^{1}P_{L}^{2}P_{L}^{3}P_{H}^{1^{\prime
}}P_{L}^{2^{\prime }}P_{L}^{3^{\prime }}(B_{1;4,5}^{+}-B_{1;4,5}^{-})\sigma
^{(5)} \\
&=&3\int \left[ P_{H}\left( \left\vert \phi \right\vert ^{4}\phi \right)
(x)\left( P_{H}\bar{\phi}\right) (x)-\left( P_{H}\phi \right) (x)P_{H}\left(
\left\vert \phi \right\vert ^{4}\bar{\phi}\right) (x)\right] \\
&&\times \left\vert P_{L}\phi \right\vert ^{4}(x)dxE_{\text{NLS}}^{H}(\phi
)^{k-1}\,d\mu (\phi )
\end{eqnarray*}%
which can be estimated in the same way as E5, we skip the details and reach%
\begin{equation*}
\left\vert \text{G4}\right\vert \lesssim
M^{2}(E^{(k)})^{1/k}(E_{H}^{(k)})^{(k-1)/k}.
\end{equation*}%
G5 and G6 are similar, so we only deal with G5. In fact,%
\begin{eqnarray*}
\text{G5} &=&\func{Tr}%
_{1}3B_{1;2,3}^{+}P_{H}^{1}P_{L}^{2}P_{L}^{3}P_{H}^{1^{\prime
}}P_{L}^{2^{\prime }}P_{L}^{3^{\prime }}(B_{2;4,5}^{+}-B_{2;4,5}^{-})\sigma
^{(5)} \\
&=&\int \left[ P_{L}\left( \left\vert \phi \right\vert ^{4}\phi \right)
\left( P_{L}\bar{\phi}\right) (x)-\left( P_{L}\phi \right) (x)P_{L}\left(
\left\vert \phi \right\vert ^{4}\bar{\phi}\right) (x)\right] \\
&&\times \left\vert P_{H}\phi \right\vert ^{2}(x)\left\vert P_{L}\phi
\right\vert ^{2}dxE_{\text{NLS}}^{H}(\phi )^{k-1}\,d\mu (\phi ) \\
&=&6\func{Im}\int P_{L}\left( \left\vert \phi \right\vert ^{4}\phi \right)
\left( P_{L}\bar{\phi}\right) (x)\left\vert P_{H}\phi \right\vert ^{2}(x) \\
&&\times \left\vert P_{L}\phi \right\vert ^{2}\left( P_{L}\bar{\phi}\right)
(x)dxE_{\text{NLS}}^{H}(\phi )^{k-1}\,d\mu (\phi )
\end{eqnarray*}%
Hence,%
\begin{equation*}
\left\vert \text{G5}\right\vert \lesssim \int \left\Vert P_{L}\left(
\left\vert \phi \right\vert ^{4}\phi \right) \right\Vert _{L^{6}}\left\Vert
P_{L}\phi \right\Vert _{L^{6}}^{3}\left\Vert P_{H}\phi \right\Vert
_{L^{6}}^{2}E_{\text{NLS}}^{H}(\phi )^{k-1}\,d\mu (\phi )
\end{equation*}%
again, by Bernstein $\Vert P_{L}f\Vert _{L_{x}^{6}}\lesssim M^{2}\Vert
f\Vert _{L_{x}^{6/5}}$,%
\begin{eqnarray*}
\left\vert \text{G5}\right\vert &\lesssim &M^{2}\int \left\Vert \left\vert
\phi \right\vert ^{4}\phi \right\Vert _{L^{\frac{6}{5}}}\left\Vert \phi
\right\Vert _{L^{6}}^{5}E_{\text{NLS}}^{H}(\phi )^{k-1}\,d\mu (\phi ) \\
&\lesssim &C_{0}^{4}M^{2}\int \left\Vert \phi \right\Vert _{L^{6}}^{6}E_{%
\text{NLS}}^{H}(\phi )^{k-1}\,d\mu (\phi ) \\
&\lesssim &M^{2}\int E_{\text{NLS}}(\phi )E_{\text{NLS}}^{H}(\phi
)^{k-1}\,d\mu (\phi ) \\
&\lesssim &M^{2}(E^{(k)})^{1/k}(E_{H}^{(k)})^{(k-1)/k}.
\end{eqnarray*}%
So far, we have estimated A1 to G6, and proved that 
\begin{equation*}
\left\vert \frac{1}{k}\partial _{t}E_{H}^{(k)}\right\vert \lesssim
M^{2}(E^{(k)})^{1/k}(E_{H}^{(k)})^{(k-1)/k}
\end{equation*}%
which is exactly (\ref{E:main-1}).

\subsubsection{Concluding the Proof of Theorem \protect\ref{Thm:UTFLinTime} 
\label{Sec:Proof of High Kinetic Eats High Energy}}

Recall $E_{H}^{(k)}(t),$ in terms of the de Finetti measure, we can rewrite
it as%
\begin{equation*}
E_{H}^{(k)}(t)=\int \left( E_{NLS}^{H}(\phi )\right) ^{k}d\mu _{t}(\phi ).
\end{equation*}%
Since the complex conjugates in $E_{NLS}^{H}$ do not effect the estimates,
we put%
\begin{eqnarray*}
E_{NLS}^{H}(\phi ) &=&\Vert P_{H}\nabla \phi \Vert _{L^{2}}^{2}+a_{3}\int
\left( P_{H}\phi \right) ^{3}\left( P_{L}\phi \right) ^{3}dx \\
&&+a_{4}\int \left( P_{H}\phi \right) ^{4}\left( P_{L}\phi \right) ^{2}dx \\
&&+a_{5}\int \left( P_{H}\phi \right) ^{4}\left( P_{L}\phi \right) ^{2}dx \\
&&+\int \left\vert P_{H}\phi \right\vert ^{6}dx
\end{eqnarray*}%
where $a_{3},a_{4},a_{5}$ are some binomial coefficients. Applying the
refined Sobolev estimates (\ref{sobolev:refine1})-(\ref{sobolev:refine3}) to
the high-low terms in $E_{NLS}^{H}(\phi )$, we attain a lower bound\footnote{%
Estimate (\ref{bound:lower high NLS energy bound}) does not hold for a
focusing problem in which the nonpositive summand $-\left\Vert P_{H}\nabla
\phi \right\Vert _{L^{6}}^{6}$ cannot be dropped.} for $E_{NLS}^{H}(\phi )$: 
\begin{eqnarray}
&&E_{NLS}^{H}(\phi )  \label{bound:lower high NLS energy bound} \\
&\geqslant &\Vert P_{H}\nabla \phi \Vert _{L^{2}}^{2}-C\Vert \nabla \phi
\Vert _{L^{2}}(\Vert \nabla \phi _{M<\bullet <R}\Vert _{L^{2}}^{2}\Vert
\nabla P_{H}\phi \Vert _{L^{2}}^{3}+\left( \frac{M}{R}\right) ^{\frac{1}{2}%
}\Vert \nabla P_{H}\phi \Vert _{L^{2}}^{5})  \notag \\
&\geqslant &\Vert P_{H}\nabla \phi \Vert _{L^{2}}^{2}-C\left[ \Vert \nabla
\phi \Vert _{L^{2}}^{2}\Vert \nabla \phi _{M<\bullet <R}\Vert
_{L^{2}}^{2}+C\left( \frac{M}{R}\right) ^{\frac{1}{2}}\Vert \nabla \phi
\Vert _{L^{2}}^{4}\right] \Vert \nabla P_{H}\phi \Vert _{L^{2}}^{2}).  \notag
\end{eqnarray}%
In the support of $d\mu _{t}(\phi )$, if we take $R=100CC_{0}^{4}M,$ (\ref%
{bound:lower high NLS energy bound}) simplifies to 
\begin{eqnarray}
E_{NLS}^{H}(\phi ) &\geqslant &\left[ 1-CC_{0}^{4}\left( \Vert \nabla \phi
_{M<\bullet <R}\Vert _{L^{2}}^{2}+\left( \frac{M}{R}\right) ^{\frac{1}{2}%
}\right) \right] \Vert P_{H}\nabla \phi \Vert _{L^{2}}^{2}
\label{bound:lower high NLS energy bound under measure} \\
&\geqslant &\left[ \frac{2}{3}-CC_{0}^{4}\Vert \nabla \phi _{M<\bullet
<R}\Vert _{L^{2}}^{2}\right] \Vert P_{H}\nabla \phi \Vert _{L^{2}}^{2} 
\notag \\
&\geqslant &\left( \frac{2}{3}-CK_{M\leq \bullet \leq R}(t)\right) \Vert
P_{H}\nabla \phi \Vert _{L^{2}}^{2}  \notag
\end{eqnarray}%
where $K_{M\leq \bullet \leq R}(t)$ is defined in (\ref{eqn:def of
intermediate kinetic}).

Recall (\ref{eqn:HUFL}) which implies that $K_{M\leq \bullet \leq
R}(0)\leqslant \frac{\varepsilon }{4}$. Our goal here is to prove that $%
K_{M\leq \bullet \leq R}(t)$ remains small for all $t\in \left( 0,\frac{1}{8}%
\left( C\langle M\rangle ^{2}\right) ^{-1}\varepsilon ^{2}\right) $ so that
we can put (\ref{bound:lower high NLS energy bound under measure}) into $%
E_{H}^{(k)}(t)$ and get 
\begin{eqnarray*}
E_{H}^{(k)}(t) &=&\int \left( E_{NLS}^{H}(\phi )\right) ^{k}d\mu _{t}(\phi )
\\
&\geqslant &\left( \frac{1}{2}\right) ^{k}\int \Vert P_{H}\nabla \phi \Vert
_{L^{2}}^{2k}d\mu _{t}(\phi )
\end{eqnarray*}%
Once we have proved the above, due to (\ref{eq conclusion of high energy
growth}), we would have%
\begin{equation*}
\left( \frac{\varepsilon }{2}\right) ^{k}\geqslant \left( \frac{1}{2}\right)
^{k}\int \Vert P_{H}\nabla \phi \Vert _{L^{2}}^{2k}d\mu _{t}(\phi )
\end{equation*}%
which is indeed%
\begin{equation*}
\limfunc{Tr}R^{(1,k)}P_{>M}^{(k)}\gamma
^{(k)}(t)P_{>M}^{(k)}R^{(1,k)}\leqslant \varepsilon ^{2k}
\end{equation*}%
for $t\in \left[ 0,\frac{1}{8}\left( C\langle M\rangle ^{2}\right)
^{-1}\varepsilon ^{2}\right] $ and hence have finished the proof of Theorem %
\ref{Thm:UTFLinTime}.

Assume for the moment that, $K_{M\leq \bullet \leq R}(t)$ is continuous in $%
t $, a result which we will prove in \S \ref{Sec:Proof of High Kinetic
Continuity} as Lemma \ref{lem:intermediate kinetic energy is continuous}, we
can let $t_{\ast }$ be the infimum of all times $t\in \left( 0,\frac{1}{8}%
\left( C\langle M\rangle ^{2}\right) ^{-1}\varepsilon \right) $ for which $%
K_{M\leq \bullet \leq R}(t)\geqslant 2\varepsilon $, then we infer\footnote{%
This is why Lemma \ref{lem:intermediate kinetic energy is continuous} is
crucial.}from the continuity of $K_{M\leq \bullet \leq R}(t)$ that $K_{M\leq
\bullet \leq R}(t_{\ast })=2\varepsilon $.

Putting (\ref{bound:lower high NLS energy bound under measure}) into $%
E_{H}^{(k)}(t_{\ast })$, we have%
\begin{equation*}
\frac{1}{2^{k}}\left( \int \Vert P_{H}\nabla \phi \Vert _{L^{2}}^{2k}d\mu
_{t_{\ast }}(\phi )\right) \leqslant \left( E_{H}^{(k)}(t_{\ast })\right)
\leqslant \left( \frac{\varepsilon }{2}\right) ^{k}
\end{equation*}%
Taking the $k$-th root and letting $k\rightarrow \infty $, we have%
\begin{equation*}
\frac{1}{2}K_{M\leqslant \bullet \leqslant R}(t_{\ast })\leqslant \frac{1}{2}%
\left( \int \Vert P_{\geqslant M}\nabla \phi \Vert _{L^{2}}^{2k}d\mu
_{t_{\ast }}(\phi )\right) ^{\frac{1}{k}}\leqslant \left(
E_{H}^{(k)}(t_{\ast })\right) ^{\frac{1}{k}}\leqslant \frac{\varepsilon }{2}
\end{equation*}%
Thus,%
\begin{equation*}
K_{M\leq \bullet \leq R}(t_{\ast })\leqslant \varepsilon
\end{equation*}%
which contradicts the assumption that $K_{M\leq \bullet \leq R}(t_{\ast
})=2\varepsilon $. That is, such a $t_{\ast }$ does not exist in $\left( 0,%
\frac{1}{8}\left( C\langle M\rangle ^{2}\right) ^{-1}\varepsilon \right) $.
Hence, we have proved that $K_{M\leq \bullet \leq R}(t)$ remains small for
all $t\in \left[ 0,\frac{1}{8}\left( C\langle M\rangle ^{2}\right)
^{-1}\varepsilon ^{2}\right] $, and established Theorem \ref{Thm:UTFLinTime}
assuming the continuity of $K_{M\leq \bullet \leq R}(t)$.

\begin{lemma}[Refined Sobolev]
\label{lemma:refined sobolev for continuity}For any dyadic level $R\geq M$,
let $\phi _{M<\bullet <R}=P_{<R}P_{>M}\phi $, we have the following refined
Sobolev estimates: 
\begin{equation}
\left\vert \int \left( P_{H}\phi \right) ^{3}\left( P_{L}\phi \right)
^{3}dx\right\vert \lesssim \Vert \nabla \phi \Vert _{L^{2}}^{3}(\Vert \nabla
\phi _{M<\bullet <R}\Vert _{L^{2}}^{2}\Vert \nabla P_{H}\phi \Vert _{L^{2}}+(%
\frac{M}{R})^{3/2}\Vert \nabla P_{H}\phi \Vert _{L^{2}}^{3})
\label{sobolev:refine1}
\end{equation}%
\begin{equation}
\left\vert \int \left( P_{H}\phi \right) ^{4}\left( P_{L}\phi \right)
^{2}dx\right\vert \lesssim \Vert \nabla \phi \Vert _{L^{2}}^{2}(\Vert \nabla
\phi _{M<\bullet <R}\Vert _{L^{2}}^{2}\Vert \nabla P_{H}\phi \Vert
_{L^{2}}^{2}+\left( \frac{M}{R}\right) \Vert \nabla P_{H}\phi \Vert
_{L^{2}}^{4})  \label{sobolev:refine2}
\end{equation}%
and%
\begin{equation}
\left\vert \int \left( P_{H}\phi \right) ^{5}\left( P_{L}\phi \right)
dx\right\vert \lesssim \Vert \nabla \phi \Vert _{L^{2}}(\Vert \nabla \phi
_{M<\bullet <R}\Vert _{L^{2}}^{2}\Vert \nabla P_{H}\phi \Vert
_{L^{2}}^{3}+\left( \frac{M}{R}\right) ^{\frac{1}{2}}\Vert \nabla P_{H}\phi
\Vert _{L^{2}}^{5})  \label{sobolev:refine3}
\end{equation}%
Similar estimates hold where any $P_{H}\phi $ is replaced by $\overline{%
P_{H}\phi }$ and any $P_{L}\phi $ is replaced by $\overline{P_{L}\phi }$, as
will be apparent from the proof.
\end{lemma}

\begin{proof}
We prove only (\ref{sobolev:refine3}) since the others are similar.
Represent the product $\left( P_{H}\phi \right) ^{5}\left( P_{L}\phi \right) 
$ as $\phi _{1}\phi _{2}\phi _{3}\phi _{4}\phi _{5}\phi _{6}$. On the
Fourier side of the integral, we have $\xi _{1}+\cdots +\xi _{6}=0$ and $%
\left\vert \xi _{6}\right\vert \leqslant M$. For $1\leqslant j\leqslant 5$,
break $\phi _{j}$ into $\sum_{M_{j}}\phi _{M_{j}}$ where each $M_{j}>M$.
Then we have to control 
\begin{equation*}
\sum_{M_{1},M_{2},M_{3},M_{4},M_{5}}\int \phi _{M_{1}}\phi _{M_{2}}\phi
_{M_{3}}\phi _{M_{4}}\phi _{M_{5}}\phi _{6}\,dx
\end{equation*}%
At the expense of a factor of $2^{5}$, we may assume the sizes of the
frequencies are ordered so that $M_{1}\geqslant M_{2}\geqslant
M_{3}\geqslant M_{4}\geqslant M_{5}>M$. Now we have that $M_{1}\sim M_{2}$
(since otherwise, if $M_{2}\leq \frac{1}{8}M_{1}$, then we cannot have $%
\left\vert \xi _{1}+...+\xi _{5}\right\vert \leqslant M$ which is forced by
the restriction $\left\vert \xi _{6}\right\vert \leqslant M$). Hence, at the
expense of a factor of $3$, we can take $M_{1}=M_{2}$. Hence the sum reduces
to 
\begin{equation*}
\sum_{M_{1},M_{3},M_{4},M_{5}>M}\int \phi _{M_{1}}^{2}\phi _{M_{3}}\phi
_{M_{4}}\phi _{M_{5}}\phi _{6}\,dx
\end{equation*}%
Now estimate as 
\begin{equation*}
\sum_{M_{1},M_{3},M_{4},M_{5}>M}\Vert \phi _{M_{1}}\Vert _{L^{2}}^{2}\Vert
\phi _{M_{3}}\Vert _{L^{\infty }}\Vert \phi _{M_{4}}\Vert _{L^{\infty
}}\Vert \phi _{M_{5}}\Vert _{L^{\infty }}\Vert \phi _{6}\Vert _{L^{\infty }}
\end{equation*}%
By Bernstein, $\Vert \phi _{j}\Vert _{L^{\infty }}\lesssim M_{j}^{1/2}\Vert
\nabla \phi _{j}\Vert _{L^{2}}$ for $3\leqslant j\leqslant 5$ and $\Vert
\phi _{6}\Vert _{L^{\infty }}\lesssim M^{\frac{1}{2}}\Vert \nabla \phi
_{6}\Vert _{L^{2}}$, and this gives the bound 
\begin{equation*}
M^{\frac{1}{2}%
}\sum_{M_{1},M_{3},M_{4},M_{5}>M}M_{1}^{-2}M_{3}^{1/2}M_{4}^{1/2}M_{5}^{1/2}%
\Vert \nabla \phi _{M_{1}}\Vert _{L^{2}}^{2}\Vert \nabla \phi _{M_{3}}\Vert
_{L^{2}}\Vert \nabla \phi _{M_{4}}\Vert _{L^{2}}\Vert \nabla \phi
_{M_{5}}\Vert _{L^{2}}\Vert \nabla \phi _{6}\Vert _{L^{2}}
\end{equation*}%
Bounding $\Vert \nabla \phi _{M_{j}}\Vert _{L^{2}}\leqslant \Vert
P_{H}\nabla \phi \Vert _{L^{2}}$ for each $3\leqslant j\leqslant 5$, and
then carrying out sum in $M_{3},M_{4},M_{5}$ using $\sum_{M_{j}\leq
M_{1}}M_{j}^{1/2}=M_{1}^{1/2}$, we obtain 
\begin{equation*}
M^{\frac{1}{2}}\sum_{M_{1}\geqslant M}M_{1}^{-1/2}\Vert \nabla \phi
_{M_{1}}\Vert _{L^{2}}^{2}\Vert P_{H}\nabla \phi \Vert _{L^{2}}^{3}\Vert
P_{L}\nabla \phi \Vert _{L^{2}}
\end{equation*}%
Split this sum into two pieces, one for $M_{1}\leqslant R$ and the other for 
$M_{1}\geqslant R$, we then have%
\begin{eqnarray*}
&&M^{\frac{1}{2}}\sum_{M_{1}\geqslant M}M_{1}^{-1/2}\Vert \nabla \phi
_{M_{1}}\Vert _{L^{2}}^{2}\Vert P_{H}\nabla \phi \Vert _{L^{2}}^{3}\Vert
P_{L}\nabla \phi \Vert _{L^{2}} \\
&\leqslant &\Vert P_{H}\nabla \phi \Vert _{L^{2}}^{3}\Vert P_{L}\nabla \phi
\Vert _{L^{2}}\sum_{R\geqslant M_{1}\geqslant M}\Vert \nabla \phi
_{M_{1}}\Vert _{L^{2}}^{2}+\Vert P_{H}\nabla \phi \Vert _{L^{2}}^{5}\Vert
P_{L}\nabla \phi \Vert _{L^{2}}M^{\frac{1}{2}}\sum_{M_{1}\geqslant
R}M_{1}^{-1/2} \\
&\lesssim &\Vert P_{L}\nabla \phi \Vert _{L^{2}}(\Vert \nabla \phi
_{M<\bullet <R}\Vert _{L^{2}}^{2}\Vert \nabla P_{H}\phi \Vert
_{L^{2}}^{3}+\left( \frac{M}{R}\right) ^{\frac{1}{2}}\Vert \nabla P_{H}\phi
\Vert _{L^{2}}^{5})
\end{eqnarray*}%
which implies (\ref{sobolev:refine3}).
\end{proof}

\subsubsection{The Continuity of the Intermediate Kinetic Energy\label%
{Sec:Proof of High Kinetic Continuity}}

\begin{lemma}
\label{lem:intermediate kinetic energy is continuous}$K_{M\leq \bullet \leq
R}(t)$, as defined in (\ref{eqn:def of intermediate kinetic}), is continuous
in $t$, in fact Lipshitz, with constant depending on $R$.
\end{lemma}

\begin{proof}
By (\ref{E:GP}), we compute%
\begin{eqnarray*}
&&\partial _{t}\left( \limfunc{Tr}P_{M<\bullet <R}^{(k)}R^{(1,k)}\gamma
^{(k)}P_{M<\bullet <R}^{(k)}R^{(1,k)}\right) \\
&=&\frac{1}{i}\sum_{j=1}^{k}\limfunc{Tr}\left( P_{M<\bullet
<R}^{(k)}R^{(1,k)}\left[ -\Delta _{j},\gamma ^{(k)}\right] P_{M<\bullet
<R}^{(k)}R^{(1,k)}\right) \\
&&+\frac{1}{i}\sum_{j=1}^{k}\limfunc{Tr}\left( P_{M<\bullet
<R}^{(k)}R^{(1,k)}(B_{j;k+1,k+2}^{+}-B_{j;k+1,k+2}^{-})\gamma
^{(k+2)}P_{M<\bullet <R}^{(k)}R^{(1,k)}\right) \\
&=&I+II.
\end{eqnarray*}

On the one hand, a typical term in $I$ reads,%
\begin{eqnarray*}
\left\vert \limfunc{Tr}P_{M<\bullet <R}^{(k)}R^{(1,k)}\Delta _{1}\gamma
^{(k)}P_{M<\bullet <R}^{(k)}R^{(1,k)}\right\vert &=&|\int \left( \int \left(
P_{M<\bullet <R}\Delta \nabla \phi \right) \left( P_{M<\bullet <R}\nabla 
\bar{\phi}\right) dx_{1}\right) \\
&&\left( \dprod\limits_{j=2}^{k}\int \left( P_{M<\bullet <R}\nabla \phi
\right) \left( P_{M<\bullet <R}\nabla \bar{\phi}\right) dx_{j}\right) d\mu
(\phi )| \\
&\leqslant &AB
\end{eqnarray*}%
where%
\begin{eqnarray*}
A &=&\left( \int \left\vert \int \left( P_{M<\bullet <R}\Delta \nabla \phi
\right) \left( P_{M<\bullet <R}\nabla \bar{\phi}\right) dx_{1}\right\vert
^{k}d\mu (\phi )\right) ^{\frac{1}{k}} \\
&\leqslant &R^{2}\left( \int \left\vert \int \left\vert P_{M<\bullet
<R}\nabla \phi \right\vert ^{2}dx_{1}\right\vert ^{k}d\mu (\phi )\right) ^{%
\frac{1}{k}}\leqslant C_{0}^{2}R^{2}
\end{eqnarray*}%
and%
\begin{eqnarray*}
B &=&\left( \int \left( \dprod\limits_{j=2}^{k}\int \left( P_{M<\bullet
<R}\nabla \phi \right) \left( P_{M<\bullet <R}\nabla \bar{\phi}\right)
dx_{j}\right) ^{\frac{k}{k-1}}d\mu (\phi )\right) ^{\frac{k-1}{k}} \\
&=&\left( \func{tr}P_{M<\bullet <R}^{(k)}R^{(1,k)}\gamma ^{(k)}P_{M<\bullet
<R}^{(k)}R^{(1,k)}\right) ^{\frac{k-1}{k}}.
\end{eqnarray*}

On the other hand, a typical term in $II$ can be estimated by%
\begin{eqnarray*}
&&\left\vert \limfunc{Tr}P_{M<\bullet
<R}^{(k)}R^{(1,k)}(B_{1;k+1,k+2}^{+})\gamma ^{(k+2)}P_{M<\bullet
<R}^{(k)}R^{(1,k)}\right\vert \\
&=&|\int (\int (P_{M<\bullet <R}\nabla (\left\vert \phi \right\vert ^{4}\phi
))(P_{M<\bullet <R}\nabla \bar{\phi})dx_{1})(\dprod\limits_{j=2}^{k}\int
(P_{M<\bullet <R}\nabla \phi )(P_{M<\bullet <R}\nabla \bar{\phi})dx_{j})d\mu
(\phi )| \\
&\leqslant &CD
\end{eqnarray*}%
where%
\begin{equation*}
C=\left( \int \left\vert \int \left( P_{M<\bullet <R}\nabla \left(
\left\vert \phi \right\vert ^{4}\phi \right) \right) \left( P_{M<\bullet
<R}\nabla \bar{\phi}\right) dx_{1}\right\vert ^{k}d\mu (\phi )\right) ^{%
\frac{1}{k}},
\end{equation*}%
and%
\begin{eqnarray*}
D &=&\left( \int \left( \dprod\limits_{j=2}^{k}\int \left( P_{M<\bullet
<R}\nabla \phi \right) \left( P_{M<\bullet <R}\nabla \bar{\phi}\right)
dx_{j}\right) ^{\frac{k}{k-1}}d\mu (\phi )\right) ^{\frac{k-1}{k}} \\
&=&\left( \limfunc{Tr}P_{M<\bullet <R}^{(k)}R^{(1,k)}\gamma
^{(k)}P_{M<\bullet <R}^{(k)}R^{(1,k)}\right) ^{\frac{k-1}{k}}.
\end{eqnarray*}%
We can use Bernstein $\left\Vert P_{M<\bullet <R}\nabla f\right\Vert _{L^{%
\frac{3}{2}}}\lesssim R^{\frac{3}{2}}\left\Vert f\right\Vert _{L^{\frac{6}{5}%
}}$ and $\left\Vert P_{M<\bullet <R}f\right\Vert _{L^{3}}\lesssim R^{\frac{1%
}{2}}\left\Vert f\right\Vert _{L^{2}}$ in $C$ and reach 
\begin{equation*}
C\lesssim \left( \int \left\vert R^{\frac{3}{2}}\left\Vert \left\vert \phi
\right\vert ^{4}\phi \right\Vert _{L^{\frac{6}{5}}}R^{\frac{1}{2}}\left\Vert
\nabla \phi \right\Vert _{L^{2}}\right\vert ^{k}d\mu (\phi )\right) ^{\frac{1%
}{k}}\lesssim C_{0}^{6}R^{2}
\end{equation*}

Putting the estimate for $I$ and $II$ together, we have%
\begin{equation*}
\left\vert \partial _{t}\limfunc{Tr}\left( P_{M<\bullet
<R}^{(k)}R^{(1,k)}\gamma ^{(k)}P_{M<\bullet <R}^{(k)}R^{(1,k)}\right)
\right\vert \lesssim kC_{0}^{6}R^{2}\left( \limfunc{Tr}P_{M<\bullet
<R}^{(k)}R^{(1,k)}\gamma ^{(k)}P_{M<\bullet <R}^{(k)}R^{(1,k)}\right) ^{%
\frac{k-1}{k}}
\end{equation*}%
Then it follows that 
\begin{equation*}
\left\vert \partial _{t}\left( \limfunc{Tr}P_{M<\bullet
<R}^{(k)}R^{(1,k)}\gamma ^{(k)}P_{M<\bullet <R}^{(k)}R^{(1,k)}\right)
^{1/k}\right\vert \lesssim R^{2}
\end{equation*}%
Integrate in time to obtain 
\begin{eqnarray*}
&&\left\vert \left( \limfunc{Tr}P_{M<\bullet <R}^{(k)}R^{(1,k)}\gamma
^{(k)}P_{M<\bullet <R}^{(k)}R^{(1,k)}\right) ^{1/k}(t_{2})-\left( \limfunc{Tr%
}P_{M<\bullet <R}^{(k)}R^{(1,k)}\gamma ^{(k)}P_{M<\bullet
<R}^{(k)}R^{(1,k)}\right) ^{1/k}(t_{1})\right\vert \\
&\lesssim &R^{2}\left\vert t_{2}-t_{1}\right\vert
\end{eqnarray*}%
with the right hand side independent of $k$.

We can then send $k\rightarrow \infty $ to obtain 
\begin{equation*}
\left\vert K_{M\leq \bullet \leq R}(t_{2})-K_{M\leq \bullet \leq
R}(t_{1})\right\vert \lesssim R^{2}\left\vert t_{2}-t_{1}\right\vert
\end{equation*}%
which is the desired continuity.
\end{proof}

Since we have now proved Lemma \ref{lem:intermediate kinetic energy is
continuous}, the argument in \S \ref{Sec:Proof of High Kinetic Eats High
Energy} is now completed and we have established Theorem \ref{Thm:UTFLinTime}%
.

\subsection{Uniqueness through HUFL\label{subsec:GPUTFLUniqueness}}

\begin{theorem}
\label{THM:UniquessAssumingUFL}For initial data satisfying HUFL, there
exists $T>0$ such that there is exactly one admissible solution to (\ref%
{hierarchy:quintic GP in differential form}) in $[0,T]$ satisfying (\ref%
{condition:kinetic energy}).
\end{theorem}

Once we have proved Theorem \ref{THM:UniquessAssumingUFL}, Theorem \ref%
{Thm:TotalUniqueness} then follows. In fact, we can write the initial datum
as 
\begin{equation*}
\gamma ^{(k)}(0)=\int \left( \left\vert \phi \right\rangle \left\langle \phi
\right\vert \right) ^{\otimes k}d\mu _{0}(\phi )
\end{equation*}%
by Theorem \ref{thm:qde}. If $\left\{ \gamma ^{(k)}(0)\right\}
_{k=1}^{\infty }$ satisfies HUFL at $t=0$, then by the same Chebyshev
argument as in \cite[Lemma 4.5]{TCNPdeFinitte}, $d\mu _{0}$ is supported in
the set 
\begin{equation*}
\{\phi \in \mathbb{B}\left( L^{2}(\mathbb{T}^{3})\right) :\left\Vert
P_{>M}\left\langle \nabla \right\rangle \phi \right\Vert _{L^{2}}\leqslant
\varepsilon \}.
\end{equation*}%
Let $S_{t}$ be the solution operator of the energy critical NLS (\ref{eqn:3d
Quintic with 1 Coupling}), then by \cite{CKSTT,IP} in which $S_{t}\phi $ is
proven to exists for all $t<+\infty $, there is a global solution to (\ref%
{hierarchy:quintic GP in differential form}) given by 
\begin{equation*}
\gamma ^{(k)}(t)=\int \left( \left\vert S_{t}\phi \right\rangle \left\langle
S_{t}\phi \right\vert \right) ^{\otimes k}d\mu _{0}(\phi )\text{.}
\end{equation*}%
Assume the maximal time of uniqueness is $T^{\ast }<+\infty $, we infer by
uniqueness at $T^{\ast }$ that, 
\begin{equation*}
\gamma ^{(k)}(T^{\ast })=\int \left( \left\vert S_{T^{\ast }}\phi
\right\rangle \left\langle S_{T^{\ast }}\phi \right\vert \right) ^{\otimes
k}d\mu _{0}(\phi ).
\end{equation*}%
Then, at $T^{\ast }$, since $d\mu _{0}$ is supported in $\{\left\Vert
P_{>M}\left\langle \nabla \right\rangle \phi \right\Vert _{L^{2}}\leqslant
\varepsilon \}$ and we have Theorem \ref{Cor:UTFLforNLS} for NLS (\ref%
{eqn:3d Quintic with 1 Coupling}), we know $\gamma ^{(k)}(T^{\ast })$
satisfies HUFL at $T^{\ast }$ and is good enough to bootstrap with Theorem %
\ref{THM:UniquessAssumingUFL} for uniqueness again. Hence, this contradicts
that the maximal time of uniqueness is $T^{\ast }<+\infty $. Thus $T^{\ast
}=+\infty $ and we have proved Theorem \ref{Thm:TotalUniqueness}.

We now start the proof of Theorem \ref{THM:UniquessAssumingUFL}. Because
hierarchy (\ref{hierarchy:quintic GP in differential form}) is linear, we
will prove that $\gamma ^{(k)}(t)=0$ if $\gamma ^{(k)}(t)=\gamma
_{1}^{(k)}(t)-\gamma _{2}^{(k)}(t)$, where $\gamma _{1}^{(k)}(t)$ and $%
\gamma _{2}^{(k)}(t)$ are two admissible solutions which satisfy (\ref%
{condition:kinetic energy}) and HUFT and subject to the same initial datum.
Without lose of generality, we prove $\gamma ^{(1)}(t)=0$.

We use (\ref{hierarchy:quintic GP in integral form}), the integral form of (%
\ref{hierarchy:quintic GP in differential form}). From now on, we will also
omit the $-ib_{0}$ in front of the coupling term in (\ref{hierarchy:quintic
GP in integral form}) so that we do not need to track its exact power. With
this notation change, (\ref{hierarchy:quintic GP in integral form}) reads 
\begin{equation}
\gamma ^{(k)}(t)=U^{(k)}(t)\gamma
^{(k)}(0)+\int_{0}^{t}U^{(k)}(t-s)B^{(k+2)}\gamma ^{(k+2)}(s)ds
\label{hierarchy:GPwithCouplingI}
\end{equation}%
We will denote $U(t)=e^{it\Delta }$ without explicitly writing out the
spatial variable. Usually, the space variable is clear from the context.

Let us define 
\begin{eqnarray*}
&&J^{(2k+1)}(\underline{t}_{2k+1})(f^{(2k+1)})(t_{1}) \\
&=&U^{(1)}(t_{1}-t_{3})B^{(3)}U^{(3)}(t_{3}-t_{5})B^{(5)}...U^{(2k-1)}(t_{2k-1}-t_{2k+1})B^{(2k+1)}f^{(2k+1)}(t_{2k+1})
\end{eqnarray*}%
with $\underline{t}_{2k+1}=\left( t_{3},t_{5},...,t_{2k-1},t_{2k+1}\right) $%
. We can then shorten%
\begin{equation*}
\gamma
^{(1)}(t_{1})=\int_{0}^{t_{1}}\int_{0}^{t_{3}}...%
\int_{0}^{t_{2k-1}}J^{(2k+1)}(\gamma ^{(2k+1)})(\underline{t}_{2k+1})d%
\underline{t}_{2k+1}\text{.}
\end{equation*}%
Before delving into the analysis, we notice that there are $(2k+1)!!2^{k}$
summands inside $\gamma ^{(1)}(t_{1})$. We reduce the total number of
summands by the following theorem which was proved by T. Chen and Pavlovic 
\cite{TChenAndNP}\ and combines the terms via the Klainerman-Machedon board
game argument \cite{KlainermanAndMachedon}.

\begin{lemma}[\protect\cite{TChenAndNP}]
\label{Lem:KLBoardGame}One can rewrite 
\begin{equation*}
\int_{0}^{t_{1}}\int_{0}^{t_{3}}...\int_{0}^{t_{2k-1}}J^{(2k+1)}(\underline{t%
}_{2k+1})(f^{(2k+1)})d\underline{t}_{2k+1}
\end{equation*}%
as a sum of at most $2^{3k-1}$ terms of the form%
\begin{equation*}
\int_{D}J^{(2k+1)}(\underline{t}_{2k+1};\mu )(f^{(2k+1)})d\underline{t}%
_{2k+1}
\end{equation*}%
where%
\begin{eqnarray*}
&&J^{(2k+1)}(\underline{t}_{2k+1};\mu )(f^{(2k+1)})(t_{1}) \\
&=&U^{(1)}(t_{1}-t_{3})B_{\mu (2);2,3}...U^{(2k-1)}(t_{2k-1}-t_{2k+1})B_{\mu
(2k);2k,2k+1}f^{(2k+1)}(t_{2k+1})
\end{eqnarray*}%
and $D\subset \left[ 0,t_{1}\right] ^{k}$, $\mu $ is a map from $%
\{2,3,...,2k\}$ to $\{1,2,3,...,2k-1\}$ such that $\mu (2)=1$ and $\mu (j)<j$
for all $j$.

That is,%
\begin{equation}
\gamma ^{(1)}(t_{1})=\sum_{\mu \in m}\int_{D}J^{(2k+1)}(\underline{t}%
_{2k+1};\mu )(\gamma ^{(2k+1)})(t_{1})d\underline{t}_{2k+1}
\label{eqn:term reduction for uniqueness}
\end{equation}%
with $\left\vert m\right\vert \leqslant 2^{3k-1}$ if we denote the set of
all such $\mu $ by $m$.
\end{lemma}

Thus, by inserting a factor of $2^{3k-1}$ into the final estimate, it
suffice to estimate a typical summand $\int_{D}J^{(2k+1)}(\underline{t}%
_{2k+1};\mu )(\gamma ^{(2k+1)})d\underline{t}_{2k+1}$. By paying an extra
factor of $2^{k}$, we can split all $B_{\mu (2l);2l,2l+1}$ inside $%
J^{(2k+1)}(\underline{t}_{2k+1};\mu )(\gamma ^{(2k+1)})$ into $B_{\mu
(2l);2l,2l+1}^{+}$ and $B_{\mu (2l);2l,2l+1}^{-}$ and focus on the estimate
of 
\begin{eqnarray}
&&J_{\pm }^{(2k+1)}(\underline{t}_{2k+1};\mu )(\gamma ^{(2k+1)})(t_{1})
\label{eqn:J+-forUniqueness} \\
&=&U^{(1)}(t_{1}-t_{3})B_{\mu (2);2,3}^{\pm
}...U^{(2k-1)}(t_{2k-1}-t_{2k+1})B_{\mu (2k);2k,2k+1}^{\pm }\gamma
^{(2k+1)}(t_{2k+1})  \notag
\end{eqnarray}%
where "$\pm $" simply means it could be a "$+$" or a "$-$".

The good thing about splitting $B_{\mu (2l);2l,2l+1}$, instead of estimating 
$J^{(2k+1)}(\underline{t}_{2k+1};\mu )(\gamma ^{(2k+1)})(t_{1})$ directly,
is that $J_{\pm }^{(2k+1)}(\underline{t}_{2k+1};\mu )(\gamma
^{(2k+1)})(t_{1})$ is indeed only one term instead of just another sum and
hence make its analysis more transparent. We will prove the following
proposition about estimating $J_{\pm }^{(2k+1)}$.

\begin{proposition}
\label{Prop:estimate for a typical term in uniqueness}For initial data
satisfying HUFL, there exists $C>0$ such that $\forall \varepsilon >0$, $%
\exists M(\varepsilon )$, we have, 
\begin{equation*}
\limfunc{Tr}\left\vert S^{(-1,1)}\int_{D}J_{\pm }^{(2k+1)}(\underline{t}%
_{2k+1};\mu )(\gamma ^{(2k+1)})(t_{1})d\underline{t}_{2k+1}\right\vert
\leqslant 2t_{1}\left( CC_{0}\right) ^{\frac{16}{5}k+2}\left( \left( t_{1}^{%
\frac{1}{4}}M^{\frac{1}{2}}\right) ^{\frac{10}{11}}CC_{0}+C\varepsilon
\right) ^{\frac{4}{5}k}
\end{equation*}%
for all $k$, provided that $t_{1}$ is smaller than the $T$ concluded in
Theorem \ref{Thm:UTFLinTime}.
\end{proposition}

We postpone the proof of Proposition \ref{Prop:estimate for a typical term
in uniqueness} to \S \ref{sec:proof of prop uniqueness}. We remark that, for
a specific $\pm $ combination, the power $\frac{4}{5}k$ could certainly be
higher, but cannot be lower.

With Proposition \ref{Prop:estimate for a typical term in uniqueness}, we
summarize the estimate for $\gamma ^{(1)}$. First, using the triangle
inequality twice, we have 
\begin{eqnarray*}
&&\sup_{t_{1}\in \lbrack 0,t]}\limfunc{Tr}\left\vert S^{(-1,1)}\gamma
^{(1)}(t_{1})\right\vert \\
&\leqslant &\sup_{t_{1}\in \lbrack 0,t]}\sum_{\mu \in m}\limfunc{Tr}%
\left\vert \int_{D}S^{(-1,1)}J^{(2k+1)}(\underline{t}_{2k+1};\mu )(\gamma
^{(2k+1)})(t_{1})d\underline{t}_{2k+1}\right\vert \\
&\leqslant &\sup_{t_{1}\in \lbrack 0,t]}\sum_{\mu \in m}\dsum\limits_{\pm }%
\limfunc{Tr}\left\vert \int_{D}S^{(-1,1)}J_{\pm }^{(2k+1)}(\underline{t}%
_{2k+1};\mu )(\gamma ^{(2k+1)})(t_{1})d\underline{t}_{2k+1}\right\vert
\end{eqnarray*}%
Plugging in Proposition \ref{Prop:estimate for a typical term in uniqueness}
gets us to%
\begin{eqnarray}
&&\sup_{t_{1}\in \lbrack 0,t]}\limfunc{Tr}\left\vert S^{(-1,1)}\gamma
^{(1)}(t_{1})\right\vert  \label{estimate:final estimate for gamma_1} \\
&\leqslant &\sup_{t_{1}\in \lbrack 0,t]}2t_{1}2^{3k-1}2^{k}\left(
CC_{0}\right) ^{\frac{16}{5}k+2}\left( \left( t_{1}^{\frac{1}{4}}M^{\frac{1}{%
2}}\right) ^{\frac{10}{11}}CC_{0}+C\varepsilon \right) ^{\frac{4}{5}k} 
\notag \\
&\leqslant &t2^{4k}\left( CC_{0}\right) ^{\frac{16}{5}k+2}\left( \left(
t_{1}^{\frac{1}{4}}M^{\frac{1}{2}}\right) ^{\frac{10}{11}}CC_{0}+C%
\varepsilon \right) ^{\frac{4}{5}k}  \notag \\
&\leqslant &\frac{t}{CC_{0}}\left( \left( t_{1}^{\frac{1}{4}}M^{\frac{1}{2}%
}\right) ^{\frac{10}{11}}CC_{0}^{5}+CC_{0}^{4}\varepsilon \right) ^{\frac{4}{%
5}k}  \notag
\end{eqnarray}%
for all $k$. Now, first choose $\varepsilon $, in $\left( \ref%
{estimate:final estimate for gamma_1}\right) $, such that $%
CC_{0}^{4}\varepsilon <\frac{1}{4}$, then choose $t<1$ such that $\left(
t_{1}^{\frac{1}{4}}M^{\frac{1}{2}}\right) ^{\frac{10}{11}}CC_{0}^{5}<\frac{1%
}{4}$, we thus have for some $t$ small enough,%
\begin{equation*}
\sup_{t_{1}\in \lbrack 0,t]}\limfunc{Tr}\left\vert S^{(-1,1)}\gamma
^{(1)}(t_{1})\right\vert \leqslant \frac{1}{CC_{0}}\left( \frac{1}{2}\right)
^{\frac{4}{5}k}\rightarrow 0\text{ as }k\rightarrow \infty \text{.}
\end{equation*}%
Bootstrapping this argument, we have 
\begin{equation*}
\sup_{t_{1}\in \lbrack 0,T]}\limfunc{Tr}\left\vert S^{(-1,1)}\gamma
^{(1)}(t_{1})\right\vert =0
\end{equation*}%
which proves Theorem \ref{THM:UniquessAssumingUFL}.

Before going into the proof of Proposition \ref{Prop:estimate for a typical
term in uniqueness}, we remark that the HUFL assumption in Theorem \ref%
{THM:UniquessAssumingUFL} can be weakened. We actually proved the following
theorem.

\begin{theorem}
Given a kinetic energy size $C_{0}$ of the admissible solution $\left\{
\gamma ^{(k)}\right\} $ to (\ref{hierarchy:quintic GP in differential form})$%
,$there is a threshold $\eta (C_{0})>0$. If there is a time $T>0$ and a
frequency $M>0$ such that%
\begin{equation*}
\sup_{t\in \lbrack 0,T]}\limfunc{Tr}S^{(1,k)}P_{>M}^{(k)}\gamma
^{(k)}(t)P_{>M}^{(k)}S^{(1,k)}\leqslant \eta ^{2k},
\end{equation*}%
then the solution is unique in $[0,T]$.
\end{theorem}

\subsubsection{Proof of Proposition \protect\ref{Prop:estimate for a typical
term in uniqueness}\label{sec:proof of prop uniqueness}}

Let us assume from now on that the two solutions $\gamma _{1}^{(k)}(t)$ and $%
\gamma _{2}^{(k)}(t)$ are weakly admissible to avoid redundancy in the
proof. By Theorem \ref{thm:qde}, there exists $d\mu _{1,t}$ and $d\mu _{2,t}$
supported on $\mathbb{B(}L^{2}\mathbb{)}$ such that%
\begin{equation}
\gamma ^{(k)}(t)=\gamma _{1}^{(k)}(t)-\gamma _{2}^{(k)}(t)=\int_{\mathbb{B(}%
L^{2}\mathbb{)}}\left( \left\vert \phi \right\rangle \left\langle \phi
\right\vert \right) ^{\otimes k}d\mu _{t}(\phi )
\label{eqn:qde rep for uniqueness}
\end{equation}%
if we define $d\mu _{t}=d\left( \mu _{1,t}-\mu _{2,t}\right) $. As stated
before, condition (\ref{HUFLinTime}) is equivalent to 
\begin{equation}
\mu _{i,t}\left( \{\phi \in L^{2}:\left\Vert SP_{\geqslant M}\phi
\right\Vert _{L_{x}^{2}}>\varepsilon \}\right) =0\text{ for }i=1,2.
\label{measure condition:UTFL}
\end{equation}%
and condition (\ref{condition:kinetic energy}) reads%
\begin{equation}
\mu _{i,t}\left( \{\phi \in L^{2}:\left\Vert S\phi \right\Vert
_{L_{x}^{2}}>C_{0}\}\right) =0\text{ for }i=1,2.
\label{measure condition:kinetic energy}
\end{equation}%
Certainly, $d\mu _{t}$ also, by definition, carries properties (\ref{measure
condition:UTFL}) and (\ref{measure condition:kinetic energy}).

We are now ready to prove Proposition \ref{Prop:estimate for a typical term
in uniqueness}. The proof to Proposition \ref{Prop:estimate for a typical
term in uniqueness} is more like an algorithm, let us first illustrate the
idea by giving an example.

\begin{example}
\label{example:uniquness example}Let us estimate 
\begin{equation*}
\limfunc{Tr}\left\vert
\int_{D}S^{(-1,1)}U_{t_{1},t_{3}}^{(1)}B_{1;2,3}^{+}U_{t_{3},t_{5}}^{(3)}B_{2;4,5}^{-}U_{t_{5},t_{7}}^{(5)}B_{3;6,7}^{+}\gamma ^{(7)}(t_{7})d%
\underline{t}_{7}\right\vert
\end{equation*}%
where we shortened $U^{(l)}(t_{l}-t_{l+2})$ as $U_{t_{l},t_{l+2}}^{(l)}$. We
will refer the $B_{\mu (2l);2l,2l+1}$ in the expression as the $l$-th
coupling.

Part I: (The Marking Part)

For the inner most coupling which is the 3rd coupling here, plugging in (\ref%
{eqn:qde rep for uniqueness}), we have%
\begin{eqnarray*}
&&B_{3;6,7}^{+}\gamma ^{(7)}(t_{7}) \\
&=&\int \phi (x_{1})\phi (x_{2})\left( \left\vert \phi \right\vert ^{4}\phi
\right) (x_{3})\phi (x_{4})\phi (x_{5})\bar{\phi}(x_{1}^{\prime })\bar{\phi}%
(x_{2}^{\prime })\bar{\phi}(x_{3}^{\prime })\bar{\phi}(x_{4}^{\prime })\bar{%
\phi}(x_{5}^{\prime })d\mu _{t_{7}}(\phi ) \\
&\equiv &\int \phi (x_{1})\phi (x_{2})Q_{R}^{(7)}(x_{3})\phi (x_{4})\phi
(x_{5})\bar{\phi}(x_{1}^{\prime })\bar{\phi}(x_{2}^{\prime })\bar{\phi}%
(x_{3}^{\prime })\bar{\phi}(x_{4}^{\prime })\bar{\phi}(x_{5}^{\prime })d\mu
_{t_{7}}(\phi ).
\end{eqnarray*}%
We denote by $Q_{R}^{(2k+1)}$, which is $Q_{R}^{(7)}$ in this example, the 
\emph{roughest quintic} term created in the most inner k-th coupling, which
is the 3rd coupling. This factor should be identified immediately in any
cases.

Then at the 2nd coupling, we first apply $U_{t_{5},t_{7}}^{(5)}$, 
\begin{eqnarray*}
&&B_{2;4,5}^{-}U_{t_{5},t_{7}}^{(5)}B_{3;6,7}^{+}\gamma ^{(7)}(t_{7}) \\
&=&B_{2;4,5}^{-}\int \left( U_{t_{5},t_{7}}\phi \right) (x_{1})\left(
U_{t_{5},t_{7}}\phi \right) (x_{2})\left( U_{t_{5},t_{7}}Q_{R}^{(7)}\right)
(x_{3})\left( U_{t_{5},t_{7}}\phi \right) (x_{4})\left( U_{t_{5},t_{7}}\phi
\right) (x_{5}) \\
&&\times \overline{\left( U_{t_{5},t_{7}}\phi \right) }(x_{1}^{\prime })%
\overline{\left( U_{t_{5},t_{7}}\phi \right) }(x_{2}^{\prime })\overline{%
\left( U_{t_{5},t_{7}}\phi \right) }(x_{3}^{\prime })\overline{\left(
U_{t_{5},t_{7}}\phi \right) }(x_{4}^{\prime })\overline{\left(
U_{t_{5},t_{7}}\phi \right) }(x_{5}^{\prime })d\mu _{t_{7}}(\phi )
\end{eqnarray*}%
then $B_{2;4,5}^{-}$ 
\begin{eqnarray*}
&&B_{2;4,5}^{-}U_{t_{5},t_{7}}^{(5)}B_{3;6,7}^{+}\gamma ^{(7)}(t_{7}) \\
&=&\int \left( U_{t_{5},t_{7}}\phi \right) (x_{1})\left( U_{t_{5},t_{7}}\phi
\right) (x_{2})\left( U_{t_{5},t_{7}}Q_{R}^{(7)}\right) (x_{3}) \\
&&\times \overline{\left( U_{t_{5},t_{7}}\phi \right) }(x_{1}^{\prime
})\left( \left\vert \left( U_{t_{5},t_{7}}\phi \right) \right\vert ^{4}%
\overline{\left( U_{t_{5},t_{7}}\phi \right) }\right) (x_{2}^{\prime })%
\overline{\left( U_{t_{5},t_{7}}\phi \right) }(x_{3}^{\prime })d\mu
_{t_{7}}(\phi ) \\
&\equiv &\int \left( U_{t_{5},t_{7}}\phi \right) (x_{1})\left(
U_{t_{5},t_{7}}\phi \right) (x_{2})\left( U_{t_{5},t_{7}}Q_{R}^{(7)}\right)
(x_{3}) \\
&&\times \overline{\left( U_{t_{5},t_{7}}\phi \right) }(x_{1}^{\prime
})\left( Q_{\phi }^{(5)}\right) (x_{2}^{\prime })\overline{\left(
U_{t_{5},t_{7}}\phi \right) }(x_{3}^{\prime })d\mu _{t_{7}}(\phi )
\end{eqnarray*}%
For the l-th coupling, where $l<k$, we denote the quintic term created in
the l-th coupling by $Q_{\phi }^{(2l+1)}$ if at least one of the factors
inside this quintic term is a $U_{t_{2l+1},t_{j}}\phi $ and $Q^{(2l+1)}$ if
none of the factors inside this quintic term is a $U_{t_{2l+1},t_{j}}\phi $
where $j$ is an index larger or equal to $2l+3$. Moreover, if one of the
factors\footnote{%
There is at most one since we splitted $B$ into $B^{+}$ and $B^{-}$.} inside
this $l$-th quintic term contains $Q_{R}^{(2k+1)}$, we put an extra $R$
subscript at the notation, that is $Q_{\phi ,R}^{(2l+1)}$ or $Q_{R}^{(2l+1)}$%
. We remark that, every one of the five factors inside $Q_{\phi }^{(2l+1)}$
or $Q^{(2l+1)}$ carries the propagator $U$ which smooths out each factor.
For the 2nd coupling in the example we are working on here, there is a
factor $U_{t_{5},t_{7}}\phi $ inside and no $Q_{R}^{(7)}$ inside, so we put
down $Q_{\phi }^{(5)}$.

Finally, at the 1st coupling,%
\begin{eqnarray*}
&&B_{1;2,3}^{+}U_{t_{3},t_{5}}^{(3)}B_{2;4,5}^{-}U_{t_{5},t_{7}}^{(5)}B_{3;6,7}^{+}\gamma ^{(7)}(t_{7})
\\
&=&B_{1;2,3}^{+}\int \left( U_{t_{3},t_{7}}\phi \right) (x_{1})\left(
U_{t_{3},t_{7}}\phi \right) (x_{2})\left( U_{t_{3},t_{7}}Q_{R}^{(7)}\right)
(x_{3}) \\
&&\times \overline{\left( U_{t_{3},t_{7}}\phi \right) }(x_{1}^{\prime
})\left( U_{t_{3},t_{5}}^{(3)}\left( Q_{\phi }^{(5)}\right) \right)
(x_{2}^{\prime })\overline{\left( U_{t_{3},t_{7}}\phi \right) }%
(x_{3}^{\prime })d\mu _{t_{7}}(\phi ) \\
&=&\int \left[ \left( U_{t_{3},t_{7}}\phi \right) (x_{1})\left(
U_{t_{3},t_{7}}\phi \right) (x_{1})\left( U_{t_{3},t_{7}}Q_{R}^{(7)}\right)
(x_{1})\left( U_{t_{3},t_{5}}^{(3)}\left( Q_{\phi }^{(5)}\right) \right)
(x_{1})\overline{\left( U_{t_{3},t_{7}}\phi \right) }(x_{1})\right] \\
&&\times \overline{\left( U_{t_{3},t_{7}}\phi \right) }(x_{1}^{\prime })d\mu
_{t_{7}}(\phi ) \\
&=&\int Q_{\phi ,R}^{(3)}(x_{1})\overline{\left( U_{t_{3},t_{7}}\phi \right) 
}(x_{1}^{\prime })d\mu _{t_{7}}(\phi ).
\end{eqnarray*}%
We used the notation $Q_{\phi ,R}^{(3)}$ because this 1st quintic term has a 
$U_{t_{2l+1},t_{j}}\phi $ factor and one of its factors contain $Q_{R}^{(7)}$%
. We have marked all the quintic terms in each coupling, we can start the
estimate part.

Part II: (The Estimate Part) First, we use Minkowski%
\begin{eqnarray*}
&&\limfunc{Tr}\left\vert
\int_{D}S^{(-1,1)}U_{t_{1},t_{3}}^{(1)}B_{1;2,3}^{+}U_{t_{3},t_{5}}^{(3)}B_{2;4,5}^{-}U_{t_{5},t_{7}}^{(5)}B_{3;6,7}^{+}\gamma ^{(7)}(t_{7})d%
\underline{t}_{7}\right\vert \\
&\leqslant &\int_{D}\limfunc{Tr}\left\vert
S^{(-1,1)}U_{t_{1},t_{3}}^{(1)}B_{1;2,3}^{+}U_{t_{3},t_{5}}^{(3)}B_{2;4,5}^{-}U_{t_{5},t_{7}}^{(5)}B_{3;6,7}^{+}\gamma ^{(7)}(t_{7})\right\vert d%
\underline{t}_{7}
\end{eqnarray*}%
Recall from Lemma \ref{Lem:KLBoardGame} that $D\subset \lbrack 0,t_{1}]^{k}$
and $k=3$ here, 
\begin{equation*}
\leqslant \int_{\lbrack 0,t_{1}]^{3}}\limfunc{Tr}\left\vert \int
S^{(-1,1)}U_{t_{1},t_{3}}^{(1)}Q_{\phi ,R}^{(3)}(x_{1})\overline{\left(
U_{t_{3},t_{7}}\phi \right) }(x_{1}^{\prime })d\mu _{t_{7}}(\phi
)\right\vert d\underline{t}_{7}
\end{equation*}%
Apply Minkowski again and Cauchy-Schwarz,%
\begin{eqnarray*}
&\leqslant &\int_{[0,t_{1}]^{3}}\int \limfunc{Tr}\left\vert
S^{(-1,1)}U_{t_{1},t_{3}}^{(1)}Q_{\phi ,R}^{(3)}(x_{1})\overline{\left(
U_{t_{3},t_{7}}\phi \right) }(x_{1}^{\prime })\right\vert \left\vert d\mu
_{t_{7}}\right\vert (\phi )d\underline{t}_{7} \\
&\leqslant &\int_{[0,t_{1}]^{3}}\int \left\Vert S^{-1}U_{t_{1},t_{3}}Q_{\phi
,R}^{(3)}\right\Vert _{L_{x}^{2}}\left\Vert S^{-1}\overline{\left(
U_{t_{1},t_{3}}U_{t_{3},t_{7}}\phi \right) }\right\Vert
_{L_{x}^{2}}\left\vert d\mu _{t_{7}}\right\vert (\phi )d\underline{t}_{7}
\end{eqnarray*}%
where $\left\vert d\mu _{t_{7}}\right\vert \leqslant d\mu _{1,t_{7}}+$ $d\mu
_{2,t_{7}}$. Because $e^{it\Delta }$ is unitary, 
\begin{eqnarray*}
&&\limfunc{Tr}\left\vert
\int_{D}S^{(-1,1)}U_{t_{1},t_{3}}^{(1)}B_{1;2,3}^{+}U_{t_{3},t_{5}}^{(3)}B_{2;4,5}^{-}U_{t_{5},t_{7}}^{(5)}B_{3;6,7}^{+}\gamma ^{(7)}(t_{7})d%
\underline{t}_{7}\right\vert \\
&\leqslant &\int_{0}^{t_{1}}dt_{7}\int \left\vert d\mu _{t_{7}}\right\vert
(\phi )\int_{[0,t_{1}]^{2}}dt_{3}dt_{5}\left\Vert S^{-1}Q_{\phi
,R}^{(3)}\right\Vert _{L_{x}^{2}}\left\Vert S^{-1}\overline{%
U_{t_{3},t_{7}}\phi }\right\Vert _{L_{x}^{2}}
\end{eqnarray*}

To carry out the $dt_{3}$ integral, we look at the quintic term generated in
the 1st coupling, which is $Q_{\phi ,R}^{(3)}$. Because it carries $\phi $
and $R$, we use (\ref{estimate:MultiLinearWithFreqLoc1}) and reach%
\begin{eqnarray*}
&\leqslant &\int_{0}^{t_{1}}dt_{7}\int \left\vert d\mu _{t_{7}}\right\vert
(\phi )\int_{[0,t_{1}]}dt_{5}\left\Vert S^{-1}Q_{R}^{(7)}\right\Vert
_{L_{x}^{2}}\left\Vert SQ_{\phi }^{(5)}\right\Vert _{L_{x}^{2}}\left\Vert
S\phi \right\Vert _{L_{x}^{2}}^{2}\left\Vert S^{-1}\phi \right\Vert
_{L_{x}^{2}} \\
&&\times \left( C\left( t_{1}^{\frac{1}{4}}M^{\frac{1}{2}}\right) ^{\frac{10%
}{11}}\left\Vert \phi \right\Vert _{H_{x}^{1}}+C\left\Vert P_{>M}S\phi
\right\Vert _{L_{x}^{2}}\right)
\end{eqnarray*}%
where we have put the factor carrying $Q_{R}^{(7)}$ in $H^{-1}$ and the
factor $\phi $ in the frequency splitting.

To carry out the $dt_{5}$ integral, we look at the quintic term generated in
the 2nd coupling, which is $Q_{\phi }^{(5)}$. Because it carries $\phi $ but
not $R$, we use (\ref{estimate:MultiLinearWithFreqLoc2}) and get to%
\begin{eqnarray*}
&\leqslant &\int_{0}^{t_{1}}dt_{7}\int \left\vert d\mu _{t_{7}}\right\vert
(\phi )\left\Vert S^{-1}Q_{R}^{(7)}\right\Vert _{L_{x}^{2}}\left\Vert S\phi
\right\Vert _{L_{x}^{2}}^{6}\left\Vert S^{-1}\phi \right\Vert _{L_{x}^{2}} \\
&&\times \left( C\left( t_{1}^{\frac{1}{4}}M^{\frac{1}{2}}\right) ^{\frac{10%
}{11}}\left\Vert \phi \right\Vert _{H_{x}^{1}}+C\left\Vert P_{>M}S\phi
\right\Vert _{L_{x}^{2}}\right) ^{2}
\end{eqnarray*}%
where we have put a $\phi $ factor in the frequency splitting.

Now, use the Sobolev inequality%
\begin{equation}
\left\Vert S^{-1}\left( \left\vert \phi \right\vert ^{4}\phi \right)
\right\Vert _{L_{x}^{2}}\leqslant C\left\Vert \phi \right\Vert
_{H_{x}^{1}}^{5}\text{,}  \label{eqn:sobolev for naked term}
\end{equation}%
we have%
\begin{eqnarray*}
&&\limfunc{Tr}\left\vert
\int_{D}S^{(-1,1)}U_{t_{1},t_{3}}^{(1)}B_{1;2,3}^{+}U_{t_{3},t_{5}}^{(3)}B_{2;4,5}^{-}U_{t_{5},t_{7}}^{(5)}B_{3;6,7}^{+}\gamma ^{(7)}(t_{7})d%
\underline{t}_{7}\right\vert \\
&\leqslant &\int_{0}^{t_{1}}dt_{7}\int \left\vert d\mu _{t_{7}}\right\vert
(\phi )C\left\Vert \phi \right\Vert _{H^{1}}^{11}\left\Vert S^{-1}\phi
\right\Vert _{L_{x}^{2}} \\
&&\times \left( C\left( t_{1}^{\frac{1}{4}}M^{\frac{1}{2}}\right) ^{\frac{10%
}{11}}\left\Vert \phi \right\Vert _{H_{x}^{1}}+C\left\Vert P_{>M}S\phi
\right\Vert _{L_{x}^{2}}\right) ^{2}
\end{eqnarray*}%
Plugging in properties (\ref{measure condition:UTFL}) and (\ref{measure
condition:kinetic energy}), we have%
\begin{eqnarray*}
&\leqslant &CC_{0}^{12}\left( C\left( t_{1}^{\frac{1}{4}}M^{\frac{1}{2}%
}\right) ^{\frac{10}{11}}C_{0}+C\varepsilon \right)
^{2}\int_{0}^{t_{1}}dt_{7}\int \left\vert d\mu _{t_{7}}\right\vert (\phi ) \\
&\leqslant &CC_{0}^{12}\left( C\left( t_{1}^{\frac{1}{4}}M^{\frac{1}{2}%
}\right) ^{\frac{10}{11}}C_{0}+C\varepsilon \right) ^{2}\times t_{1}\times 2
\end{eqnarray*}%
which is Proposition \ref{Prop:estimate for a typical term in uniqueness} in
the case of $k=3$. Hence we have ended Example \ref{example:uniquness
example}.
\end{example}

From Example \ref{example:uniquness example}, one can immediately tell that
the number of $Q_{\phi }$ couplings in the total of $k$ couplings is exactly
the decay power one puts in the end. Let us first make sure such a decay
power is at least proportional to $k$. We make the following definition to
help the presentation.

\begin{definition}
\label{def:type of coupling}For $l<k$, we say the l-th coupling is an
unclogged coupling, if the corresponding quintic term $Q^{(2l+1)}$ contains
at least one $U_{t_{2l+1},t_{j}}\phi $ factor where $j$ is an index larger
or equal to $2l+3$. If the $l$-th coupling is not unclogged, we will call it
a congested coupling.
\end{definition}

\begin{lemma}
\label{Lem:unclogged couplings}For large $k$, there are at least $\frac{4}{5}%
k$ unclogged couplings in the total of $k$ couplings when one plugs (\ref%
{eqn:qde rep for uniqueness}) into (\ref{eqn:J+-forUniqueness}).
\end{lemma}

\begin{proof}
Assume there are $j$ congested couplings, then there are $(k-1-j)$ unclogged
couplings. Before the $(k-1)$-th coupling, there are $4k-3$ copies of $U\phi 
$ available. After the 1st coupling, all of these $4k-3$ copies of $U\phi $,
except one, must be inside some quintic term. Since the $j$ congested
couplings do not consume any $U\phi $, to consume all $4k-4$ copies of $%
U\phi $, we have to have%
\begin{equation}
4k-4\leqslant 5(k-1-j)  \label{eqn:total phi vs unclogged}
\end{equation}%
because a unclogged coupling can, at most, consume $5$ copies of $U\phi $.
Inequality $(\ref{eqn:total phi vs unclogged})$ certainly cannot hold for
large $k$ if $j\geqslant \frac{k}{5}$. Hence, there are at least $\frac{4k}{5%
}$ unclogged couplings.
\end{proof}

We are now ready to present the algorithm to estimate a general $J_{\pm
}^{(2k+1)}(\underline{t}_{2k+1};\mu )(\gamma ^{(2k+1)})(t_{1})$.

\begin{itemize}
\item[Step 1] Plug (\ref{eqn:qde rep for uniqueness}) into $J_{\pm
}^{(2k+1)}(\underline{t}_{2k+1};\mu )(\gamma ^{(2k+1)})(t_{1})$ and clearly
mark down all the quintic terms generated in each coupling as described in
Example \ref{example:uniquness example}. One then reaches%
\begin{eqnarray*}
&&\limfunc{Tr}\left\vert \int_{D}\int
S^{(-1,1)}U_{t_{1},t_{3}}^{(1)}f^{(1)}(x_{1})\overline{g^{(1)}}%
(x_{1}^{\prime })d\mu _{t_{2k+1}}(\phi )d\underline{t}_{2k+1}\right\vert \\
&\leqslant &\int_{0}^{t_{1}}dt_{2k+1}\int \left\vert d\mu
_{t_{2k+1}}\right\vert (\phi
)\int_{[0,t_{1}]^{k-1}}dt_{3}dt_{5}...dt_{2k-1}\left\Vert
S^{-1}f^{(1)}\right\Vert _{L_{x}^{2}}\left\Vert S^{-1}\overline{g^{(1)}}%
\right\Vert _{L_{x}^{2}}
\end{eqnarray*}%
where one and only one of $f^{(1)}$ and $g^{(1)}$ contains the quintic term
generated at the 1st coupling. Go to Step 2.

\item[Step 2] Set counter $l=1$. Go to Step 3.

\item[Step 3] If the quintic term generated in the $l$-th coupling is a $%
Q_{\phi ,R}^{(2l+1)}$ and if we are estimating $\int dt_{2l+1}\left\Vert
S^{-1}Q_{\phi ,R}^{(2l+1)}\right\Vert _{L_{x}^{2}}$, then apply (\ref%
{estimate:MultiLinearWithFreqLoc1}), with the factor containing $%
Q_{R}^{(2k+1)}$ put in $H^{-1}$ and one $\phi $ factor spitted in frequency,
and go to Step 7. If not, go to Step 4.

\item[Step 4] If the quintic term generated in the $l$-th coupling is a $%
Q_{\phi }^{(2l+1)}$ and if we are estimating $\int dt_{2l+1}\left\Vert
S^{-1}Q_{\phi }^{(2l+1)}\right\Vert _{L_{x}^{2}}$, then use $\int
dt_{2l+1}\left\Vert S^{-1}Q_{\phi }^{(2l+1)}\right\Vert
_{L_{x}^{2}}\leqslant \int dt_{2l+1}\left\Vert SQ_{\phi
}^{(2l+1)}\right\Vert _{L_{x}^{2}}$ and apply (\ref%
{estimate:MultiLinearWithFreqLoc2}), with one $\phi $ factor spitted in
frequency, and go to Step 7. If the quintic term generated in the $l$-th
coupling is a $Q_{\phi }^{(2l+1)}$ and if we are estimating $\int
dt_{2l+1}\left\Vert SQ_{\phi }^{(2l+1)}\right\Vert _{L_{x}^{2}}$, then apply
(\ref{estimate:MultiLinearWithFreqLoc2}), with one $\phi $ factor spitted in
frequency, and go to Step 7.If not, go to Step 5.

\item[Step 5] If the quintic term generated in the $l$-th coupling is a $%
Q_{R}^{(2l+1)}$ and if we are estimating $\int dt_{2l+1}\left\Vert
S^{-1}Q_{R}^{(2l+1)}\right\Vert _{L_{x}^{2}}$, then apply (\ref%
{estimate:YounghunOld1}), with the factor containing $Q_{R}^{(2k+1)}$ put in 
$H^{-1}$, and go to Step 7. If not, go to Step 6.

\item[Step 6] If the quintic term generated in the $l$-th coupling is a $%
Q^{(2l+1)}$, and if we are estimating $\int dt_{2l+1}\left\Vert
S^{-1}Q^{(2l+1)}\right\Vert _{L_{x}^{2}}$, then use $\int
dt_{2l+1}\left\Vert S^{-1}Q^{(2l+1)}\right\Vert _{L_{x}^{2}}\leqslant \int
dt_{2l+1}\left\Vert SQ^{(2l+1)}\right\Vert _{L_{x}^{2}}$ and apply (\ref%
{estimate:YounghunOld2}) and go to Step 7. If the quintic term generated in
the $l$-th coupling is a $Q^{(2l+1)}$, and if we are estimating $\int
dt_{2l+1}\left\Vert SQ^{(2l+1)}\right\Vert _{L_{x}^{2}}$, then apply (\ref%
{estimate:YounghunOld2}) and go to Step 7.

\item[Step 7] Set counter $l=l+1$. If $l<k,$ go to Step 3. If $l=k$, go to
Step 8.

\item[Step 8] Now we are at the $k$th coupling. By Lemma \ref{Lem:unclogged
couplings}, we have used estimates (\ref{estimate:MultiLinearWithFreqLoc1})
or (\ref{estimate:MultiLinearWithFreqLoc2}) at least $\frac{4k}{5}$ times
and estimates (\ref{estimate:YounghunOld1}) or (\ref{estimate:YounghunOld2})
at most $\frac{k}{5}$ times. So in the least decaying case, we have ended up
with%
\begin{eqnarray*}
&&\limfunc{Tr}\left\vert \int_{D}\int
S^{(-1,1)}U_{t_{1},t_{3}}^{(1)}f^{(1)}(x_{1})\overline{g^{(1)}}%
(x_{1}^{\prime })d\mu _{t_{2k+1}}(\phi )d\underline{t}_{2k+1}\right\vert \\
&\leqslant &\int_{0}^{t_{1}}dt_{2k+1}\int \left\vert d\mu
_{t_{2k+1}}\right\vert (\phi )\left\Vert S^{-1}Q_{R}^{(2k+1)}\right\Vert
_{L_{x}^{2}}\left( C\left\Vert S\phi \right\Vert _{L_{x}^{2}}\right) ^{\frac{%
16}{5}k-3} \\
&&\times \left( C\left( t_{1}^{\frac{1}{4}}M^{\frac{1}{2}}\right) ^{\frac{10%
}{11}}\left\Vert \phi \right\Vert _{H_{x}^{1}}+C\left\Vert P_{>M}S\phi
\right\Vert _{L_{x}^{2}}\right) ^{\frac{4}{5}k}\text{.}
\end{eqnarray*}%
Use (\ref{eqn:sobolev for naked term}) for $\left\Vert
S^{-1}Q_{R}^{(2k+1)}\right\Vert _{L_{x}^{2}}$, 
\begin{equation*}
\leqslant \int_{0}^{t_{1}}dt_{2k+1}\int \left\vert d\mu
_{t_{2k+1}}\right\vert \left( C\left\Vert S\phi \right\Vert
_{L_{x}^{2}}\right) ^{\frac{16}{5}k+2}\left( C\left( t_{1}^{\frac{1}{4}}M^{%
\frac{1}{2}}\right) ^{\frac{10}{11}}\left\Vert \phi \right\Vert
_{H_{x}^{1}}+C\left\Vert P_{>M}S\phi \right\Vert _{L_{x}^{2}}\right) ^{\frac{%
4}{5}k}
\end{equation*}%
Put in properties (\ref{measure condition:UTFL}) and (\ref{measure
condition:kinetic energy}),%
\begin{eqnarray*}
&\leqslant &\left( CC_{0}\right) ^{\frac{16}{5}k+2}\left( C\left( t_{1}^{%
\frac{1}{4}}M^{\frac{1}{2}}\right) ^{\frac{10}{11}}C_{0}+C\varepsilon
\right) ^{\frac{4}{5}k}\int_{0}^{t_{1}}dt_{2k+1}\int \left\vert d\mu
_{t_{2k+1}}\right\vert \\
&\leqslant &\left( CC_{0}\right) ^{\frac{16}{5}k+2}\left( C\left( t_{1}^{%
\frac{1}{4}}M^{\frac{1}{2}}\right) ^{\frac{10}{11}}C_{0}+C\varepsilon
\right) ^{\frac{4}{5}k}\times t_{1}\times 2
\end{eqnarray*}%
which is Proposition \ref{Prop:estimate for a typical term in uniqueness}.
\end{itemize}

That is, assuming the multilinear estimates (\ref%
{estimate:MultiLinearWithFreqLoc1})-(\ref{estimate:YounghunOld2}) in Lemma %
\ref{Lem:MultilinearWithFreqLocal}, we have proved Theorem \ref%
{THM:UniquessAssumingUFL} and hence the large solution uniqueness Theorem %
\ref{Thm:TotalUniqueness}. We prove these multilinear estimates in \S \ref%
{sec:UniquenessMultilinearEstimate}.

\subsection{Multilinear Estimates on $\mathbb{T}^{3}$\label%
{sec:UniquenessMultilinearEstimate}}

\begin{lemma}
\label{Lem:MultilinearWithFreqLocal}\footnote{%
The factor $\left( T^{\frac{1}{4}}M_{0}^{\frac{1}{2}}\right) ^{\frac{10}{11}%
} $ can be improved to $\left( T^{\frac{1}{4}}M_{0}^{\frac{1}{2}}\right)
^{1-\varepsilon }.$We chose a concrete power here to make the proof easier
to understand.}Given any frequency $M_{0}\geqslant 0$, we have the refined
multilinear estimates%
\begin{eqnarray}
&&\left\Vert e^{\pm it\Delta }f_{1}e^{\pm it\Delta }f_{2}e^{\pm it\Delta
}f_{3}e^{\pm it\Delta }f_{4}e^{\pm it\Delta }f_{5}\right\Vert
_{L_{T}^{1}H_{x}^{-1}}  \label{estimate:MultiLinearWithFreqLoc1} \\
&\leqslant &C\left\Vert f_{1}\right\Vert _{H_{x}^{-1}}\left( T^{\frac{5}{22}%
}M_{0}^{\frac{5}{11}}\left\Vert f_{2}\right\Vert _{H^{1}}+\left\Vert
P_{>M_{0}}S_{x}f_{2}\right\Vert _{L_{x}^{2}}\right)
\dprod\limits_{j=3}^{5}\left\Vert f_{j}\right\Vert _{H^{1}},  \notag \\
&&\left\Vert e^{\pm it\Delta }f_{1}e^{\pm it\Delta }f_{2}e^{\pm it\Delta
}f_{3}e^{\pm it\Delta }f_{4}e^{\pm it\Delta }f_{5}\right\Vert
_{L_{T}^{1}H_{x}^{1}}  \label{estimate:MultiLinearWithFreqLoc2} \\
&\leqslant &C\left( T^{\frac{5}{22}}M_{0}^{\frac{5}{11}}\left\Vert
f_{1}\right\Vert _{H_{x}^{1}}+\left\Vert P_{>M_{0}}S_{x}f_{1}\right\Vert
_{L_{x}^{2}}\right) (\dprod\limits_{j=2}^{5}\left\Vert f_{j}\right\Vert
_{H_{x}^{1}}).  \notag
\end{eqnarray}%
In particular, if we set $M_{0}=0$, we have%
\begin{eqnarray}
\left\Vert e^{\pm it\Delta }f_{1}e^{\pm it\Delta }f_{2}e^{\pm it\Delta
}f_{3}e^{\pm it\Delta }f_{4}e^{\pm it\Delta }f_{5}\right\Vert
_{L_{T}^{1}H_{x}^{-1}} &\leqslant &C\left\Vert f_{1}\right\Vert
_{H_{x}^{-1}}\dprod\limits_{j=2}^{5}\left\Vert f_{j}\right\Vert _{H_{x}^{1}}
\label{estimate:YounghunOld1} \\
\left\Vert e^{\pm it\Delta }f_{1}e^{\pm it\Delta }f_{2}e^{\pm it\Delta
}f_{3}e^{\pm it\Delta }f_{4}e^{\pm it\Delta }f_{5}\right\Vert
_{L_{T}^{1}H_{x}^{1}} &\leqslant &C\dprod\limits_{j=1}^{5}\left\Vert
f_{j}\right\Vert _{H_{x}^{1}}  \label{estimate:YounghunOld2}
\end{eqnarray}
\end{lemma}

We prove Lemma \ref{Lem:MultilinearWithFreqLocal} in \S \ref{Sec:Proof of
Multilinear estimate}. The key estimate in Lemma \ref%
{Lem:MultilinearWithFreqLocal} is (\ref{estimate:YounghunOld1}) in the sense
that (\ref{estimate:MultiLinearWithFreqLoc1}), (\ref%
{estimate:MultiLinearWithFreqLoc2}), and (\ref{estimate:YounghunOld2})
follows from the proof of (\ref{estimate:YounghunOld1}) with lighter
techniques. Estimates similar to (\ref{estimate:YounghunOld1}) - (\ref%
{estimate:YounghunOld2}) used in uniqueness arguments in other settings can
be found in \cite{TCNPdeFinitte,HoTaXi14,HTX,Sohinger3,C-PUniqueness}. A
less concrete version of estimates of this type is the collapsing type
estimates in \cite%
{KlainermanAndMachedon,Kirpatrick,TChenAndNP,ChenDie,ChenAnisotropic,C-HFocusing,Sohinger,HerrSohinger}%
. Another $\mathbb{T}^{3}$ quintilinear estimate in different spaces
regarding different regularity was obtained in \cite[\S 3]{IP}. For
comparison purpose, the following is (\ref{estimate:YounghunOld1}) in the $%
U-V$ space format.

\begin{remark}
With the usual settings \cite{IP} of the $U$ and $V$ spaces, define 
\begin{equation*}
\Vert u\Vert _{X^{s}(I)}=\Vert u\Vert _{U_{\Delta
}^{2}(I;H_{x}^{s})}\,,\qquad \Vert v\Vert _{Y^{s}(I)}=\Vert v\Vert
_{DU_{\Delta }^{2}(I;H_{x}^{s})}
\end{equation*}%
for the time interval $I=[0,T]$. It is well-known that there are the
embeddings 
\begin{equation*}
\Vert u\Vert _{L_{I}^{\infty }H_{x}^{s}}\lesssim \Vert u\Vert
_{X^{s}(I)}\,,\qquad \Vert v\Vert _{Y^{s}(I)}\lesssim \Vert v\Vert
_{L_{I}^{1}H_{x}^{s}}
\end{equation*}%
Thus, using the definition of $U$ as an atomic space, (\ref%
{estimate:YounghunOld1}) implies%
\begin{equation}
\Vert u_{1}\cdots u_{5}\Vert _{Y^{-1}(I)}\lesssim \Vert u_{1}\Vert
_{X^{-1}(I)}\Vert u_{2}\Vert _{X^{1}(I)}\cdots \Vert u_{5}\Vert _{X^{1}(I)}.
\label{E:UV-version}
\end{equation}%
Since we are not using (\ref{E:UV-version}) anywhere in this paper, we omit
the details of the proof.
\end{remark}

\subsubsection{Tools to Prove Lemma \protect\ref%
{Lem:MultilinearWithFreqLocal}}

We need the following results to prove Lemma \ref%
{Lem:MultilinearWithFreqLocal}.

\begin{lemma}[Strichartz estimate on $\mathbb{T}^{3}$ \protect\cite{BD,KV}]
\label{estimate:T^3 strichartz}For $p>\frac{10}{3}$, 
\begin{equation}
\Vert P_{\leqslant M}e^{it\Delta }f\Vert _{L_{t,x}^{p}}\lesssim M^{\frac{3}{2%
}-\frac{5}{p}}\Vert P_{\leqslant M}f\Vert _{L_{x}^{2}}  \label{E:Str1}
\end{equation}
\end{lemma}

\begin{corollary}[Strichartz estimates on $\mathbb{T}^{3}$ with noncentered
frequency localization]
\label{cor:str with noncentered cube}Let $M$ be a dyadic value and let $Q$
be a (possibly) noncentered $M$-cube in Fourier space 
\begin{equation*}
Q=\left\{ \xi _{0}+\eta :\left\vert \eta \right\vert \leq M\right\} \,.
\end{equation*}%
Let $P_{Q}$ be the corresponding Littlewood-Paley projection, then by the
Galilean invariance, we have 
\begin{equation}
\Vert P_{Q}e^{it\Delta }f\Vert _{L_{t,x}^{p}}\lesssim M^{\frac{3}{2}-\frac{5%
}{p}}\Vert P_{Q}f\Vert _{L_{x}^{2}}.  \label{E:Str2}
\end{equation}%
for $p>\frac{10}{3}$. The net effect of this observation is that we pay a
factor of only $M^{\frac{3}{2}-\frac{5}{p}}$, when applying (\ref{E:Str1}).
\end{corollary}

\begin{proof}
The proof is completely standard but we include it for completeness. We will
need the following Galilean invariance property: for $\xi _{0}\in 2\pi 
\mathbb{Z}^{d}$, we have 
\begin{equation}
e^{-ix\cdot \xi _{0}}(e^{it\Delta }f)(x,t)=e^{it|\xi
_{0}|^{2}/2}[e^{it\Delta }(e^{-i\bullet \cdot \xi _{0}}f)](x-2\xi _{0}t,t).
\label{E:Gal1}
\end{equation}%
In Fourier space, let 
\begin{equation*}
Q=\left\{ \xi _{0}+\eta :\left\vert \eta \right\vert \leqslant M\right\} ,
\end{equation*}%
be a noncentered $M$-cube, and let $\tilde{Q}$ be the corresponding centered 
$M$-cube 
\begin{equation*}
\tilde{Q}=\left\{ \xi :\left\vert \xi \right\vert \leqslant M\right\} .
\end{equation*}%
Let $P_{Q}$ and $P_{\tilde{Q}}$ be the corresponding Littlewood-Paley
projections. Then since $\widehat{e^{it\Delta }f}(\xi _{0}+\eta )=\widehat{%
e^{-i\bullet \cdot \xi _{0}}e^{it\Delta }f}(\eta )$, it follows that 
\begin{equation}
\Vert P_{Q}e^{it\Delta }f\Vert _{L_{t,x}^{p}}=\Vert P_{\tilde{Q}%
}(e^{-i\bullet \cdot \xi _{0}}e^{it\Delta }f)\Vert _{L_{t,x}^{p}}
\label{E:Gal2}
\end{equation}%
Since $[f(\bullet +v)]^{\widehat{\quad }}(\xi )=e^{-i\xi \cdot v}\hat{f}(\xi
)$, we have $P_{\tilde{Q}}[f(\bullet +v)](x)=(P_{\tilde{Q}}f)(x+v)$. In
other words, $P_{\tilde{Q}}$ commutes with translations. Thus by %
\eqref{E:Gal1}, 
\begin{equation}
P_{\tilde{Q}}[e^{-i\bullet \cdot \xi _{0}}(e^{it\Delta }f)](x,t)=e^{it|\xi
_{0}|^{2}/2}[P_{\tilde{Q}}e^{it\Delta }(e^{-i\bullet \cdot \xi
_{0}}f)](x-2\xi _{0}t,t)  \label{E:Gal1b}
\end{equation}%
Plugging \eqref{E:Gal1b} into \eqref{E:Gal2}, 
\begin{equation*}
\Vert P_{Q}e^{it\Delta }f\Vert _{L_{t,x}^{p}}=\Vert P_{\tilde{Q}}e^{it\Delta
}(e^{-i\bullet \cdot \xi _{0}}f)\Vert _{L_{t,x}^{p}}
\end{equation*}%
Thus we can apply \eqref{E:Str1} with $P_{\leq M}=P_{\tilde{Q}}$ to obtain 
\begin{equation*}
\Vert P_{Q}e^{it\Delta }f\Vert _{L_{t,x}^{p}}\lesssim M^{\frac{3}{2}-\frac{%
3+2}{p}}\Vert P_{\tilde{Q}}(e^{-i\bullet \cdot \xi _{0}}f)\Vert
_{L_{x}^{2}}=M^{\frac{3}{2}-\frac{3+2}{p}}\Vert P_{Q}f\Vert _{L_{x}^{2}}
\end{equation*}
\end{proof}

\begin{lemma}[Bernstein with noncentered frequency projection]
\label{lem:benstein with noncentered cube}Let $Q$ and $M$ be as in Corollary %
\ref{cor:str with noncentered cube}, then for $1\leq p\leq q\leq \infty $ 
\begin{equation*}
\Vert P_{Q}f\Vert _{L_{x}^{q}}\lesssim M^{3(\frac{1}{p}-\frac{1}{q})}\Vert
P_{Q}f\Vert _{L_{x}^{p}}
\end{equation*}
\end{lemma}

\begin{proof}
We have 
\begin{equation*}
\Vert P_{Q}f\Vert _{L^{q}}=\Vert P_{\leq M}(e^{-i\bullet \cdot \xi
_{0}}f)\Vert _{L^{q}}\lesssim M^{3(\frac{1}{p}-\frac{1}{q})}\Vert P_{\leq
M}(e^{-i\bullet \cdot \xi _{0}}f)\Vert _{L^{p}}=M^{3(\frac{1}{p}-\frac{1}{q}%
)}\Vert P_{Q}f\Vert _{L^{p}}
\end{equation*}%
where we have applied the usual formulation of the Bernstein inequality in
the middle.
\end{proof}

\begin{lemma}[Bilinear Strichartz on $\mathbb{T}^{3}$ \protect\cite{KV,Z}]
\label{estimate:T^3 bilinear strichartz}There exists $\delta >0$ such that 
\begin{equation}
\Vert P_{M_{1}}e^{it\Delta }f_{1}\cdot P_{M_{2}}e^{it\Delta }f_{2}\Vert
_{L_{t,x}^{2}}\lesssim M_{2}^{1/2}(\frac{M_{2}}{M_{1}}+\frac{1}{M_{2}}%
)^{\delta }\Vert P_{M_{1}}f_{1}\Vert _{L_{x}^{2}}\Vert P_{M_{2}}f_{2}\Vert
_{L_{x}^{2}}  \label{E:Str3}
\end{equation}%
for any $M_{2}\leqslant M_{1}$.
\end{lemma}

With the above results, we now start proving Lemma \ref%
{Lem:MultilinearWithFreqLocal}.

\subsubsection{Proof of Lemma \protect\ref{Lem:MultilinearWithFreqLocal} 
\label{Sec:Proof of Multilinear estimate}}

We prove (\ref{estimate:MultiLinearWithFreqLoc1}) in detail. (\ref%
{estimate:MultiLinearWithFreqLoc2}) follows from the same analysis by making
some small changes in the argument.\footnote{%
See the end of this section.} By duality, we prove (\ref%
{estimate:MultiLinearWithFreqLoc1}) with the following two estimates: the
low frequency estimate%
\begin{align}
\hspace{0.3in}& \hspace{-0.3in}\iint_{x,t}e^{it\Delta }f_{1}\left(
P_{\leqslant M_{0}}e^{it\Delta }f_{2}\right) e^{it\Delta }f_{3}e^{it\Delta
}f_{4}e^{it\Delta }f_{5}gdxdt  \label{estimate:MultiLinearWithFreqLoc1-low}
\\
& \lesssim T^{\frac{5}{22}}M_{0}^{\frac{5}{11}}\Vert f_{1}\Vert
_{H_{x}^{-1}}\Vert f_{2}\Vert _{H_{x}^{1}}\Vert f_{3}\Vert _{H_{x}^{1}}\Vert
f_{4}\Vert _{H_{x}^{1}}\Vert f_{5}\Vert _{H_{x}^{1}}\Vert g\Vert
_{L_{t}^{\infty }H_{x}^{1}},  \notag
\end{align}%
and the high frequency estimate 
\begin{align}
\hspace{0.3in}& \hspace{-0.3in}\iint_{x,t}e^{it\Delta }f_{1}\left(
P_{>M_{0}}e^{it\Delta }f_{2}\right) e^{it\Delta }f_{3}e^{it\Delta
}f_{4}e^{it\Delta }f_{5}gdxdt  \label{estimate:MultiLinearWithFreqLoc1-high}
\\
& \lesssim \Vert f_{1}\Vert _{H_{x}^{-1}}\Vert P_{>M_{0}}f_{2}\Vert
_{H_{x}^{1}}\Vert f_{3}\Vert _{H_{x}^{1}}\Vert f_{4}\Vert _{H_{x}^{1}}\Vert
f_{5}\Vert _{H_{x}^{1}}\Vert g\Vert _{L_{t}^{\infty }H_{x}^{1}}.  \notag
\end{align}%
We prove the high frequency estimate (\ref%
{estimate:MultiLinearWithFreqLoc1-high}) which is at the same time (\ref%
{estimate:YounghunOld1}) first because it is in fact the key estimate, with
least room to play, in Lemma \ref{Lem:MultilinearWithFreqLocal}.

\paragraph{Proof of the high frequency estimate (\protect\ref%
{estimate:MultiLinearWithFreqLoc1-high})}

Putting $f_{2}=P_{>M_{0}}f_{2}$, (\ref{estimate:MultiLinearWithFreqLoc1-high}%
) is equivalent to 
\begin{align}
\hspace{0.3in}& \hspace{-0.3in}\iint_{x,t}e^{it\Delta }f_{1}e^{it\Delta
}f_{2}e^{it\Delta }f_{3}e^{it\Delta }f_{4}e^{it\Delta }f_{5}gdxdt
\label{estimate:MultiLinearWithFreqLoc1-high-realone} \\
& \lesssim \Vert f_{1}\Vert _{H_{x}^{-1}}\Vert f_{2}\Vert _{H_{x}^{1}}\Vert
f_{3}\Vert _{H_{x}^{1}}\Vert f_{4}\Vert _{H_{x}^{1}}\Vert f_{5}\Vert
_{H_{x}^{1}}\Vert g\Vert _{L_{t}^{\infty }H_{x}^{1}}  \notag
\end{align}%
Let $I$ denote the integral in (\ref%
{estimate:MultiLinearWithFreqLoc1-high-realone}). We insert a
Littlewood-Paley decomposition on each of the 6 factors so that 
\begin{equation*}
I=\sum_{M_{1},M_{2},M_{3},M_{4},M_{5},M}I_{M_{1},M_{2},M_{3},M_{4},M_{5},M}
\end{equation*}%
where 
\begin{equation*}
I_{M_{1},M_{2},M_{3},M_{4},M_{5},M}=\iint_{x,t}u_{1}u_{2}u_{3}u_{4}u_{5}vdxdt
\end{equation*}%
with $u_{j}=P_{M_{j}}e^{it\Delta }f_{j}$ and $v=P_{M}g$.

We have that the sum of all six frequencies is zero, so the top two
frequencies are comparable in size. Thus by symmetry we might as well assume
without loss that there are two (overlapping) cases

\noindent \emph{Case 1}. $M_{1}\sim M_{2}\geq M_{3}\geq M_{4}\geq M_{5}$ and 
$M_{1}\sim M_{2}\geq M$

\noindent \emph{Case 2}. $M_{1}\sim M\geq M_{2}\geq M_{3}\geq M_{4}\geq
M_{5} $

Note that Case 1 breaks into two cases

\noindent \emph{Case 1A}. $M_{1}\sim M_{2}\geq M_{3}\geq M_{4}\geq M_{5}$
and $M_{3}\geq M$

\noindent \emph{Case 1B}. $M_{1}\sim M_{2}\geq M\geq M_{3}\geq M_{4}\geq
M_{5}$

Let $I_{j}$ denote the integral restricted to the corresponding case. We can
also assume that the $H_{x}^{-1}$ norm goes on the $f_{1}$ term, as opposed
to $f_{j}$ for $2\leq j\leq 5$, since this is the hardest case due to our
ordering $M_{1}\geq M_{2}\geq M_{3}\geq M_{4}\geq M_{5}$. Now we begin the
proofs of each case.

\subparagraph{\noindent \textit{Case 1A of the high frequency estimate (%
\protect\ref{estimate:MultiLinearWithFreqLoc1-high})}}

Let us consider fixed $M_{1}$, $M_{2}$, $M_{3}$, $M_{4}$, $M_{5}$, and $M$
subject to the condition 
\begin{equation*}
M_{1}\sim M_{2}\geq M_{3}\geq M_{4}\geq M_{5}\text{ and }M\leq M_{3}
\end{equation*}%
Divide the $M_{1}$-dyadic space and the $M_{2}$-dyadic space into subcubes
of size $M_{3}$. Due to the frequency constraint $\xi _{2}=-(\xi _{1}+\xi
_{3}+\xi _{4}+\xi _{5}+\xi )$, for each choice $Q$ of an $M_{3}$-cube within
the $\xi _{1}$ space, the variable $\xi _{2}$ is constrained to at most $%
5^{3}$ of $M_{3}$-cubes dividing the $M_{2}$ dyadic space. For expository
convenience, we will refer to these $5^{3}$ cubes as a single cube $Q_{c}$
that corresponds to $Q$. 
\begin{align*}
I_{M_{1},\cdots ,M_{5},M}& \leq \sum_{Q}\Vert
P_{Q}u_{1}\,P_{Q_{c}}u_{2}\,u_{3}\,u_{4}\,u_{5}\,v\Vert _{L_{tx}^{1}} \\
& \lesssim \sum_{Q}\Vert P_{Q}u_{1}\,u_{3}\Vert _{L_{t,x}^{2}}\Vert
P_{Q_{c}}u_{2}\Vert _{L_{t,x}^{4}}\Vert u_{4}\Vert
_{L_{t}^{8}L_{x}^{16}}\Vert u_{5}\Vert _{L_{t}^{8}L_{x}^{16}}\Vert v\Vert
_{L_{t}^{\infty }L_{x}^{8}}
\end{align*}%
By Bernstein, 
\begin{equation*}
\lesssim M_{4}^{3/16}M_{5}^{3/16}M^{9/8}\sum_{Q}\Vert P_{Q}u_{1}\,u_{3}\Vert
_{L_{t,x}^{2}}\Vert P_{Q_{c}}u_{2}\Vert _{L_{t,x}^{4}}\Vert u_{4}\Vert
_{L_{tx}^{8}}\Vert u_{5}\Vert _{L_{tx}^{8}}\Vert v\Vert _{L_{t}^{\infty
}L_{x}^{2}}
\end{equation*}%
By bilinear Strichartz and Strichartz (\ref{E:Str1}), (\ref{E:Str2}), and (%
\ref{E:Str3}), where we note that the factor is $M_{3}^{1/4}$ instead of $%
M_{2}^{1/4}$, 
\begin{align*}
& \lesssim M_{4}^{3/16}M_{5}^{3/16}M^{9/8}M_{3}^{1/2}(\frac{M_{3}}{M_{1}}+%
\frac{1}{M_{3}})^{\delta }M_{3}^{1/4}M_{4}^{7/8}M_{5}^{7/8} \\
& \qquad \sum_{Q}\Vert P_{Q}P_{M_{1}}f_{1}\Vert _{L_{x}^{2}}\Vert
P_{M_{3}}f_{3}\Vert _{L_{x}^{2}}\Vert P_{Q_{c}}P_{M_{2}}f_{2}\Vert
_{L_{x}^{2}}\Vert P_{M_{4}}f_{4}\Vert _{L_{x}^{2}}\Vert P_{M_{5}}f_{5}\Vert
_{L_{x}^{2}}\Vert P_{M}g\Vert _{L_{t}^{\infty }L_{x}^{2}} \\
& \lesssim (\frac{M_{3}}{M_{1}}+\frac{1}{M_{3}})^{\delta
}M_{3}^{-1/4}M^{1/8}M_{4}^{1/16}M_{5}^{1/16} \\
& \qquad \sum_{Q}\Vert P_{Q}P_{M_{1}}f_{1}\Vert _{L_{x}^{2}}\Vert
P_{Q_{c}}P_{M_{2}}f_{2}\Vert _{L_{x}^{2}}\Vert P_{M_{3}}f_{3}\Vert
_{H_{x}^{1}}\Vert P_{M_{4}}f_{4}\Vert _{H_{x}^{1}}\Vert P_{M_{5}}f_{5}\Vert
_{H_{x}^{1}}\Vert P_{M}g\Vert _{L_{t}^{\infty }H_{x}^{1}}
\end{align*}%
Carrying out the sum in $Q$ via Cauchy-Schwarz, 
\begin{align*}
& \lesssim (\frac{M_{3}}{M_{1}}+\frac{1}{M_{3}})^{\delta
}M_{3}^{-1/4}M^{1/8}M_{4}^{1/16}M_{5}^{1/16} \\
& \qquad \Vert P_{M_{1}}f_{1}\Vert _{L_{x}^{2}}\Vert P_{M_{2}}f_{2}\Vert
_{L_{x}^{2}}\Vert P_{M_{3}}f_{3}\Vert _{H_{x}^{1}}\Vert P_{M_{4}}f_{4}\Vert
_{H_{x}^{1}}\Vert P_{M_{5}}f_{5}\Vert _{H_{x}^{1}}\Vert P_{M}g\Vert
_{L_{t}^{\infty }H_{x}^{1}}
\end{align*}%
Swapping an $M_{1}$ and $M_{2}$ factor (since $M_{1}\sim M_{2}$) 
\begin{align*}
& \lesssim (\frac{M_{3}}{M_{1}}+\frac{1}{M_{3}})^{\delta
}M_{3}^{-1/4}M^{1/8}M_{4}^{1/16}M_{5}^{1/16} \\
& \qquad \Vert P_{M_{1}}f_{1}\Vert _{H_{x}^{-1}}\Vert P_{M_{2}}f_{2}\Vert
_{H_{x}^{1}}\Vert P_{M_{3}}f_{3}\Vert _{H_{x}^{1}}\Vert P_{M_{4}}f_{4}\Vert
_{H_{x}^{1}}\Vert P_{M_{5}}f_{5}\Vert _{H_{x}^{1}}\Vert P_{M}g\Vert
_{L_{t}^{\infty }H_{x}^{1}}
\end{align*}%
Carrying out the sum in this case, we obtain 
\begin{align*}
I_{1A}& \lesssim \sup_{M_{3}}\Vert P_{M_{3}}f_{3}\Vert
_{H_{x}^{1}}\sup_{M_{4}}\Vert P_{M_{4}}f_{4}\Vert
_{H_{x}^{1}}\sup_{M_{5}}\Vert P_{M_{5}}f_{5}\Vert _{H_{x}^{1}}\sup_{M}\Vert
P_{M}g\Vert _{L_{t}^{\infty }H_{x}^{1}} \\
& \qquad \sum_{\substack{ M_{1}  \\ M_{1}\sim M_{2}}}\sigma (M_{1})\Vert
P_{M_{1}}f_{1}\Vert _{H_{x}^{-1}}\Vert P_{M_{2}}f_{2}\Vert _{H_{x}^{1}}
\end{align*}%
where%
\begin{eqnarray*}
\sigma (M_{1}) &=&\sum_{\substack{ M_{3}  \\ M_{3}\leq M_{1}}}(\frac{M_{3}}{%
M_{1}}+\frac{1}{M_{3}})^{\delta }M_{3}^{-1/4}\left( \sum_{\substack{ %
M_{4},M_{5}  \\ M_{5}\leq M_{4}\leq M_{3}}}M_{4}^{1/16}M_{5}^{1/16}\sum 
_{\substack{ M  \\ M\leq M_{3}}}M^{1/8}\right) \\
&=&\sum_{\substack{ M_{3}  \\ M_{3}\leq M_{1}}}(\frac{M_{3}}{M_{1}}+\frac{1}{%
M_{3}})^{\delta }=O(1).
\end{eqnarray*}%
Thus we can complete the proof of Case 1A of the high frequency case by
Cauchy-Schwarz, summing in $M_{1}$.

\paragraph{\textit{\noindent Case 1B of the high frequency estimate (\protect
\ref{estimate:MultiLinearWithFreqLoc1-high})}}

Recall in \noindent \textit{Case 1B}, $M_{1}$, $M_{2}$, $M_{3}$, $M_{4}$, $%
M_{5}$, and $M$ are subject to the condition%
\begin{equation*}
M_{1}\sim M_{2}\geq M\geq M_{3}\geq M_{4}\geq M_{5}
\end{equation*}%
We start like Case 1A but divide the $M_{1}$-dyadic space and the $M_{2}$%
-dyadic space into subcubes of size $M$ this time. 
\begin{align*}
I_{M_{1},\cdots ,M_{5},M}& \leq \sum_{Q}\Vert
P_{Q}u_{1}\,P_{Q_{c}}u_{2}\,u_{3}\,u_{4}\,u_{5}\,v\Vert _{L_{tx}^{1}} \\
& \lesssim \sum_{Q}\Vert P_{Q}u_{1}\,u_{3}\Vert _{L_{t,x}^{2}}\Vert
P_{Q_{c}}u_{2}\Vert _{L_{t,x}^{11/3}}\Vert u_{4}\Vert
_{L_{t}^{44/5}L_{x}^{\infty }}\Vert u_{5}\Vert _{L_{t}^{44/5}L_{x}^{\infty
}}\Vert v\Vert _{L_{t}^{\infty }L_{x}^{22/5}}
\end{align*}%
By Bernstein, 
\begin{equation*}
\lesssim M_{4}^{15/44}M_{5}^{15/44}M^{9/11}\sum_{Q}\Vert
P_{Q}u_{1}\,u_{3}\Vert _{L_{t,x}^{2}}\Vert P_{Q_{c}}u_{2}\Vert
_{L_{t,x}^{11/3}}\Vert u_{4}\Vert _{L_{t,x}^{44/5}}\Vert u_{5}\Vert
_{L_{t,x}^{44/5}}\Vert v\Vert _{L_{t}^{\infty }L_{x}^{2}}
\end{equation*}%
By bilinear Strichartz and Strichartz (\ref{E:Str1}), (\ref{E:Str2}), and (%
\ref{E:Str3}),\footnote{%
We used the uncommon exponent $p=\frac{11}{3}$ in (\ref{E:Str2}) in this
proof solely because it is the midpoint of $\frac{10}{3}$, where (\ref%
{E:Str1}) fails, and $4$, where this proof fails.} where we note that the
factor is $M^{3/22}$ instead of $M_{2}^{3/22}$,%
\begin{eqnarray*}
&\lesssim &M_{4}^{15/44}M_{5}^{15/44}M^{9/11}M_{3}^{1/2}(\frac{M_{3}}{M_{1}}+%
\frac{1}{M_{3}})^{\delta }M^{3/22}M_{4}^{41/44}M_{5}^{41/44} \\
&&\sum_{Q}\Vert P_{Q}P_{M_{1}}f_{1}\Vert _{L_{x}^{2}}\Vert
P_{M_{3}}f_{3}\Vert _{L_{x}^{2}}\Vert P_{Q_{c}}P_{M_{2}}f_{2}\Vert
_{L_{x}^{2}}\Vert P_{M_{4}}f_{4}\Vert _{L_{x}^{2}}\Vert P_{M_{5}}f_{5}\Vert
_{L_{x}^{2}}\Vert P_{M}g\Vert _{L_{t}^{\infty }L_{x}^{2}} \\
&\lesssim &(\frac{M_{3}}{M_{1}}+\frac{1}{M_{3}})^{\delta
}M_{3}^{-1/2}M^{-1/22}M_{4}^{3/11}M_{5}^{3/11} \\
&&\sum_{Q}\Vert P_{Q}P_{M_{1}}f_{1}\Vert _{L_{x}^{2}}\Vert
P_{Q_{c}}P_{M_{2}}f_{2}\Vert _{L_{x}^{2}}\Vert P_{M_{3}}f_{3}\Vert
_{H_{x}^{1}}\Vert P_{M_{4}}f_{4}\Vert _{H_{x}^{1}}\Vert P_{M_{5}}f_{5}\Vert
_{H_{x}^{1}}\Vert P_{M}g\Vert _{L_{t}^{\infty }H_{x}^{1}}
\end{eqnarray*}%
Carrying out the sum in $Q$ via Cauchy-Schwarz, 
\begin{align*}
& \lesssim (\frac{M_{3}}{M_{1}}+\frac{1}{M_{3}})^{\delta
}M_{3}^{-1/2}M^{-1/22}M_{4}^{3/11}M_{5}^{3/11} \\
& \qquad \Vert P_{M_{1}}f_{1}\Vert _{L_{x}^{2}}\Vert P_{M_{2}}f_{2}\Vert
_{L_{x}^{2}}\Vert P_{M_{3}}f_{3}\Vert _{H_{x}^{1}}\Vert P_{M_{4}}f_{4}\Vert
_{H_{x}^{1}}\Vert P_{M_{5}}f_{5}\Vert _{H_{x}^{1}}\Vert P_{M}g\Vert
_{L_{t}^{\infty }H_{x}^{1}}
\end{align*}%
Swapping an $M_{1}$ and $M_{2}$ factor (since $M_{1}\sim M_{2}$) 
\begin{align*}
& \lesssim (\frac{M_{3}}{M_{1}}+\frac{1}{M_{3}})^{\delta
}M_{3}^{-1/2}M^{-1/22}M_{4}^{3/11}M_{5}^{3/11} \\
& \qquad \Vert P_{M_{1}}f_{1}\Vert _{H_{x}^{-1}}\Vert P_{M_{2}}f_{2}\Vert
_{H_{x}^{1}}\Vert P_{M_{3}}f_{3}\Vert _{H_{x}^{1}}\Vert P_{M_{4}}f_{4}\Vert
_{H_{x}^{1}}\Vert P_{M_{5}}f_{5}\Vert _{H_{x}^{1}}\Vert P_{M}g\Vert
_{L_{t}^{\infty }H_{x}^{1}}
\end{align*}%
Carrying out the sum in this case, we obtain (taking without loss $\delta <%
\frac{1}{22}$) 
\begin{align*}
I_{1B}& \lesssim \sup_{M_{3}}\Vert P_{M_{3}}f_{3}\Vert
_{H_{x}^{1}}\sup_{M_{4}}\Vert P_{M_{4}}f_{4}\Vert
_{H_{x}^{1}}\sup_{M_{5}}\Vert P_{M_{5}}f_{5}\Vert _{H_{x}^{1}}\sup_{M}\Vert
P_{M}g\Vert _{L_{t}^{\infty }H_{x}^{1}} \\
& \qquad \sum_{\substack{ M_{1}  \\ M_{1}\sim M_{2}}}\sigma (M_{1})\Vert
P_{M_{1}}f_{1}\Vert _{H_{x}^{-1}}\Vert P_{M_{2}}f_{2}\Vert _{H_{x}^{1}}
\end{align*}%
where%
\begin{eqnarray*}
\sigma (M_{1}) &=&\sum_{\substack{ M,M_{3},M_{4},M_{5}  \\ M_{1}\sim
M_{2}\geq M\geq M_{3}\geq M_{4}\geq M_{5}}}(\frac{M_{3}}{M_{1}}+\frac{1}{%
M_{3}})^{\delta }M_{3}^{-1/2}M^{-1/22}M_{4}^{3/11}M_{5}^{3/11} \\
&=&\sum_{\substack{ M,M_{3}  \\ M_{1}\sim M_{2}\geq M\geq M_{3}}}(\frac{M_{3}%
}{M_{1}}+\frac{1}{M_{3}})^{\delta }M_{3}^{1/22}M^{-1/22}
\end{eqnarray*}%
Decomposing $\sigma (M_{1})$ into the two cases $\frac{M_{3}}{M_{1}}>\frac{1%
}{M_{3}}$ and $\frac{M_{3}}{M_{1}}<\frac{1}{M_{3}}$ yields two $O(1)$ sums,
that is, $\sigma (M_{1})$ is $O(1)$ and we can complete the proof of Case 1B
of the high frequency case by Cauchy-Schwarz, summing in $M_{1}$.

\subparagraph{\noindent \textit{Case 2 of the high frequency estimate (%
\protect\ref{estimate:MultiLinearWithFreqLoc1-high})}}

Recall that 
\begin{equation*}
M_{1}\sim M>M_{2}\geq M_{3}\geq M_{4}\geq M_{5}
\end{equation*}%
in Case 2. We will assume $v=P_{M_{1}}g$ for convenience.\footnote{%
To realize this, one could just double the size of the cubes in the $M$
decomposition.}We start with 
\begin{equation*}
I_{2}\lesssim \sum_{\substack{ M_{1},M_{2},M_{3},M_{4},M_{5}  \\ %
M_{1}>M_{2}\geq M_{3}\geq M_{4}\geq M_{5}}}\Vert (\langle \nabla \rangle
^{-1}u_{1})\,u_{2}\Vert _{L_{t,x}^{2}}\Vert u_{3}u_{4}u_{5}\langle \nabla
\rangle v\Vert _{L_{t,x}^{2}}
\end{equation*}%
By Cauchy-Schwarz in the $M_{1}$ sum 
\begin{equation*}
I_{2}\lesssim AB
\end{equation*}%
where $A$ and $B$ are, for fixed $M_{2},M_{3},M_{4},M_{5}$ given by 
\begin{equation*}
A=\left( \sum_{\substack{ M_{1}  \\ M_{1}\geq M_{2}}}\Vert \langle \nabla
\rangle ^{-1}u_{1}\,u_{2}\Vert _{L_{x,t}^{2}}^{2}\right) ^{1/2}\,,\qquad
B=\left( \sum_{\substack{ M_{1}  \\ M_{1}\geq M_{2}}}\Vert
u_{3}u_{4}u_{5}\langle \nabla \rangle v\Vert _{L_{x,t}^{2}}^{2}\right) ^{1/2}
\end{equation*}%
For $A$, we apply the bilinear Strichartz (\ref{E:Str3}), 
\begin{equation*}
A\lesssim M_{2}^{1/2}\Vert P_{M_{2}}f_{2}\Vert _{L^{2}}\left( \sum 
_{\substack{ M_{1}  \\ M_{1}\geq M_{2}}}\left( \frac{M_{2}}{M_{1}}+\frac{1}{%
M_{2}}\right) ^{2\delta }\Vert P_{M_{1}}f_{1}\Vert _{H^{-1}}^{2}\right)
^{1/2}
\end{equation*}%
We divide into two cases, Case 2A where $\frac{M_{2}}{M_{1}}\leq \frac{1}{%
M_{2}}$, and Case 2B, where $\frac{M_{2}}{M_{1}}\geq \frac{1}{M_{2}}$. In
Case 2A, we have 
\begin{equation*}
A_{2A}\lesssim M_{2}^{-\frac{1}{2}-\delta }\Vert P_{M_{2}}f_{2}\Vert
_{H^{1}}\Vert f_{1}\Vert _{H^{-1}}
\end{equation*}%
In Case 2B, we have $M_{1}^{1/2}\leq M_{2}\leq M_{1}$ equivalently $%
M_{2}\leq M_{1}\leq M_{2}^{2}$ as constraints. 
\begin{equation*}
A_{2B}\lesssim M_{2}^{-\frac{1}{2}+\delta }\Vert P_{M_{2}}f_{2}\Vert
_{H^{1}}\left( \sum_{\substack{ M_{1}  \\ M_{2}\leq M_{1}\leq M_{2}^{2}}}%
M_{1}^{-2\delta }\Vert P_{M_{1}}f_{1}\Vert _{H^{-1}}^{2}\right) ^{1/2}
\end{equation*}%
so that 
\begin{equation*}
A_{2B}\lesssim M_{2}^{-\frac{1}{2}+\delta }\Vert P_{M_{2}}f_{2}\Vert
_{H^{1}}\Vert P_{M_{2}\leq \bullet \leq M_{2}^{2}}f_{1}\Vert _{H^{-1-\delta
}}
\end{equation*}%
For $B$, we first write out the integral 
\begin{equation*}
B^{2}\lesssim \sum_{\substack{ M_{1}  \\ M_{1}\geq M_{2}}}%
\iint_{x,t}|u_{3}|^{2}|u_{4}|^{2}|u_{5}|^{2}|\langle \nabla \rangle
v|^{2}\,dx\,dt
\end{equation*}%
Then we sup out the $u_{3}$, $u_{4}$, and $u_{5}$ terms out of the $x$%
-integral and bring the $M_{1}$-sum inside the $t$ integral 
\begin{equation*}
B^{2}\lesssim \int_{t}\Vert u_{3}\Vert _{L_{x}^{\infty }}^{2}\Vert
u_{4}\Vert _{L_{x}^{\infty }}^{2}\Vert u_{5}\Vert _{L_{x}^{\infty }}^{2}\sum 
_{\substack{ M_{1}  \\ M_{1}\geq M_{2}}}\int_{x}|\langle \nabla \rangle
v|^{2}\,dx\,dt
\end{equation*}%
which then simplifies to 
\begin{equation*}
B\lesssim \left( \int_{t}\Vert u_{3}\Vert _{L_{x}^{\infty }}^{2}\Vert
u_{4}\Vert _{L_{x}^{\infty }}^{2}\Vert u_{5}\Vert _{L_{x}^{\infty
}}^{2}\Vert g\Vert _{H^{1}}^{2}\,dt\right) ^{1/2}.
\end{equation*}%
Now we can sup the $\Vert g\Vert _{H_{x}^{1}}$ term$\footnote{%
We could have kept $P_{\geq M_{2}}$ on $g$ but this will not help us later.}$
out of the $t$ integral and apply H\"{o}lder in $t$ to the remaining terms 
\begin{equation*}
B\lesssim \Vert u_{3}\Vert _{L_{t}^{6}L_{x}^{\infty }}\Vert u_{4}\Vert
_{L_{t}^{6}L_{x}^{\infty }}\Vert u_{5}\Vert _{L_{t}^{6}L_{x}^{\infty }}\Vert
g\Vert _{L_{t}^{\infty }H_{x}^{1}}
\end{equation*}%
By Bernstein and Strichartz (\ref{E:Str1}), and absorbing one derivative
into the $f_{3}$, $f_{4}$ and $f_{5}$ terms 
\begin{equation*}
B\lesssim M_{3}^{1/6}M_{4}^{1/6}M_{5}^{1/6}\Vert P_{M_{3}}f_{3}\Vert
_{H_{x}^{1}}\Vert P_{M_{4}}f_{4}\Vert _{H_{x}^{1}}\Vert P_{M_{5}}f_{5}\Vert
_{H_{x}^{1}}\Vert g\Vert _{L_{t}^{\infty }H_{x}^{1}}
\end{equation*}%
Putting it all together (since the $M_{1}$ sum has already been carried out) 
\begin{align*}
I_{2}& \lesssim \sum_{\substack{ M_{2},M_{3},M_{4},M_{5}  \\ M_{2}\geq
M_{3}\geq M_{4}\geq M_{5}}}M_{2}^{-1/2}M_{3}^{1/6}M_{4}^{1/6}M_{5}^{1/6}%
\Vert P_{M_{3}}f_{3}\Vert _{H_{x}^{1}}\Vert P_{M_{4}}f_{4}\Vert
_{H_{x}^{1}}\Vert P_{M_{5}}f_{5}\Vert _{H_{x}^{1}}\Vert g\Vert
_{L_{t}^{\infty }H_{x}^{1}} \\
& \qquad (M_{2}^{-\delta }\Vert f_{1}\Vert _{H^{-1}}\Vert
P_{M_{2}}f_{2}\Vert _{H^{1}}+M_{2}^{\delta }\Vert P_{M_{2}\leq \bullet \leq
M_{2}^{2}}\phi _{1}\Vert _{H^{-1-\delta }}\Vert P_{M_{2}}f_{2}\Vert _{H^{1}})
\end{align*}%
Suping out the $f_{3}$, $f_{4}$ and $f_{5}$ terms in $M_{3}$, $M_{4}$, and $%
M_{5}$ respectively 
\begin{align*}
I_{2}& \lesssim \Vert f_{3}\Vert _{H_{x}^{1}}\Vert f_{4}\Vert
_{H_{x}^{1}}\Vert f_{5}\Vert _{H_{x}^{1}}\Vert g\Vert _{L_{t}^{\infty
}H_{x}^{1}} \\
& \qquad \sum_{\substack{ M_{2},M_{3},M_{4},M_{5}  \\ M_{2}\geq M_{3}\geq
M_{4}\geq M_{5}}}M_{2}^{-1/2}M_{3}^{1/6}M_{4}^{1/6}M_{5}^{1/6} \\
& \qquad \left( M_{2}^{-\delta }\Vert f_{1}\Vert _{H^{-1}}\Vert
P_{M_{2}}f_{2}\Vert _{H^{1}}+M_{2}^{\delta }\Vert P_{M_{2}\leq \bullet \leq
M_{2}^{2}}f_{1}\Vert _{H^{-1-\delta }}\Vert P_{M_{2}}f_{2}\Vert
_{H^{1}}\right)
\end{align*}%
Carry out the $M_{3}$, $M_{4}$ and $M_{5}$ sums 
\begin{align*}
I_{2}& \lesssim \Vert f_{3}\Vert _{H_{x}^{1}}\Vert f_{4}\Vert
_{H_{x}^{1}}\Vert f_{5}\Vert _{H_{x}^{1}}\Vert g\Vert _{L_{t}^{\infty
}H_{x}^{1}} \\
& \qquad \sum_{M_{2}}(M_{2}^{-\delta }\Vert f_{1}\Vert _{H^{-1}}\Vert
P_{M_{2}}f_{2}\Vert _{H^{1}}+M_{2}^{\delta }\Vert P_{M_{2}\leq \bullet \leq
M_{2}^{2}}f_{1}\Vert _{H^{-1-\delta }}\Vert P_{M_{2}}f_{2}\Vert _{H^{1}})
\end{align*}%
This leaves two $M_{2}$ sums to carry out, namely 
\begin{equation*}
D=\sum_{M_{2}}M_{2}^{-\delta }\Vert f_{1}\Vert _{H^{-1}}\Vert
P_{M_{2}}f_{2}\Vert _{H^{1}}
\end{equation*}%
and 
\begin{equation*}
E=\sum_{M_{2}}M_{2}^{\delta }\Vert P_{M_{2}\leq \bullet \leq
M_{2}^{2}}f_{1}\Vert _{H^{-1-\delta }}\Vert P_{M_{2}}f_{2}\Vert _{H^{1}}
\end{equation*}%
For $D$, we sup out in the $f_{2}$ term and just carry out the $M_{2}$ sum 
\begin{equation*}
D\leq \Vert f_{1}\Vert _{H^{-1}}\Vert f_{2}\Vert
_{H^{1}}\sum_{M_{2}}M_{2}^{-\delta }\lesssim \Vert f_{1}\Vert _{H^{-1}}\Vert
f_{2}\Vert _{H^{1}}
\end{equation*}%
For $E$, we apply Cauchy-Schwarz in $M_{2}$ 
\begin{equation*}
E^{2}\leq \sum_{M_{2}}M_{2}^{2\delta }\Vert P_{M_{2}\leq \bullet \leq
M_{2}^{2}}f_{1}\Vert _{H^{-1-\delta }}^{2}\sum_{M_{2}}\Vert
P_{M_{2}}f_{2}\Vert _{H^{1}}^{2}
\end{equation*}%
so that 
\begin{equation*}
E^{2}=\Vert f_{2}\Vert _{H^{1}}^{2}\sum_{M_{2}}M_{2}^{2\delta }\Vert
P_{M_{2}\leq \bullet \leq M_{2}^{2}}f_{1}\Vert _{H^{-1-\delta }}^{2}
\end{equation*}%
Now in the sum, decompose the frequency region $M_{2}\leq \bullet \leq
M_{2}^{2}$ into dyadic pieces (labeled again by $M_{1}$) 
\begin{equation*}
E^{2}=\Vert f_{2}\Vert _{H^{1}}^{2}\sum_{\substack{ M_{1},M_{2}  \\ %
M_{2}\leq M_{1}\leq M_{2}^{2}}}M_{2}^{2\delta }M_{1}^{-2\delta }\Vert
P_{M_{1}}f_{1}\Vert _{H^{-1}}^{2}
\end{equation*}%
Then bring the $M_{2}$ sum to the inside, (in which the sum is constrained
to $M_{2}\leq M_{1}$) in order to get 
\begin{equation*}
E^{2}=\Vert f_{2}\Vert _{H^{1}}^{2}\Vert f_{1}\Vert _{H^{-1}}^{2}
\end{equation*}%
which completes the estimate of \textit{Case 2 of (\ref%
{estimate:MultiLinearWithFreqLoc1-high-realone})}. So far, we have obtained 
\textit{(\ref{estimate:MultiLinearWithFreqLoc1-high-realone}), the high
frequency part of (\ref{estimate:MultiLinearWithFreqLoc1})}

\bigskip

\paragraph{Proof of the low frequency estimate (\protect\ref%
{estimate:MultiLinearWithFreqLoc1-low})}

\textit{We can now start the proof of (\ref%
{estimate:MultiLinearWithFreqLoc1-low}) which reads }%
\begin{align*}
\hspace{0.3in}& \hspace{-0.3in}\iint_{x,t}e^{it\Delta }f_{1}\left(
P_{\leqslant M_{0}}e^{it\Delta }f_{2}\right) e^{it\Delta }f_{3}e^{it\Delta
}f_{4}e^{it\Delta }f_{5}gdxdt \\
& \lesssim T^{\frac{5}{22}}M_{0}^{\frac{5}{11}}\Vert f_{1}\Vert
_{H_{x}^{-1}}\Vert f_{2}\Vert _{H_{x}^{1}}\Vert f_{3}\Vert _{H_{x}^{1}}\Vert
f_{4}\Vert _{H_{x}^{1}}\Vert f_{5}\Vert _{H_{x}^{1}}\Vert g\Vert
_{L_{t}^{\infty }H_{x}^{1}},
\end{align*}

Since in our proof, we assume $M_{1}\geq M_{2}\geq M_{3}\geq M_{4}\geq M_{5}$%
, we cannot assume that $P_{<M_{0}}$ sits on $f_{2}$, but have to allow for $%
f_{1}$, $f_{2}$, $f_{3}$, $f_{4}$, or $f_{5}$. The most difficult case is
when $P_{<M_{0}}$ lands on $f_{5}$ (the least gain case) and the $H_{x}^{-1}$
norm goes on the $f_{1}$ term, so we will assume so in the proofs of the
cases.

\subparagraph{\textit{Case 1A of the low frequency estimate (\protect\ref%
{estimate:MultiLinearWithFreqLoc1-low})}}

Again, we have 
\begin{equation*}
M_{1}\sim M_{2}\geq M_{3}\geq M_{4}\geq M_{5}\text{ and }M\leq M_{3}
\end{equation*}%
For this Case 1A, we will consider%
\begin{equation*}
I_{M_{1},\cdots ,M_{4},M}=\sum_{M_{5}}I_{M_{1},\cdots ,M_{4},M_{5},M}
\end{equation*}%
and we replace $\sum_{M_{5},M_{4}\geq M_{5}}u_{5}=e^{it\Delta
}\sum_{M_{5},M_{4}\geq M_{5}}P_{\leq M_{0}}P_{M_{5}}f_{5}$ with $%
u_{5}=e^{it\Delta }P_{\leq M_{0}}P_{\leq M_{4}}f_{5}$ to shorten some
notations. Below, we will use that, by first Bernstein with the $P_{\leq
M_{0}}$ factor, then Sobolev, 
\begin{equation*}
\Vert u_{5}\Vert _{L_{t}^{\infty }L_{x}^{\infty }}\lesssim M_{0}^{1/2}\Vert
P_{\leq M_{4}}e^{it\Delta }f_{5}\Vert _{L_{t}^{\infty }L_{x}^{6}}\lesssim
M_{0}^{1/2}\Vert P_{\leq M_{4}}f_{5}\Vert _{H^{1}}.
\end{equation*}

As before, divide the $M_{1}$-dyadic space and the $M_{2}$-dyadic space into
subcubes of size $M$, we reach 
\begin{align*}
I_{M_{1},\cdots ,M_{4},M}& \leq \sum_{Q}\Vert
P_{Q}u_{1}\,P_{Q_{c}}u_{2}\,u_{3}\,u_{4}\,u_{5}\,v\Vert _{L_{t,x}^{1}} \\
& \lesssim T^{1/4}\sum_{Q}\Vert P_{Q}u_{1}\,u_{3}\Vert _{L_{t,x}^{2}}\Vert
P_{Q_{c}}u_{2}\Vert _{L_{t,x}^{4}}\Vert u_{4}\Vert _{L_{t}^{\infty
}L_{x}^{8}}\Vert u_{5}\Vert _{L_{t}^{\infty }L_{x}^{\infty }}\Vert v\Vert
_{L_{t}^{\infty }L_{x}^{8}}
\end{align*}%
By Bernstein, 
\begin{equation*}
\lesssim T^{1/4}M_{0}^{1/2}M_{4}^{9/8}M^{9/8}\sum_{Q}\Vert
P_{Q}u_{1}\,u_{3}\Vert _{L_{t,x}^{2}}\Vert P_{Q_{c}}u_{2}\Vert
_{L_{t,x}^{4}}\Vert u_{4}\Vert _{L_{t}^{\infty }L_{x}^{2}}\Vert P_{\leq
M_{4}}f_{5}\Vert _{L_{t}^{\infty }H_{x}^{1}}\Vert v\Vert _{L_{t}^{\infty
}L_{x}^{2}}
\end{equation*}%
By bilinear Strichartz and Strichartz (\ref{E:Str1}), (\ref{E:Str2}), and (%
\ref{E:Str3}), where the factor is $M_{3}^{1/4}$ instead of $M_{2}^{1/4}$, 
\begin{align*}
& \lesssim T^{1/4}M_{0}^{1/2}M_{4}^{9/8}M^{9/8}M_{3}^{1/2}(\frac{M_{3}}{M_{1}%
}+\frac{1}{M_{3}})^{\delta }M_{3}^{1/4} \\
& \qquad \sum_{Q}\Vert P_{Q}P_{M_{1}}f_{1}\Vert _{L_{x}^{2}}\Vert
P_{M_{3}}f_{3}\Vert _{L_{x}^{2}}\Vert P_{Q_{c}}P_{M_{2}}f_{2}\Vert
_{L_{x}^{2}}\Vert P_{M_{4}}f_{4}\Vert _{L_{x}^{2}}\Vert P_{\leq
M_{4}}f_{5}\Vert _{L_{x}^{2}}\Vert P_{M}g\Vert _{L_{t}^{\infty }L_{x}^{2}} \\
& \lesssim T^{1/4}M_{0}^{1/2}(\frac{M_{3}}{M_{1}}+\frac{1}{M_{3}})^{\delta
}M_{3}^{-1/4}M^{1/8}M_{4}^{1/8} \\
& \qquad \sum_{Q}\Vert P_{Q}P_{M_{1}}f_{1}\Vert _{L_{x}^{2}}\Vert
P_{Q_{c}}P_{M_{2}}f_{2}\Vert _{L_{x}^{2}}\Vert P_{M_{3}}f_{3}\Vert
_{H_{x}^{1}}\Vert P_{M_{4}}f_{4}\Vert _{H_{x}^{1}}\Vert P_{\leq
M_{4}}f_{5}\Vert _{H_{x}^{1}}\Vert P_{M}g\Vert _{L_{t}^{\infty }H_{x}^{1}}
\end{align*}%
Carrying out the sum in $Q$ via Cauchy-Schwarz, 
\begin{align*}
& \lesssim T^{1/4}M_{0}^{1/2}(\frac{M_{3}}{M_{1}}+\frac{1}{M_{3}})^{\delta
}M_{3}^{-1/4}M^{1/8}M_{4}^{1/8} \\
& \qquad \Vert P_{M_{1}}f_{1}\Vert _{L_{x}^{2}}\Vert P_{M_{2}}f_{2}\Vert
_{L_{x}^{2}}\Vert P_{M_{3}}f_{3}\Vert _{H_{x}^{1}}\Vert P_{M_{4}}f_{4}\Vert
_{H_{x}^{1}}\Vert P_{\leq M_{4}}f_{5}\Vert _{H_{x}^{1}}\Vert P_{M}g\Vert
_{L_{t}^{\infty }H_{x}^{1}}
\end{align*}%
Swapping an $M_{1}$ and $M_{2}$ factor (since $M_{1}\sim M_{2}$) 
\begin{align*}
& \lesssim T^{1/4}M_{0}^{1/2}(\frac{M_{3}}{M_{1}}+\frac{1}{M_{3}})^{\delta
}M_{3}^{-1/4}M^{1/8}M_{4}^{1/8} \\
& \qquad \Vert P_{M_{1}}f_{1}\Vert _{H_{x}^{-1}}\Vert P_{M_{2}}f_{2}\Vert
_{H_{x}^{1}}\Vert P_{M_{3}}f_{3}\Vert _{H_{x}^{1}}\Vert P_{M_{4}}f_{4}\Vert
_{H_{x}^{1}}\Vert P_{\leq M_{4}}f_{5}\Vert _{H_{x}^{1}}\Vert P_{M}g\Vert
_{L_{t}^{\infty }H_{x}^{1}}
\end{align*}%
Carrying out the sum in this case, we obtain 
\begin{align*}
I_{1A}& \lesssim T^{1/4}M_{0}^{1/2}\sup_{M_{3}}\Vert P_{M_{3}}f_{3}\Vert
_{H_{x}^{1}}\sup_{M_{4}}\Vert P_{M_{4}}f_{4}\Vert
_{H_{x}^{1}}\sup_{M_{4}}\Vert P_{\leq M_{4}}f_{5}\Vert
_{H_{x}^{1}}\sup_{M}\Vert P_{M}g\Vert _{L_{t}^{\infty }H_{x}^{1}} \\
& \qquad \sum_{\substack{ M_{1}  \\ M_{1}\sim M_{2}}}\sigma (M_{1})\Vert
P_{M_{1}}f_{1}\Vert _{H_{x}^{-1}}\Vert P_{M_{2}}f_{2}\Vert _{H_{x}^{1}}
\end{align*}%
where 
\begin{align*}
\sigma (M_{1})& =\sum_{\substack{ M_{3}  \\ M_{3}\leq M_{1}}}(\frac{M_{3}}{%
M_{1}}+\frac{1}{M_{3}})^{\delta }M_{3}^{-1/4}\left( \sum_{\substack{ M_{4} 
\\ M_{4}\leq M_{3}}}M_{4}^{1/8}\sum_{\substack{ M  \\ M\leq M_{3}}}%
M^{1/8}\right) \\
& =\sum_{\substack{ M_{3}  \\ M_{3}\leq M_{1}}}(\frac{M_{3}}{M_{1}}+\frac{1}{%
M_{3}})^{\delta }=O(1)
\end{align*}%
Thus we can complete the proof of Case 1A of the low frequency case by
Cauchy-Schwarz, summing in $M_{1}$.

\subparagraph{\noindent \textit{Case 1B of the low frequency estimate (%
\protect\ref{estimate:MultiLinearWithFreqLoc1-low})}}

Recall Case 1B in which,%
\begin{equation*}
M_{1}\sim M_{2}\geq M\geq M_{3}\geq M_{4}\geq M_{5}
\end{equation*}%
and we divide the $M_{1}$-dyadic space and the $M_{2}$-dyadic space into
subcubes of size $M$. We have\footnote{%
This is where we could not obtain $T^{\frac{1}{4}}M^{\frac{1}{2}}$ but $%
\left( T^{\frac{1}{4}}M^{\frac{1}{2}}\right) ^{1-\varepsilon }$ because (\ref%
{E:Str1}) fails at $p=\frac{10}{3}.$} 
\begin{align*}
I_{M_{1},\cdots ,M_{5},M}& \leq \sum_{Q}\Vert
P_{Q}u_{1}\,P_{Q_{c}}u_{2}\,u_{3}\,u_{4}\,u_{5}\,v\Vert _{L_{t,x}^{1}} \\
& \lesssim T^{5/22}\sum_{Q}\Vert P_{Q}u_{1}\,u_{3}\Vert _{L_{t,x}^{2}}\Vert
P_{Q_{c}}u_{2}\Vert _{L_{t,x}^{11/3}}\Vert u_{4}\Vert _{L_{t}^{\infty
}L_{x}^{\infty }}\Vert u_{5}\Vert _{L_{t}^{\infty }L_{x}^{\infty }}\Vert
v\Vert _{L_{t}^{\infty }L_{x}^{22/5}}
\end{align*}%
By Bernstein, 
\begin{equation*}
\lesssim T^{5/22}M_{0}^{5/11}M_{4}^{3/2}M_{5}^{23/22}M^{9/11}\sum_{Q}\Vert
P_{Q}u_{1}\,u_{3}\Vert _{L_{t,x}^{2}}\Vert P_{Q_{c}}u_{2}\Vert
_{L_{t,x}^{11/3}}\Vert u_{4}\Vert _{L_{t}^{\infty }L_{x}^{2}}\Vert
u_{5}\Vert _{L_{t}^{\infty }L_{x}^{2}}\Vert v\Vert _{L_{t}^{\infty
}L_{x}^{2}}
\end{equation*}%
where we note that for the $u_{5}$ we used both the $P_{\leq M_{0}}$ and $%
P_{M_{5}}$ factors to estimate 
\begin{equation*}
\Vert P_{\leq M_{0}}P_{M_{5}}u_{5}\Vert _{L_{t}^{\infty }L_{x}^{\infty
}}\lesssim M_{0}^{5/11}M_{5}^{23/22}\Vert u_{5}\Vert _{L_{t}^{\infty
}L_{x}^{2}}.
\end{equation*}%
By bilinear Strichartz and Strichartz (\ref{E:Str1}), (\ref{E:Str2}), and (%
\ref{E:Str3}), where the factor is $M^{3/22}$ instead of $M_{2}^{3/22}$, 
\begin{align*}
I_{M_{1},\cdots ,M_{5},M}& \lesssim
T^{5/22}M_{0}^{5/11}M_{4}^{3/2}M_{5}^{23/22}M^{9/11}M_{3}^{1/2}(\frac{M_{3}}{%
M_{1}}+\frac{1}{M_{3}})^{\delta }M^{3/22} \\
& \qquad \sum_{Q}\Vert P_{Q}P_{M_{1}}f_{1}\Vert _{L_{x}^{2}}\Vert
P_{M_{3}}f_{3}\Vert _{L_{x}^{2}}\Vert P_{Q_{c}}P_{M_{2}}f_{2}\Vert
_{L_{x}^{2}}\Vert P_{M_{4}}f_{4}\Vert _{L_{x}^{2}}\Vert P_{M_{5}}f_{5}\Vert
_{L_{x}^{2}}\Vert P_{M}g\Vert _{L_{t}^{\infty }L_{x}^{2}}
\end{align*}%
which rearranges to%
\begin{align*}
& \lesssim T^{5/22}M_{0}^{5/11}(\frac{M_{3}}{M_{1}}+\frac{1}{M_{3}})^{\delta
}M_{3}^{-1/2}M^{-1/22}M_{4}^{1/2}M_{5}^{1/22} \\
& \qquad \sum_{Q}\Vert P_{Q}P_{M_{1}}f_{1}\Vert _{L_{x}^{2}}\Vert
P_{Q_{c}}P_{M_{2}}f_{2}\Vert _{L_{x}^{2}}\Vert P_{M_{3}}f_{3}\Vert
_{H_{x}^{1}}\Vert P_{M_{4}}f_{4}\Vert _{H_{x}^{1}}\Vert P_{M_{5}}f_{5}\Vert
_{H_{x}^{1}}\Vert P_{M}g\Vert _{L_{t}^{\infty }H_{x}^{1}}
\end{align*}%
Carrying out the sum in $Q$ via Cauchy-Schwarz, 
\begin{align*}
& \lesssim T^{5/22}M_{0}^{5/11}(\frac{M_{3}}{M_{1}}+\frac{1}{M_{3}})^{\delta
}M_{3}^{-1/2}M^{-1/22}M_{4}^{1/2}M_{5}^{1/22} \\
& \qquad \Vert P_{M_{1}}f_{1}\Vert _{L_{x}^{2}}\Vert P_{M_{2}}f_{2}\Vert
_{L_{x}^{2}}\Vert P_{M_{3}}f_{3}\Vert _{H_{x}^{1}}\Vert P_{M_{4}}f_{4}\Vert
_{H_{x}^{1}}\Vert P_{M_{5}}f_{5}\Vert _{H_{x}^{1}}\Vert P_{M}g\Vert
_{L_{t}^{\infty }H_{x}^{1}}
\end{align*}%
Swapping an $M_{1}$ and $M_{2}$ factor (since $M_{1}\sim M_{2}$) 
\begin{align*}
& \lesssim T^{5/22}M_{0}^{5/11}(\frac{M_{3}}{M_{1}}+\frac{1}{M_{3}})^{\delta
}M_{3}^{-1/2}M^{-1/22}M_{4}^{1/2}M_{5}^{1/22} \\
& \qquad \Vert P_{M_{1}}f_{1}\Vert _{H_{x}^{-1}}\Vert P_{M_{2}}f_{2}\Vert
_{H_{x}^{1}}\Vert P_{M_{3}}f_{3}\Vert _{H_{x}^{1}}\Vert P_{M_{4}}f_{4}\Vert
_{H_{x}^{1}}\Vert P_{M_{5}}f_{5}\Vert _{H_{x}^{1}}\Vert P_{M}g\Vert
_{L_{t}^{\infty }H_{x}^{1}}
\end{align*}

Carrying out the sum in this case, we obtain (taking without loss $\delta <%
\frac{1}{22}$) 
\begin{align*}
I_{1B}& \lesssim T^{5/22}M_{0}^{5/11}\sup_{M_{3}}\Vert P_{M_{3}}f_{3}\Vert
_{H_{x}^{1}}\sup_{M_{4}}\Vert P_{M_{4}}f_{4}\Vert
_{H_{x}^{1}}\sup_{M_{5}}\Vert P_{M_{5}}f_{5}\Vert _{H_{x}^{1}}\sup_{M}\Vert
P_{M}g\Vert _{L_{t}^{\infty }H_{x}^{1}} \\
& \qquad \sum_{\substack{ M_{1}  \\ M_{1}\sim M_{2}}}\sigma (M_{1})\Vert
P_{M_{1}}f_{1}\Vert _{H_{x}^{-1}}\Vert P_{M_{2}}f_{2}\Vert _{H_{x}^{1}}
\end{align*}%
where 
\begin{eqnarray*}
\sigma (M_{1}) &=&\sum_{\substack{ M,M_{3},M_{4},M_{5}  \\ M_{1}\sim
M_{2}\geq M\geq M_{3}\geq M_{4}\geq M_{5}}}(\frac{M_{3}}{M_{1}}+\frac{1}{%
M_{3}})^{\delta }M_{3}^{-1/2}M^{-1/22}M_{4}^{1/2}M_{5}^{1/22} \\
&=&\sum_{\substack{ M,M_{3}  \\ M_{3}\leq M\leq M_{1}}}(\frac{M_{3}}{M_{1}}+%
\frac{1}{M_{3}})^{\delta }M_{3}^{1/22}M^{-1/22}=O(1)
\end{eqnarray*}%
Therefore we can complete the proof of Case 1B of the low frequency case by
Cauchy-Schwarz, summing in $M_{1}$.

\subparagraph{\noindent \textit{Case 2 of the low frequency estimate (%
\protect\ref{estimate:MultiLinearWithFreqLoc1-low})}}

Recall Case 2 in which 
\begin{equation*}
M_{1}\sim M\geq M_{2}\geq M_{3}\geq M_{4}\geq M_{5}
\end{equation*}%
and we will again write $v=P_{M_{1}}g$ for convenience. Like in Case 1A of 
\textit{(\ref{estimate:MultiLinearWithFreqLoc1-low}), }we consider%
\begin{equation*}
I_{M_{1},\cdots ,M_{4},M}=\sum_{M_{5}}I_{M_{1},\cdots ,M_{4},M_{5},M},
\end{equation*}%
replace $\sum_{M_{5},M_{4}\geq M_{5}}u_{5}=e^{it\Delta
}\sum_{M_{5},M_{4}\geq M_{5}}P_{\leq M_{0}}P_{M_{5}}f_{5}$ with $%
u_{5}=e^{it\Delta }P_{\leq M_{0}}P_{\leq M_{4}}f_{5}$ to shorten some
notations, and use the estimate 
\begin{equation*}
\Vert u_{5}\Vert _{L_{t}^{\infty }L_{x}^{\infty }}\lesssim M_{0}^{1/2}\Vert
P_{\leq M_{4}}e^{it\Delta }f_{5}\Vert _{L_{t}^{\infty }L_{x}^{6}}\lesssim
M_{0}^{1/2}\Vert P_{\leq M_{4}}f_{5}\Vert _{H^{1}}
\end{equation*}

We start with 
\begin{equation*}
I_{2}\lesssim \sum_{\substack{ M_{1},M_{2},M_{3},M_{4}  \\ M_{1}\geq
M_{2}\geq M_{3}\geq M_{4}}}\Vert (\langle \nabla \rangle
^{-1}u_{1})\,u_{2}\Vert _{L_{t,x}^{2}}\Vert u_{3}u_{4}u_{5}\langle \nabla
\rangle v\Vert _{L_{t,x}^{2}}
\end{equation*}%
By Cauchy-Schwarz in the $M_{1}$ sum 
\begin{equation*}
I_{2}\lesssim AB
\end{equation*}%
where $A$ and $B$ are, for fixed $M_{2},M_{3},M_{4}$ given by 
\begin{equation*}
A=\left( \sum_{\substack{ M_{1}  \\ M_{1}\geq M_{2}}}\Vert \langle \nabla
\rangle ^{-1}u_{1}\,u_{2}\Vert _{L_{x,t}^{2}}^{2}\right) ^{1/2}\,,\qquad
B=\left( \sum_{\substack{ M_{1}  \\ M_{1}\geq M_{2}}}\Vert
u_{3}u_{4}u_{5}\langle \nabla \rangle v\Vert _{L_{x,t}^{2}}^{2}\right) ^{1/2}
\end{equation*}%
Term $A$ is the same as in the high frequency estimate, and we have 
\begin{eqnarray*}
A_{2A} &\lesssim &M_{2}^{-\frac{1}{2}-\delta }\Vert P_{M_{2}}f_{2}\Vert
_{H^{1}}\Vert f_{1}\Vert _{H^{-1}} \\
A_{2B} &\lesssim &M_{2}^{-\frac{1}{2}+\delta }\Vert P_{M_{2}}f_{2}\Vert
_{H^{1}}\Vert P_{M_{2}\leq \bullet \leq M_{2}^{2}}f_{1}\Vert _{H^{-1-\delta
}}
\end{eqnarray*}%
For $B$, as in the high frequency estimate, it becomes 
\begin{equation*}
B\lesssim \left( \int_{t}\Vert u_{3}\Vert _{L_{x}^{\infty }}^{2}\Vert
u_{4}\Vert _{L_{x}^{\infty }}^{2}\Vert u_{5}\Vert _{L_{x}^{\infty
}}^{2}\Vert g\Vert _{H^{1}}^{2}\,dt\right) ^{1/2}
\end{equation*}%
and we H\"{o}lder in $t$ to get to 
\begin{equation*}
B\lesssim T^{1/4}\Vert u_{3}\Vert _{L_{t}^{4}L_{x}^{\infty }}\Vert
u_{4}\Vert _{L_{t}^{\infty }L_{x}^{\infty }}\Vert u_{5}\Vert _{L_{t}^{\infty
}L_{x}^{\infty }}\Vert g\Vert _{L_{t}^{\infty }H_{x}^{1}}
\end{equation*}%
By Bernstein and Strichartz (\ref{E:Str1}), and absorbing one derivative
into the $f_{3}$, $f_{4}$ and $f_{5}$ terms 
\begin{equation*}
B\lesssim T^{1/4}M_{0}^{1/2}M_{4}^{1/2}\Vert P_{M_{3}}f_{3}\Vert
_{H_{x}^{1}}\Vert P_{M_{4}}f_{4}\Vert _{H_{x}^{1}}\Vert P_{\leq
M_{4}}f_{5}\Vert _{H_{x}^{1}}\Vert g\Vert _{L_{t}^{\infty }H_{x}^{1}}
\end{equation*}%
Putting it all together (since the $M_{1}$ sum has already been carried out) 
\begin{align*}
I_{2}& \lesssim T^{1/4}M_{0}^{1/2}\sum_{\substack{ M_{2},M_{3},M_{4}  \\ %
M_{2}\geq M_{3}\geq M_{4}}}M_{2}^{-1/2}M_{4}^{1/2}\Vert P_{M_{3}}f_{3}\Vert
_{H_{x}^{1}}\Vert P_{M_{4}}f_{4}\Vert _{H_{x}^{1}}\Vert P_{\leq
M_{4}}f_{5}\Vert _{H_{x}^{1}}\Vert g\Vert _{L_{t}^{\infty }H_{x}^{1}} \\
& \qquad (M_{2}^{-\delta }\Vert f_{1}\Vert _{H^{-1}}\Vert
P_{M_{2}}f_{2}\Vert _{H^{1}}+M_{2}^{\delta }\Vert P_{M_{2}\leq \bullet \leq
M_{2}^{2}}f_{1}\Vert _{H^{-1-\delta }}\Vert P_{M_{2}}f_{2}\Vert _{H^{1}})
\end{align*}%
Suping out the $f_{3}$, $f_{4}$ and $f_{5}$ terms in $M_{3}$ and $M_{4}$, 
\begin{align*}
I_{2}& \lesssim T^{1/4}M_{0}^{1/2}\Vert f_{3}\Vert _{H_{x}^{1}}\Vert
f_{4}\Vert _{H_{x}^{1}}\Vert f_{5}\Vert _{H_{x}^{1}}\Vert g\Vert
_{L_{t}^{\infty }H_{x}^{1}} \\
& \qquad \sum_{\substack{ M_{2},M_{3},M_{4}  \\ M_{2}\geq M_{3}\geq M_{4}}}%
M_{2}^{-1/2}M_{4}^{1/2}\left( M_{2}^{-\delta }\Vert f_{1}\Vert
_{H^{-1}}\Vert P_{M_{2}}f_{2}\Vert _{H^{1}}+M_{2}^{\delta }\Vert
P_{M_{2}\leq \bullet \leq M_{2}^{2}}f_{1}\Vert _{H^{-1-\delta }}\Vert
P_{M_{2}}f_{2}\Vert _{H^{1}}\right)
\end{align*}%
Carry out the $M_{3}$ and $M_{4}$ sums 
\begin{align*}
I_{2}& \lesssim T^{1/4}M_{0}^{1/2}\Vert f_{3}\Vert _{H_{x}^{1}}\Vert
f_{4}\Vert _{H_{x}^{1}}\Vert f_{5}\Vert _{H_{x}^{1}}\Vert g\Vert
_{L_{t}^{\infty }H_{x}^{1}} \\
& \qquad \sum_{M_{2}}(M_{2}^{-\delta }\Vert f_{1}\Vert _{H^{-1}}\Vert
P_{M_{2}}f_{2}\Vert _{H^{1}}+M_{2}^{\delta }\Vert P_{M_{2}\leq \bullet \leq
M_{2}^{2}}f_{1}\Vert _{H^{-1-\delta }}\Vert P_{M_{2}}f_{2}\Vert _{H^{1}})
\end{align*}%
The two $M_{2}$ sums is the same as in high frequency case, thus we have
completed the proof of Case 2 of the low frequency case and hence \textit{(%
\ref{estimate:MultiLinearWithFreqLoc1-low})}. Together with (\ref%
{estimate:MultiLinearWithFreqLoc1-high}), we have proved (\ref%
{estimate:MultiLinearWithFreqLoc1}).

\bigskip

\paragraph{Proof of (\protect\ref{estimate:MultiLinearWithFreqLoc2})}

Estimate (\ref{estimate:MultiLinearWithFreqLoc2}) is a straightforward
corollary of the proof of (\ref{estimate:MultiLinearWithFreqLoc1}). Indeed,
consider the two cases

\noindent \emph{Case 1}. $M_{1}\sim M_{2}\geq M_{3}\geq M_{4}\geq M_{5}$ and 
$M_{1}\sim M_{2}\geq M$

\noindent \emph{Case 2}. $M_{1}\sim M\geq M_{2}\geq M_{3}\geq M_{4}\geq
M_{5} $

In Case 1, in the proof of (\ref{estimate:MultiLinearWithFreqLoc1}) we had 
\begin{equation*}
\Vert P_{M_{1}}f_{1}\Vert _{H^{-1}}\,,\qquad \Vert P_{M_{2}}f_{2}\Vert
_{H^{1}}\,,\qquad \Vert P_{M}g\Vert _{H^{1}}
\end{equation*}%
or equivalently 
\begin{equation*}
\Vert P_{M_{1}}f_{1}\Vert _{L^{2}}\,,\qquad \Vert P_{M_{2}}f_{2}\Vert
_{L^{2}}\,,\qquad M\Vert P_{M}g\Vert _{L^{2}}
\end{equation*}%
and, for (\ref{estimate:MultiLinearWithFreqLoc2}), we would now like to have 
\begin{equation*}
\Vert P_{M_{1}}f_{1}\Vert _{H^{1}}\,,\qquad \Vert P_{M_{2}}f_{2}\Vert
_{H^{1}}\,,\qquad \Vert P_{M}g\Vert _{H^{-1}}
\end{equation*}%
or equivalently 
\begin{equation*}
M_{1}\Vert P_{M_{1}}f_{1}\Vert _{L^{2}}\,,\qquad M_{2}\Vert
P_{M_{2}}f_{2}\Vert _{L^{2}}\,,\qquad M^{-1}\Vert P_{M}g\Vert _{L^{2}}
\end{equation*}%
Thus (\ref{estimate:MultiLinearWithFreqLoc2})'s Case 1 follows immediately
from the previous analysis because summands inside $\sigma (M_{1})$ got
multiplied by $M^{2}/M_{1}M_{2}\lesssim 1$.

In Case 2, the proof is identical. In (\ref{estimate:MultiLinearWithFreqLoc1}%
), we had $\Vert P_{M_{1}}f_{1}\Vert _{H^{-1}}$ and $\Vert P_{M}g\Vert
_{H^{1}}$, and now we want $\Vert P_{M_{1}}f_{1}\Vert _{H^{1}}$ and $\Vert
P_{M}g\Vert _{H^{-1}}$ for (\ref{estimate:MultiLinearWithFreqLoc2}). Since $%
M_{1}\sim M$, we can trade derivatives between these two factors. We omit
further details.

\appendix

\section{Uniform in Time Frequency Localization for NLS\label%
{Sec:AppendixUTFL}}

\begin{definition}
\label{D:uniform}Suppose that $u\in C_{[0,1]}^{0}\dot{H}_{x}^{1}$. We say
that $u$ satisfies \emph{uniform in time frequency localization} (UTFL$_{%
\text{NLS}}$) if for each $\varepsilon >0$ there exists $M(\varepsilon )$
such that 
\begin{equation}
\Vert P_{>M}\nabla u\Vert _{L_{[0,1]}^{\infty }L_{x}^{2}}\leq \varepsilon
\label{E:uniform}
\end{equation}
\end{definition}

\begin{theorem}
\label{Cor:UTFLforNLS}Any solution to the 3D quintic NLS 
\begin{equation}
i\partial _{t}u+\Delta u\pm |u|^{4}u=0
\label{eqn:3d quintic NLS in Appendix}
\end{equation}%
on $0\leq t\leq 1$ belonging to $C_{t}^{0}\dot{H}_{x}^{1}\cap \dot{C}_{t}^{1}%
\dot{H}_{x}^{-1}$ satisfies UTFL$_{\text{NLS}}$.
\end{theorem}

Notice that, Theorem \ref{Cor:UTFLforNLS} holds no matter (\ref{eqn:3d
quintic NLS in Appendix}) is focusing or defocusing and one has uniform in
time frequency localization for arbitrarily large time while Theorem \ref%
{Thm:UTFLinTime} is a small time defocusing result. This is certainly
because there is no hierarchical structure, which gets in the way in the
analysis of (\ref{hierarchy:quintic GP in differential form}), for the NLS.
On the other hand, we note that Theorem \ref{Cor:UTFLforNLS} only applies to
solutions of (\ref{eqn:3d quintic NLS in Appendix}) with bounded $\dot{H}%
^{1} $ norm on a closed interval, thus blow ups in the focusing case do not
contradict Theorem \ref{Cor:UTFLforNLS}.

\paragraph{Proof of Theorem \protect\ref{Cor:UTFLforNLS}}

For convenience, we fix the sign in (\ref{eqn:3d quintic NLS in Appendix})
and assume that $u$ solves the 3D quintic \emph{focusing} NLS 
\begin{equation*}
i\partial _{t}u+\Delta u+|u|^{4}u=0
\end{equation*}%
The same argument applies to the defocusing case. Let 
\begin{equation}
C_{1}=\max (\Vert \nabla u(t)\Vert _{L_{[0,1]}^{\infty }L_{x}^{2}},1)
\label{E:C1}
\end{equation}%
We first prove the following lemma which stems from an energy decomposition
into low and high frequency components, and a control in time of the low
frequency part.

\begin{lemma}
\label{L:local-control} Let $0<\varepsilon \ll C_{1}^{-3}$ and suppose that
for some $t_{0}\in \lbrack 0,1]$ and some $M$, 
\begin{equation*}
\Vert P_{>M}\nabla u(t_{0})\Vert _{L_{x}^{2}}\leqslant \tfrac{1}{2}%
\varepsilon
\end{equation*}%
Then there exists $\delta >0$ (which can be taken $\sim
C_{1}^{-10}M^{-2}\varepsilon ^{2}$) such that 
\begin{equation*}
\Vert P_{>M}\nabla u(t)\Vert _{L_{x}^{2}}\leqslant \varepsilon
\end{equation*}%
for all $t\in (t_{0}-\delta ,t_{0}+\delta )$.
\end{lemma}

\begin{proof}
The energy 
\begin{equation*}
E(u)=\frac{1}{2}\Vert \nabla u\Vert _{L^{2}}^{2}-\frac{1}{6}\Vert u\Vert
_{L^{6}}^{6}
\end{equation*}%
is conserved in time. Let $u_{H}=P_{>M}u$ and $u_{L}=P_{\leqslant M}u$. For
convenience we will ignore the complex conjugates for intermediate terms
since they will just be estimated in a way that does not matter whether
there are complex conjugates or not and expand the sixth power as 
\begin{equation*}
|u|^{6}=|u_{L}+u_{H}|^{6}=|u_{L}|^{6}+6u_{L}^{5}u_{H}+15u_{L}^{4}u_{H}^{2}+20u_{L}^{3}u_{H}^{3}+15u_{L}^{2}u_{H}^{4}+6u_{L}u_{H}^{5}+|u_{H}|^{6}
\end{equation*}%
Let 
\begin{equation}
E_{L}(u)=\frac{1}{2}\Vert \nabla u_{L}\Vert _{L^{2}}^{2}-\frac{1}{6}\int
|u_{L}|^{6}-\int u_{L}^{5}u_{H}-\frac{15}{6}\int u_{L}^{4}u_{H}^{2}
\label{E:Eminus}
\end{equation}%
and 
\begin{equation}
E_{H}(u)=\frac{1}{2}\Vert \nabla u_{H}\Vert _{L^{2}}^{2}-\frac{20}{6}\int
u_{L}^{3}u_{H}^{3}-\frac{15}{6}\int u_{L}^{2}u_{H}^{4}-\int u_{L}u_{H}^{5}-%
\frac{1}{6}\int |u_{H}|^{6}  \label{E:Eplus}
\end{equation}%
so that 
\begin{equation}
E(u)=E(u(t))=E_{H}(u(t))+E_{L}(u(t))  \label{E:en-decomp}
\end{equation}%
We do not expect the same for the components $E_{H}(u(t))$ and $E_{L}(u(t))$
although $E(u)$ is independent of $t$. We will prove the following bound for 
$E_{L}(u(t))$ giving short time control: 
\begin{equation}
\left\vert E_{L}(u(t))-E_{L}(u(t_{0}))\right\vert \lesssim
C_{1}^{10}M^{2}|t-t_{0}|  \label{E:en-minus-bd}
\end{equation}%
Since $E(u)$ is conserved, \eqref{E:en-decomp} implies 
\begin{equation*}
E_{H}(u(t))-E_{H}(u(t_{0}))=E_{L}(u(t_{0}))-E_{L}(u(t))
\end{equation*}%
and hence \eqref{E:en-minus-bd} implies the same bound for $E_{H}(u(t))$: 
\begin{equation}
\left\vert E_{H}(u(t))-E_{H}(u(t_{0}))\right\vert \lesssim
C_{1}^{10}M^{2}|t-t_{0}|  \label{E:en-plus-bd}
\end{equation}%
Since $\Vert \nabla u_{H}(t_{0})\Vert _{L^{2}}$ is a small quantity, this
will imply control on $\Vert \nabla u_{H}(t)\Vert _{L^{2}}$ by continuity of 
$\nabla u_{H}(t)$ in $t$ (mapping into $L_{x}^{2}$). More precisely, we
argue as follows. Let $t_{H}>t_{0}$ be the first forward time at which $%
\Vert \nabla u_{H}(t_{H})\Vert _{L_{x}^{2}}=\varepsilon $. We will obtain a
contradiction unless $t_{H}\geqslant t_{0}+\delta $, where $\delta \sim
C_{1}^{-10}M^{-2}\varepsilon ^{2}$.

Indeed, for any $t_{0}\leq t\leq t_{H}$, we have from \eqref{E:Eplus} and H%
\"{o}lder that 
\begin{equation*}
\left\vert E_{H}(u)-\frac{1}{2}\Vert \nabla u_{H}\Vert
_{L_{x}^{2}}^{2}\right\vert \lesssim \Vert u_{L}\Vert _{L_{x}^{6}}^{3}\Vert
u_{H}\Vert _{L_{x}^{6}}^{3}+\Vert u_{L}\Vert _{L_{x}^{6}}^{2}\Vert
u_{H}\Vert _{L_{x}^{6}}^{4}+\Vert u_{L}\Vert _{L_{x}^{6}}\Vert u_{H}\Vert
_{L_{x}^{6}}^{5}+\Vert u_{H}\Vert _{L_{x}^{6}}^{6}
\end{equation*}%
Using that $\Vert u_{L}\Vert _{L_{x}^{6}}\lesssim C_{1}$ by Sobolev, and $%
\Vert u_{H}\Vert _{L_{x}^{6}}\lesssim \Vert \nabla u_{H}\Vert
_{L_{x}^{2}}\leq \varepsilon $, we have 
\begin{equation}
\left\vert E_{H}(u)-\frac{1}{2}\Vert \nabla u_{H}\Vert
_{L_{x}^{2}}^{2}\right\vert \lesssim C_{1}^{3}\varepsilon
^{3}+C_{1}^{2}\varepsilon ^{4}+C_{1}\varepsilon ^{5}+\varepsilon
^{6}\leqslant \frac{1}{32}\varepsilon ^{2}  \label{E:un7}
\end{equation}%
where the last inequality holds provided $\varepsilon \ll C_{1}^{-3}$. By %
\eqref{E:un7} and \eqref{E:en-plus-bd}, if $|t_{H}-t_{0}|\ll
C_{1}^{-10}M^{-2}\varepsilon ^{2}$, then%
\begin{eqnarray*}
&&\left\vert \left\Vert \nabla u_{H}(t_{H})\right\Vert
_{L_{x}^{2}}^{2}-\left\Vert \nabla u_{H}(t_{0})\right\Vert
_{L_{x}^{2}}^{2}\right\vert \\
&\leqslant &2\left\vert E_{H}(u(t_{H}))-\tfrac{1}{2}\Vert \nabla
u_{H}(t_{H})\Vert _{L_{x}^{2}}^{2}\right\vert +2\left\vert E_{H}(u(t_{0}))-%
\tfrac{1}{2}\Vert \nabla u_{H}(t_{0})\Vert _{L_{x}^{2}}^{2}\right\vert \\
&&+2\left\vert E_{H}(u(t_{H}))-E_{H}(u(t_{0}))\right\vert \\
&\leqslant &\tfrac{1}{4}\varepsilon ^{2}
\end{eqnarray*}%
This is a contradiction with the assumption that $\Vert \nabla
u_{H}(t_{H})\Vert _{L_{x}^{2}}=\varepsilon $ and $\Vert \nabla
u_{H}(t_{0})\Vert _{L_{x}^{2}}\leqslant \frac{1}{2}\varepsilon $, which
together imply that 
\begin{equation*}
\left\vert \Vert \nabla u_{H}(t_{H})\Vert _{L_{x}^{2}}^{2}-\Vert \nabla
u_{H}(t_{0})\Vert _{L_{x}^{2}}^{2}\right\vert \geqslant \frac{3}{4}%
\varepsilon ^{2}.
\end{equation*}%
Thus it follows that the assumption $\left\vert t_{H}-t_{0}\right\vert \ll
C_{1}^{-10}M^{-2}\epsilon ^{2}$ is false, and we have $\left\vert
t_{H}-t_{0}\right\vert \gtrsim C_{1}^{-10}M^{-2}\epsilon ^{2}$. A similar
argument applies in the backward time direction $t<t_{0}$, yielding the
result.

It remains to prove the claimed bound \eqref{E:en-minus-bd} on $E_{L}(u(t))$%
. For this, it clearly suffices to prove 
\begin{equation}
\left\vert \partial _{t}E_{L}(u(t))\right\vert \lesssim C_{1}^{10}M^{2}
\label{E:en-minus-bd2}
\end{equation}%
We will compute each term of $\partial _{t}E_{L}(u(t))$ and estimate each
term using Bernstein and Sobolev. First, we have 
\begin{equation*}
\partial _{t}\int |\nabla u_{L}|^{2}=-2\func{Re}\int \Delta \bar{u}%
_{L}\,\partial _{t}u_{L}=2\func{Im}\int \Delta \bar{u}_{L}(\Delta
u_{L}+(|u|^{4}u)_{L})=2\func{Im}\int \Delta \bar{u}_{L}(|u|^{4}u)_{L}
\end{equation*}%
Hence 
\begin{equation}
\left\vert \partial _{t}\int |\nabla u_{L}|^{2}\right\vert \leqslant 2\Vert
\Delta u_{L}\Vert _{L_{x}^{6}}\Vert (|u|^{4}u)_{L}\Vert
_{L_{x}^{6/5}}\lesssim M^{2}\Vert u\Vert _{L_{x}^{6}}^{6}\lesssim
M^{2}C_{1}^{6}  \label{E:un6}
\end{equation}%
Next, we address the power $6$ terms, ignoring real/imaginary parts, complex
conjugates, and $\pm $ signs. We are justified in doing so, since in the
end, every term is just estimated by absolute value.%
\begin{eqnarray}
\partial _{t}\int u_{L}^{4}u_{H}^{2} &=&4\int u_{L}^{3}\left( \partial
_{t}u_{L}\right) u_{H}^{2}+2\int u_{L}^{4}u_{H}\left( \partial
_{t}u_{H}\right)  \label{E:un1} \\
&=&4\int u_{L}^{3}\left( \Delta u_{L}\right) u_{H}^{2}+4\int u_{L}^{3}\left(
u^{5}\right) _{L}u_{H}^{2}  \notag \\
&&+2\int u_{L}^{4}u_{H}\left( u^{5}\right) _{H}+2\int u_{L}^{4}u_{H}\left(
\Delta u_{H}\right)  \notag
\end{eqnarray}%
The first three terms are estimated directly: 
\begin{eqnarray}
\left\vert \int u_{L}^{3}\Delta u_{L}u_{H}^{2}\right\vert &\leqslant &\Vert
u_{L}\Vert _{L_{x}^{6}}^{3}\Vert \Delta u_{L}\Vert _{L_{x}^{6}}\Vert
u_{H}\Vert _{L_{x}^{6}}^{2}  \label{E:un2} \\
&\leqslant &\Vert u\Vert _{L_{x}^{6}}^{3}M^{2}\Vert u_{L}\Vert
_{L_{x}^{6}}\Vert u\Vert _{L_{x}^{6}}^{2}  \notag \\
&\lesssim &M^{2}C_{1}^{6}  \notag
\end{eqnarray}%
The second term: 
\begin{eqnarray*}
\left\vert \int u_{L}^{3}(u^{5})_{L}u_{H}^{2}\right\vert &\leqslant &\Vert
u_{L}\Vert _{L_{x}^{\infty }}^{3}\Vert (u^{5})_{L}\Vert _{L_{x}^{3/2}}\Vert
u_{H}\Vert _{L_{x}^{6}}^{2} \\
&\leqslant &\Vert u_{L}\Vert _{L_{x}^{\infty }}^{3}\Vert (u^{5})_{L}\Vert
_{L_{x}^{3/2}}\Vert u\Vert _{L_{x}^{6}}^{2} \\
&\lesssim &\Vert u_{L}\Vert _{L_{x}^{\infty }}^{3}\Vert (u^{5})_{L}\Vert
_{L_{x}^{3/2}}C_{1}^{2}
\end{eqnarray*}%
By Bernstein, $\Vert u_{L}\Vert _{L_{x}^{\infty }}\lesssim M^{1/2}C_{1}$.
Also, $\Vert (u^{5})_{L}\Vert _{L_{x}^{3/2}}\lesssim M^{1/2}\Vert u^{5}\Vert
_{L_{x}^{6/5}}\lesssim M^{1/2}C_{1}^{5}$. That is, 
\begin{equation}
\left\vert \int u_{L}^{3}(u^{5})_{L}u_{H}^{2}\right\vert \lesssim
M^{2}C_{1}^{10}  \label{E:un3}
\end{equation}%
The third term 
\begin{eqnarray}
\left\vert \int u_{L}^{4}u_{H}(u^{5})_{H}\right\vert &\lesssim &\Vert
u_{L}\Vert _{L_{x}^{\infty }}^{4}\Vert u_{H}\Vert _{L_{x}^{6}}\Vert
(u^{5})_{H}\Vert _{L_{x}^{6/5}}  \label{E:un4} \\
&\lesssim &\left( M^{1/2}C_{1}\right) ^{4}\Vert u\Vert _{L_{x}^{6}}\Vert
u^{5}\Vert _{L_{x}^{6/5}}  \notag \\
&\lesssim &\left( M^{1/2}C_{1}\right) ^{4}C_{1}^{6}  \notag
\end{eqnarray}%
where by Bernstein, $\Vert u_{L}\Vert _{L_{x}^{\infty }}\lesssim
M^{1/2}C_{1} $. For the fourth term, we must first apply integration by
parts: 
\begin{equation*}
\int u_{L}^{4}u_{H}\Delta u_{H}=-4\int u_{L}^{3}\nabla u_{L}u_{H}\nabla
u_{H}-\int u_{L}^{4}\nabla u_{H}\nabla u_{H}
\end{equation*}%
By H\"{o}lder, 
\begin{equation*}
\left\vert \int u_{L}^{4}u_{H}\Delta u_{H}\right\vert \leqslant 4\Vert
u_{L}\Vert _{L_{x}^{\infty }}^{3}\Vert \nabla u_{L}\Vert _{L_{x}^{3}}\Vert
u_{H}\Vert _{L_{x}^{6}}\Vert \nabla u_{H}\Vert _{L_{x}^{2}}+\Vert u_{L}\Vert
_{L_{x}^{\infty }}^{4}\Vert \nabla u_{H}\Vert _{L_{x}^{2}}^{2}
\end{equation*}%
Again, by Bernstein, $\Vert u_{L}\Vert _{L_{x}^{\infty }}\lesssim
C_{1}M^{1/2}$ and by Sobolev, $\Vert \nabla u_{L}\Vert _{L_{x}^{3}}\lesssim
M^{1/2}\Vert \nabla u_{L}\Vert _{L_{x}^{2}}\lesssim C_{1}M^{1/2}$. Hence 
\begin{equation}
\left\vert \int u_{L}^{4}u_{H}\Delta u_{H}\right\vert \lesssim M^{2}C_{1}^{6}
\label{E:un5}
\end{equation}%
Using \eqref{E:un2}, \eqref{E:un3}, \eqref{E:un4}, \eqref{E:un5} as
estimates for the terms in the right side of \eqref{E:un1}, we obtain 
\begin{equation}
\left\vert \partial _{t}\int u_{L}^{4}u_{H}^{2}\right\vert \lesssim
M^{2}C_{1}^{10}  \label{E:un8}
\end{equation}%
Similar methods apply to the other two (easier) terms $\int |u_{L}|^{6}$ and 
$\int u_{L}^{5}u_{H}$ in \eqref{E:Eminus}. Collecting \eqref{E:un6}, %
\eqref{E:un8}, and the corresponding estimates for $\int |u_{L}|^{6}$ and $%
\int u_{L}^{5}u_{H}$, we obtain \eqref{E:en-minus-bd2}.
\end{proof}

We complete the proof of Theorem \ref{Cor:UTFLforNLS} with the following
compactness argument. Fix $0<\varepsilon \ll 1$. For each $t_{0}\in \lbrack
0,1]$, there exists $M(t_{0})$ such that 
\begin{equation*}
\Vert P_{>M(t_{0})}\nabla u(t_{0})\Vert _{L_{x}^{2}}\leq \frac{1}{2}%
\varepsilon
\end{equation*}%
By Lemma \ref{L:local-control}, there exists $\delta >0$ (depending on $%
M(t_{0})$, $C_{1}$, and $\varepsilon $) such that if we set $%
I_{t_{0}}=(t_{0}-\delta ,t_{0}+\delta )$, then 
\begin{equation*}
\Vert P_{>M(t_{0})}\nabla u(t)\Vert _{L_{I_{t_{0}}}^{\infty }L_{x}^{2}}\leq
\varepsilon
\end{equation*}%
The collection of intervals $\{I_{t}\}_{t\in \lbrack 0,1]}$ covers $[0,1]$.
Pass to a finite subcover $I_{t_{1}},\ldots ,I_{t_{k}}$, and then let 
\begin{equation*}
M=\max (M(t_{1}),\ldots ,M(t_{k}))
\end{equation*}%
Then clearly 
\begin{equation*}
\Vert P_{>M}\nabla u(t)\Vert _{L_{[0,1]}^{\infty }L_{x}^{2}}\leq \varepsilon
\end{equation*}%
since this holds on each of the subintervals $I_{t_{j}}$, $1\leq j\leq k$.

\section{An Example of the Bernstein Inequalities on $\mathbb{T}^{d}$\label%
{Sec:AppendixBernstein}}

Take $M\geq 1$. For $\xi \in \mathbb{Z}^{d}$, let $\mathbf{1}_{\leq M}(\xi )$
denote the characteristic function that projects to tuples $\xi =(\xi
_{1},\ldots \xi _{d})$ so that $|\xi _{j}|\leq M$ for each $1\leq j\leq d$.
We define 
\begin{equation*}
P_{\leq M}f(x)=\sum_{\xi \in \mathbb{Z}}e^{ix\cdot \xi }\mathbf{1}_{\leq
M}(\xi )\hat{f}(\xi )
\end{equation*}%
so that 
\begin{equation*}
P_{\leq M}f(x)=K_{M}\ast f(x)
\end{equation*}%
where 
\begin{equation*}
K_{M}(x)=\sum_{\xi \in \mathbb{Z}^{d}}\mathbf{1}_{\leq M}(\xi )e^{ix\cdot
\xi }=\prod_{j=1}^{d}\sum_{|\xi _{j}|\leq M}e^{ix_{j}\xi
_{j}}=\prod_{j=1}^{d}K_{M}(x_{j})
\end{equation*}%
where on the right, each $K_{M}(x_{j})$ is the $d=1$ version of the kernel.
By summing a geometric series, we compute 
\begin{equation*}
K_{M}(x_{j})=\frac{\sin (M+1)x_{j}}{\sin x_{j}}
\end{equation*}%
$K_{M}(x_{j})$ is, of course, $2\pi $ periodic, but due to the fact that $%
K_{M}(x_{j}\pm \pi )=\pm K_{M}(x_{j})$, with signs depending on the parity
of $M$, we have that $|K_{M}(x_{j})|$ is in fact $\pi $-periodic. By
separately considering $|x_{j}|\leq \frac{1}{M}$ (and there using the
approximations $\sin (M+1)x_{j}\approx (M+1)x_{j}$ and $\sin x_{j}\approx
x_{j}$) and the region $\frac{1}{M}\leq |x_{j}|\leq \frac{\pi }{2}$ (there
just using the crude approximation for the denominator $|\sin x_{j}|\geq 
\frac{1}{2}|x_{j}|$), we obtain that for $|x_{j}|\leq \frac{\pi }{2}$, we
have the pointwise bound 
\begin{equation*}
|K_{M}(x_{j})|\leq 
\begin{cases}
\frac{1}{|x_{j}|} & \text{if }\frac{1}{M}\leq |x_{j}|\leq \frac{\pi }{2} \\ 
M & \text{if }|x_{j}|\leq \frac{1}{M}%
\end{cases}%
\end{equation*}%
From this we obtain that for $1<q\leq \infty $, 
\begin{equation*}
\Vert K_{M}(x)\Vert _{L^{q}(\mathbb{T}^{d})}=\prod_{j=1}^{d}\Vert
K_{M}(x_{j})\Vert _{L^{q}(\mathbb{T})}\lesssim M^{d(1-\frac{1}{q})}
\end{equation*}%
Now consider $\frac{1}{r}=\frac{1}{p}-\frac{s}{d}$ with $1\leq p<r\leq
\infty $, $s>0$. The bound 
\begin{equation*}
\Vert P_{\leq M}f\Vert _{L^{r}(\mathbb{T}^{d})}\lesssim M^{s}\Vert f\Vert
_{L^{p}(\mathbb{T}^{d})}
\end{equation*}%
follows from Young's convolution inequality ($q>1$ since $p<r$). The case of 
$r=p$ (in which $q=1$) requires a separate argument, since $\Vert K_{M}\Vert
_{L^{1}}\approx \log M$. In this case, write 
\begin{equation*}
K_{M}(x_{j})(x_{j}-y_{j})=e^{ix_{j}M}\frac{1}{2i\sin (x_{j}-y_{j})}%
e^{-iy_{j}M}-e^{-ix_{j}M}\frac{1}{2i\sin (x_{j}-y_{j})}e^{iy_{j}M}
\end{equation*}%
so that, with $\kappa _{j}$ the operator of convolution in the $j$-th
component with $\func{pv}\frac{1}{2i\sin x_{j}}$ and $e_{M,j}^{\pm }$ the
operator of multiplication by $e^{\pm ix_{j}M}$, 
\begin{equation*}
P_{\leq M}=\prod_{j=1}^{M}(e_{j}^{+}\kappa _{j}e_{j}^{-}-e_{j}^{-}\kappa
_{j}e_{j}^{+})
\end{equation*}%
Then 
\begin{equation*}
\Vert P_{\leq M}f\Vert _{L^{r}}\lesssim \Vert f\Vert _{L^{r}}
\end{equation*}%
follows for $1<r<\infty $ by the boundedness of the Hilbert transform. 
\footnote{%
For a measure space $X$, if $T_{j}$ are each one-dimensional operators
acting in the $j$-th component and are bounded as one-dimensional operators $%
T_{j}:L^{r}(X)\rightarrow L^{r}(X)$, then $T=T_{1}\cdots T_{d}$ is bounded
as a $d$-dimensional operator $T:L^{r}(X^{d})\rightarrow L^{r}(X^{d})$. This
follows readily by Fubini's theorem. In fact, one can generalize this to a
statement for $T_{j}:L^{p}(X)\rightarrow L^{r}(X)$ for $r\geq p$ by using
Minkowski's integral inequality.}

For $M\geq 2$, the projection $P_{M}$ is defined as 
\begin{equation*}
P_{M}=P_{\leq M}-P_{\leq M/2}
\end{equation*}%
and thus the same estimates apply.

\end{document}